\documentclass[a4paper, reqno, 10pt]{amsart}

\usepackage[usenames,dvipsnames]{color}

\usepackage{amsthm,amsfonts,amssymb,amsmath,amsxtra,amsrefs}
\usepackage[all]{xy}
\SelectTips{cm}{}
\usepackage{xr-hyper}
\usepackage[colorlinks=
   citecolor=Black,
   linkcolor=Red,
   urlcolor=Blue]{hyperref}
\usepackage{verbatim}

\usepackage[margin=1.25in]{geometry}

\usepackage{mathrsfs}

\RequirePackage{xspace}
\RequirePackage{etoolbox}
\RequirePackage{varwidth}
\RequirePackage{enumitem}
\RequirePackage{tensor}
\RequirePackage{mathtools}
\RequirePackage{longtable}
\RequirePackage{multirow}

\usepackage[refpage]{nomencl}
\usepackage{ifthen}
\usepackage{lipsum}

\makeindex

\makenomenclature

\newcommand{\sC}{\ensuremath{\mathscr{C}}\xspace}

\newcommand{\sH}{\ensuremath{\mathscr{H}}\xspace}

\newcommand{\sL}{\ensuremath{\mathscr{L}}\xspace}

\newcommand{\sP}{\ensuremath{\mathscr{P}}\xspace}

\newcommand{\sZ}{\ensuremath{\mathscr{Z}}\xspace}

\newcommand{\fkm}{\ensuremath{\mathfrak{m}}\xspace}

\newcommand{\heart}{{\heartsuit}}

\newcommand{\nat}{{\natural}}

\newcommand{\diam}{{\Diamond}}

\newcommand{\BA}{\ensuremath{\mathbb {A}}\xspace}
\newcommand{\BB}{\ensuremath{\mathbb {B}}\xspace}
\newcommand{\BC}{\ensuremath{\mathbb {C}}\xspace}

\newcommand{\BF}{\ensuremath{\mathbb {F}}\xspace}
\newcommand{\BG}{\ensuremath{\mathbb {G}}\xspace}

\newcommand{\BI}{\ensuremath{\mathbb {I}}\xspace}
\newcommand{\BJ}{\ensuremath{\mathbb {J}}\xspace}
\newcommand{\BK}{\ensuremath{\mathbb {K}}\xspace}

\newcommand{\BO}{\ensuremath{\mathbb {O}}\xspace}
\newcommand{\BP}{\ensuremath{\mathbb {P}}\xspace}
\newcommand{\BQ}{\ensuremath{\mathbb {Q}}\xspace}
\newcommand{\BR}{\ensuremath{\mathbb {R}}\xspace}

\newcommand{\BZ}{\ensuremath{\mathbb {Z}}\xspace}

\newcommand{\CA}{\ensuremath{\mathcal {A}}\xspace}

\newcommand{\CC}{\ensuremath{\mathcal {C}}\xspace}

\newcommand{\CE}{\ensuremath{\mathcal {E}}\xspace}
\newcommand{\CF}{\ensuremath{\mathcal {F}}\xspace}
\newcommand{\CG}{\ensuremath{\mathcal {G}}\xspace}

\newcommand{\CI}{\ensuremath{\mathcal {I}}\xspace}

\newcommand{\CK}{\ensuremath{\mathcal {K}}\xspace}
\newcommand{\CL}{\ensuremath{\mathcal {L}}\xspace}
\newcommand{\CM}{\ensuremath{\mathcal {M}}\xspace}
\newcommand{\CN}{\ensuremath{\mathcal {N}}\xspace}
\newcommand{\CO}{\ensuremath{\mathcal {O}}\xspace}

\newcommand{\CQ}{\ensuremath{\mathcal {Q}}\xspace}

\newcommand{\CY}{\ensuremath{\mathcal {Y}}\xspace}

\newcommand{\RR}{\ensuremath{\mathrm {R}}\xspace}

\newcommand{\ab}{{\mathrm{ab}}}
\newcommand{\Ad}{{\mathrm{Ad}}}

\DeclareMathOperator{\Aut}{Aut}
\DeclareMathOperator{\aut}{aut}

\newcommand{\Ch}{{\mathrm{Ch}}}

\DeclareMathOperator{\coker}{coker}

\newcommand{\cl}{{\mathrm{cl}}}

\newcommand{\codim}{{\mathrm {codim}}}

\DeclareMathOperator{\cusp}{cusp}

\DeclareMathOperator{\der}{der}

\newcommand{\Div}{{\mathrm{Div}}}
\renewcommand{\div}{{\mathrm{div}}}

\DeclareMathOperator{\Eis}{Eis}
\DeclareMathOperator{\End}{End}
\DeclareMathOperator{\Ext}{Ext}

\DeclareMathOperator{\Frob}{Frob}

\DeclareMathOperator{\Fr}{Fr}

\DeclareMathOperator{\Gal}{Gal}
\newcommand{\GL}{\mathrm{GL}}

\DeclareMathOperator{\Hom}{Hom}

\newcommand{\id}{\ensuremath{\mathrm{id}}\xspace}

\newcommand{\Ind}{{\mathrm{Ind}}}

\newcommand{\inv}{{\mathrm{inv}}}

\DeclareMathOperator{\Jac}{Jac}

\DeclareMathOperator{\Ker}{Ker}

\newcommand{\naive}{\ensuremath{\mathrm{naive}}\xspace}

\DeclareMathOperator{\Nm}{Nm}

\DeclareMathOperator{\Quot}{Quot}

\newcommand{\PGL}{{\mathrm{PGL}}}
\DeclareMathOperator{\Pic}{Pic}

\renewcommand{\Re}{{\mathrm{Re}}}
\newcommand{\red}{\ensuremath{\mathrm{red}}\xspace}

\DeclareMathOperator{\Res}{Res}

\newcommand{\rs}{\ensuremath{\mathrm{rs}}\xspace}

\newcommand{\Sat}{{\mathrm{Sat}}}

\newcommand{\SL}{{\mathrm{SL}}}
\DeclareMathOperator{\Spec}{Spec}

\newcommand{\Sp}{{\mathrm{Sp}}}
\newcommand{\Span}{{\mathrm{Span}}}

\DeclareMathOperator{\Supp}{Supp}

\DeclareMathOperator{\sgn}{sgn}

\DeclareMathOperator\Tor{Tor}
\DeclareMathOperator{\tr}{tr}
\DeclareMathOperator{\Tr}{Tr}

\DeclareMathOperator{\vol}{vol}

\newcommand{\Bun}{{\mathrm{Bun}}}

\newcommand{\Sht}{{\mathrm{Sht}}}

\newcommand{\Hk}{{\mathrm{Hk}}}

\newcommand{\Gr}{\mathrm{Gr}}

\newcommand{\Mat}{\mathrm{Mat}}

\newcommand{\val}{\mathrm{val}}

\newcommand{\Coh}{\mathrm{Coh}}

\newcommand{\wt}{\widetilde}
\newcommand{\wh}{\widehat}

\newcommand{\pair}[1]{\langle {#1} \rangle}

\newcommand{\ds}{\displaystyle}
\newcommand{\ov}{\overline}
\newcommand{\ul}{\underline}
\newcommand{\bu}{\bullet}

\newcommand{\incl}{\hookrightarrow}
\newcommand{\lra}{\longrightarrow}
\newcommand{\imp}{\Longrightarrow}

\newcommand{\bs}{\backslash}

\newcommand{\ep}{\varepsilon}

\def\AA{\mathbb{A}}
\def\BB{\mathbb{B}}
\def\CC{\mathbb{C}}
\def\DD{\mathbb{D}}

\def\FF{\mathbb{F}}
\def\GG{\mathbb{G}}

\def\II{\mathbb{I}}
\def\JJ{\mathbb{J}}

\def\OO{\mathbb{O}}
\def\PP{\mathbb{P}}
\def\QQ{\mathbb{Q}}
\def\RR{\mathbb{R}}

\def\ZZ{\mathbb{Z}}

\def\calA{\mathcal{A}}
\def\calB{\mathcal{B}}
\def\calC{\mathcal{C}}
\def\calD{\mathcal{D}}
\def\calE{\mathcal{E}}
\def\calF{\mathcal{F}}
\def\calG{\mathcal{G}}
\def\calH{\mathcal{H}}
\def\calI{\mathcal{I}}
\def\calJ{\mathcal{J}}
\def\calK{\mathcal{K}}
\def\calL{\mathcal{L}}
\def\calM{\mathcal{M}}
\def\calN{\mathcal{N}}
\def\calO{\mathcal{O}}

\def\calU{\mathcal{U}}

\def\calX{\mathcal{X}}

\def\bR{\mathbf{R}}
\def\bL{\mathbf{L}}

\newcommand\frN{\mathfrak{N}}

\newcommand\frX{\mathfrak{X}}

\newcommand\frZ{\mathfrak{Z}}

\newcommand\tilA{\widetilde{A}}

\newcommand\tilW{\widetilde{W}}

\newcommand\tilZ{\widetilde{Z}}

\newcommand\tilh{\widetilde{h}}

\newcommand\tilq{\widetilde{q}}

\newcommand{\isom}{\stackrel{\sim}{\to}}
\newcommand{\surj}{\twoheadrightarrow}
\newcommand{\bij}{\leftrightarrow}
\renewcommand{\c}{\circ}
\renewcommand{\l}{\lambda}
\renewcommand{\L}{\Lambda}
\newcommand{\om}{\omega}
\newcommand{\Om}{\Omega}

\newcommand{\leftexp}[2]{{\vphantom{#2}}^{#1}{#2}}

\newcommand{\Ql}{\QQ_{\ell}}
\newcommand{\Qlbar}{\overline{\QQ}_\ell}
\newcommand{\kbar}{\overline{k}}
\newcommand{\Qbar}{\overline{\QQ}}
\newcommand{\const}[1]{\QQ_{\ell,#1}}
\newcommand{\twtimes}[1]{\stackrel{#1}{\times}}
\newcommand{\Ltimes}{\stackrel{\bL}{\otimes}}
\newcommand{\jiao}[1]{\langle{#1}\rangle}
\newcommand\un{\underline}
\newcommand\one{\mathbf{1}}

\newcommand{\homog}[2]{\textup{H}_{#1}({#2})}  
\newcommand{\cohog}[2]{\textup{H}^{#1}({#2})}     
\newcommand{\cohoc}[2]{\textup{H}_{c}^{#1}({#2})}     
\newcommand{\hBM}[2]{\textup{H}^{\textup{BM}}_{#1}({#2})}  

\newcommand{\oll}[1]{\overleftarrow{#1}}
\newcommand{\orr}[1]{\overrightarrow{#1}}
\newcommand{\olr}[1]{\overleftrightarrow{#1}}

\newcommand\tcM{\widetilde{\calM}}
\newcommand\tcN{\widetilde{\calN}}
\newcommand\tcA{\widetilde{\calA}}

\renewcommand\div{\textup{div}}
\newcommand\AJ{\textup{AJ}}
\newcommand\add{\textup{add}}
\newcommand{\Gm}{\GG_m}
\newcommand{\Ga}{\GG_a}
\newcommand\hX{\wh{X}}
\newcommand\hM{\wh{\calM}}
\newcommand\pr{\textup{pr}}
\newcommand\Prym{\textup{Prym}}
\newcommand\supp{\textup{supp}}

\newcommand\Fix{\textup{Fix}}

\newcommand\inst{\textup{inst}}
\newcommand\hs{\heartsuit}
\renewcommand\ds{\diamondsuit}
\newcommand\Ah{\calA^{\heartsuit}}
\newcommand\Mh{\calM^{\heartsuit}}

\newcommand\Ads{\calA^{\diamondsuit}}
\newcommand\Mds{\calM^{\diamondsuit}}

\newcommand{\bsH}{\overline{\sH_{\ell}}}
\newcommand\tsH{\widetilde{\sH_{\ell}}}
\newcommand{\gen}{\overline{\eta}}
\renewcommand\Mc{M^{\circ}}

\newcommand\tl{\widetilde{\ell}}
\newcommand\tf{\widetilde{f}}
\newcommand\tD{\widetilde{D}}
\newcommand\tB{\widetilde{B}}
\newcommand\tfZ{\widetilde{\mathfrak{Z}}}

\newcommand\cCh{{\vphantom{\Ch}}_{c}{\Ch}}

\newcommand{\sslash}{\mathbin{/\mkern-6mu/}}

\newtheorem{theorem}{Theorem}
\newtheorem{prop}[theorem]{Proposition}
\newtheorem{lem}[theorem]{Lemma}
\newtheorem{lemma}[theorem]{Lemma}

\newtheorem{cor}[theorem]{Corollary}
\newtheorem{thm}[theorem]{Theorem}

\theoremstyle{definition}
\newtheorem{defn}[theorem]{Definition}

\newtheorem{remark}[theorem]{Remark}

\newcommand{\matrixx}[4]
{\left[ \begin{array}{cc}
  #1 &  #2  \\
  #3 &  #4  \\
 \end{array}\right]}

\numberwithin{equation}{section}
\numberwithin{theorem}{section}

\setitemize[0]{leftmargin=*,itemsep=\the\smallskipamount}
\setenumerate[0]{leftmargin=*,itemsep=\the\smallskipamount}

\renewcommand{\to}{%
   \ifbool{@display}{\longrightarrow}{\rightarrow}%
   }
\let\shortmapsto\mapsto
\renewcommand{\mapsto}{%
   \ifbool{@display}{\longmapsto}{\shortmapsto}%
   }
\newlength{\olen}
\newlength{\ulen}
\newlength{\xlen}
\newcommand{\xra}[2][]{%
   \ifbool{@display}%
      {\settowidth{\olen}{$\overset{#2}{\longrightarrow}$}%
       \settowidth{\ulen}{$\underset{#1}{\longrightarrow}$}%
       \settowidth{\xlen}{$\xrightarrow[#1]{#2}$}%
       \ifdimgreater{\olen}{\xlen}%
          {\underset{#1}{\overset{#2}{\longrightarrow}}}%
          {\ifdimgreater{\ulen}{\xlen}%
             {\underset{#1}{\overset{#2}{\longrightarrow}}}
             {\xrightarrow[#1]{#2}}}}%
      {\xrightarrow[#1]{#2}}
   }
\makeatother
\newcommand{\xyra}[2][]{%
   \settowidth{\xlen}{$\xrightarrow[#1]{#2}$}%
   \ifbool{@display}%
      {\settowidth{\olen}{$\overset{#2}{\longrightarrow}$}%
       \settowidth{\ulen}{$\underset{#1}{\longrightarrow}$}%
       \ifdimgreater{\olen}{\xlen}%
          {\mathrel{\xymatrix@M=.12ex@C=3.2ex{\ar[r]^-{#2}_-{#1} &}}}%
          {\ifdimgreater{\ulen}{\xlen}%
             {\mathrel{\xymatrix@M=.12ex@C=3.2ex{\ar[r]^-{#2}_-{#1} &}}}
             {\mathrel{\xymatrix@M=.12ex@C=\the\xlen{\ar[r]^-{#2}_-{#1} &}}}}}%
      {\mathrel{\xymatrix@M=.12ex@C=\the\xlen{\ar[r]^-{#2}_-{#1} &}}}%
   }
\makeatletter
\newcommand{\xla}[2][]{%
   \ifbool{@display}%
      {\settowidth{\olen}{$\overset{#2}{\longleftarrow}$}%
       \settowidth{\ulen}{$\underset{#1}{\longleftarrow}$}%
       \settowidth{\xlen}{$\xleftarrow[#1]{#2}$}%
       \ifdimgreater{\olen}{\xlen}%
          {\underset{#1}{\overset{#2}{\longleftarrow}}}%
          {\ifdimgreater{\ulen}{\xlen}%
             {\underset{#1}{\overset{#2}{\longleftarrow}}}
             {\xleftarrow[#1]{#2}}}}%
      {\xleftarrow[#1]{#2}}
   }
\newcommand{\isoarrow}{%
   \ifbool{@display}{\overset{\sim}{\longrightarrow}}{\xrightarrow\sim}%
   }
\renewcommand{\lra}{%
   \ifbool{@display}{\longleftrightarrow}{\leftrightarrow}%
   }

\begin{document}

\thanks{Research of Z.Yun partially supported by the Packard Foundation and the NSF grant DMS $\#1302071$. Research of W. Zhang partially supported by the NSF grant DMS $\#$1301848 and $\#$1601144,  and a Sloan research fellowship. }

\title[taylor expansion]{Shtukas and the Taylor expansion of $L$-functions} 
\author{Zhiwei Yun}
\address{Zhiwei Yun: Department of Mathematics, Yale University, 10 Hillhouse Avenue, New Haven, CT 06511}
\email{zhiwei.yun@yale.edu}
\author{Wei Zhang}
\address{Wei Zhang: Department of Mathematics, Columbia University, MC 4423, 2990 Broadway, New York, NY 10027}
\email{wzhang@math.columbia.edu}


\date{\today}

\begin{abstract}
We define the Heegner--Drinfeld cycle on the moduli stack of Drinfeld Shtukas of rank two with $r$-modifications for an even integer $r$. We prove an identity between
\begin{enumerate}
\item The $r$-th central derivative of the quadratic base change
$L$-function associated to an everywhere unramified cuspidal
automorphic representation $\pi$ of $\PGL_{2}$;
\item The self-intersection number of the $\pi$-isotypic component of the Heegner--Drinfeld cycle.
\end{enumerate}
This identity can be viewed as a function-field analog of the
Waldspurger and Gross--Zagier formula for higher derivatives of
$L$-functions.
\end{abstract}

\maketitle

\tableofcontents

\section{Introduction}

In this paper we prove a formula for the arbitrary order central derivative of a certain class of $L$-functions over a {\em function field} $F=k(X)$, for a curve $X$ over a finite field $k$ of characteristic $p>2$. The $L$-function under consideration is associated to a cuspidal automorphic representation of  $\PGL_{2,F}$, or rather, its base change to a quadratic field extension of $F$. The $r$-th central derivative of our $L$-function is expressed in terms of the intersection number of the ``Heegner--Drinfeld cycle" on a moduli stack denoted by $\Sht^{r}_{G}$ in the introduction, where $G=\PGL_2$. The moduli stack $\Sht^{r}_{G}$ is closely related to the moduli stack of Drinfeld Shtukas of rank two with $r$-modifications. One important feature of this stack is that it admits a natural fibration over the $r$-fold self-product $X^r$ of the curve $X$ over $\Spec k$
$$
\xymatrix{ \Sht^r_{G}\ar[r] &X^r}.
$$ 
The very existence of such moduli stacks presents a striking difference between a function field and a number field. In the number field case, the analogous spaces only exist (at least for the time being)  when $r\leq 1$. 
When $r=0$, the moduli stack $\Sht^{0}_{G}$ is the constant groupoid over $k$
\begin{equation}
\label{eqn Bun(k)}
\Bun_G(k)\simeq G(F)\bs \left(G(\BA)/K\right),
\end{equation}
where $\BA$ is the ring of ad\`eles, and $K$ a maximal compact open subgroup of $G(\BA)$. 
The double coset in the RHS of \eqref{eqn Bun(k)} remains meaningful for a number field $F$ (except that one cannot demand the archimedean component of $K$ to be open). When $r=1$ the analogous space in the case $F=\BQ$ is the moduli stack of elliptic curves, which lives over $\Spec\BZ$. From such perspectives, our formula can be viewed as a simultaneous generalization (for function fields) of the Waldspurger formula \cite{W} (in the case of $r=0$) and the Gross--Zagier formula \cite{GZ} (in the case of $r=1$).

Another noteworthy feature of our work is that we need not restrict ourselves to the leading coefficient in the Taylor expansion of the $L$-functions: our formula is about the $r$-th Taylor coefficient of the $L$-function regardless whether $r$ is the central vanishing order or not. This leads us to speculate that, contrary to the usual belief, central derivatives of arbitrary order of motivic $L$-functions (for instance, those associated to elliptic curves) should bear some geometric meaning in the number field case. However, due to the lack of the analog of $\Sht^{r}_{G}$ for arbitrary $r$ in the number field case, we could not formulate a precise conjecture.

Finally we note that, in the current paper, we restrict ourselves to everywhere unramified cuspidal automorphic representations. One consequence is that we only need to consider the even $r$ case. Ramifications, particularly the odd $r$ case,  will be considered in subsequent work.

Now we give more details of our main theorems.

\subsection{Some notation} Throughout the paper,
let $k=\BF_q$ be a finite field of characteristic $p>2$. \index{$k=\BF_q$}%
Let $X$ be a geometrically connected smooth proper curve over $k$. Let $\nu: X'\to X$ be a finite \'etale cover of degree $2$ such that $X'$ is also geometrically connected. \index{$X, X'$}%
\index{$\nu:X'\to X$}%
Let $\sigma\in\Gal(X'/X)$ be the nontrivial involution. 
\index{$\sigma:X'\to X'$}%
Let $F=k(X)$ and $F'=k(X')$ be their function fields. 
\index{$F, F'$}%
Let $g$ and $g'$ be the genera of $X$ and $X'$, then $g'=2g-1$.
\index{$g, g'$}%

We denote the set of closed points (places) of $X$ by $|X|$. 
\nomenclature{$|X|$}{set of closed points (places) of $X$}%
For $x\in |X|$, let $\calO_{x}$ be the completed local ring of $X$ at $x$ and let $F_{x}$ be its fraction field. \index{$\calO_{x},F_{x}, k_{x}$}%
Let $\BA=\prod'_{x\in|X|}F_x$ be the ring of ad\`eles, \index{$\BA$}%
and $\BO=\prod_{x\in|X|}\calO_{x}$ the ring of integers inside $\BA$.  
\index{$\BO$}%
Similar notation applies to $X'$. Let 
$$\xymatrix{\eta_{F'/F}: F^{\times}\bs\BA^{\times}/\BO^{\times}\ar[r]&\{\pm1\}}$$\index{$\eta_{F'/F}$}%
 be the character corresponding to the \'etale double cover $X'$ via class field theory.

Let $G=\PGL_2$. \index{$G=\PGL_{2}$}%
Let $K=\prod_{x\in|X|} K_x$ where $K_x= G(\calO_{x})$. \index{$K_{x},  K$}%
The (spherical) Hecke algebra $\sH$ is the $\QQ$-algebra of bi-$K$-invariant functions $C_c^\infty(G(\BA)/\!\!/K,\BQ)$ with the product given by convolution.\index{$\sH$, Hecke algebra}%

\subsection{$L$-functions}\label{ss:intro L} Let $\calA=C_c^\infty(G(F)\bs G(\BA)/K,\QQ)$ be the space of everywhere unramified $\QQ$-valued automorphic functions for $G$. \index{$\calA,\calA_\pi$}%
Then $\calA$ is an $\sH$-module. By an everywhere unramified cuspidal automorphic representation $\pi$ of $G(\BA_{F})$ we mean an $\sH$-submodule $\calA_{\pi}\subset \calA$ that is irreducible over $\QQ$. 

For every such $\pi$, $\End_{\sH}(\calA_{\pi})$ is a number field $E_{\pi}$, which we call the {\em coefficient field} of $\pi$.\index{$\pi,E_\pi$} Then by the commutativity of $\sH$, $\calA_{\pi}$ is a one-dimensional $E_{\pi}$-vector space. If we extend scalars to $\CC$, $\calA_{\pi}$ splits into one-dimensional $\sH_{\CC}$-modules $\calA_{\pi}\otimes_{E_{\pi},\iota}\CC$, one for each embedding $\iota:E_{\pi}\incl\CC$, and each $\calA_{\pi}\otimes_{E_{\pi},\iota}\CC\subset\calA_{\CC}$ is the unramified vectors of an everywhere unramified cuspidal automorphic representation in the usual sense.

The standard (complete) $L$-function $L(\pi,s)$ is a polynomial of degree $4(g-1)$ in $q^{-s-1/2}$ with coefficients in the ring of integers $\calO_{E_{\pi}}$. Let $\pi_{F'}$ be the base change to $F'$, and let $L(\pi_{F'},s)$ be the standard $L$-function of $\pi_{F'}$. This $L$-function is a product of two $L$-functions associated to cuspidal automorphic representations of $G$ over $F$:
$$
L(\pi_{F'},s)=L(\pi,s)L(\pi\otimes\eta_{F'/F},s).
$$\index{$L(\pi,s), L(\pi_{F'},s)$}%
Therefore $L(\pi_{F'},s)$ is a polynomial of degree $8(g-1)$ in $q^{-s-1/2}$ with coefficients in $E_{\pi}$. It satisfies a functional equation
$$
L(\pi_{F'},s)=\epsilon(\pi_{F'},s) L(\pi_{F'},1-s),
$$
where the epsilon factor takes a simple form
$$
\epsilon(\pi_{F'},s)=q^{-8(g-1)(s-1/2)}.
$$
\index{$\epsilon(\pi_{F'},s)$}%
Let $L(\pi,\Ad,s)$ be the adjoint $L$-function of $\pi$. \index{$L(\pi,\Ad,s)$}%
Denote
\begin{equation}
\label{eqn sL}
\sL(\pi_{F'},s)= \epsilon(\pi_{F'},s)^{-1/2}\frac{L(\pi_{F'},s)}{L(\pi,\Ad,1)},
\end{equation}\index{$\sL(\pi_{F'},s)$}%
where the the square root is understood as
$$
 \epsilon(\pi_{F'},s)^{-1/2}:=q^{4(g-1)(s-1/2)}.
$$
Then we have a functional equation:
$$
\sL(\pi_{F'},s)=\sL(\pi_{F'},1-s).
$$
Note that the constant factor $L(\pi,\Ad,1)$ in $\sL(\pi_{F'},s)$ does not affect the functional equation, and it shows up only through the calculation of the Petersson inner product of a spherical vector in $\pi$ (see the proof of Theorem \ref{p:Jpi}). 

Consider the Taylor expansion at the central point $s=1/2$:
$$
\sL(\pi_{F'},s)= \sum_{r\geq 0} \sL^{(r)}(\pi_{F'},1/2)\frac{(s-1/2)^r}{r!},
$$
i.e.,
$$
\sL^{(r)}(\pi_{F'},1/2)=\frac{d^r}{ds^r}\Big|_{s=0}  \left(   \epsilon(\pi_{F'},s)^{-1/2}\frac{L(\pi_{F'},s)}{L(\pi,\Ad,1)}\right).
$$\index{$\sL^{(r)}(\pi_{F'},1/2)$}%
If $r$ is odd, by the functional equation we have
$$
\sL^{(r)}(\pi_{F'},1/2)=0.
$$
Since $L(\pi,\Ad,1)\in E_\pi$, we have $\sL(\pi_{F'},s)\in E_{\pi}[q^{-s-1/2},q^{s-1/2}]$. It follows that
\begin{equation*}
\sL^{(r)}(\pi_{F'},1/2)\in E_{\pi}\cdot (\log q)^{r}.
\end{equation*} 
The main result of this paper is to relate each even degree Taylor coefficient to the self-intersection numbers of a certain algebraic cycle on the moduli stack of Shtukas.
We give two formulations of our main results, one using certain subquotient of the rational Chow group, and the other using $\ell$-adic cohomology.

\subsection{The Heegner--Drinfeld cycles}
From now on, we let $r$ be an {\em even} integer. In \S\ref{ss:ShtG}, we will introduce moduli stack $\Sht^{r}_{G}$ of Drinfeld Shtukas with $r$-modifications for the group $G=\PGL_{2}$.  The stack $\Sht^{r}_{G}$ is a  Deligne--Mumford stack over $X^r$ and the natural morphism 
$$
\xymatrix{\pi_{G}: \Sht^r_{G}\ar[r] &X^r} \index{$\pi_G$}%
$$
is smooth of relative dimension $r$, and locally of finite type.  Let $T=(\Res_{F'/F}\BG_m)/\BG_m$ be the non-split torus associated to the double cover $X'$ of $X$.\index{$T$, torus}%
 In \S\ref{ss:ShtT}, we will introduce the moduli stack $\Sht^{\mu}_{T}$ of $T$-Shtukas, depending on the choice of an $r$-tuple of signs $\mu\in\{\pm\}^{r}$ satisfying certain balance conditions in \S\ref{sss:mu}. Then we have a similar map
$$
\xymatrix{\pi^{\mu}_{T}: \Sht^{\mu}_{T}\ar[r]& X'^r}	\index{$\Sht^{\mu}_{T}$}	\index{$\pi^{\mu}_{T}$}%
$$
which is a torsor under the finite Picard stack $\Pic_{X'}(k)/\Pic_{X}(k)$. In particular,  $\Sht^{\mu}_{T}$ is a proper smooth Deligne--Mumford stack over $\Spec k$.

There is a natural finite morphism of stacks over $X^r$
$$
\xymatrix{\Sht^{\mu}_{T}\ar[r]& \Sht^{r}_{G}}.
$$
It induces a finite morphism
$$
\xymatrix{\theta^{\mu}:\Sht^{\mu}_{T}\ar[r]& \Sht'^{r}_{G}:=\Sht_{G}^r\times_{X^r}X'^r}.	\index{$\Sht'^{r}_{G}$}	\index{$\theta^{\mu}$}%
$$
This defines a class in the Chow group
$$
\theta^{\mu}_{*}[\Sht^{\mu}_{T}]\in \Ch_{c,r}(\Sht'^{r}_{G})_\BQ. \index{$\theta^{\mu}_{*}[\Sht^{\mu}_{T}]$} \index{$\Ch_{c,r}$}%
$$
Here $\Ch_{c,r}(-)_{\QQ}$ means the Chow group of proper cycles of dimension $r$, tensored over $\QQ$. See \S\ref{ss:proper cycles} for details. In analogy to the classical Heegner cycles \cite{GZ}, we will call $\theta^{\mu}_{*}[\Sht^{\mu}_{T}]$ the {\em Heegner--Drinfeld cycle} in our setting.

\subsection{Main results: cycle-theoretic version}\label{ss:intro cycle}
The Hecke algebra $\sH$ acts on the Chow group
$\Ch_{c,r}(\Sht'^{r}_{G})_\BQ$ as correspondences. Let $\wt W\subset \Ch_{c,r}(\Sht'^{r}_{G})_\BQ$ be the sub $\sH$-module generated by the Heegner--Drinfeld cycle $\theta^{\mu}_{*}[\Sht^{\mu}_{T}]$.
There is a bilinear and symmetric intersection pairing \footnote{In this paper, the intersection pairing on the Chow groups will be denoted by $\jiao{\cdot,\cdot}$, and other pairings (those on the quotient of the Chow groups, and the cup product pairing on cohomology) will be denoted by $(\cdot,\cdot)$.}
\begin{equation}\label{W intersection}
\xymatrix{\jiao{\cdot,\cdot}_{\Sht'^{r}_{G}}:\wt W\times\wt W\ar[r]& \BQ.}
\end{equation}
Let $\wt W_0$ be the kernel of the pairing, i.e.,
$$
\wt W_0=\left\{z\in\wt W\big |\, (z,z')=0, \mbox{ for all } z'\in \wt W \right\}.
$$
The pairing $\jiao{\cdot,\cdot}_{\Sht'^{r}_{G}}$ then induces a {\em non-degenerate} pairing on the quotient $W:=\wt W/\wt W_0$ \index{$W, \wt W_0, \wt W$}%
\begin{align}\label{eqn W int}
\xymatrix{(\cdot,\cdot): W\times W\ar[r]& \BQ}.	
\end{align}

The Hecke algebra $\sH$ acts on $W$. For any ideal $\calI\subset \sH$, let 
$$W[\calI]=\left\{w\in W \big |\,\calI\cdot w=0\right\}.$$ Let $\pi$ be an everywhere  unramified cuspidal automorphic representation of $G$ with coefficient field $E_{\pi}$, and let $\l_{\pi}: \sH\to E_{\pi}$ be the associated character, whose kernel $\fkm_{\pi}$ is a maximal ideal of $\sH$. Let
\begin{equation*}
W_{\pi}=W[\fkm_{\pi}]\subset W
\end{equation*}\index{$W[\calI], W_\pi, \lambda_{\pi}, \fkm_{\pi}$}%
be the $\l_{\pi}$-eigenspace of $W$. This is an $E_{\pi}$-vector space. Let $\calI_{\Eis}\subset \sH$ be the Eisenstein ideal as defined in Definition \ref{def:Eis} and define
$$
W_{\Eis}=W[\calI_{\Eis}].\index{$\calI_{\Eis}$}%
$$
\begin{thm}\label{th:intro W decomp} 
We have an orthogonal decomposition of $\sH$-modules
\begin{equation}\label{intro W decomp}
W=W_{\Eis} \oplus \left(\bigoplus_{\pi} W_{\pi}\right) ,
\end{equation}
where $\pi$ runs over the finite set of everywhere unramified cuspidal automorphic representation of $G$, and $W_{\pi}$ is an $E_{\pi}$-vector space of dimension at most one.
\end{thm}
The proof will be given in \S\ref{proof:intro W decomp}. In fact one can also show that $W_{\Eis}$ is a free rank one module over $\QQ[\Pic_{X}(k)]^{\iota_{\Pic}}$ (for notation see \S\ref{sss:Pic inv}), but we shall omit the proof of this fact.

The $\QQ$-bilinear pairing $(\cdot,\cdot)$ on $W_{\pi}$ can be lifted to an $E_{\pi}$-bilinear symmetric pairing
\begin{equation}\label{eq jiao pi}
\xymatrix{(\cdot,\cdot)_{\pi}: W_{\pi}\times W_{\pi}\ar[r]& E_{\pi}}
\end{equation}\index{$(\cdot,\cdot)_{\pi}$}%
where for $w,w'\in W_{\pi}$, $(w,w')_{\pi}$ is the unique element in $E_{\pi}$ such that $\Tr_{E_{\pi}/\QQ}(e\cdot (w,w')_{\pi})=(ew,w')$.

We now present the cycle-theoretic version of our main result.
\begin{thm}\label{th:main cycle} 
Let $\pi$ be an everywhere unramified cuspidal automorphic representation of $G$ with coefficient field $E_{\pi}$. Let $[\Sht^{\mu}_{T}]_{\pi}\in W_{\pi}$ be the projection of the image of  $\theta^{\mu}_{*}[\Sht^{\mu}_{T}]\in \tilW$ in $W$ to the direct summand $W_{\pi}$ under the decomposition \eqref{intro W decomp}. Then we have an equality in $E_{\pi}$ 
$$
 \frac{1}{2(\log q)^{r} } \,\lvert \omega_X \rvert \,\sL^{(r)}\left(\pi_{F'},1/2\right) =\Big([\Sht^{\mu}_{T}]_{\pi},\quad [\Sht^{\mu}_{T}]_{\pi} \Big)_{\pi},
$$
where $\omega_X$ is the canonical divisor of $X$, and $ \lvert \omega_X \rvert=q^{-\deg \omega_X}$. \index{$\omega_X$}%
\end{thm}
The proof will be completed in \S\ref{proof:main cycle}.

\begin{remark}Assume that $r=0$. Then our formula is equivalent to the Waldspurger formula \cite{W} for an everywhere unramified cuspidal automorphic representation $\pi$. More precisely, for any nonzero $\phi\in \pi^K$, the Waldspurger formula is the identity 
\begin{align*}
\frac{1}{2}\,\lvert \omega_X\rvert\,\sL(\pi_{F'},1/2)=\frac{\left|  \int_{T(F)\bs T(\BA)}\phi(t)\,dt\right|^2 }{\pair{\phi,\phi}_{{\rm Pet}}},
\end{align*}
where $\pair{\phi,\phi}_{{\rm Pet}}$ is the Petersson inner product \eqref{eqn Pet}, and the measure on $G(\BA)$ (resp. $T(\BA)$) is chosen such that $\vol(K)=1$ (resp. $\vol (T(\BO))=1$). 

\end{remark}
\begin{remark}Our $E_\pi$-valued intersection paring is similar to the N\'eron--Tate height pairing with coefficients in \cite[\S1.2.4]{YZZ}.
\end{remark}

\subsection{Main results:  cohomological version}
Let $\ell$ be a prime number different from $p$. Consider the middle degree cohomology with compact support 
$$V'_{\Ql}=\cohoc{2r}{(\Sht'^{r}_{G})\otimes_{k}\kbar,\Ql}(r).$$
In the main body of the paper, we simply denote this by $V'$. This vector space is endowed with the cup product 
$$
\xymatrix{(\cdot,\cdot): V'_{\Ql} \times V'_{\Ql}\ar[r]&\Ql}.
$$
Then for any maximal ideal $\fkm\subset\sH_{\Ql}$, we define the generalized eigenspace of $V'_{\Ql}$ with respect to $\fkm$ by
\begin{equation*}
V'_{\Ql,\fkm}=\cup_{i>0}V'_{\Ql}[\fkm^i].
\end{equation*}
We also define the Eisenstein part of $V'_{\Ql}$ by
$$
V'_{\Ql,\Eis}=\cup_{i>0}V'_{\Ql}[\calI_{\Eis}^i].
\index{$V_{\Ql}, V_{\Ql,\fkm}, V'_{\Ql,\Eis}$}%
$$

We remark that in the cycle-theoretic version (cf. \S\ref{ss:intro cycle}), the generalized eigenspace coincides with the eigenspace because the space $W$ is a cyclic module over the Hecke algebra.

\begin{thm}[see Theorem \ref{th:Spec decomp Sht'} for a more precise statement]\label{th:intro V decomp} 
We have an orthogonal decomposition of $\sH_{\Ql}$-modules
\begin{equation}\label{intro V decomp}
V'_{\Ql}=V'_{\Ql,\Eis}\oplus\left(\bigoplus_{\fkm}V'_{\Ql,\fkm}\right),
\end{equation}
where $\fkm$ runs over a finite set of maximal ideals of $\sH_{\Ql}$ whose residue fields $E_{\fkm}:=\sH_{\Ql}/\fkm$ are finite extensions of $\Ql$, and each $V'_{\Ql,\fkm}$ is an $\sH_{\Ql}$-module of finite dimension over $\Ql$ supported at the maximal ideal $\fkm$.
\end{thm}

The action of $\sH_{\Ql}$ on $V'_{\Ql,\fkm}$ factors through the completion $\wh{\sH}_{\Ql,\fkm}$ with residue field $E_{\fkm}$. Since $E_{\fkm}$ is finite \'etale over $\Ql$, and $\wh{\sH}_{\Ql,\fkm}$ is a complete local (hence henselian) $\Ql$-algebra with residue field $E_{\fkm}$, Hensel's lemma implies that there is a unique section $E_{\fkm}\to \wh{\sH}_{\Ql,\fkm}$ (the minimal polynomial of every element $\ov h\in E_{\fkm}$ over $\Ql$ has a unique root $h\in \wh{\sH}_{\Ql,\fkm}$ whose reduction is $\ov h$). Hence each  $V'_{\Ql,\fkm}$ is also an $E_{\fkm}$-vector space in a canonical way. As in the case of $W_{\pi}$, using the $E_{\fkm}$-action on $V'_{\Ql,\fkm}$, the $\Ql$-bilinear pairing on $V'_{\Ql,\fkm}$ may be lifted to an $E_{\fkm}$-bilinear symmetric pairing
\begin{equation*}
\xymatrix{ (\cdot,\cdot)_{\fkm}: V'_{\Ql,\fkm}\times V'_{\Ql,\fkm}\ar[r]& E_{\fkm}}.\index{$(\cdot,\cdot)_{\fkm}$}%
\end{equation*}

Note that, unlike \eqref{intro W decomp}, in the decomposition \eqref{intro V decomp} we can not be sure whether all $\fkm$ are automorphic (i.e., the homomorphism $\sH\to E_{\fkm}$ is the character by which $\sH$ acts on the unramified line of an irreducible automorphic representation). However, for an everywhere unramified cuspidal automorphic representation $\pi$ of $G$ with coefficient field $E_{\pi}$,  we may extend $\l_{\pi}:\sH\to E_{\pi}$ to $\Ql$ to get
\begin{equation*}
\xymatrix{ \l_{\pi}\otimes\Ql:\sH_{\Ql}\ar[r] & E_{\pi}\otimes\Ql\cong\prod_{\l|\ell}E_{\pi,\l} }
\end{equation*}
where $\l$ runs over places of $E_{\pi}$ above $\ell$. Let $\fkm_{\pi,\l}$ be the maximal ideal of $\sH_{\Ql}$ obtained as the kernel of the $\l$-component of the above map $\sH_{\Ql}\to E_{\pi,\l}$. 

To alleviate notation, we denote $V'_{\Ql,\fkm_{\pi,\l}}$ simply by $V'_{\pi,\l}$, and denote the $E_{\pi,\l}$-bilinear pairing $(\cdot,\cdot)_{\fkm_{\pi,\l}}$ on $V'_{\pi,\l}$ by
\begin{equation*}
\xymatrix{ (\cdot,\cdot)_{\pi,\l}: V'_{\pi,\l}\times V'_{\pi,\l}\ar[r]& E_{\pi,\l}.} \index{$(\cdot,\cdot)_{\pi,\l}$}%
\end{equation*}

We now present the cohomological version of our main result.

\begin{thm}\label{th:main coho} Let $\pi$ be an everywhere unramified cuspidal automorphic representation of $G$ with coefficient field $E_{\pi}$. Let $\l$ be a place of $E_{\pi}$ above $\ell$. Let $[\Sht^{\mu}_{T}]_{\pi,\l}\in V'_{\pi,\l}$ be the projection of the cycle class $\cl(\theta^{\mu}_{*}[\Sht^{\mu}_{T}]) \in V'_{\Ql}$ to the direct summand $V'_{\pi,\l}$ under the decomposition \eqref{intro V decomp}. Then we have an equality in $E_{\pi,\l}$
$$
 \frac{1}{2(\log q)^{r} } \,\lvert \omega_X \rvert \,\sL^{(r)}\left(\pi_{F'},1/2\right)=\Big([\Sht^{\mu}_{T}]_{\pi,\l},\quad [\Sht^{\mu}_{T}]_{\pi,\l} \Big)_{\pi,\l}.
$$
In particular, the RHS also lies in $E_{\pi}$.
\end{thm}
The proof will be completed in in \S\ref{proof:main coho}.

\subsection{Two other results}

We have the following positivity result. This may be seen as an evidence of the Hodge standard conjecture (on the positivity of intersection pairing) for a sub-quotient of the Chow group of middle dimensional cycles on $\Sht'^{r}_G$.
\begin{thm}\label{th: index} Let $W_{{\rm cusp}}$ be the orthogonal complement of $W_{\Eis}$ in $W$ (cf. \eqref{intro W decomp}). Then the restriction to $W_{{\rm cusp}}$  of the intersection pairing $\xymatrix{(\cdot,\cdot)}$ in \eqref{eqn W int} is positive definite.  
\end{thm}
\begin{proof}The assertion is equivalent to the positivity for the restriction to $W_\pi$ of the intersection pairing, for all $\pi$ in \eqref{intro W decomp}. Fix such a $\pi$. Then the coefficient field $E_\pi$ is a totally real number field because the Hecke operators $\sH$ act on the positive definite inner product space $\calA\otimes_{\QQ}\RR$ (under the Petersson inner product) by self-adjoint operators. For an embedding  $\iota:E_\pi\to \BR$, we define
$$
W_{\pi,\iota}:=W_{\pi}\otimes_{E_\pi,\iota}\BR.
$$ Extending scalars from $E_\pi$ to $\BR$ via $\iota$, the pairing \eqref{eq jiao pi} induces an $\BR$-bilinear symmetric pairing
$$
\xymatrix{(\cdot,\cdot)_{\pi,\iota}: W_{\pi,\iota} \times W_{\pi,\iota}\ar[r]& \BR.}
$$
It suffices to show that, for every embedding  $\iota:E_\pi\to \BR$, the pairing $(\cdot,\cdot)_{\pi,\iota}$ is positive definite. The $\BR$-vector space $W_{\pi,\iota}$ is at most one-dimensional, with a generator given by $[\Sht^{\mu}_{T}]_{\pi,\iota}=[\Sht^{\mu}_{T}]_{\pi}\otimes 1$. The embedding $\iota$ gives an irreducible cuspidal automorphic representation $\pi_\iota$ with $\BR$-coefficient. Then Theorem \ref{th:main cycle} implies that 
$$
 \frac{1}{2(\log q)^{r} } \,\lvert \omega_X \rvert \,\sL^{(r)}\left(\pi_{\iota,F'},1/2\right) =\Big([\Sht^{\mu}_{T}]_{\pi,\iota},\quad [\Sht^{\mu}_{T}]_{\pi,\iota} \Big)_{\pi,\iota}\in \BR.
$$
It is easy to see that $L(\pi_{\iota},\Ad,1)>0$.  By Theorem \ref{th:pos}, we have 
$$
\sL^{(r)}\left(\pi_{\iota,F'},1/2\right)\geq 0.
$$
It follows that $$\Big([\Sht^{\mu}_{T}]_{\pi,\iota},\quad [\Sht^{\mu}_{T}]_{\pi,\iota} \Big)_{\pi,\iota}\geq 0.$$
This completes the proof.
\end{proof}

Another result is a ``Kronecker limit formula" for function fields. Let $L(\eta,s)$ be the (complete) L-function associated to the Hecke character $\eta$. \index{$L(\eta,s)$}%
\begin{thm}\label{th:limit}When $r>0$ is even, we have 
$$
\jiao{\theta^{\mu}_{*}[\Sht^{\mu}_{T}],\quad  \theta^{\mu}_{*}[\Sht^{\mu}_{T}]}_{\Sht'^{r}_{G}}=\frac{2^{r+2}}{(\log q)^r} L^{(r)}(\eta,0).
$$
\end{thm}
The proof will be given in \S\ref{proof:limit}. For the case $r=0$, see Remark \ref{rem limit r=0}.
\begin{remark}
To obtain a similar formula for the odd order derivatives $L^{(r)}(\eta,0)$, we need moduli spaces analogous to $\Sht^\mu_T$ and $\Sht'^{r}_G$ for odd $r$. We will return to this in future work.
\end{remark}

\subsection{Outline of the proof of the main theorems}
\subsubsection{Basic strategy} The basic strategy is to compare two relative trace formulae. A  relative trace formula (abbreviated as RTF) is an equality between a spectral expansion and an orbital integral expansion. We have two RTFs, an ``analytic" one for the $L$-functions, and a ``geometric" one for the intersection numbers, corresponding to the two sides of the desired equality in Theorem \ref{th:main coho}.

We may summarize the strategy of the proof into the following diagram
\begin{equation}\label{eqn outline}
\xymatrix{\text{Analytic:} & \sum_{u\in\PP^{1}(F)-\{1\}}\BJ_{r}(u,f)\ar@{=}[r]^(.7){\S\ref{s:RTF}}\ar@{-}[d]^{\sim \text{Th \ref{th:IJ}}}\ar@{}[dr]|{\Rightarrow} & \BJ_{r}(f)\ar@{=}[r]^{\S\ref{s:a spec}}\ar@{-}[d]^{\text{Th \ref{th:full IJ}}} \ar@{}[dr]|{\Rightarrow} & \sum_{\pi}\BJ_{r}(\pi, f)\ar@{-}[d]^{\Rightarrow \text{Th \ref{th:main coho}}}\\
\text{Geometric:} & \sum_{u\in\PP^{1}(F)-\{1\}}\BI_{r}(u,f)\ar@{=}[r]^(.7){\S\ref{s:alt}} & \BI_{r}(f)\ar@{=}[r]^{\S\ref{s:c spec}} & \sum_{\fkm}\BI_{r}(\fkm, f)}
\end{equation}
The vertical lines mean equalities after dividing the first row by $(\log q)^{r}$.

\subsubsection{The analytic side} We start with the analytic RTF. To an $f\in \sH$ (or more generally, $C_c^\infty(G(\BA))$), one first associates an automorphic kernel function $\BK_{f}$ on $G(\BA)\times G(\BA)$ and then a regularized integral: 
\begin{align*}
\BJ(f,s)=\int^{{\rm reg}}_{[A]\times [A]}\BK_{f}(h_1,h_2)\lvert h_1h_2\rvert^s \eta(h_2)\,dh_1\,dh_2. \index{$\BJ(f,s)$}%
\end{align*}
Here $A$ is the diagonal torus of $G$, and $[A]=A(F)\bs A(\BA)$. \index{$A$}\index{$[A]$}%
We refer to \S\ref{ss:RTF} for the definition of the weighted factors, and the regularization. Informally, we may view this integral as a weighted (naive) intersection number on the constant groupoid $\Bun_G(k)$ (the moduli stack of Shtukas with $r=0$ modifications) between $\Bun_A(k)$ and its Hecke translation under $f$ of $\Bun_A(k)$.  

The resulting $\BJ(f,s)$
belongs to $\BQ[q^{-s}, q^{s}]$. For an $f$ in the Eisenstein ideal $\CI_{\Eis}$ (cf. \S\ref{ss:eis}),  the spectral decomposition of $\BJ(f,s)$ takes a simple form: it is the sum of 
\begin{align*}
\BJ_{\pi}(f,s)=\frac{1}{2}\, |\omega_X| \,\sL(\pi_{F'},s+1/2)\,\lambda_\pi(f)			\index{$\BJ_\pi(f,s)$}%
\end{align*}
where $\pi$ runs over all everywhere unramified cuspidal automorphic representations $\pi$ of $G$ with $\Qlbar$-coefficients (cf. Prop. \ref{p:Jpi}). We define $\BJ_{r}(f)$ to be the $r$-th derivative
\begin{align*}
\BJ_r(f):= \left(\frac{d}{ds}\right)^r\Big|_{s=0}\BJ(f,s). \index{$\BJ_r(f)$}%
\end{align*}

We point out that in the case of $r=0$, the relative trace formula in question was first introduced by Jacquet \cite{J86}, in his reproof of Waldspurger's formula. In the case of $r=1$, a variant was first considered in \cite{Z12} (for number fields).

\subsubsection{The geometric side} Next we consider the geometric RTF. We consider the Heegner--Drinfeld cycle $\theta^{\mu}_{*}[\Sht^{\mu}_{T}]$ and its translation by the Hecke correspondence given by $f\in\sH$, both being cycles on the ambient stack $\Sht'^{r}_{G}$. We define $\BI_{r}(f)$ to be their intersection number
\begin{align*}
\BI_r(f):=\jiao{\theta^{\mu}_{*}[\Sht^{\mu}_{T}],\quad f*\theta^{\mu}_{*}[\Sht^{\mu}_{T}]}_{\Sht'^{r}_{G}}\quad\in\QQ,\quad f\in \sH.		\index{$\BI_r(f)$}%
\end{align*}
To decompose this spectrally according to the Hecke action, we have two perspectives, one viewing  the Heegner--Drinfeld cycle as an element in the Chow group modulo numerical equivalence, the other considering the cycle class of the  Heegner--Drinfeld cycle in the $\ell$-adic cohomology.  In either case, when $f$ is in a certain power of $\calI_{\Eis}$, the spectral decomposition (\S\ref{s:c spec}, or Theorem \ref{th:intro V decomp}) of $W_{\Qlbar}$ or $V'_{\Qlbar}$ as an $\sH_{\Qlbar}$-module expresses $\BI_{r}(f)$ as a sum of
\begin{align*}
\BI_{r}(\fkm,f)=\Big([\Sht^{\mu}_{T}]_{\fkm}, \quad f*[\Sht^{\mu}_{T}]_{\fkm}\Big)
\index{$\BI_r(\pi,f)$}%
\end{align*}
where $\fkm$ runs over a finite set of maximal ideals of $\sH_{\Qlbar}$ whose corresponding generalized eigenspaces appear discretely in $W_{\Qlbar}$ or $V'_{\Qlbar}$.  We remark that the method of the proof of the spectral decomposition in Theorem \ref{th:intro V decomp} can potentially be applied to moduli of Shtukas for more general groups $G$, which should lead to a better understanding of the cohomology of these moduli spaces.

We point out here that we use the same method as in \cite{Z12}  to set up the geometric RTF, although in \cite{Z12} only the case of $r=1$ was considered. In the case $r=0$, Jacquet used an integration of kernel function to set up an RTF for the $T$-period integral, which is equivalent to our geometric RTF because in this case $\Sht^{\mu}_{T}$ and $\Sht^{r}_{G}$ become discrete stacks $\Bun_{T}(k)$ and $\Bun_{G}(k)$. Our geometric formulation treats all values of $r$ uniformly.

\subsubsection{The key identity} In view of the spectral decompositions of both $\BI_{r}(f)$ and $\BJ_{r}(f)$, to prove the main Theorem \ref{th:main coho} for all $\pi$ simultaneously, it suffices to establish the following key identity  (cf. Theorem \ref{th:full IJ})
\begin{equation}\label{key id}
\BI_{r}(f)=(\log q)^{-r}\BJ_{r}(f)\in\BQ,\quad \text{ for all }f\in\sH.
\end{equation}
This key identity also allows us to deduce Theorem \ref{th:intro W decomp} on the spectral decomposition of the space $W$ of cycles  from the spectral decomposition of $\BJ_r$.  Theorems \ref{th:main cycle} then follows easily from Theorem \ref{th:main coho}.

Since half of the paper is devoted to the proof of the key identity \eqref{key id}, we comment on its proof in more detail. The spectral decompositions allow us to reduce to proving  \eqref{key id} for sufficiently many functions $f\in \sH$, indexed by effective divisors on $X$ with large degree compared to the genus of $X$ (cf. Theorem \ref{th:IJ}). Most of the algebro-geometric part of this paper is devoted to the proof of the key identity \eqref{key id} for those Hecke functions. 

In \S\ref{s:orb}, we interpret the orbital integral expansion of $\BJ_r(f)$ (the upper left sum in \eqref{eqn outline}) as a certain weighted counting of effective divisors on the curve $X$. The geometric ideas used in the part are close to those in the proof of various fundamental lemmas by Ng\^o \cite{Ngo} and by the first-named author \cite{Y-FL}, although the situation is much simpler in the current paper. In \S\ref{s:alt}, we interpret the intersection number $\BI_{r}(f)$ as the trace of a correspondence acting on the cohomology of a certain variety. This section involves new geometric ideas that did not appear in the treatment of the fundamental lemma type problems. This is also the most technical part of the paper, making use of the general machinery on intersection theory reviewed or developed in Appendix A.

After the preparations in \S\ref{s:orb} and \S\ref{s:alt}, our situation can be summarized as follows. For an integer $d\geq0$, we have fibrations  
$$
f_{\CN}:\CN_{d}=\bigsqcup_{\un{d}}\CN_{\un{d}}\to \CA_d,\quad
f_{\CM}:\CM_{d} \to \CA_d,
$$
where $\un{d}$ runs over all quadruples $(d_{11},d_{12},d_{21},d_{22})\in\ZZ^{4}_{\geq0}$ such that $d_{11}+d_{22}=d=d_{12}+d_{21}$. We need to show that the direct image complexes
$\bR f_{\calM,*}\Ql $ and $\bR f_{{\calN},*}L_{d}$ are isomorphic to each other, where $L_{d}$ is a local system of rank one coming from the double cover $X'/X$. When $d$ is sufficiently large, we show that both complexes are shifted perverse sheaves, and are obtained by middle extension from a dense open subset of $\calA_{d}$ over which both can be explicitly calculated (cf. Prop \ref{p:Rf decomp} and \ref{p:RfN decomp}). The isomorphism between the two complexes over the entire base $\calA_{d}$ then follows by the functoriality of the middle extension. The strategy used here is the {\em perverse continuation principle} coined by Ng\^o, which has already played a key role in all known geometric proofs of fundamental lemmas, see \cite{Ngo} and \cite{Y-FL}. 

\begin{remark} One feature of our proof of the key identity \eqref{key id} is that it is entirely global, in the sense that we do {\em not} reduce to the comparison of local orbital integral identities, as opposed to what one usually does when comparing two trace formulae. Therefore our proof is different from Jacquet's in the case $r=0$ in that his proof is essentially local (this is inevitable because he also considers the number field case).

Another remark is that our proof of \eqref{key id} in fact gives a term-by-term identity of the orbital expansion of both $\BJ_{r}(f)$ and $\BI_{r}(f)$, as indicated in the left column of \eqref{eqn outline}, although this is not logically needed for our main results. However, such more refined identities (for more general $G$) will be needed in the proof of the arithmetic fundamental lemma for function fields, a project to be completed in near future \cite{Y11}.
\end{remark}

\subsection{A guide for readers} Since this paper uses a mixture of tools from automorphic representation theory, algebraic geometry and sheaf theory, we think it might help orient the readers by providing a brief summary of the contents and the background knowledge required for each section. 
 
First we give the Leitfaden.

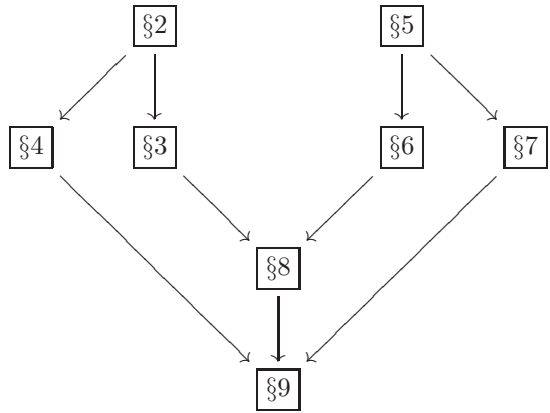
\begin{equation*}
\xymatrix{ &  \fbox{\S\ref{s:RTF}} \ar[dl]\ar[d] & & \fbox{\S\ref{s:Sht}}\ar[d]\ar[dr]\\
\fbox{\S\ref{s:a spec}}\ar[ddrr] &  \fbox{\S\ref{s:orb}}\ar[dr] & & \fbox{\S\ref{s:alt}}\ar[dl] & \fbox{\S\ref{s:c spec}}\ar[ddll]\\
&&\fbox{\S\ref{s:most}}\ar[d]\\
&&\fbox{\S\ref{s:pf}}}
\end{equation*}

Section \ref{s:RTF} sets up the relative trace formula following Jacquet's approach \cite{J86} to the Waldspurger formula. This section is purely representation-theoretic.

Section \ref{s:orb} gives a geometric interpretation of the orbital integrals involved in the relative trace formula introduced in \S\ref{s:RTF}. We express these orbital integrals as the trace of Frobenius on the cohomology of certain varieties, in the similar spirit of the proof of various fundamental lemmas (\cite{Ngo}, \cite{Y-FL}). This section involves both orbital integrals and some algebraic geometry but not yet perverse sheaves.

Section \ref{s:a spec} relates the spectral side of the relative trace formula in \S\ref{s:RTF} to automorphic $L$-functions. Again this section is purely representation-theoretic.

Section \ref{s:Sht} introduces the geometric players in our main theorem: moduli stacks $\Sht^{r}_{G}$ of Drinfeld Shtukas, and Heegner--Drinfeld cycles on them. We give self-contained definitions of these moduli stacks, so no prior knowledge of Shtukas is assumed, although some experience with the moduli stack of bundles will help.

Section \ref{s:alt} is the technical heart of the paper, aiming to prove Theorem \ref{th:I trace}. The proof involves studying several auxiliary moduli stacks and uses heavily the intersection-theoretic tools reviewed and developed in Appendix \ref{s:int}. The first-time readers may skip the proof and only read the statement of Theorem \ref{th:I trace}.

Section \ref{s:c spec} gives a decomposition of the cohomology of $\Sht^{r}_{G}$ under the action of the Hecke algebra, generalizing the classical spectral decomposition for the space automorphic forms. The idea is to remove the analytic ingredients from the classical treatment of spectral decomposition, and to use solely commutative algebra (in particular, we crucially use the {\em Eisenstein ideal} introduced in \S\ref{s:a spec}). For first-time readers, we suggest read \S\ref{ss:coho Sht}, then jump directly to Definition \ref{def:bsH} and continue from there. What he/she will miss in doing this is the study of the geometry of $\Sht^{r}_{G}$ near infinity (horocycles), which requires some familiarity with the moduli stack of bundles, and the formalism of $\ell$-adic sheaves.

Section \ref{s:most} combines the geometric formula for orbital integrals established in \S\ref{s:orb} and the trace formula for the intersection numbers established in \S\ref{s:alt} to prove the key identity \eqref{key id} for most Hecke functions. The proofs in this section involve perverse sheaves.

Section \ref{s:pf} finishes the proofs of our main results. Assuming results from the previous sections, most argument in this section only involves commutative algebra.

Both appendices can be read independently of the rest of the paper.

Appendix \ref{s:int} reviews the intersection theory on algebraic stacks following Kresch \cite{Kresch}, with two new results that are used in \S\ref{s:alt} for the calculation of the intersection number of Heegner--Drinfeld cycles. The first result, called the {\em Octahedron Lemma} (Theorem \ref{th:oct}), is an elaborated version of the following simple principle: in calculating the intersection product of several cycles, one can combine terms and change the orders arbitrarily. The second result is a Lefschetz trace formula for the intersection of a correspondence with the graph of the Frobenius map, building on results of Varshavsky \cite{Var}.

Appendix \ref{s positivity} proves a positivity result for central derivatives of automorphic $L$-functions, assuming the generalized Riemann hypothesis in the case of number fields. The main body of the paper only considers $L$-functions for function fields, for which the positivity result can be proved in an elementary way (see Remark \ref{r:func field pos}).

\subsection{Further notation}

\subsubsection{Function field notation} For $x\in |X|$, let $\varpi_x$ be a uniformizer of $\CO_x$,  $k_x$ be the residue field of $x$, $d_x=[k_x:k]$, and $q_x=\# k_x=q^{d_x}$.\index{$\varpi_x,\CO_x,k_x,d_x,q_x$}%
The valuation map is a homomorphism
$$
\xymatrix{\val\colon\BA^\times \ar[r]&\BZ }
$$\index{$\val$}%
such that $\val(\varpi_{x})=d_{x}$. The normalized absolute value on $\AA^{\times}$ is defined as
\begin{align*}
   \begin{gathered}
	|\cdot|  \colon
	\xymatrix@R=0ex{
	  \BA^\times \ar[r]  & \BQ^\times_{>0}\subset\BR^\times  \\
	a  \ar@{|->}[r]  & q^{-\val(a)}.
	}
	\end{gathered}
\end{align*} 
Denote the kernel of the absolute value by
$$
\BA^1=\Ker(|\cdot|).
$$\index{$\BA^1$}%
We have the global and local zeta function  
$$\zeta_F(s)=\prod_{x\in |X|}\zeta_{x}(s), \quad\zeta_x(s)=\frac{1}{1-q_x^{-s}}.\index{$\zeta_F, \zeta_x$}%
$$
Denote by $\Div(X)\cong \BA^{\times}/\OO^{\times}$ the group of divisors on $X$.\index{$\Div(X)$}%

\subsubsection{Group-theoretic notation} Let $\BG$ be an algebraic group over $k$. We will view it as an algebraic group over $F$ by extension of scalars. We will abbreviate $[\BG]=\BG(F)\bs \BG(\BA)$. \index{$[\BG]=\BG(F)\bs \BG(\BA)$}%
Unless otherwise stated, the Haar measure on the group $\BG(\BA)$ will be chosen such that the natural maximal compact open subgroup $\BG(\BO)$ has volume equal to one. For example,  the measure on $\BA^\times$, resp. $G(\BA)$ is such that $\vol\left(\BO^\times\right)=1$, resp. $\vol(K)=1.$

\subsubsection{Algebro-geometric notation} In the main body of the paper, all geometric objects are algebraic stacks over the finite field $k=\FF_{q}$. For such a stack $S$,  let $\Fr_S:S\to S$ be the absolute $q$-Frobenius endomorphism that raises functions to their $q$-th powers.\index{$\Fr_S$}%

For an algebraic stack $S$ over $k$, we write $\cohog{*}{S\otimes_{k}\kbar}$ (resp. $\cohoc{*}{S\otimes_{k}\kbar}$) for the \'etale cohomology (resp. \'etale cohomology with compact support) of the base change $S\otimes_{k}\kbar$ with $\Ql$-coefficients. \index{$\cohog{*}{S\otimes_{k}\kbar}, \cohoc{*}{S\otimes_{k}\kbar}$}%
The $\ell$-adic homology $\homog{*}{S\otimes_{k}\kbar}$ and Borel-Moore homology $\hBM{*}{S\otimes_{k}\kbar}$ are defined as the graded duals of $\cohog{*}{S\otimes_{k}\kbar}$  and $\cohoc{*}{S\otimes_{k}\kbar}$ respectively. \index{$\homog{*}{S\otimes_{k}\kbar}, \hBM{*}{S\otimes_{k}\kbar}$}%
We use $D^{b}_{c}(S)$ to denote the derived category of $\Ql$-complexes for the \'etale topology of $S$, as defined in \cite{LO}. We use $\DD_{S}$ to denote the dualizing complex of $S$ with $\Ql$-coefficients.
\index{$D^{b}_{c}(S), \DD_{S}$}

\subsection*{Acknowledgement} We thank Akshay Venkatesh for a key conversation that inspired our use of the Eisenstein ideal, and Dorian Goldfeld and Peter Sarnak for their help on Appendix \ref{s positivity}.  We thank Benedict Gross for his comments, Michael Rapoport for communicating us comments from the participants of ARGOS in Bonn, and Shouwu Zhang for carefully reading the first draft of the paper and providing many useful suggestions. We thank  the Mathematisches Forschungsinstitut Oberwolfach to host the Arbeitsgemeinschaft in April 2017 devoted to this paper, and we thank the participants, especially Jochen Heinloth and Yakov Varshavsky, for their valuable feedbacks.

\part{The analytic side}

\section{The relative trace formula}\label{s:RTF}
In this section we set up the relative trace formula following Jacquet's approach \cite{J86} to the Waldspurger formula.

\subsection{Orbits}
In this subsection $F$ is allowed to be an arbitrary field. Let $F'$ be a semisimple quadratic $F$-algebra, i.e., it is either the split algebra $F\oplus F$ or a quadratic field extension of $F$. Denote by $\Nm:F'\to F$ the norm map.

Denote $G=\PGL_{2,F}$ and $A$ the subgroup of diagonal matrices in $G$. \index{$A$}%
We consider the action of $A\times A$ on $G$ where $(h_1,h_2)\in A\times A$ acts by $(h_1,h_2)g=h_1^{-1}gh_2$. We define an $A\times A$-invariant morphism:
\begin{align}\label{inv G}
   \begin{gathered}
	\inv  \colon
	\xymatrix@R=0ex{
	  G \ar[r]  &  \BP^1_F -\{1\}  \\
	\gamma  \ar@{|->}[r]  & \frac{bc}{ad}
	}
	\end{gathered}
\end{align}\index{$\inv$}%
where $\matrixx{a}{b}{c}{d} \in \GL_2$ is a lifting of $\gamma$. 
We say that $\gamma\in G$ is {\em $A\times A$-regular semisimple} if 
$$\inv(\gamma)\in\BP^1_F -\{0,1,\infty\},$$
 or equivalently all $a,b,c,d$ are invertible in terms of the lifting of $\gamma$. Let $G_\rs$ be the open subscheme of $A\times A$-regular semisimple locus.
 A section of the restriction of the morphism $\inv$ to $G_\rs$ is given by
\begin{align}\label{sec inv}
   \begin{gathered}
	\gamma  \colon
	\xymatrix@R=0ex{
	  \BP^1_F -\{0,1,\infty\} \ar[r]  & G    \\
	u  \ar@{|->}[r]  &\gamma(u)=\matrixx{1}{u}{1}{1}
		}
	\end{gathered}.\index{$\gamma,\gamma(u)$}%
\end{align} 
 
Now we consider the induced map on the $F$-points $\inv: G(F)\to \BP^1(F) -\{1\}$, and the action of $A(F)\times A(F)$ on $G(F)$. Denote by
$\BO_\rs(G)=A(F)\bs G_\rs(F)/A(F)$ the set of orbits in $G_\rs(F)$ under the action of $A(F)\times A(F)$. They will be called the regular semisimple orbits. It is easy to see that the map $\inv: G_\rs(F)\to \BP^1(F) -\{0,1,\infty\}$ induces a bijection
$$\inv: \BO_\rs(G) \longrightarrow \BP^1(F) -\{0,1,\infty\}.\index{$G_\rs, \BO_{\rs}$}%
$$
 A convenient set of representative of $\BO_\rs(G)$ is given by
$$
\BO_\rs(G)\simeq \left\{\gamma(u)= \matrixx{1}{u}{1}{1}\biggm|  u\in \BP^1(F) -\{0,1,\infty\}\right\}.
$$
There are six non-regular-semisimple orbits in $G(F)$, represented respectively by
$$
\begin{array}{ccc}
 1=\matrixx{1}{}{}{1},  & n_+=\matrixx{1}{1}{}{1},  & n_-=\matrixx{1}{}{1}{1},  \\
w=\matrixx{}{1}{1}{}, & wn_+=\matrixx{}{1}{1}{1}, &wn_-=\matrixx{1}{1}{1}{},
\end{array}\index{$n_\pm, w$}%
$$
where the first three (the last three, resp.) have $\inv=0$ ($\infty$, resp.)

\subsection{Jacquet's RTF}\label{ss:RTF}Now we return to the setting of the introduction. In particular, we have $\eta=\eta_{F'/F}$.
In  \cite{J86} Jacquet constructs an RTF to study the central value of $L$-functions of the same type as ours (mainly in the number field case).  Here we modify his RTF to study higher derivatives. 

For $f\in C_c^\infty(G(\BA))$, we
consider the automorphic kernel function
\begin{align}\label{eqn kernel}
\BK_{f}(g_1,g_2)=\sum_{\gamma\in G(F)} f(g_1^{-1}\gamma g_2),\quad g_1,g_2\in G(\BA).\index{$\BK_f$}%
\end{align}
We will define a distribution, given by a regularized integral
\begin{align*}
\BJ(f,s)=\int^{{\rm reg}}_{[A]\times [A]}\BK_{f}(h_1,h_2)\lvert h_1h_2\rvert^s \eta(h_2)\,dh_1\,dh_2.\index{$\BJ(f,s)$}%
\end{align*}
Here we recall that $[A]=A(F)\bs A(\BA),\index{$[A]$}%
$ and for $h=\matrixx{a}{}{}{d}\in A(\BA)$ we write for simplicity
$$
\big|h\big|=\big|a/d\big|,\quad \eta(h)=\eta(a/d).
$$
The integral is not always convergent but can be regularized in a way analogous to \cite{J86}. 
For an integer $n$, consider the ``annulus" 
$$
\BA^\times_n:=\bigg\{ x\in\BA^\times\biggm|
\val(x)= n \bigg\}.\index{$\BA^\times_n$}%
$$This is a torsor under the group $\BA^1=\BA^\times_0$. Let $A(\BA)_n$ be the subset of $A(\BA)$ defined by  
\[
 A(\BA)_n=  \biggl\{ \matrixx{a}{}{}{d}\in A(\BA)
	\biggm|
a/d\in \BA^\times_{n}
	  \biggl\}.
\] Then we define, for $(n_1,n_2)\in\BZ^2$, 
\begin{align}\label{eqn J c}
\BJ_{n_1,n_2}(f,s)=\int_{[A]_{n_1}\times [A]_{n_2}}\BK_{f}(h_1,h_2)\lvert h_1h_2\rvert^s \eta(h_2)\,dh_1\,dh_2.\index{$\BJ_{n_1,n_2}(f,s)$}%
\end{align}
The integral \eqref{eqn J c} is clearly absolutely convergent and equal to a Laurent polynomial in $q^{s}$.
\begin{prop}\label{prop J n1n2}
 The integral $\BJ_{n_1,n_2}(f,s)$ vanishes when $|n_1|+| n_2|$ is sufficiently large. 
\end{prop}
Granting this proposition, we then define 
\begin{align}\label{eqn J(f,s)}
\BJ(f,s):=\sum_{(n_1,n_2)\in\BZ^2} \BJ_{n_1,n_2}(f,s).		\index{$\BJ(f,s)$}%
\end{align}This is a Laurent polynomial in $q^{s}$.

The proof of Proposition \ref{prop J n1n2} will occupy \S2.3-2.5.

\subsection{A finiteness lemma}
For an $(A\times A)(F)$-orbit of $\gamma$, we define 
\begin{align}\label{eqn K a}
\BK_{f,\gamma}(h_1,h_2)= \sum_{\delta\in A(F)\gamma A(F)} f(h_1^{-1}\delta h_2),\quad h_1,h_2\in A(\BA).\index{$\BK_{f,\gamma}$}%
\end{align}
Then we have
\begin{align}\label{eqn sum}
\BK_{f}(h_1,h_2)=\sum_{\gamma \in A(F)\bs G(F)/ A(F) } \BK_{f,\gamma}(h_1,h_2).\index{$\BK_f$}%
\end{align}

\begin{lem}\label{lem finte}
The sum in \eqref{eqn sum}  has only finitely many non-zero terms.
\end{lem}
\begin{proof}
Denote by $G(F)_u$ the fiber of $u$ under the (surjective) map \eqref{inv G} $$\inv:G(F)\to  \BP^1(F)-\{1\}.$$ We then have a decomposition of $G(F)$ as a disjoint union
$$
G(F)=\coprod_{u\in \BP^1(F)-\{1\}} G(F)_u.
$$There is exactly one (three, resp.)  $(A\times A)(F)$-orbit in $G(F)_a$
when $u\in \BP^1(F)-\{0,1,\infty\}$ (when $u\in \{0,\infty\}$, resp.). It suffices to show that for
all but finitely many $u\in \BP^1(F)-\{0,1,\infty\}$, the kernel function $\BK_{f,\gamma(u)}(h_1,h_2)$ vanishes identically on $A(\BA)\times A(\BA)$.

Consider the map 
$$
\tau:=\frac{\inv}{1-\inv}\colon G(\BA)\to \BA.
$$
The map $\tau$ is continuous and takes constant values on $A(\BA)\times A(\BA)$-orbits. For $\BK_{f,\gamma(u)}(h_1,h_2)$ to be nonzero, the invariant $\tau(\gamma(u))=\frac{u}{1-u}$ must be in the image of $supp(f)$, the support of the function $f$.  Since $supp(f)$ is compact, so is its image under $\tau$. On the other hand, the invariant $\tau(\gamma(u))=\frac{u}{1-u}$ belongs to $F$. Since the intersection of a compact set $supp(f)$ with a discrete set $F$ in $\BA$ must have finite cardinality,  the kernel function $\BK_{f,\gamma(u)}(h_1,h_2)$ is nonzero for only finitely many $u$.
\end{proof}

For $\gamma\in A(F)\bs G(F)/A(F)$, we define 
\begin{align}\label{eqn J n1n2 gamma}
\BJ_{n_1,n_2}(\gamma,f,s)=\int_{[A]_{n_1}\times [A]_{n_2}}\BK_{f,\gamma}(h_1,h_2)\lvert h_1h_2 \rvert^s \eta(h_2)\,dh_1\,dh_2.\index{$\BJ_{n_1,n_2}(\gamma,f,s)$}%
\end{align}
Then we have 
\begin{align*}
\BJ_{n_1,n_2}(f,s)=\sum_{\gamma \in A(F)\bs G(F)/ A(F) }  \BJ_{n_1,n_2}(\gamma,f,s).\index{$\BJ_{n_1,n_2}(f,s)$}%
\end{align*}

By the previous lemma, the above sum has only finitely many nonzero terms. Therefore, to show Prop. \ref{prop J n1n2}, it suffices to show 
\begin{prop}\label{prop J n gamma}
For any $\gamma\in G(F)$, the integral $\BJ_{n_1,n_2}(\gamma,f,s)$ vanishes when $|n_1|+| n_2|$ is sufficiently large. 
\end{prop}
Granting this proposition, we may define the (weighted) orbital integral
\begin{align}\label{eqn J(gamma,f,s)}
\BJ(\gamma,f,s):=\sum_{(n_1,n_2)\in\BZ^2} \BJ_{n_1,n_2}(\gamma,f,s).\index{$\BJ(\gamma,f,s)$}%
\end{align}
To show Prop. \ref{prop J n gamma}, we distinguish two cases according to whether $\gamma$ is regular semisimple.

\subsection{Proof of Proposition \ref{prop J n gamma}: regular semisimple orbits} \label{ss:rs orb}

For $u\in \BP^1(F)-\{0,1,\infty\}$, the fiber $G(F)_u=\inv^{-1}(u)$ is a single $A(F)\times A(F)$-orbit of $\gamma(u)$, and the stabilizer of $\gamma(u)$ is trivial.
We may rewrite \eqref{eqn J n1n2 gamma} as
\begin{align}\label{eqn J(a)1}
\BJ_{n_1,n_2}(\gamma(u),f,s)=\int^{}_{A(\BA)_{n_1} \times A(\BA)_{n_2}}f(h_1^{-1}\gamma(u)h_2)\lvert h_1h_2\rvert^s \eta(h_2)\,dh_1\,dh_2.
\end{align}
For the regular semisimple $\gamma=\gamma(u)$, the map \begin{align*}\iota_{\gamma}: (A \times A)(\BA)&\to G(\BA)\\(h_1,h_2)&\mapsto h_1^{-1}\gamma h_2
\end{align*}
is a closed embedding.  It follows that the function $f\circ \iota_{\gamma}$ has compact support, hence belongs to $C_c^\infty((A \times A)(\BA))$. Therefore, the integrand in \eqref{eqn J(a)1} vanishes when $|n_1|+|n_2|\gg 0$ (depending on $f$ and $\gamma(u)$).

\subsection{Proof of Proposition \ref{prop J n gamma}: non-regular-semisimple orbits}
Let $u\in \{0,\infty\}$. We only consider the case $u=0$ since the other case is completely analogous. There are three orbit representatives $\{1,n_+,n_-\}$.

It is easy to see that for $\gamma=1$, we have for all $(n_1,n_2)\in \BZ^2$, 
$$\BJ_{n_1,n_2}(\gamma,f,s)=0,
$$
because $\eta|_{\BA^1}$ is a nontrivial character.

Now we consider the case $\gamma=n_+$; the remaining case $\gamma=n_-$ is similar. Define a function
\begin{align}\label{eqn phi x y}
\phi(x,y)=f\left(\matrixx{x}{y}{}{1}\right),\quad (x,y)\in \BA^\times\times\BA. 
\end{align}
Then we have $\phi\in C_c^\infty(\BA^\times\times\BA)$. The integral $\BJ_{n_1,n_2}(n_+,f,s)$ is given by
\begin{align}\label{eqn J 0 phi}
\int_{\BA^\times_{n_1}\times \BA^\times_{n_2}} \phi \left(x^{-1}y,x^{-1}\right)\eta(y) |xy|^s \,d^\times x\,d^\times y,
\end{align}
where we use the multiplicative measure $d^\times x$ on $\BA^\times$. We substitute $y$ by $xy$, and then $x$ by $x^{-1}$:
$$
\int_{Z(n_1,n_2)}\phi \left(y,x\right)\eta(xy) |x|^{-2s}|y|^s \,d^\times x\,d^\times y,
$$
where 
 $$Z(n_1,n_2)=\left\{(x,y)\mid x\in A(\BA)_{-n_1}, x^{-1}y\in  A(\BA)_{n_2}\right\}.$$

Since $C_c^\infty(\BA^\times \times\BA)\simeq C_c^\infty(\BA^\times) \otimes C_c^\infty(\BA)$,
we may reduce to the case $\phi(x,y)=\phi_1(x)\phi_2(y)$ where $\phi_1\in C_c^\infty(\BA^\times), \phi_2\in C_c^\infty(\BA)$. Moreover, by writing $\phi_1$ as a finite linear combination, each supported on a single $\BA_n^\times$, we may even assume that $supp(\phi_1)$ is contained in $\BA^{\times}_{n}$, for  some $n\in\BZ$. The last integral is equal to
$$
\left(\int_{\BA^\times_{n}} \phi_1 \left( y\right) \eta(y) |y|^s\,d^\times y\right)\left( \int_{\BA^\times_{-n_1}\cap \BA^\times_{-n_2+n} } \phi_2\left(x\right)\eta(x) |x|^{-2s} \, d^\times x\right).
$$

Finally we recall that, from Tate's thesis,  for any $\varphi\in C_c^\infty(\BA)$, the integral on an annulus
\begin{align*}
\int _{ \BA^\times_n}\varphi\left(x \right)\eta(x) |x|^{2s} \,d^\times x,
\end{align*}
vanishes when $|n|\gg 0$. We briefly recall how this is proved.  It is clear if $n\ll 0$. Now assume that $n\gg 0$. We rewrite the integral as
\begin{align*}
\int _{ F^\times\bs \BA^\times_n}\sum_{\alpha\in F^\times}\varphi \left(\alpha x \right)\eta(x) |x|^{2s} \,d^\times x,
\end{align*} 
The Fourier transform of $\varphi$, denoted by $\wh{\varphi}$, still lies in $C_c^\infty(\BA)$.  By the Poisson summation formula, we have
\begin{align}\label{eqn PSF}
\sum_{\alpha\in F^\times}\varphi (\alpha x)=-\varphi(0)+|x|^{-1}\wh{\varphi}(0)+|x|^{-1}\sum_{\alpha\in F^\times}\wh{\varphi} (\alpha /x),
\end{align}
By the boundedness of the support of $\wh{\varphi}$, the sum over $F^\times$ on the RHS vanishes when $\val(x)=n\gg 0$.
Finally we note that the the integral of the remaining two terms on the RHS of \eqref{eqn PSF} vanishes because $\eta$ is nontrivial on $F^\times\bs\BA^1$.

This completes the proof of Prop. \ref{prop J n gamma}, and Prop.  \ref{prop J n1n2}.

\subsection{The distribution $\BJ$}Now $\BJ(f,s)$ is a Laurent polynomial in $q^{s}$. Consider the $r$-th derivative
\begin{align*}
\BJ_r(f):= \left(\frac{d}{ds}\right)^r\Big|_{s=0}\BJ(f,s).\index{$\BJ_r(f)$}%
\end{align*}
For $\gamma\in A(F)\bs G(F)/ A(F)$, we define
\begin{align*}
\BJ_r(\gamma,f):= \left(\frac{d}{ds}\right)^r\Big|_{s=0}\BJ(\gamma,f,s).\index{$\BJ_r(\gamma,f)$}%
\end{align*}
We then have an expansion (cf.\eqref{eqn J(f,s)})
\begin{align*}
\BJ(f,s)=\sum_{\gamma\in A(F)\bs G(F)/ A(F)}\BJ(\gamma,f,s),
\end{align*}
and (cf. \eqref{eqn J(gamma,f,s)})
\begin{align}\label{eqn orb decomp}
\BJ_r(f)=\sum_{\gamma\in A(F)\bs G(F)/ A(F)}\BJ_r(\gamma,f).
\end{align}
We define 
\begin{align}\label{orb J(u,s)}
\BJ(u,f,s)=\sum_{\gamma\in A(F)\bs G(F)_u/A(F)}\BJ(\gamma,f,s),\quad u\in \BP^1(F)-\{1\}.\index{$\BJ(u,f,s)$}%
\end{align}
and
\begin{align}\label{orb J}
\BJ_r(u,f)=\sum_{\gamma\in A(F)\bs G(F)_u/A(F)}\BJ_r(\gamma,f),\quad u\in \BP^1(F)-\{1\}.\index{$\BJ_r(u,f)$}%
\end{align}
Then we have a slightly coarser decomposition than \eqref{eqn orb decomp}
\begin{align}\label{eqn u decomp}
\BJ_r(f)=\sum_{u\in \BP^1(F)-\{1\}}\BJ_r(u,f).
\end{align}

\subsection{A special test function $f=\one_K$.}
\begin{prop}
For the test function 
$$f=\one_K,$$
we have 
\begin{align}\label{eqn J irr}
\BJ(u,\one_K,s)=\begin{cases}L(\eta,2s)+L(\eta,-2s) & \text{ if } u\in\{0,\infty\},\\
1 & \text{ if } u\in k-\{0,1\},\\
0 & \text{ otherwise}.
\end{cases}
\end{align}
\end{prop}
\begin{proof}
We first consider the case $u\in\PP^{1}(F)-\{0,1,\infty\}$. In this case, we have
\begin{eqnarray*}
\BJ(u,\one_K,s)&=&\int_{\BA^{\times}\times\BA^{\times}}\one_{K}\left(\matrixx{x^{-1}}{0}{0}{1}\matrixx{1}{u}{1}{1}\matrixx{y}{0}{0}{1}\right)|xy|^{s}\eta(y)\,d^{\times}x\,d^{\times}y\\
&=&\sum_{x,y\in \BA^{\times}/\BO^{\times}}\one_{K}\left(\matrixx{x^{-1}y}{x^{-1}u}{y}{1}\right)|xy|^{s}\eta(y).
\end{eqnarray*}
The integrand is nonzero if and only if $g=\matrixx{x^{-1}y}{x^{-1}u}{y}{1}\in K$. This is equivalent to the condition that $g^{2}_{ij}/\det(g)\in\BO$, where $\{g_{ij}\}_{1\leq i,j\leq 2}$ are the entries of $g$. We have $\det(g)=x^{-1}y(1-u)$, therefore $g\in K$ is equivalent to
\begin{equation}\label{a1 a2 cond}
x^{-1}y(1-u)^{-1}\in \BO, x^{-1}y^{-1}u^{2}(1-u)^{-1}\in \BO, xy(1-u)^{-1}\in \BO, \text{ and } xy^{-1}(1-u)^{-1}\in\BO.
\end{equation}
Multiplying the first and last condition we get $(1-u)^{-1}\in\BO $. Therefore $1-u\in F^{\times}$ must be a constant function, i.e.,  $u\in k-\{0,1\}$. This shows that $\BJ(u,\one_K,s)=0$ when $u\in F- k$.

When $u\in k-\{0,1\}$, the conditions \eqref{a1 a2 cond} become
\begin{equation*}
x^{-1}y\in \BO, x^{-1}y^{-1}\in \BO, xy\in\BO, \text{ and } xy^{-1}\in\BO.
\end{equation*}
These together imply that $x,y\in\BO^{\times}$. Therefore the integrand is nonzero only when both $x$ and $y$ are in the unit coset of $\BA^{\times}/\BO^{\times}$, and the integrand is equal to 1 when this happens.  This proves $\BJ(u,\one_{K},s)=1$ when $u\in k-\{0,1\}$.

Next we consider the case $u=0$. For $f=\one_K$ and $\gamma=n_{+}$, we have in \eqref{eqn phi x y} $\phi=\phi_1\otimes\phi_2$ where
$$ \phi_1=\one_{\BO^\times},\quad\phi_2=\one_{\BO}.
$$
Therefore we have 
$$
\BJ(n_+,\one_K,s)=\int_{\BA^\times}\phi_2(x)\eta(x)|x|^{-2s}\, d^\times x=L(\eta,-2s).
$$
Similarly we have 
$$
\BJ(n_-,\one_K,s)=L(\eta,2s).
$$
This proves the equality \eqref{eqn J irr} for $u=0$. The case for $u=\infty$ is analogous. 
\end{proof}

\begin{cor}\label{c:one K} We have
\begin{equation*}
\JJ_{r}(\one_{K})=\begin{cases}4L(\eta,0)+q-2=4\frac{\#\Jac_{X'}(k)}{\#\Jac_{X}(k)}+q-2 & r=0;\\
2^{r+2}\left(\frac{d}{ds}\right)^{r}\Big|_{s=0}L(\eta,s) & r>0 \text{ even};\\
0 & r>0 \text{ odd}.\end{cases}
\end{equation*}
\end{cor}

\section{Geometric interpretation of orbital integrals}\label{s:orb}
In this section, we will give a geometric interpretation the orbital integrals $\BJ(\gamma, f,s)$ (cf. \eqref{eqn J(gamma,f,s)}) as a certain weighted counting of effective divisors on the curve $X$, when $f$ is in the unramified Hecke algebra. 

\subsection{A basis for the Hecke algebra}\label{ss:test fun}
Let $x\in|X|$. In the case $G=\PGL_{2}$, the local unramified Hecke algebra $\sH_{x}$ is the polynomial algebra $\QQ[h_{x}]$ where $h_{x}$ is the characteristic function of the $G(\calO_{x})$-double coset of $\matrixx{\varpi_{x}}{0}{0}{1}$, and $\varpi_{x}\in \calO_{x}$  is a uniformizer. For each integer $n\geq 0$, consider the set $\Mat_{2}(\calO_{x})_{v_{x}(\det)=n}$ of matrices $A\in\Mat_{2}(\calO_{x})$ such that $v_{x}(\det(A))=n$. Let $M_{x,n}$ be the image of $\Mat_{2}(\calO_{x})_{v_{x}(\det)=n}$ in $G(F_{x})$. Then $M_{x,n}$ is a union of $G(\calO_{x})$-double cosets. We define $h_{nx}$ to be the characteristic function
\begin{equation}\label{hlocal}
h_{nx}=\one_{M_{x,n}}.
\end{equation}
Then $\{h_{nx}\}_{n\geq 0}$ is a $\QQ$-basis for $\sH_{x}$.

Now consider the global unramified Hecke algebra $\sH=\otimes_{x\in|X|}\sH_{x}$, which is a polynomial ring over $\QQ$ with infinitely generators $h_{x}$. For each effective divisor $D=\sum_{x\in|X|}n_{x}\cdot x$, we can define an element $h_{D}\in\sH$ using
\begin{equation}\label{fD}
h_{D}=\otimes_{x\in |X|}h_{n_{x}x}\index{$\sH_x, h_{nx}, h_D$}%
\end{equation}
where $h_{n_{x}x}$ is defined in \eqref{hlocal}. It is easy to see that the set $\{h_{D}| $D$ \text{ effective divisor on }$X$\}$ is a $\QQ$-basis for $\sH$. \index{$\sH$, Hecke algebra}%

The goal of the next few subsections is to give a geometric interpretation the orbital integral $\BJ(\gamma, h_{D},s)$. We begin by defining certain moduli spaces.

\subsection{Global moduli space for orbital integrals} 

\subsubsection{}\label{ss:sym power} 
For $d\in\ZZ$, we consider the Picard stack $\Pic^{d}_{X}$ of lines bundles over $X$ of degree $d$. Note that $\Pic^{d}_{X}$ is a $\Gm$-gerbe over its coarse moduli space. Let $\hX_{d}\to\Pic^{d}_{X}$ be the universal family of sections of line bundles, i.e., an $S$-point of $\hX_{d}$ is a pair $(\CL,s)$, where $\CL$ is a line bundle over $X\times S$ such that $\deg\calL|_{X\times\{t\}}=d$ for all geometric points $t$ of $S$,  and $s\in\cohog{0}{X\times S,\calL}$.\index{$\Pic_X^d$}\index{$\hX_d$}%

When $d<0$, $\hX_{d}\cong\Pic^{d}_{X}$ since all global sections of all line bundles $\calL\in\Pic^{d}_{X}$ vanish.  When $d\geq0$, let $X_{d}=X^{d}\!\!\sslash \!\! S_{d}$ be the $d$-th symmetric power of $X$. Then there is an open embedding $X_{d}\incl \hX_{d}$ as the open locus of nonzero sections, with complement isomorphic to $\Pic^{d}_{X}$. \index{$X_d$}%

For $d_{1},d_{2}\in\ZZ$, we have a morphism
\begin{equation*}
\wh{\add}_{d_{1},d_{2}}: \hX_{d_{1}}\times\hX_{d_{2}}\to \hX_{d_{1}+d_{2}}\index{$\wh{\add}_{d_1,d_2}, \add_{d_1,d_2}$}%
\end{equation*}
sending $((\CL_{1}, s_{1}), (\CL_{2},s_{2}))$ to $(\CL_{1}\otimes\CL_{2},s_{1}\otimes s_{2})$. The restriction of $\wh{\add}_{d_{1},d_{2}}$ to the open subset $X_{d_{1}}\times X_{d_{2}}$ becomes the addition map for divisors $\add_{d_{1},d_{2}}: X_{d_{1}}\times X_{d_{2}}\to X_{d_{1}+d_{2}}$.

\subsubsection{The moduli space $\calN_{\un{d}}$}\label{sss:calN} Let $d\geq0$ be an integer. Let $\Sigma_{d}$ be the set of  quadruple of non-negative integers $\un{d}=(d_{ij})_{i,j\in\{1,2\}}$  satisfying $d_{11}+d_{22}=d_{12}+d_{21}=d$. \index{$\Sigma_d, \un{d}$}%

For $\un{d}\in\Sigma_{d}$, we consider the moduli functor  $\tcN_{\un{d}}$ classifying $(\calK_{1}, \calK_{2}, \calK'_{1}, \calK'_{2}, \varphi)$ where
\begin{itemize}
\item $\calK_{i},\calK'_{i}\in \Pic_{X}$ with $\deg\calK'_{i}-\deg\calK_{j}=d_{ij}$.\index{$\calK_{i},\calK'_{i}$}%
\item $\varphi:\calK_{1}\oplus\calK_{2}\to \calK'_{1}\oplus\calK'_{2}$ is an $\calO_{X}$-linear map. We express it as a matrix
\begin{equation*}
\varphi=\matrixx{\varphi_{11}}{\varphi_{12}}{\varphi_{21}}{\varphi_{22}}
\end{equation*}
where $\varphi_{ij}:\calK_{j}\to\calK'_{i}$.
\item If $d_{11}< d_{22}$, then $\varphi_{11}\neq0$ otherwise $\varphi_{22}\neq0$. If $d_{12}<d_{21}$ then $\varphi_{12}\neq 0$ otherwise $\varphi_{21}\neq0$. Moreover, at most one of the four maps $\varphi_{ij}, i,j\in\{1,2\}$ can be zero. 
\end{itemize}
The Picard stack $\Pic_{X}$ acts on $\tcN_{\un{d}}$ by tensoring  each $\calK_{i}$ and $\calK'_{j}$ with the same line bundle. Let $\calN_{\un{d}}$ be the quotient stack $\tcN_{\un{d}}/\Pic_{X}$, which will turn out to be representable by a scheme over $k$.  We remark that the artificial-looking last condition in the definition of $\calN_{\un{d}}$ is to guarantee that $\calN_{\un{d}}$ is separated. \index{$\tcN_{\un{d}},\calN_{\un{d}}$}%

\subsubsection{The base $\calA_{d}$}\label{sss:Ad} Let $\calA_{d}$ be the moduli stack of triples $(\Delta,a,b)$ where $\Delta\in\Pic^{d}_{X}$, $a$ and $b$ are sections of $\Delta$ with the open condition that $a$ and $b$ are not simultaneously zero. Then we have an isomorphism
\begin{equation}\label{AXX}
\calA_{d}\cong\hX_{d}\times_{\Pic^{d}_{X}}\hX_{d}-Z_{d}      \index{$\calA_d$} \index{$Z_d$}%
\end{equation}
where $Z_{d}\cong\Pic^{d}_{X}$ is the image of the diagonal zero sections $(0,0):\Pic^{d}_{X}\incl \hX_{d}\times_{\Pic^{d}_{X}}\hX_{d}$. 

We claim that $\calA_{d}$ is a scheme. In fact, $\calA$ is covered by two opens $V=\hX_{d}\times_{\Pic^{d}_{X}}X_{d}$ and $V'=X_{d}\times_{\Pic^{d}_{X}}\hX_{d}$. Both $V$ and $V'$ are schemes because the  map $\hX_{d}\to\Pic^{d}_{X}$ is schematic.

We have a map 
\begin{equation*}
\delta:\calA_{d}\to\hX_{d}			\index{$\delta:\calA_{d}\to\hX_{d}$}%
\end{equation*}
given by $(\Delta,a,b)\mapsto (\Delta, a-b)$. 

\subsubsection{The open part $\Ah_{d}$}\label{sss:Ah} Later we will consider the open subscheme $\Ah_{d}\subset\calA_{d}$ defined by the condition $a\neq b$, i.e., the preimage of $X_{d}$ under the map $\delta:\calA_{d}\to \hX_{d}$. \index{$\calA_{d}^\heart$}%

\subsubsection{} To a point $(\calK_{1}, \calK_{2}, \calK'_{1}, \calK'_{2}, \varphi)\in\tcN_{\un{d}}$ we attach the following maps
\begin{itemize}
\item $a:=\varphi_{11}\otimes\varphi_{22}: \calK_{1}\otimes\calK_{2}\to\calK'_{1}\otimes\calK'_{2}$;
\item $b:=\varphi_{12}\otimes\varphi_{21}:\calK_{1}\otimes\calK_{2}\to\calK'_{2}\otimes\calK'_{1}\cong\calK'_{1}\otimes\calK'_{2}$.
\end{itemize}
Both $a$ and $b$ can be viewed as sections of the line bundle $\Delta=\calK'_{1}\otimes\calK'_{2}\otimes\calK^{-1}_{1}\otimes\calK^{-1}_{2}\in\Pic^{d}_{X}$. Clearly this assignment $(\calK_{1}, \calK_{2}, \calK'_{1}, \calK'_{2}, \varphi)\mapsto (\Delta, a,b)$ is invariant under the action of $\Pic_{X}$ on $\tcN_{\un{d}}$. Therefore we get a map
\begin{equation*}
f_{\calN_{\un{d}}}:\calN_{\un{d}}\to \calA_{d}.\index{$f_{\calN_{\un{d}}}$}%
\end{equation*}
The composition $\delta\circ f_{\calN_{\un{d}}}:\calN_{\un{d}}\to \hX_{d}$ takes $(\calK_{1},\calK_{2},\calK'_{1},\calK'_{2},\varphi)$ to $\det(\varphi)$ as a section of $\Delta=\calK'_{1}\otimes\calK'_{2}\otimes\calK^{-1}_{1}\otimes\calK^{-1}_{2}$.

\subsubsection{Geometry of $\calN_{\un{d}}$}
Fix $\un{d}=(d_{ij})\in \Sigma_{d}$. For $i,j\in\{1,2\}$, we have a morphism $\jmath_{ij}:\calN_{\un{d}}\to\hX_{d_{ij}}$ sending $(\calK_{1},\cdots,\calK'_{2},\varphi)$ to the section $\varphi_{ij}$ of the line bundle $\calL_{ij}:=\calK'_{i}\otimes\calK^{-1}_{j}\in\Pic^{d_{ij}}_{X}$. We have canonical isomorphisms $\calL_{11}\otimes\calL_{22}\cong\calL_{12}\otimes\calL_{21}\cong \Delta=\calK'_{1}\otimes\calK'_{2}\otimes\calK^{-1}_{1}\otimes\calK^{-1}_{2}$. Thus we get a morphism
\begin{eqnarray}\label{jm}
\jmath_{\un{d}}=(\jmath_{ij})_{i,j}: \calN_{\un{d}}&\to& (\hX_{d_{11}}\times \hX_{d_{22}})\times_{\Pic^{d}_{X}}(\hX_{d_{12}}\times \hX_{d_{21}}).
\end{eqnarray}
Here the fiber product on the right side is formed using the maps $\hX_{d_{11}}\times \hX_{d_{22}}\to\Pic^{d_{11}}_{X}\times\Pic^{d_{22}}_{X}\xrightarrow{\otimes}\Pic^{d}_{X}$ and $\hX_{d_{12}}\times \hX_{d_{21}}\to \Pic^{d_{12}}_{X}\times\Pic^{d_{21}}_{X}\xrightarrow{\otimes} \Pic^{d}_{X}$.

\begin{prop}\label{p:Nsm} Let $\un{d}\in\Sigma_{d}$.
\begin{enumerate}
\item The morphism $\jmath_{\un{d}}$ is an open embedding, and $\calN_{\un{d}}$ is a geometrically connected scheme over $k$. 
\item If $d\geq 2g'-1=4g-3$, $\calN_{\un{d}}$ is smooth over $k$ of dimension $2d-g+1$. 
\item We have a commutative diagram
\begin{equation}\label{f add}
\xymatrix{\calN_{\un{d}}\ar@{^{(}->}[r]^(.2){\jmath_{\un{d}}}\ar[d]^{f_{\calN_{\un{d}}}} & (\hX_{d_{11}}\times \hX_{d_{22}})\times_{\Pic^{d}_{X}}(\hX_{d_{12}}\times \hX_{d_{21}})\ar[d]^{\wh{\add}_{d_{11},d_{22}}\times\wh{\add}_{d_{12},d_{21}}}\\
\calA_{d}\ar@{^{(}->}[r]  & \hX_{d}\times_{\Pic^{d}_{X}}\hX_{d}}
\end{equation}
Moreover, the map $f_{\calN_{\un{d}}}$ is proper.
\end{enumerate}
\end{prop}
\begin{proof}
(1) We abbreviate $\Pic^{d}_{X}$ by $P^{d}$. Let $Z_{\un{d}}\subset (\hX_{d_{11}}\times \hX_{d_{22}})\times_{\Pic^{d}_{X}}(\hX_{d_{12}}\times \hX_{d_{21}})$ be the closed substack consisting of $((\calL_{ij}, s_{ij})\in\hX_{d_{ij}})_{1\leq i,j\leq 2}$ such that
\begin{itemize}
\item Either two of $\{s_{ij}\}_{1\leq i,j\leq 2}$ are zero;
\item Or $s_{11}=0$ if $d_{11}<d_{22}$; 
\item Or  $s_{22}=0$ if $d_{11}\geq d_{22}$;
\item Or $s_{12}=0$ if $d_{12}<d_{21}$;
\item Or $s_{21}=0$ if $d_{12}\geq d_{21}$.
\end{itemize}
By the definition of $\calN_{\un{d}}$, we have a Cartesian diagram
\begin{equation*}
\xymatrix{\calN_{\un{d}}\ar[r]^(.4){\jmath_{\un{d}}}\ar[d]^{\l} & (\hX_{d_{11}}\times \hX_{d_{22}})\times_{P^{d}}(\hX_{d_{12}}\times \hX_{d_{21}})-Z_{\un{d}}\ar[d]\\
P^{d_{11}-d_{12}}\times P^{d_{11}}\times P^{d_{21}}\ar[r]^(.4){\rho} & (P^{d_{11}}\times P^{d_{22}})\times_{P^{d}}(P^{d_{12}}\times P^{d_{21}})}
\end{equation*}
Here $\l$ sends $(\calK_{1},\cdots, \calK'_{2},\varphi)$ to $(\calX_{2}=\calK_{2}\otimes\calK^{-1}_{1}, \calX'_{1}=\calK'_{1}\otimes\calK^{-1}_{1}, \calX'_{2}=\calK'_{2}\otimes\calK^{-1}_{1})$, and $\rho$ sends $(\calX_{2},\calX'_{1},\calX'_{2})$ to $(\calX'_{1}, \calX'_{2}\otimes\calX^{-1}_{2}, \calX'_{1}\otimes\calX^{-1}_{2}, \calX'_{2})$. Note that $ \rho$ is an isomorphism. Therefore $\jmath_{\un{d}}$ is an isomorphism. Since the geometric fibers of $\l$ are connected, and $P^{d_{11}-d_{12}}\times P^{d_{11}}\times P^{d_{21}}$ is geometrically connected, so is $\calN_{\un{d}}$.

The stack $\calN_{\un{d}}$ is covered by four open substacks $U_{ij}, i,j\in\{1,2\}$ where $U_{ij}$ is the locus where only $\varphi_{ij}$ is allow to be zero. Each $U_{ij}$ is a scheme over $k$. In fact, for example, $U_{11}$ is an open substack of $(\hX_{d_{11}}\times X_{d_{22}})\times_{P^{d}}(X_{d_{12}}\times X_{d_{21}})$, and the latter is a scheme since the morphism $\hX_{d_{11}}\to P^{d_{11}}$ is schematic.

(2) We first show that $\calN_{\un{d}}$ is smooth when $d\geq2g'-1=4g-3$. For this we only need to show that $U_{ij}$ is smooth (see the proof of part (1) for the definition of $U_{ij}$). By the definition of $\calN_{\un{d}}$, $\varphi_{ij}$ is allowed to be zero only when $d_{ij}\geq d/2$, which implies that $d_{ij}\geq 2g-1$. Therefore, we need $U_{ij}$ to cover $\calN_{\un{d}}$ only when $d_{ij}\geq 2g-1$; otherwise $\varphi_{ij}$ is never zero and the rest of the $U_{i',j'}$ still cover $\calN_{\un{d}}$.  Therefore, we only need to prove the smoothness of $U_{ij}$ under the assumption that $d_{ij}\geq d/2$. Without loss of generality we argue for $i=j=1$. Then $d_{11}\geq 2g-1$ implies that the Abel-Jacobi map $\AJ_{d_{11}}:\hX_{d_{11}}\to P^{d_{11}}$ is smooth of relative dimension $d_{11}-g+1$. We have a Cartesian diagram
\begin{equation*}
\xymatrix{U_{11}\ar[r]\ar[d] & \hX_{d_{11}}\ar[d]^{\AJ_{d_{11}}}\\
X_{d_{22}}\times X_{d_{12}}\times X_{d_{21}}\ar[r] & P^{d_{11}}}
\end{equation*} 
where the bottom horizontal map is given by $(\calL_{22},s_{22},\calL_{12},s_{12}, \calL_{21},s_{21})\mapsto \calL_{12}\otimes\calL_{21}\otimes\calL_{22}^{-1}$.
Therefore $U_{11}$ is smooth over $X_{d_{22}}\times X_{d_{12}}\times X_{d_{21}}$ with relative dimension $d_{11}-g+1$, and $U_{11}$ is itself smooth over $k$ of dimension $2d-g+1$.

(3) The commutativity of the diagram \eqref{f add} is clear from the definition of $\jmath_{\un{d}}$. Finally we show that $f_{\calN_{\un{d}}}:\calN_{\un{d}}\to \calA_{d}$ is proper.  Note that $\calA_{d}$ is covered by open subschemes $V=\hX_{d}\times_{P^{d}}X_{d}$ and $V'=X_{d}\times_{P^{d}}\hX_{d}$ whose preimages under $f_{\calN_{\un{d}}}$ are $U_{11}\cup U_{22}$ and $U_{12}\cup U_{21}$ respectively. Therefore it suffices to show that $f_{V}: U_{11}\cup U_{22}\to V$ and $f_{V'}:U_{12}\cup U_{21}\to V'$ are both proper. 

We argue for the properness of $f_{V}$. There are two cases: either $d_{11}\geq d_{22}$ or $d_{11}<d_{22}$.

When $d_{11}\geq d_{22}$, by the last condition in the definition of $\calN_{\un{d}}$, $\varphi_{22}$ is never zero, hence $U_{11}\cup U_{22}=U_{11}$. By part (2), the map $f_{V}$ becomes
\begin{equation*}
(\hX_{d_{11}}\times X_{d_{22}})\times_{P^{d}}(X_{d_{12}}\times X_{d_{21}})\to \hX_{d}\times_{P^{d}} X_{d}.
\end{equation*}
Therefore it suffices to show that the restriction of the addition map
\begin{equation*}
\alpha=\wh{\add}_{d_{11},d_{22}}|_{\hX_{d_{11}}\times X_{d_{22}}}: \hX_{d_{11}}\times X_{d_{22}}\to \hX_{d}
\end{equation*}
is proper. We may factor  $\alpha$ as the composition of the closed embedding $\hX_{d_{11}}\times X_{d_{22}}\to \hX_{d}\times X_{d_{22}}$ sending $(\calL_{11},s_{11}, D_{22})$ to $(\calL_{11}(D_{22}), s_{11}, D_{22})$ and the projection $\hX_{d}\times X_{d_{22}}\to \hX_{d}$, and the properness of $\alpha$ follows.

The case $d_{11}<d_{22}$ is argued in the same way. The properness of $f_{V'}$ is also proved in the similar way. This finishes the proof of the properness of $f_{\calN_{\un{d}}}$. 
\end{proof}

\subsection{Relation with orbital integrals} In this subsection we relate the derivative orbital integral $\BJ(\gamma,h_{D}, s)$ to the cohomology of fibers of $f_{\calN_{\un{d}}}$.

\subsubsection{The local system $L_{\un{d}}$}\label{sss:Ld} Recall that $\nu: X'\to X$ is a geometrically connected \'etale double cover with the nontrivial involution $\sigma\in\Gal(X'/X)$. Let $L=(\nu_{*}\Ql)^{\sigma=-1}$. This is a rank one local system on $X$ with $L^{\otimes2}\cong\Ql$. Since we have a canonical isomorphism $\homog{1}{X,\ZZ/2\ZZ}\cong \homog{1}{\Pic^{n}_{X}, \ZZ/2\ZZ}$, each $\Pic^{n}_{X}$ carries a rank one local system $L_{ n}$ corresponding to $L$. By abuse of notation, we also denote the pullback of $L_{ n}$ to $\hX_{n}$ by $L_{n}$. Note that the pullback of $L_{ n}$ to $X_{n}$ via the Abel-Jacobi map $X_{n}\to \Pic^{n}_{X}$ is the descent of $L^{\boxtimes n}$ along the natural map $X^{n}\to X_{n}$. 

Using the map $\jmath_{\un{d}}$ \eqref{jm}, we define the following local system $L_{ \un{d}}$ on $\calN_{\un{d}}$:
\begin{equation*}
L_{\un{d}}:=\jmath^{*}_{\un{d}}(L_{d_{11}}\boxtimes\Ql\boxtimes L_{d_{12}}\boxtimes\Ql).\index{$L_{\un{d}}$}%
\end{equation*}

\subsubsection{}\label{ss:AD} Fix $D\in X_{d}(k)$. Let $\calA_{D}\subset \calA_{d}$ be the fiber of $\calA_{d}$ over $D$ under the map $\delta: \calA_{d}\to \hX_{d}$. Then $\calA_{D}$ classifies triples $(\calO_{X}(D), a,b)$ in $\calA_{d}$ with the condition that $a-b$ is the tautological section $1\in\Gamma(X,\calO_{X}(D))$. Such a triple is determined uniquely by the section $a\in \Gamma(X,\calO_{X}(D))$. Therefore we get canonical isomorphisms (viewing the RHS as an affine spaces over $k$)
\begin{eqnarray}\label{AD OD}
\calA_{D}&\cong&\Gamma(X,\calO_{X}(D)).\index{$\calA_{D}$}%
\end{eqnarray}
On the level of $k$-points, we have an injective map
\begin{eqnarray*}
\inv_{D}: \calA_{D}(k)\cong \Gamma(X,\calO_{X}(D)) &\incl& \PP^{1}(F)-\{1\}\\
(\calO_{X}(D), a,a-1)\bij a &\mapsto& (a-1)/a=1-a^{-1}.\index{$\inv_{D}$}%
\end{eqnarray*}

\begin{prop}\label{p:geom orb int} Let $D\in X_{d}(k)$ and consider the test function $h_{D}$ defined in \eqref{fD}. Let $u\in \PP^{1}(F)-\{1\}$. 
\begin{enumerate}
\item If $u$ is not in the image of $\inv_{D}$, then $\BJ(\gamma,h_{D}, s)=0$ for any $\gamma\in A(F)\backslash G(F)/A(F)$ with $\inv(\gamma)=u$;
\item If $u=\inv_{D}(a)$ for some $a\in\calA_{D}(k)=\Gamma(X,\calO_{X}(D))$, and $u\notin\{0,1,\infty\}$ (i.e., $a\notin\{0,1\}$), then 
\begin{equation*}
\BJ(\gamma(u), h_{D}, s)=\sum_{\un{d}\in\Sigma_{d}}q^{(2d_{12}-d)s}\Tr\left(\Frob_{a}, \left(\bR f_{\calN_{\un{d}},*}L_{\un{d}}\right)_{\ov{a}}\right).
\end{equation*}
\item Assume $d\geq 2g'-1=4g-3$. If $u=0$ then it corresponds to $a=1\in\calA_{D}(k)$; if $u=\infty$ then it corresponds to $a=0\in\calA_{D}(k)$. In both cases we have
\begin{equation}\label{u 0 infty}
\sum_{\inv(\gamma)=u}\BJ(\gamma, h_{D}, s)=\sum_{\un{d}\in\Sigma_{d}}q^{(2d_{12}-d)s}\Tr\left(\Frob_{a}, \left(\bR f_{\calN_{\un{d}},*}L_{\un{d}}\right)_{\ov{a}}\right).
\end{equation}
Here the sum on the LHS is over the three irregular double cosets $\gamma\in\{1, n_{+},n_{-}\}$ if $u=0$, and over $\gamma\in\{w, wn_{+},wn_{-}\}$ if $u=\infty$.
\end{enumerate}
\end{prop}
\begin{proof}
We first make some general construction. Let $\tilA\subset \GL_{2}$ be the diagonal torus and let 
 $\wt\gamma\in \GL_{2}(F)-(\tilA(F)\cup w\tilA(F))$ with image $\gamma\in G(F)$. Let $\alpha: \tilA\to \Gm$ be the simple root $\matrixx{a}{}{}{d}\mapsto a/d$. Let $Z\cong\Gm\subset\tilA$ be the center of $\GL_{2}$. We may rewrite $\JJ(\gamma,h_{D},s)$ as an orbital integral on $\tilA(\BA)$-double cosets on $\GL_{2}(\BA)$ (cf. \eqref{eqn J(a)1}, \eqref{eqn phi x y}, \eqref{eqn J 0 phi})
\begin{equation}\label{J for GL2}
\JJ(\gamma,h_{D},s)=\int_{\Delta (Z(\BA))\bs (\tilA \times\tilA)(\BA)} \tilh_{D}(t'^{-1}\wt\gamma t)|\alpha(t)\alpha(t')|^{s}\eta(\alpha(t))\,dt\,dt'.
\end{equation}
Here for $D=\sum_{x}n_{x} x$, $\tilh_{D}=\otimes_{x}\tilh_{n_{x}x}$ is an element in the global unramified Hecke algebra for $\GL_{2}$, where $\tilh_{n_{x}x}$ is the characteristic function of $\Mat_{2}(\calO_{x})_{v_{x}(\det)=n_{x}}$, cf. \S\ref{ss:test fun}. 
 
Using the isomorphism $\tilA(\BA)/\prod_{x\in |X|}\tilA(\calO_{x})\cong (\BA^{\times}/\OO^{\times})^{2}\cong \Div(X)^{2}$ given by taking the divisors of the two diagonal entries, we may further write the RHS of \eqref{J for GL2} as a sum over divisors $E_{1}, E_{2}, E'_{1}, E'_{2}\in\Div(X)$, up to simultaneous translation by $\Div(X)$. Suppose $t\in \tilA(\BA)$ gives the pair $(E_{1},E_{2})$ and $t'\in \tilA(\BA)$ gives the pair $(E'_{1}, E'_{2})$, then the integrand $\tilh_{D}(t'^{-1}\wt\gamma t)$ takes value 1 if and only if the rational map $\wt\gamma:\calO^{2}_{X}\dashrightarrow\calO_{X}^{2}$ given by the matrix $\wt\gamma$ fits into a commutative diagram
\begin{equation}\label{gamma phi}
\xymatrix{\calO^{2}_{X}\ar@{-->}[r]^{\wt\gamma} & \calO^{2}_{X}\\
\calO_{X}(-E_{1})\oplus\calO_{X}(-E_{2})\ar[r]^{\varphi_{\wt\gamma}}\ar@{^{(}->}[u] & \calO_{X}(-E'_{1})\oplus\calO_{X}(-E'_{2})\ar@{^{(}->}[u]}
\end{equation}
Here the vertical maps are the natural inclusions, and $\varphi_{\wt\gamma}$ is an injective map of $\calO_{X}$-modules such that $\det(\varphi_{\wt\gamma})$ has divisor $D$.
The integrand $\tilh_{D}(t'^{-1}\wt\gamma t)$ is zero otherwise. 

Let $\wt\frN_{D, \wt\gamma}\subset \Div(X)^{4}$ be the set of quadruples of divisors $(E_{1},E_{2},E'_{1},E'_{2})$ such that $\wt\gamma$ fits into a diagram \eqref{gamma phi} and $\det(\varphi_{\wt\gamma})$ has divisor $D$. Let $\frN_{D,\wt\gamma}=\wt\frN_{D,\wt\gamma}/\Div(X)$ where $\Div(X)$ acts by simultaneous translation on the divisors $E_{1},E_{2},E'_{1}$ and $E'_{2}$. \index{$\wt\frN_{D, \wt\gamma},\frN_{D,\wt\gamma}$}%

We have $|\alpha(t)\alpha(t')|^{s}=q^{-\deg(E_{1}-E_{2}+E'_{1}-E'_{2})s}$. Viewing $\eta$ as a character on the id\`ele class group $F^{\times}\backslash \BA^{\times}_{F}/\prod_{x\in |X|}\calO^{\times}_{x}\cong\Pic_{X}(k)$, we have $\eta(\alpha(t))=\eta(E_{1})\eta(E_{2})=\eta(E_{1}-E'_{1})\eta(E_{2}-E'_{1})$. Therefore
\begin{equation}\label{frN gamma}
\BJ(\gamma, h_{D}, s)=\sum_{(E_{1},E_{2},E'_{1},E'_{2})\in\frN_{D, \wt\gamma}}q^{-\deg(E_{1}-E_{2}+E'_{1}-E'_{2})s}\eta(E_{1}-E'_{1})\eta(E_{2}-E'_{1}).
\end{equation}

(1) Since $u=0$ and $\infty$ are in the image of $\inv_{D}$, we may assume that $u\notin\{0,1,\infty\}$. For $\gamma\in G(F)$ with invariant $u$, any lifting $\wt\gamma$ of $\gamma$ in $\GL_{2}(F)$ does not lie in $\tilA$ or $w\tilA$. Therefore the previous discussion applies to $\wt\gamma$. Suppose $\BJ(\gamma,h_{D},s)\neq0$, then $\frN_{D,\wt\gamma}\neq\varnothing$. Take a point $(E_{1},E_{2},E'_{1},E'_{2})\in\frN_{D,\wt\gamma}$, the map $\det(\varphi_{\wt\gamma})$ gives an isomorphism $\calO_{X}(-E'_{1}-E'_{2})\cong\calO_{X}(-E_{1}-E_{2}+D)$. Taking $a=\varphi_{\wt\gamma,11}\varphi_{\wt\gamma,22}:\calO_{X}(-E_{1}-E_{2})\to\calO_{X}(-E'_{1}-E'_{2})$, then $a$ can be viewed as a section of $\calO_{X}(D)$ satisfying $1-a^{-1}=\inv(\gamma)$. Therefore $u=\inv(\gamma)=\inv_{D}(a)$ is in the image of $\inv_{D}$.

(2) When $u\notin\{0,1,\infty\}$, recall $\gamma(u)$ is the image of $\wt\gamma(u)=\matrixx{1}{u}{1}{1}$. Let $\calN_{\un{d}, a}$ be the fiber of $\calN_{\un{d}}$ over $a\in\calA_{D}(k)$. Let $\calN_{a}=\coprod_{\un{d}\in\Sigma_{d}}\calN_{\un{d}, a}$. We have a map
\begin{eqnarray*}
\nu_{u}: \frN_{D,\wt\gamma(u)}&\to& \calN_{a}(k)\\
(E_{1},E_{2},E'_{1},E'_{2})&\mapsto&(\calO_{X}(-E_{1}), \calO_{X}(-E_{2}), \calO_{X}(-E'_{1}), \calO_{X}(-E'_{2}), \varphi_{\wt\gamma(u)}).
\end{eqnarray*}
We show that this map is bijective by constructing an inverse. For $(\calK_{1},\calK_{2},\calK'_{1},\calK'_{2},\varphi_{ij})\in\calN_{\un{d}, a}(k)$, we may assume $\calK_{1}=\calO_{X}$ (since we mod out by the action of $\Pic_{X}$ in the end). Let $S=|\div(a)|\cup|\div(a-1)|\cup |D|$ be a finite collection of places of $X$. Then each $\varphi_{ij}$ is an isomorphism over $U=X-S$. In particular, we get isomorphisms $\varphi_{11}:\calO_{U}\cong \calK'_{1}|_{U}$, $\varphi_{21}:\calO_{U}\cong\calK'_{2}|_{U}$ and $\varphi^{-1}_{22}\varphi_{21}:\calO_{U}\cong\calK_{2}|_{U}$. Let $E'_{1}, E'_{2}$ and $E_{2}$ be the {\em negative} of the divisors of the isomorphisms $\varphi_{11}$, $\varphi_{21}$ and $\varphi^{-1}_{22}\varphi_{21}$, viewed as rational maps between line bundles on $X$. Set $E_{1}=0$. Then we have $\calK_{i}=\calO_{X}(-E_{i})$ and $\calK'_{i}=\calO_{X}(-E'_{i})$ for $i=1,2$. The map $\varphi$ guarantees that the quadruple $(E_{1}=0, E_{2},E'_{1},E'_{2})\in\frN_{D,\wt\gamma(u)}$. This gives a map $ \calN_{a}(k)\to\frN_{D,\wt\gamma(u)}$, which is easily seen to be inverse to $\nu_{u}$.

By the Lefschetz trace formula, we have
\begin{eqnarray*}
&&\sum_{\un{d}\in\Sigma_{d}}q^{(2d_{12}-d)s}\Tr\left(\Frob_{a}, \left(\bR f_{\calN_{\un{d}},*}L_{\un{d}})_{\ov{a}}\right)\right)\\
&=&\sum_{(\calK_{1},\calK_{2},\calK'_{1},\calK'_{2},\varphi)\in\calN_{a}(k)}q^{(2d_{12}-d)s}\eta(D_{11})\eta(D_{12})
\end{eqnarray*}
where $D_{ij}$ is the divisor of $\varphi_{ij}$. Moreover, under the isomorphism $\nu_{u}$, the term $q^{-\deg(E_{1}-E_{2}+E'_{1}-E'_{2})s}$ corresponds to $q^{(2d_{12}-d)s}$ where $d_{12}=\deg(D_{12})$. Therefore Part (2) follows from the bijectivity of $\nu_{u}$ and \eqref{frN gamma}.

(3) We treat the case $u=0$ (i.e.,  $a=1$), and the case $u=\infty$ is similar. Let $\frN'_{D,n_{+}}$ be the set of triples of {\em effective} divisors $(D_{11}, D_{12}, D_{22})$ such that $D_{11}+D_{22}=D$. Then we have a bijection
\begin{eqnarray*}
\frN_{D,n_{+}}&\isom & \frN'_{D,n_{+}} \\
(E_{1},E_{2},E'_{1}, E'_{2})&\mapsto & (E_{1}-E'_{1}, E_{2}-E'_{1}, E_{2}-E'_{2}).
\end{eqnarray*}
Using this bijection, we may rewrite \eqref{frN gamma} as
\begin{eqnarray}
\notag\BJ(n_{+},h_{D},s)&=&\sum_{(D_{11},D_{12}, D_{22})\in\frN'_{D,n_{+}}} q^{(2\deg(D_{12})-d)s}\eta(D_{11})\eta(D_{12})\\
\notag &=&q^{-ds}\sum_{D_{12}\geq0} q^{2s\deg(D_{12})}\eta(D_{12})\cdot\sum_{\substack{ D_{11}+D_{22}=D\\ D_{11},D_{22}\geq0} }\eta(D_{11})\\
\label{frN n+}&=&q^{-ds}L(-2s, \eta)\sum_{0\leq D_{11}\leq D}\eta(D_{11})
\end{eqnarray}
Similarly, let $\frN'_{D,n_{-}}$ be the set of triples of {\em effective} divisors $(D_{11}, D_{21}, D_{22})$ such that $D_{11}+D_{22}=D$. Then we have a bijection $\frN_{D,n_{-}}\bij\frN'_{D,n_{-}}$ and an identity
\begin{eqnarray}
\notag\BJ(n_{-},h_{D},s)&=&\sum_{(D_{11},D_{21}, D_{22})\in\frN'_{D,n_{-}}} q^{(d-2\deg(D_{21}))s}\eta(D_{21})\eta(D_{22})\\
\label{frN n-}&=&q^{ds}L(2s, \eta)\sum_{0\leq D_{22}\leq D}\eta(D_{22}).
\end{eqnarray}

We now introduce a subset $\frN^{\hs}_{D,n_{+}}\subset\frN'_{D,n_{+}}$ consisting of those $(D_{11},D_{12},D_{22})$ such that $\deg(D_{12})<d/2$; similarly we introduce $\frN^{\hs}_{D,n_{-}}\subset\frN'_{D,n_{-}}$ consisting of those $(D_{11},D_{21},D_{22})$ such that $\deg(D_{21})\leq d/2$. Then the same argument as Part (2) gives a bijection
\begin{equation*}
\nu_{n_{\pm}}:\frN^{\hs}_{D,n_{+}}\coprod \frN^{\hs}_{D,n_{-}}\isom \calN_{a}(k):=\coprod_{\un{d}\in\Sigma_{d}}\calN_{\un{d},a}(k).
\end{equation*}
Here the degree constraints $\deg(D_{12})< d/2$ or $\deg(D_{21})\leq d/2$ come from the last condition in the definition of $\calN_{\un{d}}$ in \S\ref{sss:calN}.
 
Using the Lefschetz trace formula, we get
\begin{eqnarray}\notag
&&\sum_{\un{d}\in\Sigma_{d}}q^{(2d_{12}-d)s}\Tr\left(\Frob_a,\cohoc{*}{\calN_{\un{d},a}\otimes_{k}\kbar, L_{\un{d}}}\right)\\
\notag&=&\sum_{(D_{11},D_{12},D_{22})\in\frN^{\hs}_{D,n_{+}}} q^{(2\deg(D_{12})-d)s}\eta(D_{11})\eta(D_{12})\\
\notag&&+\sum_{(D_{11},D_{21},D_{22})\in\frN^{\hs}_{D,n_{-}}} q^{(d-2\deg(D_{21}))s}\eta(D_{21})\eta(D_{22})\\
\label{trace n+}&=&q^{-ds}\sum_{D_{12}\geq0,\deg(D_{12})<d/2} q^{2\deg(D_{12})s}\eta(D_{12})\sum_{0\leq D_{11}\leq D}\eta(D_{11})\\
\label{trace n-}&&+q^{ds}\sum_{D_{21}\geq0,\deg(D_{21})\leq d/2} q^{-2\deg(D_{21})s}\eta(D_{21})\sum_{0\leq D_{22}\leq D}\eta(D_{22}).
\end{eqnarray}
The only difference between the term in \eqref{trace n+} and the RHS of \eqref{frN n+} is that we have restricted the range of the summation to effective divisors $D_{12}$ satisfying $\deg(D_{12})<d/2$. However, since $\eta$ is a nontrivial id\`ele class character, the Dirichlet $L$-function $L(s,\eta)=\sum_{E\geq 0}q^{-\deg(E)s}\eta(E)$ is a polynomial in $q^{-s}$ of degree $2g-2<d/2$.  Therefore\eqref{trace n+} is the same as \eqref{frN n+}. Similarly, \eqref{trace n-} is the same as \eqref{frN n-}. We conclude that 
\begin{eqnarray}\label{Npm}
\sum_{\un{d}\in\Sigma_{d}}q^{(2d_{12}-d)s}\Tr\left(\Frob_a,\cohoc{*}{\calN_{\un{d},a}\otimes_{k}\kbar, L_{\un{d}}}\right)=\BJ(n_{+},h_{D},s)+\BJ(n_{-},h_{D},s).
\end{eqnarray}

Finally, observe that
\begin{equation}\label{van orb}
\BJ(1,h_{D},s)=0
\end{equation}
because $\eta$ restricts nontrivially to the centralizer of $\gamma=1$. 
Putting together \eqref{Npm} and the vanishing \eqref{van orb}, we get \eqref{u 0 infty}.
\end{proof}

\begin{cor}\label{c:Jr geom} For $D\in X_{d}(k)$ and $u\in\PP^{1}(F)-\{1\}$, we have
\begin{equation*}
\BJ_{r}(u, h_{D})=\begin{cases}(\log q)^{r}\sum_{\un{d}\in\Sigma_{d}}(2d_{12}-d)^{r}\Tr\left(\Frob_{a}, \left(\bR f_{\calN_{\un{d}},*}L_{\un{d}}\right)_{\ov{a}}\right) & \textup{if } u=\inv_{D}(a), a\in \calA_{D}(k);\\
0 & \textup{otherwise.}
\end{cases}
\end{equation*}
\end{cor}

\section{Analytic spectral decomposition}\label{s:a spec}

In this section we express the spectral side of the relative trace formula in \S\ref{s:RTF} in terms of  automorphic $L$-functions.

\subsection{The Eisenstein ideal}  \label{ss:eis}
Consider the Hecke algebra $\sH=\otimes_{x\in |X|}\sH_{x}$. We also consider the Hecke algebra $\sH_{A}$ for the diagonal torus $A=\Gm$ of $G$. Then $\sH_{A}=\otimes_{x\in |X|}\sH_{A,x}$ with $\sH_{A,x}=\QQ[F^{\times}_{x}/\calO_{x}^{\times}]=\QQ[t_{x}, t_{x}^{-1}]$, and $t_{x}$ stands for the characteristic function of $\varpi_{x}^{-1}\calO^{\times}_{x}$, where $\varpi_{x}$ is a uniformizer of $F_{x}$.  \index{$\sH$, Hecke algebra}%
 \index{$\sH_A, \sH_{A,x}$}%

Recall we have a basis $\{h_{D}\}$ for $\sH$ indexed by effective divisors $D$ on $X$. For fixed $x\in |X|$, $h_{x}\in\sH_{x}$ and $\sH_{x}\cong\QQ[h_{x}]$ is a polynomial algebra with generator $h_{x}$.\index{$\sH_x, h_{nx}, h_D$}%
 
\subsubsection{The Satake transform} To avoid  introducing $\sqrt{q}$, we normalize the Satake transform  in the following way
\begin{eqnarray*}
\Sat_{x}: \sH_{x}&\to& \sH_{A,x}\\
h_{x} &\mapsto& t_{x}+q_{x}t^{-1}_{x}
\end{eqnarray*}
where $q_{x}=\#k_{x}$. Consider the involution $\iota_{x}$ on $\sH_{A,x}$ sending $t_{x}$ to $q_{x}t^{-1}_{x}$. Then $\Sat_{x}$ identifies $\sH_{x}$ with the subring of $\iota_{x}$-invariants of $\sH_{A,x}$. This normalization of the Satake transform is designed to make it compatible with constant term operators, see Lemma \ref{l:horo action}.  Let 
\begin{equation*}
\Sat: \sH\to \sH_{A}\index{$\Sat, \Sat_{x}$}%
\end{equation*}
be the tensor product of all $\Sat_{x}$.

\subsubsection{}\label{sss:Pic inv} We have natural homomorphisms between abelian groups: 
\begin{equation*}
\xymatrix{    \BA^{\times}/\OO^{\times} \ar[d] \ar[r]^{\simeq} &   \Div(X)  \ar[d]\\
 F^{\times}\backslash\BA^{\times}/\OO^{\times} \ar[r]^{\simeq}& \Pic_{X}(k)  .}
\end{equation*} 
In particular, the top row above gives a canonical isomorphism $\sH_{A}=\QQ[\BA^{\times}/\OO^{\times}]\cong \QQ[\Div(X)]$, the group algebra of $\Div(X)$. 

Define an involution $\iota_{\Pic}$ on $\QQ[\Pic_{X}(k)]$ by
\begin{equation*}
\iota_{\Pic}(\one_{\CL})=q^{\deg\CL}\one_{\CL^{-1}}.
\end{equation*}
Here $\one_{\CL}\in\QQ[\Pic_{X}(k)]$ is the characteristic function of the point $\CL\in\Pic_{X}(k)$.  Since the action of $\otimes_{x}\iota_{x}$ on $\sH_{A}\cong\QQ[\Div(X)]$ is compatible with the involution $\iota_{\Pic}$ on $\QQ[\Pic_{X}(k)]$ under the projection $\QQ[\Div(X)]\to \QQ[\Pic_{X}(k)]$, we see that the image of the composition
\begin{equation*}
\sH\xrightarrow{\Sat}\sH_{A}\cong\QQ[\Div(X)]\surj \QQ[\Pic_{X}(k)] 
\end{equation*}
lies in the $\iota_{\Pic}$-invariants. Therefore the above composition gives a ring homomorphism
\begin{equation}\label{eqn I eis}
a_{\Eis}:\sH\to \QQ[\Pic_{X}(k)]^{\iota_{\Pic}}=:\sH_{\Eis}.\index{$a_{\Eis}$} \index{$\sH_{\Eis}$}%
\end{equation}

\begin{defn}\label{def:Eis} We define the {\em Eisenstein ideal} $\calI_{\Eis}\subset\sH$ to be the kernel of the ring homomorphism $a_{\Eis}$ in \eqref{eqn I eis}.\index{$\calI_{\Eis}$}%
\end{defn}

The ideal $\calI_{\Eis}$ is the analog of the Eisenstein ideal of Mazur in the function field setting. Taking the spectra we get a morphism of affine schemes
\begin{equation*}
\Spec(a_{\Eis}): Z_{\Eis}:=\Spec\QQ[\Pic_{X}(k)]^{\iota_{\Pic}}\to \Spec\sH. \index{$Z_{\Eis}$}%
\end{equation*}

\begin{lemma}\label{l:Eis fin gen}
\begin{enumerate}
\item For any $x\in |X|$, under the ring homomorphism $a_{\Eis}$, $\QQ[\Pic_{X}(k)]^{\iota_{\Pic}}$ is finitely generated as an $\sH_{x}$-module. 
\item The map $a_{\Eis}$ is surjective, hence $\Spec(a_{\Eis})$ is a closed embedding. \footnote{This result is not used in an essential way in the rest of paper.}
\end{enumerate}
\end{lemma}
\begin{proof} (1) We have an exact sequence $0\to \Jac_{X}(k)\to \Pic_{X}(k)\to \ZZ\to 0$ with $\Jac_{X}(k)$ finite. Let $x\in|X|$, then the map $\ZZ\to \Pic_{X}(k)$ sending $n\mapsto \calO_{X}(nx)$ has finite cokernel since $\Jac_{X}(k)$ is finite. Therefore $\QQ[\Pic_{X}(k)]$ is finitely generated as a $\sH_{A,x}\cong \QQ[t_{x},t^{-1}_{x}]$-module. On the other hand, via $\Sat_{x}$, $\sH_{A,x}$ is a finitely generated $\sH_{x}$-module (in fact a free module of rank two over $\sH_{x}$). Therefore $\QQ[\Pic_{X}(k)]$ is a finitely generated module over the noetherian ring $\sH_{x}$, hence so is its submodule $\QQ[\Pic_{X}(k)]^{\iota_{\Pic}}$.  

(2) For proving surjectivity we may base change the situation to $\Qlbar$.  Let $\frZ_{\Eis}=\Spec\Qlbar[\Pic_{X}(k)]^{\iota_{\Pic}}$, and we still use $\Spec(a_{\Eis})$ to denote $\frZ_{\Eis}\to \Spec\sH_{\Qlbar}$. We first check that $\Spec(a_{\Eis})$ is injective on $\Qlbar$-points. Identifying $\Pic_{X}(k)$ with the abelianized Weil group $W(X)^{\ab}$ via class field theory, the set $\frZ_{\Eis}(\Qlbar)$ are in natural bijection with Galois characters $\chi: W(X)\to \Qlbar^{\times}$ up to the equivalence relation $\chi\sim \chi^{-1}(-1)$ (where $(-1)$ means Tate twist). Suppose $\chi_{1}$ and $\chi_{2}$ are two such characters that pullback to the same homomorphism $\sH\to \Qlbar[\Pic_{X}(k)]\xrightarrow{\chi_{i}} \Qlbar$, then $\chi_{1}(a_{\Eis}(h_{x}))=\chi_{1}(\Frob_{x})+q_{x}\chi_{1}(\Frob^{-1}_{x})=\chi_{2}(\Frob_{x})+q_{x}\chi_{2}(\Frob^{-1}_{x})=\chi_{2}(a_{\Eis}(h_{x}))$  for all $x$. Consider the two-dimensional representation $\rho_{i}=\chi_{i}\oplus\chi^{-1}_{i}(-1)$  of $W(X)$. Then $\Tr(\rho_{1}(\Frob_{x}))=\Tr(\rho_{2}(\Frob_{x}))$ for all $x$. By Chebotarev density, this implies that $\rho_{1}$ and $\rho_{2}$ are isomorphic to each other (since they are already semisimple). Therefore either $\chi_{1}=\chi_{2}$ or $\chi_{1}=\chi^{-1}_{2}(-1)$. In any case $\chi_{1}$ and $\chi_{2}$ define the same $\Qlbar$-point of $\frZ_{\Eis}$. We are done.

Next, we show that $\Spec(a_{\Eis})$ is injective on tangent spaces at $\Qlbar$-points. Let $\tfZ_{\Eis}=\Spec\Qlbar[\Pic_{X}(k)]$. Then $\tfZ_{\Eis}$ is a disjoint union of components indexed by characters $\chi_{0}: \Jac_{X}(k)\to \Qlbar^{\times}$, and each component is a torsor under $\Gm$. The scheme $\frZ_{\Eis}$ is the quotient $\tfZ_{\Eis}\sslash \jiao{\iota_{\Pic}}$. For a character $\chi: \Pic_{X}(k)\to\Qlbar^{\times}$ with restriction  $\chi_{0}$ to $\Jac_{X}(k)$, we may identify its component $\tfZ_{\chi_{0}}$ with $\Gm$ in such a way that $s\in \Gm$ corresponds to the character $\chi\cdot s^{\deg}:\Pic_{X}(k)\to \Qlbar^{\times}, \calL\mapsto \chi(\calL)s^{\deg\calL}$. The map $\Spec(a_{\Eis})$ pulled back to $\tfZ_{\chi_{0}}$ then gives a morphism
\begin{equation*}
b: \Gm\cong\tfZ_{\chi_{0}}\to \frZ_{\Eis}\to \Spec\sH_{\Qlbar}\cong \BA^{|X|}
\end{equation*} 
given by the formula
\begin{equation}\label{formula b}
\Gm\ni s\mapsto  \left(\chi(t_{x})s^{d_{x}}+q_{x}\chi(t^{-1}_{x})s^{-d_{x}}\right)_{x\in |X|}
\end{equation}
where $d_{x}=[k_{x}:k]$. The derivative $\frac{db}{ds}$ at $s=1$ is then the vector $(d_{x}(\chi(t_{x})-q_{x}\chi(t_{x}^{-1})))_{x\in|X|}$. This is identically zero only when $\chi(t_{x})=\pm q^{1/2}_{x}$ for all $x$, hence if and only if $\chi^{2}=q^{\deg}=\Qlbar(-1)$. Therefore when $\chi^{2}\neq\Qlbar(-1)$, we have proved that the tangent map of $b$ at $s=1$ is nonzero, hence a fortiori the tangent map of $\Spec(a_{\Eis})$ at the image of $\chi$ is nonzero. If $\chi^{2}=\Qlbar(-1)$, $\chi$ is a fixed point under $\iota_{\Pic}$. The component $\tfZ_{\chi_{0}}$ is then stable under $\iota_{\Pic}$ which acts by $s\mapsto s^{-1}$, and its image $\frZ_{\chi_{0}}\subset \frZ_{\Eis}$ is a component isomorphic to $\BA^{1}$ with affine coordinate $z=s+s^{-1}$. Therefore we may factor $b$ into two steps
\begin{equation*}
b: \tfZ_{\chi_{0}}\cong\Gm\xrightarrow{s\mapsto z=s+s^{-1}}\BA^{1}\cong \frZ_{\chi_{0}}\xrightarrow{c}\Spec\sH_{\Qlbar} \cong \BA^{|X|}
\end{equation*}
where $c$ is the restriction of $\Spec(a_{\Eis})$ to $Z_{\chi_{0}}$. By chain rule we have $\frac{dc}{dz}\frac{dz}{ds}=\frac{db}{ds}$. Using this we see that the derivative $\frac{dc}{dz}$ at $z=s+s^{-1}$ is the vector
\begin{equation*}
\left(d_{x}\chi(t_{x})\frac{s^{d_{x}}-s^{-d_{x}}}{s-s^{-1}}\right)_{x\in |X|}
\end{equation*}
(using that $\chi(t_{x})=q_{x}\chi(t^{-1}_{x})$).  Evaluating at $s=1$ we get the vector $(\chi(t_{x})d^{2}_{x})_{x\in |X|}$, which is nonzero. We have checked that the tangent map of $\Spec(a_{\Eis})$ is also injective at the image of those points $\chi\in\tfZ_{\Eis}(\Qlbar)$ such that $\chi^{2}=\Ql(-1)$.  Therefore the tangent map of $\Spec(a_{\Eis})$ is injective at all $\Qlbar$-points. Combining the two injectivity results we conclude that $\Spec(a_{\Eis})$ is a closed immersion and hence $a_{\Eis}$ is surjective.
\end{proof}

\subsection{Spectral decomposition of the kernel function}
Recall that we have defined the automorphic kernel function by \eqref{eqn kernel}.
For a cuspidal automorphic representation $\pi$ (in the usual sense, i.e., an irreducible sub-representation of the $\BC$-values automorphic functions), we define the $\pi$-component of the kernel function as (cf. \cite[\S7.1(1)]{J86})
\begin{align}\label{eqn K pi}
\BK_{f,\pi}(x,y)=\sum_{\phi}\pi(f)\phi(x)\ov{\phi(y)},	\index{$\BK_{f,\pi}$}%
\end{align}
where the sum runs over an orthonormal basis $\{\phi\}$ of $\pi$. The cuspidal kernel function is defined as
\begin{align}\label{eqn sum Jpi}
\BK_{f,{\rm cusp}}=\sum_{\pi} \BK_{f,\pi},	\index{$\BK_{f,{\rm cusp}}$}%
\end{align}
where the sum runs over all cuspidal automorphic representations $\pi$ of $G$.  Note that this is a finite sum. 

Similarly, we define the special (residual) kernel function (cf. \cite[\S7.4]{J86})
\begin{align*}
\BK_{f,{\rm sp}}(x,y):=\sum_{\chi}\pi(f)\chi(x)\ov{\chi(y)},		\index{$\BK_{f,{\rm sp}}$}%
\end{align*}
where the sum runs over all one-dimensional automorphic representations $\pi=\chi$, indeed solely characters of order two:
\begin{equation*}
\xymatrix{   \chi\colon G(\BA)\ar[r] & F^\times\bs \BA^\times/(\BA^\times)^2\ar[r]&  \{\pm 1\}.}
\end{equation*}

\begin{thm}\label{thm K eis=0}Let $f\in\CI_{\Eis}$ be in the Eisenstein ideal $\CI_{\Eis}\subset \sH$. Then we have
\[
\BK_f=\BK_{f,{\rm cusp}}+\BK_{f,{\rm sp}}.
\]
\end{thm}

\begin{proof}
To show this, we need to recall the Eisenstein series (cf. \cite[\S8.4]{J86}). We fix an $\alpha\in \BA^\times$ with valuation one and we then have a direct product
$$\BA^\times=\BA^1\times \alpha^\BZ.$$ 
For a character $\chi:F^\times\bs \BA^1\to\BC^\times$, we extend it as a character of $F^\times\bs \BA^\times$, by demanding $\chi(\alpha)=1$. Moreover, we define a character for any $u\in\BC$
\begin{align*}
   \begin{gathered}
	\chi_u  \colon
	\xymatrix@R=0ex{
	  \BA^\times \ar[r]  & \BC^\times  \\
	a  \ar@{|->}[r]  & \chi(a)|a|^u
	}.
	\end{gathered}
\end{align*} 
We also define
\[
   \begin{gathered}
	\delta_B \colon
	\xymatrix@R=0ex{
	  B(\BA)   \ar[r]  & \BA^\times \\
	\mbox{$\left[ \begin{array}{cc}
 a &  b  \\
  &  d 
 \end{array}\right] $  }\ar@{|->}[r]  & a/d
	}
	\end{gathered},
\qquad\text{and}\qquad	
   \begin{gathered}
\chi\colon
	\xymatrix@R=0ex{
	  B(\BA) \ar[r]  & \BC^\times\\ 	 b
 \ar@{|->}[r]  &\chi(a/d)
	}
	\end{gathered}.
\]

For $u\in \BC$, the (induced) representation $\rho_{\chi,u}$ of $G(\BA)=\PGL_2(\BA)$ is defined to be the right translation on the space $V_{\chi,u}$
of smooth functions $$\phi: G(\BA)\to \BC$$ such that 
$$
\phi\left(b g\right)=\chi \left(b \right)|\delta_B(b)|^{u+\frac{1}{2}}\phi(g),\quad b\in B(\BA),\quad g\in G(\BA).
$$Note that we have
$
\rho_{\chi,u}=\rho_{\chi,u+\frac{2\pi i}{\log q}}.
$
By restriction to $K$, the space $V_{\chi,u}$ is canonically identified with the space
of smooth functions 
$$
V_{\chi}\colon=\biggl\{\phi: K\to \BC, \mbox{ smooth}\biggm|\phi\left(b k\right)=\chi\left(b \right)\phi(k),\quad b \in K\cap B(\BA) \biggr\}.
$$
This space is endowed with a natural inner product 
\begin{align}\label{eqn:inner prod}
(\phi,\phi')=\int_{K}\phi(k)\ov{\phi'(k)}dk.
\end{align}
Let $\phi\in V_\chi$, we denote by $\phi(g,u,\chi)$ the corresponding function in $V_{\chi,u}$, i.e.,
$$
\phi(g,u,\chi)=\chi\left(b \right)\Big|\delta_B(b)\Big|^{u+\frac{1}{2}}\phi(k)
$$
if we write $g=b k$ where $b\in B(\BA), k\in K$.

For $\phi\in V_\chi$,
the Eisenstein series is defined as (the analytic continuation of)
$$
E(g,\phi,u,\chi)=\sum_{\gamma\in B(F)\bs G(F)}\phi(\gamma g,u,\chi).
$$
Let $\{\phi_i\}_{i}$ be an orthonormal basis of the Hermitian space $V_{\chi}$. We define
\begin{align}\label{eqn K eis chi}
\BK_{f,\Eis,\chi}(x,y):=\frac{\log q}{2\pi i}\,\sum_{  i,j}\int_0^{\frac{2\pi i}{\log q}} (\rho_{\chi,u}(f)\phi_i,\phi_j)\, E(x,\phi_i,u,\chi)\ov {E(y,\phi_j,u,\chi)}\,du ,
\end{align}
where the inner product is given by \eqref{eqn:inner prod} via the identification $V_{\chi,u}\simeq V_\chi$. We set (cf., \cite[\S8.4]{J86})
\begin{align}\label{eqn K eis}
\BK_{f,{\rm Eis}}:=\sum_{\chi} \BK_{f,{\rm Eis},\chi}  \index{$\BK_{f,{\rm Eis}}$}%
\end{align}
where the sum runs over all characters $\chi$ of $F^\times\bs \BA^1$. Since our test function $f$ is in the spherical Hecke algebra $\sH$, for $\BK_{f,{\rm Eis},\chi}$ to be nonzero, the character $\chi$ is necessarily unramified everywhere. Therefore the sum over $\chi$ is in fact finite.

By \cite[\S7.1(4)]{J86}, we have a spectral decomposition of the automorphic kernel function $\BK_f$ defined by \eqref{eqn kernel}
\begin{align}\label{eqn K decomp}
\BK_f=\BK_{f,{\rm cusp}}+\BK_{f,{\rm sp}}+\BK_{f,\Eis}.
\end{align}
Therefore it remains to show that $\BK_{f,\Eis}$ vanishes if $f$ lies in the Eisenstein ideal $\CI_{\Eis}$.

We may assume that $\chi$ is unramified.  Then we have 
\begin{align}\label{eqn K eis chi 1}
\BK_{f,\Eis,\chi}(x,y)=\frac{\log q}{2\pi i}\int_0^{\frac{2\pi i}{\log q}} (\rho_{\chi,u}(f)\phi,\phi)\, E(x,\phi,u,\chi)\ov {E(y,\phi,u,\chi)}\,du,
\end{align}
where $\phi=\one_{K}\in V_\chi$ (we are taking the Haar measure on $G(\BA)$ such that $\vol(K)=1$).

Recall that the Satake transform $\Sat$ has the property that, for all unramified characters $\chi$, and all $u\in \BC$, we have 
$$
\tr \rho_{\chi,u}(f)=\chi_{u+1/2}(\Sat(f)),
$$
where we extend $\chi_{u+1/2}$ to a homomorphism $\sH_{A,\BC}\simeq \BC[\Div(X)]\to \BC$. 
Since $\chi_u: A(\BA)/(A(\BA)\cap K)\simeq \Div(X)\to \BC^\times$ factors through $\Pic_X(k)$, we have
$$
\tr \rho_{\chi,u}(f)=\chi_{u+1/2}(a_{\Eis}(f)),
$$
Then we may rewrite \eqref{eqn K eis chi 1} as
\begin{align*}
\BK_{f,\Eis,\chi}(x,y)=\frac{\log q}{2\pi i}\int_0^{\frac{2\pi i}{\log q}} \chi_{u+1/2}(a_{\Eis}(f)) \, E(x,\phi,u,\chi)\ov {E(y,\phi,u,\chi)}\,du.
\end{align*}
In particular, if $f$ lies in the Eisenstein ideal, then $a_{\Eis}(f)=0$, and hence the integrand vanishes.
This completes the proof.
\end{proof}

\subsection{The cuspidal kernel}

Let $\pi$ be a cuspidal automorphic representation of $G(\BA)$, endowed with the natural Hermitian form given by the Petersson inner product:
\begin{align}\label{eqn Pet}
\langle\phi,\phi'\rangle_{{\rm Pet}}:=\int_{[G]}\phi(g)\ov{\phi'(g)}dg,\quad \phi,\phi'\in\pi. 
\end{align}
We abbreviate the notation to $\pair{\phi,\phi'}$.
For a character $\chi:F^\times\bs\BA^\times\to \BC^\times$, the $(A,\chi)-$period integral for $\phi\in \pi$ is defined as
\begin{align}\label{eqn P-chi}
\sP_\chi(\phi,s):=\int_{[A]}\phi(h)\chi(h)\big|h\big |^s\, dh. \index{$\sP_\chi(\phi,s)$}%
\end{align}
 We simply write $\sP(\phi,s)$ if $\chi=\one$ is trivial. This
is absolutely convergent for all $s\in\BC$. 

 The spherical character (relative to $(A\times A, 1\times \eta)$) associated to $\pi$ is a distribution on $G(\BA)$ defined by
\begin{align}\label{eqn dist J pi}
\BJ_{\pi}(f,s)=\sum_{\phi}\frac{ \sP(\pi(f)\phi,s)\sP_\eta(\ov{\phi},s)}{\pair{\phi,\phi}},\quad f\in C_c^\infty(G(\BA)),
\end{align}
where the sum runs over an orthogonal basis $\{\phi\}$  of $\pi$. This is a finite sum, and the result is independent of the choice of the basis.

\begin{lem}\label{l:J Eis ideal} Let $f$ be a function in the Eisenstein ideal $\CI_{\Eis}\subset \sH$.
Then we have
$$
\BJ(f,s)=\sum_{\pi }\,\BJ_\pi(f,s),
$$
where the sum runs over all (everywhere unramified) cuspidal automorphic representations $\pi$ of $G(\BA)$.
\end{lem}
\begin{proof}
For $\ast= {\rm cusp}, {\rm sp}$ or $\pi$, we define $
\BJ_{\ast}(f,s)$ by replacing $\BK_f$ by $\BK_{f,\ast}$ in both \eqref{eqn J c} and \eqref{eqn J(f,s)}. To make sense of this, we need to show the analogous statements to Proposition \ref{prop J n1n2}. When $\ast={\rm sp}$, 
we note that, for any character $\chi:\BA^\times\to\BC^\times$, one of $\chi$ and $\chi\eta$ must be nontrivial on $\BA^1$. It follows that for any $(n_1,n_2)\in\BZ^2$ we have
$$
\int_{[A]_{n_1}\times[A]_{n_2}}\chi(h_1)\chi^{-1}(h_2)\lvert h_1h_2\rvert^s \eta(h_2)\,dh_1\,dh_2=0.
$$
Consequently we have $$\BJ_{{\rm sp}}(f,s)=0.$$
When $\ast=\pi$, we need to show that, for any $\phi\in\pi$, the following integral vanishes if $|n|\gg 0$
$$
\int_{[A]_n}\phi(h)\chi(h)\big|h\big |^s\, dh.
$$
 But this follows from the fact that $\phi$ is cuspidal, particularly $\phi(h)=0$ if $h\in [A]_{n}$ and $|n|\gg 0$. This also shows that this definition of $\BJ_\pi(f,s)$ coincides with \eqref{eqn dist J pi}. 
The case $\ast={\rm cusp}$ follows from the case for $\ast=\pi$ and the finite sum decomposition \eqref{eqn sum Jpi}. We then have
$$
\BJ_{{\rm cusp}}(f,s)=\sum_{\pi}\BJ_\pi(f,s).
$$
The proof is complete, noting that, by Theorem \ref{thm K eis=0}, we have 
\begin{align*}
\BJ(f,s)=\BJ_{{\rm cusp}}(f,s)+\BJ_{{\rm sp}}(f,s).
\end{align*}
\end{proof}

\begin{prop}\label{p:Jpi}
Let $\pi$ be a cuspidal automorphic representation of $G(\BA)$, unramified everywhere. Let $\lambda_\pi: \sH\to \BC$ be the homomorphism associated to $\pi$. Then we have
\begin{align*}
\BJ_{\pi}(f,s)=\frac{1}{2}\,|\omega_X|\, \sL(\pi_{F'},s+1/2)\lambda_\pi(f).\index{$\BJ_\pi(f,s)$}%
\end{align*}
\end{prop}
\begin{proof}
Write  $\pi=\otimes_{x\in|X|}\pi_x$  and  let $\phi$ be a nonzero vector in the one-dimensional space $\pi^K$.  Since $f\in \sH$ is bi-$K$-invariant, the sum in \eqref{eqn dist J pi} is reduced to one term
\begin{align}\label{eqn J vol(K)}
\BJ_{\pi}(f,s)=\frac{ \sP(\phi,s)\sP_\eta(\ov{\phi},s)}{\pair{\phi,\phi}_{{\rm Pet}}}\lambda_\pi(f)\vol(K),
\end{align}
where we may choose any measure on $G(\BA)$, and then define the Petersson inner product using the same measure. We will choose the Tamagawa measure on $G(\BA)$ in this proof. To decompose the Tamagawa measure into local measures, we fix a nontrivial additive character associated to a nonzero meromorphic differential form $c$ on $X$:
 \[
   \begin{gathered}
	\psi \colon
	\xymatrix@R=0ex{
	  \BA   \ar[r]  & \BC^\times 
	}
	\end{gathered}.
	\]
We note that the character $\psi$ is defined by $\psi (a)=\psi_{\BF_p}\left(\sum _{x\in |X|}\Tr _{k_x/\BF_p} \left(\Res_x (ca)\right)\right)$
 where $\psi_{\BF_p} $ is a fixed nontrivial character  $\BF_p\to \BC^\times.$

We decompose $\psi=\prod_{x\in|X|}\psi_x$ where $\psi_x$ is a character of $F_x$. This gives us a self-dual measure $dt=dt_{\psi_x}$ on $F_x$, a measure $d^\times t=\zeta_x(1)\frac{dt}{|t|}$ on $F_x^\times$, and the product measure on $\BA^\times$. We then choose the Haar measure $dg_x=\zeta_{x}(1)|\det(g_x)|^{-2}\prod_{1\leq i,j\leq 2}dg_{ij}$ on $\GL_2(F_x)$ where $g_x=(g_{ij})\in\GL_2(F_x)$. The measure on $G(F_x)$ is then the quotient measure, and the Tamagawa measure on $G(\BA)$ decomposes $dg=\prod_{x\in |X|}\,dg_x$. 
Note that under such a choice of measures, we have
\begin{align}\label{eqn meas1}
&\vol(\BO^\times)=\vol(\BO)=\lvert \omega_X\rvert^{1/2},\\ \label{eqn meas2}
&\vol(K)=\zeta_F(2)^{-1}\vol(\BO)^3=\zeta_F(2)^{-1}\lvert \omega_X\rvert^{3/2}.
\end{align}

To compute the period integrals, we use the Whittaker models with respect to the character $\psi$. Denote the Whittaker model of $\pi_x$ by $W_{\psi_x}$. Write the $\psi$-Whittaker coefficient $W_\phi$ as a product $\otimes_{x\in|X|} W_x$, where $W_x\in W_{\psi_x}$.

Let $L(\pi_x\times\wt\pi_x,s)$, resp. $L(\pi\times\wt\pi,s)$ denote the local, resp. global Rankin--Selberg L-functions.
By  \cite[Prop. 3.1]{Z14} there are invariant inner products $\theta^\nat_x$ on the Whittaker models $W_{\psi_x}$ 
$$
\theta^\nat_x(W_x,W_x'):=\frac{1}{L(\pi_x\times\wt\pi_x,1)}\int_{F_x^\times}W_x\left(\matrixx{a}{}{}{1}\right)\ov {W'}_{x}\left(\matrixx{a}{}{}{1}\right)\,d^\times a,
$$
 such that
$$
\pair{\phi,\phi}_{{\rm Pet}}=2\frac{\Res_{s=1} L(\pi\times\wt\pi,s)}{\vol(F^\times\bs \BA^1)}\prod_{x\in|X|}\theta_x^\nat(W_{x}, W_x).
$$
Note that 
$$
\Res_{s=1} L(\pi\times \wt\pi,s)=L(\pi,\Ad,1) \,\Res_{s=1} \zeta_F(s)=L(\pi,\Ad,1)\,\vol(F^\times\bs \BA^1).
$$
Hence we have 
$$
\pair{\phi,\phi}_{{\rm Pet}}=2 L(\pi,\Ad,1)\prod_{x\in|X|}\theta_x^\nat(W_{x}, W_x).
$$
Moreover, when $\psi_x$ is unramified,  we have $\theta_x^\nat(W_{x}, W_x)=\vol(K_x)=\zeta_x(2)^{-1}$ (cf. {\em loc. cit.}).

In \cite[Prop. 3.3]{Z14} there are linear functionals $\lambda^\nat_x$ on the Whittaker models $W_{\psi_x}$
$$
\lambda^\nat_x(W_x,\chi_x,s):=\frac{1}{L(\pi_x\otimes\chi_x,s+1/2)}\int_{F_x^\times}W_x\left(\matrixx{a}{}{}{1}\right)\chi_x(a)|a|^s\,d^\times a
$$ such that
 $$
\sP_\chi(\phi,s)=L(\pi\otimes\chi,s+1/2)\prod_{x\in|X|}\lambda^\nat_x(W_{x}, \chi_x, s).
$$
While in {\em loc. cit.} we only treated the case $s=0$, the same argument goes through. 
Moreover, when $\psi_x$ and $\chi_x$ are unramified,  we have $\lambda^\nat_x=1$.

We now have 
\begin{align}\label{eqn lambda}
\frac{ \sP(\phi,s)\sP_\eta(\ov{\phi},s)}{\pair{\phi,\phi}_{{\rm Pet}}}=\lvert \omega_X\rvert^{-1}\,\frac{L(\pi_{F'},s+1/2)}{2L(\pi,\Ad,1)}\prod_{x\in |X|} \xi_{x,\psi_x}(W_x,\eta_x,s),
\end{align}
where the constant $|\omega_X|^{-1}$ is caused by the choice of measures (cf. \eqref{eqn meas1}), and the local term at a place $x$ is
\begin{align}\label{eqn xi}
\xi_{x,\psi_x}(W_x,\eta_x,s):=
\frac{\lambda_x^\nat( W_{x}, {\bf1}_x, s)\lambda_x^\nat(\ov W_{x}, \eta_x, s)}{\theta_x^\nat(W_x,W_x)}.
\end{align}
Note that the local term $\xi_{x,\psi_x}$ is now independent of the choice of the nonzero vector $W_x$ in the one-dimensional space $W_{\psi_x}^{K_x}$. We thus simply write it as
\begin{align*}
\xi_{x,\psi_x}(\eta_x,s):=\xi_{x,\psi_x}(W_x,\eta_x,s).
\end{align*}
When $\psi_x$ is unramified, we have 
\begin{align*}
\xi_{x,\psi_x}(\eta_x,s)=\zeta_x(2).
\end{align*}

We want to know how $\xi_{x,\psi_x}$ depends on $\psi_x$. Let $ c_x\in F_x^\times$,  and denote by $\psi_{x,c_x}$ the twist $\psi_{x,c_x}(t)=\psi_x(c_x t)$.

\begin{lem}\label{l:xi} For any unramified character $\chi_x$ of $F_x^\times$, we have
\begin{align*}
\xi_{x,\psi_{x,c_x}}(\chi_x,s)=\chi^{-1}(c_x)|c_x|^{-2s+1/2}\xi_{x,\psi_{x}}(\chi_x,s).
\end{align*}
\end{lem}
\begin{proof}The self-dual measure on $F_x$ changes according to the following rule
$$
dt_{\psi_{x,c_x}}=|c_x|^{1/2}\,  dt_{\psi_{x}}.
$$
Then the multiplicative measure on $F_x^\times $ changes by the same multiple. Now we compare $\xi_{x,\psi_x}$ and $\xi_{x,\psi_{x,c_x}}$ using the same measure on $F_x^\times$ to define the integrals. 
 
There is a natural isomorphism between the Whittaker models $W_{\psi_x}\simeq W_{\psi_{x,c_x}}$, preserving the natural inner product $\theta^\nat_x$.  We write $\lambda^\nat_{\psi_{x}}$ to indicate the dependence on $\psi_x$. Then we have for  any character $\chi_x:F_x^\times\to\BC^\times$:
$$
\lambda^\nat_{\psi_{x,c_x}}( W_{x}, \chi_x, s)=\chi^{-1}(c_x)|c_x|^{-s}\lambda^\nat_{\psi_{x}}( W_{x}, \chi_x, s).
$$
This completes the proof of Lemma \ref{l:xi}.
\end{proof}

Let $\psi_x$  have conductor $c_x^{-1}\CO_{x}$. Then the id\`ele class of $(c_x)_{x\in|X|}$ in $\Pic_X(k)$ is the class of $\div(c)$ and hence the class of $\omega_X$. Hence we have 
$$
\prod_{x\in|X|}\lvert c_x\rvert=\lvert \omega_X\rvert=q^{-\deg\omega_X }=q^{-(2g-2)}.\index{$\omega_X$}%
$$
 This shows that the product in \eqref{eqn lambda} is equal to 
\begin{align*}
\prod_{x\in|X|}\xi_{x,\psi_x}(\eta_x,s)&=\eta(\omega_X)\prod_{x\in|X|}\zeta_{x}(2)\lvert c_x\rvert^{-2s+1/2}
\\&=\eta(\omega_X)\lvert \omega_X\rvert^{1/2}\,\zeta_F(2)\,q^{4(g-1)s}.
\end{align*}
We claim that
$$\eta(\omega_X)=1.
$$
In fact,  this follows from
$$
\eta(\omega_X)=\prod_{x\in|X|}\epsilon(\eta_x,1/2,\psi_x)=\epsilon(\eta,1/2)=1,
$$
where $\epsilon(\eta,s)$ is in the functional equation (of the complete $L$-function) $L(\eta,s)=\epsilon(\eta,s)L(\eta,1-s)$.

We thus have
\begin{align*}
\frac{ \sP(\phi,s)\sP_\eta(\ov{\phi},s)}{\pair{\phi,\phi}_{{\rm Pet}}}=\frac{1}{2} \,\lvert \omega_X\rvert^{-1/2}\, \zeta_F(2) \,\sL(\pi_{F'},s+1/2).
\end{align*}
Together with \eqref{eqn J vol(K)} and \eqref{eqn meas2}, the proof of Proposition \ref{p:Jpi} is complete.
\end{proof}

\subsection{Change of coefficients}
\label{ss:coefficient}

Let $E$ be algebraic closed field containing $\BQ$. We consider the space of $E$-valued automorphic functions $\CA_E=C_c^\infty(G(F)\bs G(\BA)/K,E)$, and its subspace $\CA_{E,0}$ of cuspidal automorphic functions. For an irreducible $\sH_E$-module $\pi$ in $\CA_{E,0}$, let $\lambda_\pi:\sH\to E$ be the associated homomorphism. The L-function $\sL(\pi_{F'},s+1/2)$ is a well-defined element in $E[q^{-s},q^{s}]$. Recall that $f\in \sH$, the distribution $\BJ(f,s)$ defines an element in $\BQ[q^{-s},q^{s}]$ (cf. \S2).

\begin{thm}\label{th: coeff E}
Let $f$ be a function in the Eisenstein ideal $\CI_{\Eis}\subset \sH$. Then  we have an equality in $E[q^{-s},q^{s}]$:
$$
\BJ(f,s)=\frac{1}{2}\,|\omega_X|\, \sum_{\pi}\,  \sL(\pi_{F'},s+1/2)\,\lambda_\pi(f),
$$
where the sum runs over all irreducible $\sH_E$-module $\pi$ in the $E$-vector space $\CA_{E,0}$.
\end{thm}

\begin{proof}
It suffices to show this when $E=\ov\BQ$, and we fix an embedding $\ov\BQ\incl\BC$.
For $f\in \CI_{\Eis}$, then Theorem \ref{thm K eis=0} on the kernel functions remains valid if we understand the sum in \eqref{eqn sum Jpi} over $\pi$ as  $\sH_E$-submodule. In fact, to prove Theorem \ref{thm K eis=0}, we are allowed to extend $E=\ov\BQ$ to $\BC$.

Since a cuspidal $\phi$ has compact support, the integral $\sP_\chi(\phi,s)$ defined by \eqref{eqn P-chi} for $\chi\in\{1,\eta\}$ reduces to a finite sum. In particular, it defines an element in $E[q^{-s},q^{s}]$.
Therefore the equalities in Lemma \ref{l:J Eis ideal} and Proposition \ref{p:Jpi} hold, when both sides are viewed as elements in $E[q^{-s},q^{s}]$, and $\lambda_\pi$ as an $E=\ov\BQ$-valued homomorphism. This completes the proof.
\end{proof}

\part{The geometric side}

\section{Moduli spaces of Shtukas}\label{s:Sht}
The notion of rank $n$ Shtukas (or $F$-sheaves) with one upper and one lower modifications was introduced by Drinfeld \cite{D87}. It was generalized to an arbitrary reductive group $G$ and arbitrary number and type of modifications by Varshavsky \cite{Va}. In this section, we will review the definition of rank $n$ Shtukas, and then specialize to the case of $G=\PGL_{2}$ and the case of $T$ a nonsplit torus. Then we define Heegner--Drinfeld cycles to set up notation for the geometric side of the main theorem.

\subsection{The moduli of rank $n$ Shtukas}
\subsubsection{}\label{sss:arb mu} We fix the following data.
\begin{itemize}
\item $r\geq0$ is an integer;
\item $\mu=(\mu_1,...,\mu_r)$ is an ordered sequence of dominant coweights for $\GL_{n}$, where each $\mu_{i}$ is either equal to $\mu_{+}=(1,0,...,0)$ or equal to $\mu_{-}=(0,...,0,-1)$. \index{$\mu,\mu_\pm$}%
\end{itemize}
To such a tuple $\mu$ we assign an $r$-tuple of signs
\begin{equation*}
\sgn(\mu)=(\sgn(\mu_{1}),\cdots,\sgn(\mu_{r}))\in\{\pm1\}^{r}
\end{equation*}
where $\sgn(\mu_{\pm})=\pm1$. \index{$\sgn(\mu)$}%

\subsubsection{Parity condition}\label{sss:mu} At certain places we will impose the following conditions on the data $(r,\mu)$ above:
\begin{itemize}
\item $r$ is even;
\item Exactly half of $\mu_i$ are $\mu_{+}$, and the other half are $\mu_{-}$. Equivalently $\sum_{i=1}^{r}\sgn(\mu_{i})=0$.
\end{itemize}

\subsubsection{The Hecke stack} We denote by $\Bun_{n}$ the moduli stack of rank $n$ vector bundles on $X$. By definition, for any $k$-scheme $S$, $\Bun_{n}(S)$ is the groupoid of vector bundles over $X\times S$ of rank $n$. It is well-known that $\Bun_{n}$ is a smooth algebraic stack over $k$ of dimension $n^{2}(g-1)$.\index{$\Bun_n$}
 
\begin{defn}\label{def:Hk} Let $\mu$ be as in \S\ref{sss:arb mu}. The {\em Hecke stack} $\Hk^{\mu}_{n}$ is the stack whose $S$-points $\Hk^{\mu}_{n}(S)$ is the groupoid of the following data: 
\begin{enumerate}
\item A sequence of vector bundles $(\CE_0,\CE_1,\cdots,\CE_r)$ of rank $n$ on $X\times S$;
\item Morphisms $x_i: S\to X$ for $i=1,\cdots, r$, with graphs $\Gamma_{x_{i}}\subset X\times S$;
\item Isomorphisms of vector bundles
$$
f_i: \CE_{i-1}|_{X\times S-\Gamma_{x_i}}\isom \CE_{i}|_{X\times S-\Gamma_{x_i}}, \quad i=1,2,...,r,
$$
such that
\begin{itemize}
\item If $\mu_{i}=\mu_{+}$, then $f_{i}$ extends to an injective map $\CE_{i-1}\to \CE_{i}$ whose cokernel  is an  invertible sheaf on the graph $\Gamma_{x_i}$;
\item If $\mu_{i}=\mu_{-}$, then $f^{-1}_{i}$ extends to an injective map $\CE_{i}\to \CE_{i-1}$ whose cokernel  is an  invertible sheaf on the graph $\Gamma_{x_i}$.
\end{itemize}
\index{$\Hk^{\mu}_{n}$}%
\end{enumerate}
\end{defn}

For each $i=0,\cdots, r$, we have a map
\begin{equation*}
p_{i}: \Hk^{\mu}_{n}\to \Bun_{n}
\end{equation*}
sending $(\calE_{0},\cdots, \calE_{r}, x_{1},\cdots, x_{r}, f_{1},\cdots, f_{r})$ to $\CE_{i}$.  We also have a map
\begin{equation*}
p_{X}: \Hk^{\mu}_{n}\to X^{r}
\end{equation*}
recording the points $(x_1,...,x_r)\in X^{r}$.

\begin{remark}\label{r:HkG smooth} The morphism $(p_{0},p_{X}): \Hk^{\mu}_{n}\to \Bun_{n}\times X^{r}$ is representable, proper and smooth of relative dimension $r(n-1)$. Its fibers are iterated $\PP^{n-1}$-bundles. In particular, $\Hk^{\mu}_{n}$ is a smooth algebraic stack over $k$ because $\Bun_{n}$ is. 
\end{remark}

\subsubsection{The moduli stack of rank $n$ Shtukas}\label{sss:Sht n}

\begin{defn}\label{def:Sht} Let $\mu$ satisfy the conditions in \S\ref{sss:mu}.
The {\em moduli stack $\Sht^{\mu}_{n}$ of $\GL_{n}$-Shtukas of type $\mu$} is the fiber product
\begin{equation}\label{def Shtn}
\xymatrix{\Sht_n^{\mu}\ar[rr]\ar[d]&&
\Hk^{\mu}_{n} \ar[d]_{(p_{0},p_{r})}\\
\Bun_n\ar[rr]^{(\id, \Fr)}&& \Bun_n\times \Bun_n}		\index{$\Sht_n^{\mu}$}%
\end{equation}
\end{defn}

By definition, we have a morphism
\begin{equation*}
\pi^{\mu}_{n}: \Sht^{\mu}_{n}\to \Hk^{\mu}_{n}\xrightarrow{p_{X}}X^{r}. \index{$\pi^{\mu}_{n}$}%
\end{equation*}

\subsubsection{}\label{sss:points Shtn} 
Let $S$ be a scheme over $k$. For a vector bundle $\CE$ on $X\times S$, we denote
 $$^\tau\CE:=(\id_X\times\Fr_S)^\ast\CE.		\index{$^\tau\CE$}%
 $$  
An object in the groupoid $\Sht^{\mu}_{n}(S)$ is called a {\em Shtuka of type $\mu$ over $S$}.  Concretely, a Shtuka of type $\mu$ over $S$ is the following data:
\begin{enumerate}
\item $(\CE_0,\CE_1,\cdots,\CE_r; x_{1},\cdots,x_{r};f_{1},\cdots, f_{r})$ as in Definition \ref{def:Hk};
\item An isomorphism $\iota: \CE_r\simeq \,^\tau \CE_0$.
\end{enumerate}

The basic geometric properties of $\Sht^{\mu}_{n}$ are summarized in the following theorem.
\begin{thm}[Drinfeld \cite{D87} for $r=2$ ; Varshavsky {\cite[Prop 2.16, Thm 2.20]{Var}} in general]\label{th:Sht smooth}
\hfill
\begin{enumerate}
\item The stack $\Sht_{n}^{\mu}$ is a Deligne--Mumford stack locally of finite type. 
\item The morphism $\pi^{\mu}_{n}:\Sht_{n}^{\mu}\to X^r$ is separated and smooth of relative dimension $r(n-1)$.
\end{enumerate}
\end{thm}
We briefly comment on the proof of the separatedness of $\pi^{\mu}_{n}$. Pick a place $x\in |X|$, and consider the restriction of $\pi^{\mu}_{n}$ to $(X-\{x\})^{r}$. By \cite[Prop 2.16(a)]{Var}, $\Sht^{\mu}_{n}|_{(X-\{x\})^{r}}$ is an increasing union of open substacks $\frX_{1}\subset \frX_{2}\subset\cdots $ where each $\frX_{i}\cong[V_{i}/G_{i}]$ is the quotient of a quasi-projective scheme $V_{i}$ over $k$ by a finite discrete group $G_{i}$. These $V_{i}$ are obtained as moduli of Shtukas with level structures at $x$ and then truncated using stability conditions. Therefore each map $\frX_{i}\to (X-\{x\})^{r}$ is separated, hence so is $\pi^{\mu}_{n}|_{(X-\{x\})^{r}}$. Since $X^{r}$ is covered by open subschemes of the form $(X-\{x\})^{r}$, the map $\pi^{\mu}_{n}$ is separated.

\subsubsection{}\label{sss:Pic action} The Picard stack $\Pic_{X}$ of line bundles on $X$ acts on $\Bun_{n}$ and on $\Hk^{\mu}_{n}$ by tensoring on the vector bundles.

Similarly, the groupoid $\Pic_{X}(k)$ of line bundles over $X$ acts on $\Sht^{\mu}_{n}$. For a line bundle $\calL$ over $X$ and $(\calE_{i};x_{i};f_{i};\iota)\in\Sht^{\mu}_{n}(S)$, we define $\calL\cdot (\calE_{i};x_{i};f_{i};\iota)$ to be $(\calE_{i}\otimes_{\calO_{X}}\calL;x_{i};f_{i}\otimes\id_{\calL};\iota')$ where $\iota'$ is the isomorphism 
\begin{equation*}
\CE_{r}\otimes_{\calO_{X}}\calL\xrightarrow{\iota\otimes\id_{\calL}}((\id_{X}\times\Fr_{S})^{*}\CE_{0})\otimes_{\calO_{X}}\calL\cong(\id_{X}\times\Fr_{S})^{*}(\CE_{0}\otimes_{\calO_{X}}\calL)=\leftexp{\tau}{(\CE_{0}\otimes_{\calO_{X}}\calL)}.
\end{equation*}

\subsection{Moduli of Shtukas for $G=\PGL_{2}$}\label{ss:ShtG} 
Now we move on to $G$-Shtukas where $G=\PGL_{2}$. Let $\Bun_{G}$ be the moduli stack of $G$-torsors over $X$. \index{$\Bun_G$}%

\subsubsection{Quotient by a Picard stack} Here and later we will consider quotients of the form $[\CY/\CQ]$ where $\CY$ is an algebraic stack and $\CQ$ is a Picard stack such as $\Pic_{X}$ or $\Pic_{X}(k)$. Making such a quotient involves considering 2-categories {\em a priori}.  However, according to \cite[Lemme 4.7]{NgoH}, whenever the automorphism group of the identity object in $\CQ$ injects to the automorphism groups of objects in $\CY$, the quotient $[\CY/\CQ]$ makes sense as a stack. This injectivity condition will be satisfied in all situations we encounter in this paper. 

We have $\Bun_{G}\cong\Bun_{2}/\Pic_{X}$, where $\Pic_{X}$ acts on $\Bun_{2}$ by tensoring.

For each $\mu$ as in \S\ref{sss:arb mu}, we define
\begin{equation*}
\Hk^{\mu}_{G}:=\Hk^{\mu}_{2}/\Pic_{X}.\index{$\Hk^\mu_G$}%
\end{equation*}
For $\mu$ satisfying \S\ref{sss:mu}, we define
\begin{equation*}
\Sht^{\mu}_{G}:=\Sht^{\mu}_{2}/\Pic_{X}(k).\index{$\Sht^\mu_G$}%
\end{equation*}
The actions of $\Pic_{X}$ and $\Pic_{X}(k)$ are those introduced in \S\ref{sss:Pic action}. The maps $p_{i}:\Hk^{\mu}_{2}\to \Bun_{2}$ are $\Pic_{X}$-equivariant, and induce maps
\begin{equation}\label{pi Hkmu}
p_{i}: \Hk^{\mu}_{G}\to \Bun_{G}, \quad 0\leq i\leq r.
\end{equation}

\begin{lemma}\label{l:Hk indep mu}
For different choices $\mu$ and $\mu'$ as in \S\ref{sss:arb mu}, there are canonical isomorphisms $\Hk^{\mu}_{2}\cong\Hk^{\mu'}_{2}$ and  $\Hk^{\mu}_{G}\cong\Hk^{\mu'}_{G}$.  Moreover, these isomorphisms respect the maps $p_{i}$ in \eqref{pi Hkmu}.
\end{lemma}
\begin{proof} For $\mu^{r}_{+}:=(\mu_{+},\cdots, \mu_{+})$, we denote the corresponding Hecke stack by $\Hk^{r}_{2}$.  The $S$-points of $\Hk^{r}_{2}$ classify a sequence of rank two vector bundles on $X\times S$ together with embeddings
\begin{equation*}
\CE_{0}\xrightarrow{f_{1}}\CE_{1}\xrightarrow{f_{2}}\cdots\xrightarrow{f_{r}}\CE_{r}
\end{equation*}
such that the cokernel of $f_{i}$ is an invertible sheaf supported on the graph of a morphism $x_{i}:S\to X$.

We construct a morphism
\begin{equation*}
\phi_{\mu}: \Hk^{\mu}_{2}\to\Hk^{r}_{2}.
\end{equation*}
Consider a point $(\CE_{i};x_{i};f_{i})\in\Hk^{\mu}_{2}(S)$. For $i=1,\cdots, r$, we define a divisor on $X\times S$
\begin{equation*}
D_{i}:=\sum_{1\leq j\leq i,\mu_{j}=\mu_{-}}\Gamma_{x_{j}}.
\end{equation*}
Then we define
\begin{equation*}
\CE'_{i}=\CE_{i}(D_{i}).
\end{equation*}
If $\mu_{i}=\mu_{+}$, then $D_{i-1}=D_{i}$, the map $f_{i}$ induces an embedding $f'_{i}:\CE'_{i-1}=\CE_{i-1}(D_{i-1})\to \CE_{i}(D_{i-1})=\CE'_{i}$. If $\mu_{i}=\mu_{-}$, then $D_{i}=D_{i-1}+\Gamma_{x_{i}}$, and the map $f_{i}: \CE_{i}\to \CE_{i-1}$ induces an embedding $\CE_{i-1}\to \CE_{i}(\Gamma_{x_{i}})$, and hence an embedding $f'_{i}:\CE'_{i-1}=\CE_{i-1}(D_{i-1})\to \CE_{i}(D_{i-1}+\Gamma_{x_{i}})=\CE'_{i}$. The map $\phi_{\mu}$ sends $(\CE_{i};x_{i};f_{i})$ to $(\CE'_{i};x_{i};f'_{i})$. 

We also have a morphism
\begin{eqnarray*}
\psi_{\mu}: \Hk^{r}_{2}&\to&\Hk^{\mu}_{2}\\
(\CE'_{i};x_{i};f'_{i}) &\mapsto& (\CE'_{i}(-D_{i});x_{i};f_{i}). 
\end{eqnarray*}
It is easy to check that $\phi_{\mu}$ and $\psi_{\mu}$ are inverse to each other. This way we get a canonical isomorphism $\Hk^{\mu}_{2}\cong\Hk^{r}_{2}$, which is clearly $\Pic_{X}$-equivariant. Therefore all $\Hk^{\mu}_{G}$ are also canonically isomorphic to each other. In the construction of $\phi_{\mu}$, the vector bundles $\CE_{i}$ only change by tensoring with line bundles, therefore the image of $\CE_{i}$ in $\Bun_{G}$ remain unchanged. This shows that the canonical isomorphisms between the $\Hk^{\mu}_{G}$ respect the maps $p_{i}$ in \eqref{pi Hkmu}.
\end{proof}

\begin{lemma}\label{l:indep mu} There is a canonical Cartesian diagram
\begin{equation}\label{ShtrG}
\xymatrix{\Sht^{\mu}_{G}\ar[d]\ar[rr]& & \Hk^{\mu}_{G}\ar[d]^{(p_{0},p_{r})}\\
\Bun_{G}\ar[rr]^{(\id,\Fr)} && \Bun_{G}\times \Bun_{G}}		\index{$\Sht^\mu_G$}%
\end{equation}
In particular, for different choices of $\mu$ satisfying the conditions in \S\ref{sss:mu}, the stacks $\Sht^{\mu}_{G}$ are canonically isomorphic to each other.
\end{lemma}
\begin{proof}This follows from the Cartesian diagram \eqref{def Shtn} divided termwisely by the Cartesian diagram
\begin{equation*}
\xymatrix{\Pic_{X}(k)\ar[rr]\ar[d]&&
\Pic_{X} \ar[d]_{\Delta}\\
\Pic_{X}\ar[rr]^{(\id, \Fr)}&& \Pic_{X}\times \Pic_{X}}
\end{equation*}
\end{proof}

By the above lemmas, we may unambiguously use the notation
\begin{equation}\label{HkrG}
\Sht^{r}_{G}; \quad \Hk^{r}_{G} 		\index{$\Hk^r_G$}\index{$\Sht^r_G$}%
\end{equation} 
for $\Sht^{\mu}_{G}$ and $\Hk^{\mu}_{G}$ with any choice of $\mu$. If $r$ is fixed from the context, we may also drop $r$ from the notation and write simply $\Sht_{G}$. The morphism $\pi^{\mu}_{2}:\Sht^{\mu}_{2}\to X^{r}$ is invariant under the action of $\Pic_{X}(k)$ and induces a morphism
\begin{equation*}
\pi_{G}: \Sht^{r}_{G}\to X^{r}.\index{$\pi_G$}%
\end{equation*}

Theorem \ref{th:Sht smooth} has the following immediate consequence.
\begin{cor}\label{c:ShtG smooth}
\begin{enumerate}
\item The stack $\Sht^{r}_{G}$ is a Deligne--Mumford stack locally of finite type.
\item The morphism $\pi_{G}:\Sht^{r}_{G}\to X^r$ is separated and smooth of relative dimension $r$.
\end{enumerate}
\end{cor}

\subsection{Hecke correspondences}

We define the rational Chow group of proper cycles $\Ch_{c,i}(\Sht^{r}_{G})_{\QQ}$ as in \S\ref{ss:proper cycles}. As in \S\ref{sss:corr operator}, we also have a $\QQ$-algebra $\cCh_{2r}(\Sht^{r}_{G}\times\Sht^{r}_{G})_{\QQ}$ that acts on $\Ch_{c,i}(\Sht^{r}_{G})_{\QQ}$. The goal of this subsection is to define a ring homomorphism from the unramified Hecke algebra
$\sH=C_{c}(K\bs G(\BA)/K,\QQ)$ to $\cCh_{2r}(\Sht^{r}_{G}\times\Sht^{r}_{G})_{\QQ}$.\index{$\cCh_{2r}$}\index{$\Ch_{c,i}$}%

\subsubsection{The stack $\Sht^{r}_{G}(h_{D})$}\label{sss:Sht hD} Recall from \S\ref{ss:test fun} that we have a basis $h_{D}$ of $\sH$ indexed by effective divisors $D$ on $X$. For each effective divisor $D=\sum_{x\in |X|}n_{x}x$ we shall define a self-correspondence $\Sht^{r}_{G}(h_{D})$ of $\Sht^{r}_{G}$ over $X^{r}$:
\begin{equation*}
\xymatrix{ & \Sht^{r}_{G}(h_{D})\ar[dl]_{\oll{p}}\ar[dr]^{\orr{p}}\\
\Sht^{r}_{G}\ar[dr] & & \Sht^{r}_{G}\ar[dl]\\
& X^{r}}						\index{$\oll{p},\orr{p}$} \index{$\Sht^{r}_{G}(h_{D})$}%
\end{equation*}
For this, we first fix a $\mu$ as in \S\ref{sss:mu}. We introduce a self-correspondence $\Sht^{\mu}_{2}(h_{D})$ of $\Sht^{\mu}_{2}$ whose $S$-points is the groupoid classifying the data
\begin{enumerate}
\item Two objects $(\CE_{i};x_{i};f_{i};\iota)$ and $(\CE'_{i}; x_{i};f'_{i};\iota')$ of $\Sht^{\mu}_{2}(S)$ with the same collection of points $x_{1},\cdots, x_{r}$ in $X(S)$;
\item For each $i=0,\cdots, r$, an embedding of coherent sheaves $\phi_{i}: \CE_{i}\incl \CE'_{i}$ such that $\det(\phi_{i}):\det\CE_{i}\incl \det\CE'_{i}$ has divisor $D\times S\subset X\times S$. 
\item The following diagram is commutative
\begin{equation}\label{Sht hD}
\xymatrix{\calE_{0}\ar@{-->}[r]^{f_{1}}\ar[d]^{\phi_{0}} & \calE_{1}\ar@{-->}[r]^{f_{2}}\ar[d]^{\phi_{1}} & \cdots\ar@{-->}[r]^{f_{r}} &\calE_{r}\ar[d]^{\phi_{r}}\ar[r]^{\iota} & \leftexp{\tau}{\calE_{0}}\ar[d]^{\leftexp{\tau}{\phi_{0}}}\\
\calE'_{0}\ar@{-->}[r]^{f'_{1}} & \calE'_{1}\ar@{-->}[r]^{f'_{2}} & \cdots\ar@{-->}[r]^{f'_{r}} &\calE'_{r}\ar[r]^{\iota'} & \leftexp{\tau}{\calE'_{0}}}
\end{equation}
\end{enumerate}

There is a natural action of $\Pic_{X}(k)$ on $\Sht^{\mu}_{2}(h_{D})$ by tensoring on each $\CE_{i}$ and $\CE'_{i}$. We define
\begin{equation*}
\Sht^{r}_{G}(h_{D}):=\Sht^{\mu}_{2}(h_{D})/\Pic_{X}(k).
\end{equation*}
Using Lemma \ref{l:Hk indep mu}, it is easy to check that $\Sht^{r}_{G}(h_{D})$ is canonically independent of the choice of $\mu$. The two maps $\oll{p},\orr{p}:\Sht^{r}_{G}(h_{D})\to \Sht^{r}_{G}$ send the data above to the image of $(\CE_{i};x_{i};f_{i};\iota)$ and $(\CE'_{i}; x_{i};f'_{i};\iota')$ in $\Sht^{r}_{G}$ respectively.

\begin{lemma}\label{l:Sht hD proper} The maps $\oll{p}$,$\orr{p}$ as well as $(\oll{p},\orr{p}): \Sht^{r}_{G}(h_{D})\to \Sht^{r}_{G}\times\Sht^{r}_{G}$ are representable and proper.
\end{lemma}
\begin{proof}
Once the bottom row of the diagram \eqref{Sht hD} is fixed, the choices of the vertical maps $\phi_{i}$ for $i=1,\cdots,r$ form a closed subscheme of the product of Quot schemes $\prod_{i=1}^{r}\Quot^{d}(\calE'_{i})$, where $d=\deg D$, which is proper. Therefore $\orr{p}$ is representable and proper. Same argument applied to the dual of the diagram \eqref{Sht hD} proves that $\oll{p}$ is proper. 

The representability of $(\oll{p},\orr{p})$ is obvious from the definition, since its fibers are closed subschemes of $\prod_{i=1}^{r}\Hom(\CE_{i},\CE_{i}')$. Since $\Sht^{r}_{G}$ is separated by Corollary \ref{c:ShtG smooth} and $\oll{p}$ is proper, $(\oll{p},\orr{p})$ is also proper.
\end{proof}

\begin{lemma}\label{l:dim hD} The geometric fibers of the map $\Sht^{r}_{G}(h_{D})\to X^{r}$ have dimension $r$.
\end{lemma}
The proof of this lemma will be postponed to \S\ref{sss:proof dim hD}, because the argument will involve some auxiliary moduli spaces that we will introduce in \S\ref{ss:aux}.

Granting Lemma \ref{l:dim hD}, we have $\dim\Sht^{r}_{G}(h_{D})=2r$. By Lemma \ref{l:Sht hD proper}, it makes sense to push forward the fundamental cycle of $\Sht^{r}_{G}(h_{D})$ along the proper map $(\oll{p},\orr{p})$. Therefore $(\oll{p},\orr{p})_{*}[\Sht^{r}_{G}(h_{D})]$ defines an element in $\cCh_{2r}(\Sht^{r}_{G}\times\Sht^{r}_{G})_{\QQ}$ (because $\oll{p}$ is also proper). We define the $\QQ$-linear map
\begin{eqnarray}\label{define H}
H:\sH&\to&  \cCh_{2r}(\Sht^{r}_{G}\times\Sht^{r}_{G})_{\QQ}\\
 h_{D}&\mapsto& (\oll{p}\times \orr{p})_{*}[\Sht^{r}_{G}(h_{D})], \quad \textup{ for all effective divisors } D.		\index{$H(h_D)$}%
\end{eqnarray}

\begin{prop}\label{p:Hecke action on Chow} The linear map $H$ in \eqref{define H} is a ring homomorphism.
\end{prop}
\begin{proof} Since $h_{D}$ form a $\QQ$-basis of $\sH$, it suffices to show that\begin{equation}\label{H mult}
H(h_{D}h_{D'})=H(h_{D})*H(h_{D'})\in\cCh_{2r}(\Sht^{r}_{G}\times\Sht^{r}_{G})_{\QQ}
\end{equation}
for any two effective divisors $D$ and $D'$.

Let $U=X-|D|-|D'|$. Since $h_{D}h_{D'}$ is a linear combination of $h_{E}$ for effective divisors $E\leq D+D'$ such that $D+D'-E$ has even coefficients, the cycle $H(h_{D}h_{D'})$  is supported on $\cup_{E\leq D+D', D+D'-E\textup{ even}} \Sht^{r}_{G}(h_{E})=\Sht^{r}_{G}(h_{D+D'})$. The cycle $H(h_{D})*H(h_{D'})$ is supported on the image of the projection
\begin{equation*}
\pr_{13}: \Sht^{r}_{G}(h_{D})\times_{\orr{p},\Sht^{r}_{G},\oll{p}}\Sht^{r}_{G}(h_{D'})\to\Sht^{r}_{G}\times\Sht^{r}_{G}
\end{equation*}
which is easily seen to be contained in $\Sht^{r}_{G}(h_{D+D'})$. We see that both sides of \eqref{H mult} are supported on $Z:=\Sht^{r}_{G}(h_{D+D'})$.

By Lemma \ref{l:dim hD} applied to $Z=\Sht^{r}_{G}(h_{D+D'})$, the dimension of $Z-Z|_{U^{r}}$ is strictly less than $2r$. Therefore, the restriction map induces an isomorphism
\begin{equation}\label{restrict Chow}
\Ch_{2r}(Z)_{\QQ}\isom\Ch_{2r}(Z|_{U^{r}})_{\QQ}.
\end{equation}

Restricting the definition of $H$ to $U^{r}$, we get a linear map $H_{U}:\sH\to  \cCh_{2r}(\Sht^{r}_{G}|_{U^{r}}\times\Sht^{r}_{G}|_{U^{r}})_{\QQ}$. For any effective divisor $E$ supported on $|D|\cup |D'|$, the two projections $\oll{p},\orr{p}: \Sht^{r}_{G}(h_{E})|_{U^{r}}\to \Sht^{r}_{G}|_{U^{r}}$ are finite \'etale. The equality
\begin{equation}\label{HU mult}
H_{U}(h_{D}h_{D'})=H_{U}(h_{D})*H_{U}(h_{D'})\in\Ch_{2r}(Z|_{U^{r}})_{\QQ}
\end{equation}
is well-known. By \eqref{restrict Chow}, this implies the equality \eqref{H mult} where both sides are interpreted as elements in $\Ch_{2r}(Z)_{\QQ}$, and a fortiori as elements in $\cCh_{2r}(\Sht^{r}_{G}\times\Sht^{r}_{G})_{\QQ}$. 
\end{proof}

\begin{remark} Let $g=(g_{x})\in G(\BA)$, and let $f=\one_{KgK}\in\sH$ be the characteristic function of the double coset $KgK$ in $G(\BA)$. Traditionally, one defines a self-correspondence $\Gamma(g)$ of $\Sht^{r}_{G}|_{(X-S)^{r}}$ over $(X-S)^{r}$, where $S$ is the finite set of places where $g_{x}\notin K_{x}$ (see \cite[Construction 2.20]{VL}). The two projections $\oll{p},\orr{p}: \Gamma(g)\to \Sht^{r}_{G}|_{(X-S)^{r}}$ are finite \'etale. The disadvantage of this definition is that we need to remove the bad points $S$ which depend on $f$, so one is forced to work only with the generic fiber of $\Sht^{r}_{G}$ over $X^{r}$ if one wants to consider the actions of all Hecke functions.  Our definition of $H(f)$ for any $f\in \sH$ gives a correspondence for the whole $\Sht^{r}_{G}$. It is easy to check that for $f=\one_{KgK}$, our cycle $H(f)|_{(X-S)^{r}}$, which is a linear combination of the cycles $\Sht^{r}_{G}(h_{D})|_{(X-S)^{r}}$ for divisors $D$ supported on $S$, is the same cycle as $\Gamma(g)$. Therefore our definition of the Hecke algebra action extends the traditional one. 
\end{remark}

\subsubsection{A variant}\label{sss:Hk Sht'}
Later we will consider the stack $\Sht'^{r}_{G}:=\Sht^{r}_{G}\times_{X^{r}}X'^{r}$ defined using the double cover $X'\to X$. Let $\Sht'^{r}_{G}(h_{D})=\Sht^{r}_{G}(h_{D})\times_{X^{r}}X'^{r}$. Then we have natural maps
\begin{equation*}
\oll{p}',\orr{p}':\Sht'^{r}_{G}(h_{D})\to \Sht'^{r}_{G}.
\end{equation*}
The analogs of Lemma \ref{l:Sht hD proper} and \ref{l:dim hD} for $\Sht'^{r}_{G}(h_{D})$ follow from the original statements. The map $h_{D}\mapsto (\oll{p}'\times \orr{p}')_{*}[\Sht'^{r}_{G}(h_{D})]\in  \cCh_{2r}(\Sht'^{r}_{G}\times\Sht'^{r}_{G})_{\QQ}$ then gives a ring homomorphism $H'$:
\begin{equation*}
H':\sH\to  \cCh_{2r}(\Sht'^{r}_{G}\times\Sht'^{r}_{G})_{\QQ}.
\end{equation*}

\subsubsection{Notation}\label{sss:H action on Chow} By \S\ref{sss:corr operator}, the $\QQ$-algebra $\cCh_{2r}(\Sht'^{r}_{G}\times\Sht'^{r}_{G})_{\QQ}$ acts on $\Ch_{c,*}(\Sht'^{r}_{G})_{\QQ}$. Hence the Hecke algebra $\sH$ also acts on $\Ch_{c,*}(\Sht'^{r}_{G})_{\QQ}$ via the homomorphism $H'$. For $f\in \sH$, we denote its action on $\Ch_{c,*}(\Sht'^{r}_{G})_{\QQ}$ by
\begin{equation*}
f*(-):\Ch_{c,*}(\Sht'^{r}_{G})_{\QQ} \to \Ch_{c,*}(\Sht'^{r}_{G})_{\QQ}.		\index{$f*(-)$}%
\end{equation*}

Recall the Chow group $\Ch_{c,*}(\Sht^{r}_{G})_{\QQ}$ (or $\Ch_{c,*}(\Sht'^{r}_{G})_{\QQ}$) is equipped with an intersection pairing between complementary degrees, see \S\ref{sss:int pairing}.

\begin{lemma}\label{l:Chow self adj}
The action of any $f\in\sH$ on $\Ch_{c,*}(\Sht^{r}_{G})_{\QQ}$ or $\Ch_{c,*}(\Sht'^{r}_{G})_{\QQ}$ is self-adjoint with respect to the intersection pairing.
\end{lemma}
\begin{proof} It suffices to prove self-adjointness for $h_{D}$ for all effective divisors $D$. We give the argument for $\Sht^{r}_{G}$ and the case of $\Sht'^{r}_{G}$ can be proved in the same way. For $\zeta_{1}\in \Ch_{c,i}(\Sht^{r}_{G})_{\QQ}$ and $\zeta_{2}\in\Ch_{c,2r-i}(\Sht^{r}_{G})_{\QQ}$, the intersection number $\jiao{h_{D}*\zeta_{1},\zeta_{2}}_{\Sht^{r}_{G}}$ is the same as the following intersection number in $\Sht^{r}_{G}\times\Sht^{r}_{G}$
\begin{equation*}
\jiao{\zeta_{1}\times\zeta_{2}, (\oll{p},\orr{p})_{*}[\Sht^{r}_{G}(h_{D})]}_{\Sht^{r}_{G}\times\Sht^{r}_{G}}.
\end{equation*}
We will construct an involution $\tau$ on $\Sht^{r}_{G}(h_{D})$ such that the following diagram is commutative
\begin{equation}\label{trans Sht}
\xymatrix{\Sht^{r}_{G}(h_{D})\ar[d]^{(\oll{p}, \orr{p})}\ar[rr]^{\tau}    &&  \Sht^{r}_{G}(h_{D})\ar[d]^{(\oll{p}, \orr{p})}  \\
\Sht^{r}_{G}\times\Sht^{r}_{G}\ar[rr]^{\sigma_{12}} && \Sht^{r}_{G}\times \Sht^{r}_{G}}
\end{equation}
Here $\sigma_{12}$ in the bottom row means flipping two factors. Once we have such a diagram, we can apply $\tau$ to $\Sht^{r}_{G}(h_{D})$ and $\sigma_{12}$ to $\Sht^{r}_{G}\times\Sht^{r}_{G}$ and get
\begin{equation*}
\jiao{\zeta_{1}\times\zeta_{2}, (\oll{p},\orr{p})_{*}[\Sht^{r}_{G}(h_{D})]}_{\Sht^{r}_{G}\times\Sht^{r}_{G}}=\jiao{\zeta_{2}\times\zeta_{1}, (\oll{p},\orr{p})_{*}[\Sht^{r}_{G}(h_{D})]}_{\Sht^{r}_{G}\times\Sht^{r}_{G}}
\end{equation*}
which is the same as the self-adjointness for $h*(-)$:
\begin{equation*}
\jiao{h_{D}*\zeta_{1},\zeta_{2}}_{\Sht^{r}_{G}}=\jiao{h_{D}*\zeta_{2},\zeta_{1}}_{\Sht^{r}_{G}}=\jiao{\zeta_{1}, h_{D}*\zeta_{2}}_{\Sht^{r}_{G}}.
\end{equation*}

We pick any $\mu$ as in \S\ref{sss:mu} and identify $\Sht^{r}_{G}$ with $\Sht^{\mu}_{G}=\Sht^{\mu}_{2}/\Pic_{X}(k)$. We use $-\mu$ to denote the negated tuple if we think of $\mu\in\{\pm1\}^{r}$  using the $\sgn$ map. We consider the composition
\begin{equation*}
\delta: \Sht^{\mu}_{G}\xrightarrow{\delta'}\Sht^{-\mu}_{G}\cong\Sht^{\mu}_{G}
\end{equation*}
where $\delta'(\CE_{i};x_{i};f_{i}; \iota)=(\CE^{\vee}_{i}; x_{i}; (f^{\vee}_{i})^{-1};(\iota^{\vee})^{-1})$ and the second map is the canonical isomorphism $\Sht^{-\mu}_{G}\cong \Sht^{\mu}_{G}$ given by Lemma \ref{l:indep mu}.

Similarly we define $\tau$ as the composition
\begin{equation}\label{tau for Sht hD}
\tau: \Sht^{\mu}_{G}(h_{D})\xrightarrow{\tau'}\Sht^{-\mu}_{G}(h_{D})\cong\Sht^{\mu}_{G}(h_{D})
\end{equation}
where $\tau'$ sends the diagram \eqref{Sht hD} to the diagram
\begin{equation}\label{Sht hD tr}
\xymatrix{\calE'^{\vee}_{0}\ar@{-->}[r]^{f'^{\vee-1}_{1}}\ar[d]^{\phi^{\vee}_{0}} & \calE'^{\vee}_{1}\ar@{-->}[r]^{f'^{\vee-1}_{2}}\ar[d]^{\phi^{\vee}_{1}} & \cdots\ar@{-->}[r]^{f'^{\vee-1}_{r}} &\calE'^{\vee}_{r}\ar[d]^{\phi^{\vee}_{r}}\ar[r]^{\iota'^{\vee-1}} & \leftexp{\tau}{\calE'^{\vee}_{0}}\ar[d]^{\leftexp{\tau}{\phi^{\vee}_{0}}}\\
\calE^{\vee}_{0}\ar@{-->}[r]^{f^{\vee-1}_{1}} & \calE^{\vee}_{1}\ar@{-->}[r]^{f^{\vee-1}_{2}} & \cdots\ar@{-->}[r]^{f^{\vee-1}_{r}} &\calE^{\vee}_{r}\ar[r]^{\iota^{\vee-1}} & \leftexp{\tau}{\calE^{\vee}_{0}}}
\end{equation}
and the second map in \eqref{tau for Sht hD} is the canonical isomorphism $\Sht^{-\mu}_{G}(h_{D})\cong \Sht^{\mu}_{G}(h_{D})$ given by the analog of Lemma \ref{l:indep mu}. It is clear from the definition that if we replace the bottom arrow of \eqref{trans Sht} with $\sigma_{12}\circ(\delta\times\delta)$ (i.e., the map $(a,b)\mapsto (\delta(b),\delta(a))$), the diagram is commutative. 

We claim that $\delta$ is the identity map for $\Sht^{\mu}_{G}$. In fact, $\delta$ turns $(\CE_{i};x_{i};f_{i}; \iota)\in\Sht^{\mu}_{G}$ into $(\CE^{\vee}_{i}(D_{i}); x_{i}; (f^{\vee}_{i})^{-1}; (\iota^{\vee})^{-1})$, where $D_{i}=\sum_{1\leq j\leq i}\sgn(\mu_{j})\Gamma_{x_{j}}$. Note that we have a canonical isomorphism $\CE^{\vee}_{i}\cong\CE_{i}\otimes(\det\CE_{i})^{-1}$, and isomorphisms $\det\CE_{i}\cong(\det\CE_{0})(D_{i})$ induced by the $f_{i}$. Therefore we get a canonical isomorphism $\CE^{\vee}_{i}(D_{i})\cong\CE_{i}\otimes(\det\CE_{i})^{-1}\otimes\calO(D_{i})\cong \CE_{i}\otimes(\det\CE_{0})^{-1}$ compatibly with the maps $(f^{\vee}_{i})^{-1}$ and $f_{i}$, and also compatible with $(\iota^{\vee})^{-1}$ and $\iota$. Therefore  $\delta(\CE_{i};x_{i};f_{i}; \iota)$ is canonically isomorphic to $(\CE_{i};x_{i};f_{i}; \iota)$ up to tensoring with $\det(\CE_{0})$. This shows that $\delta$ is the identity map of $\Sht^{\mu}_{G}$.

Since $\delta=\id$,  the diagram \eqref{trans Sht} is also commutative. This finishes the proof.
\end{proof}

\subsection{Moduli of Shtukas for the torus $T$}\label{ss:ShtT} 

\subsubsection{} Recall that $\nu: X'\to X$  is an \'etale double covering with $X'$ also geometrically connected. Let $ \sigma\in\Gal(X'/X)$ be the non-trivial involution.

Let $\wt{T}$ be the two-dimensional torus over $X$ defined as
\begin{equation*}
\wt{T}:=\Res_{X'/X}\Gm.		\index{$\wt{T}$}%
\end{equation*}
We have a natural homomorphism $\Gm\to\wt{T}$. We define a one-dimensional torus over $X$
\begin{equation*}
T:=\wt{T}/\Gm=(\Res_{X'/X}\Gm)/\Gm.		\index{$T$, torus}%
\end{equation*}

Let $\Bun_{T}$ be the moduli stack of $T$-torsors over $X$. Then we have a canonical isomorphism of stacks
\begin{equation*}
\Bun_{T}\cong \Pic_{X'}/\Pic_{X}.		\index{$\Bun_T$}%
\end{equation*}
In particular, $\Bun_{T}$ is a Deligne--Mumford stack whose coarse moduli space is a group scheme with two components, and its neutral component is an abelian variety over $k$.

\subsubsection{}
Specializing Definition \ref{def:Hk} to the case $n=1$ and replacing the curve $X$ with its double cover $X'$, we get the Hecke stack $\Hk^{\mu}_{1,X'}$. This makes sense for any tuple $\mu$ as in \S\ref{sss:arb mu}.

Now assume that $\mu$ satisfies the conditions in \S\ref{sss:mu}. We may view each $\mu_{i}$ as a coweight for $\GL_{1}=\Gm$ in an obvious way: $\mu_{+}$ means $1$ and $\mu_{-}$ means $-1$. Specializing Definition \ref{def:Sht} to the case $n=1$ and replacing $X$ with  $X'$, we get the moduli stack $\Sht^{\mu}_{1,X'}$ of rank one Shtukas over $X'$ of type $\mu$. We define
\begin{equation*}
\Sht^{\mu}_{\wt{T}}:=\Sht^{\mu}_{1,X'}.		\index{$\Sht^{\mu}_{\wt{T}}$}%
\end{equation*}
We have a Cartesian  diagram
\begin{equation*}
\xymatrix{\Sht_{\wt{T}}^{\mu}\ar[d] \ar[rr] & & \Hk^{\mu}_{1,X'}\ar[d]^{(p_{0},p_{r})}\\
\Pic_{X'}\ar[rr]^{(\id, \Fr)}&& \Pic_{X'}\times\Pic_{X'}}
\end{equation*}
We also have a morphism
\begin{equation*}
\pi^{\mu}_{\wt{T}}: \Sht_{\wt{T}}^{\mu}\to \Hk^{\mu}_{1,X'}\xrightarrow{p_{X'}}X'^{r}.
\end{equation*}

\subsubsection{}\label{sss:points ShtT} Fix $\mu$ as in \S\ref{sss:mu}. Concretely, for any $k$-scheme $S$, $\Sht_{\wt{T}}^{\mu}(S)$ classifies the following data
\begin{enumerate}
\item A line bundle $\calL$ over $X'\times S$;
\item Morphisms $x'_{i}:S\to X'$ for $i=1,\cdots, r$, with graphs $\Gamma_{x'_{i}}\subset X'\times S$;
\item An isomorphism
\begin{equation*}
\iota: \CL\left(\sum_{i=1}^r \sgn(\mu_i) \Gamma_{x'_i}\right)\isom \leftexp{\tau}{\CL}:=(\id\times \Fr_S)^\ast\CL.
\end{equation*}
Here the signs $\sgn(\mu_\pm)=\pm1$ are defined in \S\ref{sss:arb mu}. 
\end{enumerate}
This description of points appears to be simpler than its counterpart in \S\ref{sss:points Shtn}:  the other line bundles $\calL_{i}$ are canonically determined by $\calL_{0}$ and $x'_{i}$ using the formula
\begin{equation}\label{define Li}
\calL_{i}=\calL_{0}\left(\sum_{1\leq j\leq i}\sgn(\mu_{j})\Gamma_{x'_{j}}\right).
\end{equation}

\subsubsection{}\label{sss:HkT smooth} The Picard stack $\Pic_{X'}$, and hence $\Pic_{X}$, acts on $\Hk^{\mu}_{1,X'}$. We consider the quotient
\begin{equation}\label{HkrT}
\Hk^{\mu}_{T}:=\Hk^{\mu}_{1,X'}/\Pic_{X}.		\index{$\Hk^{\mu}_{T}$}%
\end{equation}
In fact we have a canonical isomorphism $\Hk^{\mu}_{T}\cong \Bun_{T}\times X'^{r}$ sending $(\calL_{i};x'_{i};f_{i})$ to $(\calL_{0};x'_{i})$. In particular, $\Hk^{\mu}_{T}$ is a smooth and proper Deligne--Mumford stack of pure dimension $r+g-1$ over $k$.

\subsubsection{} The groupoid $\Pic_{X'}(k)$ acts on $\Sht_{\wt{T}}^{\mu}$ by tensoring on the line bundle $\calL$. We consider the restriction of this action to $\Pic_{X}(k)$ via the pullback map $\nu^{*}: \Pic_{X}(k)\to \Pic_{X'}(k)$. We define
\begin{equation*}
\Sht^{\mu}_{T}:=\Sht^{\mu}_{\wt{T}}/\Pic_{X}(k).		\index{$\Sht^{\mu}_{T}$}
\end{equation*}
The analog of Lemma \ref{l:indep mu} gives a Cartesian diagram
\begin{equation}\label{ShtrT}
\xymatrix{\Sht^{\mu}_{T}\ar[d]\ar[rr] && \Hk^{\mu}_{T}\ar[d]^{(p_{0},p_{r})}\\
\Bun_{T}\ar[rr]^{(\id,\Fr)} && \Bun_{T}\times \Bun_{T}}
\end{equation}
Since the morphism $\pi^{\mu}_{\wt{T}}$ is invariant under $\Pic_{X}(k)$, we get a morphism
\begin{equation*}
\pi^{\mu}_{T}: \Sht^{\mu}_{T}\to X'^{r}.		\index{$\pi^{\mu}_{T}$}%
\end{equation*}

\begin{lemma}\label{l:ShtT proper}
The morphism $\pi^{\mu}_{T}$ is a torsor under  the finite Picard groupoid $\Pic_{X'}(k)/\Pic_{X}(k)$. In particular, $\pi^{\mu}_{T}$ is finite \'etale, and the stack $\Sht^{\mu}_{T}$ is a smooth proper Deligne--Mumford stack over $k$ of pure dimension $r$. 
\end{lemma}
\begin{proof}
This description given in \S\ref{sss:points ShtT} gives a Cartesian diagram
\begin{equation}\label{Sht wT}
\xymatrix{\Sht_{\wt{T}}^{\mu}\ar[d]^{\pi^{\mu}_{\wt{T}}}  \ar[rr]& &\Pic_{X'}\ar[d]^{\id-\Fr}\\
X'^r \ar[rr]^{\phi} &&  \Pic^0_{X'}}
\end{equation}
where $\phi(x'_{1},\cdots, x'_{r})=\calO_{X'}(\sum _{i=1}^r\sgn(\mu_i) x'_i)$.  Dividing the top row of the diagram \eqref{Sht wT} by $\Pic_{X}(k)$ we get a Cartesian diagram
\begin{equation*}
\xymatrix{\Sht^{\mu}_{T}\ar[d]^{\pi^{\mu}_{T}}  \ar[rr]& &\Pic_{X'}/(\Pic_{X}(k))\ar[d]^{\id-\Fr}\\
X'^r \ar[rr]^{\phi} &&  \Pic^0_{X'}}
\end{equation*}
Since the right vertical map $\id-\Fr: \Pic_{X'}/(\Pic_{X}(k))\to \Pic^{0}_{X'}$ is a torsor under $\Pic_{X'}(k)/\Pic_{X}(k)$, so is $\pi^{\mu}_{T}$.
\end{proof}

\subsubsection{Changing $\mu$}\label{HkT indep mu} For a different choice $\mu'$ as in \ref{sss:arb mu}, we have a canonical isomorphism
\begin{equation}\label{HktT mu mu'}
\Hk^{\mu}_{\wt{T}}\isom\Hk^{\mu'}_{\wt{T}}
\end{equation}
sending $(\calL_{i};x'_{i};f_{i})$ to $(\CK_{i}; y'_{i}; g_{i})$ where
\begin{equation}\label{y'}
y'_{i}=\begin{cases} x'_{i} & \text{ if } \mu_{i}=\mu'_{i} \\
\sigma(x'_{i}) & \text{ if } \mu_{i}\neq\mu'_{i} \end{cases}
\end{equation}and
\begin{equation}\label{Ki in L0}
\CK_{i}=\CL_{0}\left(\sum_{1\leq j\leq i}\sgn(\mu'_{j})\Gamma_{y'_{j}}\right).
\end{equation}
The rational maps $g_{i}: \CK_{i-1}\dashrightarrow\CK_{i}$ is the one corresponding to the identity map on $\CL_{0}$ via the description   \eqref{Ki in L0}. Note that we have
\begin{equation*}
\CK_{i}=\CL_{i}\otimes_{\calO_{X\times S}}\calO_{X\times S}\left(\sum_{1\leq j\leq i}\frac{\sgn(\mu'_{j})-\sgn(\mu_{j})}{2}\Gamma_{x_{j}}\right)
\end{equation*}
where $x_{i}:S\to X$ is the image of $x'_{i}$. Therefore $\CK_{i}$ has the same image as $\CL_{i}$ in $\Bun_{T}$. The isomorphism \eqref{HktT mu mu'} induces an isomorphism
\begin{equation}\label{HkT mu mu'}
\Hk^{\mu}_{T}\isom \Hk^{\mu'}_{T}.
\end{equation}
From the construction and the above discussion, this isomorphism preserves the maps $p_{i}$ to $\Bun_{T}$ but {\em does not preserve} the projections to $X'^{r}$ (it only preserves the further projection to $X^{r}$).

Since the isomorphism \eqref{HkT mu mu'} preserves the maps $p_{0}$ and $p_{r}$, the diagram \eqref{ShtrT} implies a canonical isomorphism
\begin{equation}\label{ShtT mu mu'}
\iota_{\mu,\mu'}: \Sht^{\mu}_{T}\isom\Sht^{\mu'}_{T}.
\end{equation}
Just as the map \eqref{HkT mu mu'}, $\iota_{\mu,\mu'}$ does not respect the maps $\pi^{\mu}_{T}$ and $\pi^{\mu'}_{T}$ from $\Sht^{\mu}_{T}$ and $\Sht^{\mu'}_{T}$ to $X'^{r}$: it only respects their further projections to $X^{r}$.

\subsection{The Heegner--Drinfeld cycles} 

\subsubsection{}\label{sss:Pi} We have a morphism
\begin{eqnarray*}
\Pi: \Bun_{T}&\to& \Bun_{G}\\
(\calL\mod\Pic_{X})&\mapsto&(\nu_{*}\calL\mod\Pic_{X})		\index{$\Pi$}%
\end{eqnarray*}

\subsubsection{} For any $\mu$ as in \S\ref{sss:mu} we define a morphism
\begin{equation*}
\wt\theta^{\mu}: \Sht^{\mu}_{\wt{T}}\to \Sht^{\mu}_{2}		
\end{equation*}
as follows. Let $(\CL;x'_{i};\iota)\in\Sht^{\mu}_{\wt{T}}(S)$ as in the description in \S\ref{sss:points ShtT}. Let $\CL_{0}=\CL$ and we may define the line bundles $\CL_{i}$ using \eqref{define Li}. Then there are natural maps $g_{i}:\CL_{i-1}\incl\CL_{i}$ if $\mu_{i}=\mu_{+}$ or $g_{i}: \CL_{i}\incl\CL_{i-1}$ if $\mu_{i}=\mu_{-}$. Let $\nu_S=\nu\times\id_{S}: X'\times S\to X\times S$ the base change of $\nu$. We define
\begin{equation*}
\CE_{i}=\nu_{S*}\CL_{i}
\end{equation*}
with the maps $f_{i}:\CE_{i-1}\to\CE_{i}$ or $\CE_{i}\to \CE_{i-1}$ induced from $g_{i}$. The isomorphism $\iota$ then induces an isomorphism 
\begin{equation*}
\jmath:\CE_{r}=\nu_{S*}\CL_{r}\xrightarrow{\nu_{S*}\iota}\nu_{S*}(\id_{X'}\times \Fr_{S})^{*}\CL_{0}\cong(\id_{X}\times \Fr_{S})^{*}\nu_{S*}{\CL}_{0}=\leftexp{\tau}{\CE_{0}}.
\end{equation*}Let $x_i=\nu\circ x_i'$. The morphism $\wt\theta^{\mu}$ then sends $(\CL;x'_{i};\iota)$ to $(\CE_{i};x_{i};f_{i};\jmath)$.  Clearly $\wt\theta^\mu$ is equivariant with respect to the $\Pic_{X}(k)$-actions. Passing to the quotients, we get a morphism
\begin{equation*}
\ov\theta^{\mu}: \Sht^{\mu}_{T}\to \Sht^{\mu}_{G}.
\end{equation*}
For a different $\mu'$,   the canonical isomorphism $\iota_{\mu,\mu'}$ in \eqref{ShtT mu mu'} intertwines the maps $\ov\theta^{\mu}$ and $\ov\theta^{\mu'}$, i.e., we have a commutative diagram
\begin{equation*}
\xymatrix{\Sht^{\mu}_{T}\ar[d]^{\iota_{\mu,\mu'}}\ar[rr]^{\ov\theta^{\mu}} && \Sht^{\mu}_{G}\ar[d]\\
\Sht^{\mu'}_{T}\ar[rr]^{\ov\theta^{\mu'}} &&\Sht^{\mu'}_{G} }
\end{equation*}
where the right vertical map is the canonical isomorphism in Lemma \ref{l:indep mu}.
By our identification of $\Sht^{\mu}_{G}$ for different $\mu$ (cf. \eqref{HkrG}), we  get a morphism, still denoted by $\ov\theta^{\mu}$,
\begin{equation*}
\ov\theta^{\mu}: \Sht^{\mu}_{T}\to \Sht^{r}_{G}.
\end{equation*}

\subsubsection{} By construction we have a commutative diagram
\begin{equation*}
\xymatrix{\Sht^{\mu}_{T}\ar[d]^{\pi^{\mu}_{T}}\ar[rr]^{\ov\theta^{\mu}} && \Sht^{r}_{G}\ar[d]^{\pi_{G}}\\
X'^{r}\ar[rr]^{\nu^{r}} && X^{r}}
\end{equation*}
Recall that
\begin{equation*}
\Sht'^{r}_{G}:=\Sht^{r}_{G}\times_{X^{r}}X'^{r}.		\index{$\Sht'^{r}_{G}$}%
\end{equation*}
Then the map $\ov\theta^{\mu}$ factors through a morphism
\begin{equation*}
\theta^{\mu}: \Sht^{\mu}_{T}\to \Sht'^{r}_{G}		\index{$\theta^{\mu}$}%
\end{equation*}
over $X'^{r}$. Since $\Sht^{\mu}_{T}$ is proper of dimension $r$, $\theta^{\mu}_{*}[\Sht^{\mu}_{T}]$ is a proper cycle class in $\Sht'^{r}_{G}$ of dimension $r$.

\begin{defn} The {\em Heegner--Drinfeld cycle} of type $\mu$ is the direct image of $[\Sht^{\mu}_{T}]$ under $\theta^{\mu}$:
\begin{equation*}
\theta^{\mu}_{*}[\Sht^{\mu}_{T}]\in \Ch_{c,r}(\Sht'^{r}_{G})_{\QQ}.		\index{$\theta^{\mu}_{*}[\Sht^{\mu}_{T}]$}%
\end{equation*}
\end{defn}

Recall from Proposition \ref{p:Hecke action on Chow} and \S\ref{sss:H action on Chow} that we have an action of $\sH$ on $\Ch_{c,r}(\Sht'^{r}_{G})_{\QQ}$. Since
\begin{equation*}
\dim \Sht^{\mu}_{T}=r=\frac{1}{2}\dim\Sht'^{r}_{G},
\end{equation*}
both $\theta^{\mu}_{*}[\Sht^{\mu}_{T}]$ and $f*\theta^{\mu}_{*}[\Sht^{\mu}_{T}]$ for any function $f\in\sH$ are {\em proper} cycle classes in $\Sht^{r}_{G}$ of complementary dimension, and they define elements in $\Ch_{c,r}(\Sht'^{r}_{G})_{\QQ}$. The following definition then makes sense.

\begin{defn} Let $f\in\sH$ be an unramified  Hecke function. We define the following intersection number 
\begin{align*}
\BI_r(f):=\jiao{\theta^{\mu}_{*}[\Sht^{\mu}_{T}],\quad f*\theta^{\mu}_{*}[\Sht^{\mu}_{T}]}_{\Sht'^{r}_{G}}\in\QQ.		\index{$\BI_r(f)$}%
\end{align*}
\end{defn}

\subsubsection{Changing $\mu$} For different $\mu$ and $\mu'$ as in \S\ref{sss:mu}, the Heegner--Drinfeld cycles $\theta^{\mu}_{*}[\Sht^{\mu}_{T}]$ and $\theta^{\mu'}_{*}[\Sht^{\mu'}_{T}]$ are different. Therefore, a priorily the intersection number $\BI_{r}(f)$ depends on $\mu$. However we have

\begin{lemma} The intersection number $\BI_{r}(f)$ for any $f\in\sH$ is independent of the choice of $\mu$.
\end{lemma}
\begin{proof} Let $Z^{\mu}$ denote the cycle  $\theta^{\mu}_{*}[\Sht^{\mu}_{T}]$.
Using the isomorphism $\iota_{\mu,\mu'}$ in \eqref{ShtT mu mu'}, we see that $Z^{\mu}$ and $Z^{\mu'}$ are transformed to each other under the involution $\sigma(\mu,\mu'):\Sht'^{r}_{G}=\Sht^{r}_{G}\times_{X^{r}}X'^{r}\to \Sht^{r}_{G}\times_{X^{r}}X'^{r}=\Sht'^{r}_{G}$ which is the identity on $\Sht^{r}_{G}$ and on $X'^{r}$ sends $(x'_{1},\cdots, x'_{r})$ to $(y'_{1},\cdots, y'_{r})$ using the formula \eqref{y'}. Since $\sigma(\mu,\mu')$ is the identity on $\Sht^{r}_{G}$, it commutes with the Hecke action on $\Ch_{c,r}(\Sht'^{r}_{G})_\BQ$. Therefore we have
\begin{eqnarray*}
&&\jiao{Z^{\mu},f*Z^{\mu}}_{\Sht'^{r}_{G}}=\jiao{\sigma(\mu,\mu')_{*}Z^{\mu}, \sigma(\mu,\mu')_{*}(f*Z^{\mu})}_{\Sht'^{r}_{G}}\\
&=&\jiao{\sigma(\mu,\mu')_{*}Z^{\mu}, f*(\sigma(\mu,\mu')_{*}Z^{\mu})}_{\Sht'^{r}_{G}}=\jiao{Z^{\mu'},f*Z^{\mu'}}_{\Sht'^{r}_{G}}.
\end{eqnarray*}
\end{proof}

\section{Intersection number as a trace}\label{s:alt}
The goal of this section is to turn the intersection number $\BI_{r}(h_{D})$ into the trace of an operator acting on the cohomology of a certain variety. This will be accomplished in Theorem \ref{th:I trace}. To state the theorem, we need to introduce certain moduli spaces similar to $\calN_{\un{d}}$ defined in \S\ref{sss:calN}.

\subsection{Geometry of $\calM_{d}$}\label{ss:Md}

\subsubsection{}\label{sss:M} Recall $\nu: X'\to X$ is a geometrically connected \'etale double cover. We will use the notation $\hX'_{d}$ and $X'_{d}$ as in \S\ref{ss:sym power}. We have the norm map $\wh{\nu}_{d}:\hX'_{d}\to\hX_{d}$ sending $(\calL,\alpha\in\Gamma(X',\calL))$ to $(\Nm(\calL),\Nm(\alpha)\in\Gamma(X,\Nm(\calL)))$.\index{$\wh{\nu}_{d}$}%

Let $d\geq0$ be an integer. Let $\tcM_{d}$ be the moduli functor whose $S$-points is the groupoid of $(\calL, \calL', \alpha,\beta)$ where
\begin{itemize}
\item $\calL,\calL'\in \Pic(X'\times S)$ such that $\deg(\calL'_{s})-\deg(\calL_{s})=d$ for all geometric points $s\in S$;
\item $\alpha:\calL\to \calL'$ is an $\calO_{X'}$-linear map;
\item $\beta:\calL\to \sigma^{*}\calL'$ is an $\calO_{X'}$-linear map;
\item For each geometric point $s\in S$, the restrictions $\alpha|_{X'\times s}$ and $\beta|_{X'\times s}$ are not both zero.
\end{itemize}
There is a natural action of $\Pic_{X}$ on $\tcM_{d}$ by tensoring: $\calK\in \Pic_{X}$ sends $(\calL,\calL',\alpha,\beta)$ to $(\calL\otimes\nu^{*}\calK,\calL'\otimes\nu^{*}\calK,\alpha\otimes\id_{\calK}, \beta\otimes\id_{\calK})$.  We define
\begin{equation*}
\calM_{d}:=\tcM_{d}/\Pic_{X}.		\index{$\tcM_{d}, \calM_{d}$}%
\end{equation*}

\subsubsection{} To $(\calL,\calL',\alpha,\beta)\in\tcM_{d}$, we may attach
\begin{itemize}
\item $a:=\Nm(\alpha): \Nm(\calL)\to\Nm(\calL')$;
\item $b:=\Nm(\beta):\Nm(\calL)\to \Nm(\sigma^{*}\calL')=\Nm(\calL')$.
\end{itemize}
Both $a$ and $b$ are sections of the same line bundle $\Delta=\Nm(\calL')\otimes\Nm(\calL)^{-1}\in\Pic^{d}_{X}$, and they are not simultaneously zero. The assignment $(\calL,\calL',\alpha,\beta)\mapsto (\Delta,a,b)$ is invariant under the the action of $\Pic_{X}$ on $\tcM_{d}$, and it induces a morphism
\begin{equation*}
f_{\calM}: \calM_{d}\to \calA_{d}.		\index{$f_{\calM}$}%
\end{equation*}
Here $\calA_{d}$ is defined in \S\ref{sss:Ad}.

\subsubsection{}\label{push L} Given $(\calL,\calL',\alpha,\beta)\in\tcM_{d}$, there is a canonical way to attach an $\calO_{X}$-linear map $\psi: \nu_{*}\calL\to\nu_{*}\calL'$ and vice versa. In fact, by adjunction, a map $\psi: \nu_{*}\calL\to\nu_{*}\calL'$ is the same as a map $\nu^{*}\nu_{*}\calL\to\calL'$. Since $\nu^{*}\nu_{*}\calL\cong\calL\oplus\sigma^{*}\calL$ canonically,  the datum of $\psi$ is the same as a map of $\calO_{X'}$-modules $\calL\oplus\sigma^{*}\calL\to \calL'$, and we name the two components of this map by $\alpha$ and $\sigma^{*}\beta$. Note that we have a canonical isomorphism
\begin{equation*}
\Nm(\calL)\cong\det(\nu_{*}\calL)\otimes\det(\nu_{*}\calO)^{\vee},
\end{equation*}
and likewise for $\calL'$. Therefore $\det(\psi): \det(\nu_{*}\calL)\to \det(\nu_{*}\calL')$ can be identified with a map $\Nm(\calL)\to \Nm(\calL')$, which is given by
\begin{equation}\label{det psi}
\det(\psi)=\Nm(\alpha)-\Nm(\beta)=a-b:\Nm(\calL)\to\Nm(\calL').
\end{equation}
The composition $\delta\circ f_{\calM}:\calM_{d}\to \calA_{d}\to \hX_{d}$ takes $(\calL,\calL',\alpha,\beta)$ to the pair $(\Delta=\Nm(\calL')\otimes\Nm(\calL)^{-1}, \det(\psi))$.

\subsubsection{} We give another description of $\calM_{d}$. We have a map $\iota_{\alpha}:\calM_{d}\to \hX'_{d}$ sending $(\calL,\calL',\alpha,\beta)$ to the line bundle $\calL'\otimes\calL^{-1}$ and its section given by $\alpha$. Similarly we have a map $\iota_{\beta}:\calM_{d}\to \hX'_{d}$ sending $(\calL,\calL',\alpha,\beta)$ to the line bundle $\sigma^{*}\calL'\otimes\calL^{-1}$ and its section given by $\beta$. Note that the line bundles underlying $\iota_{\alpha}(\calL,\calL',\alpha,\beta)$ and $\iota_{\beta}(\calL,\calL',\alpha,\beta)$ have the same norm $\Delta=\Nm(\calL')\otimes\Nm(\calL)^{-1}\in\Pic^{d}_{X}$. Since $\alpha$ and $\beta$ are not both zero, we get a map
\begin{equation*}
\iota=(\iota_{\alpha},\iota_{\beta}): \calM_{d}\to \hX'_{d}\times_{\Pic^{d}_{X}}\hX'_{d}-Z'_{d}
\end{equation*}
where the fiber product on the RHS is taken with respect to the map $\hX'_{d}\to \Pic^{d}_{X'}\xrightarrow{\Nm}\Pic^{d}_{X}$, and $Z'_{d}:=\Pic^{d}_{X'}\times_{\Pic^{d}_{X}}\Pic^{d}_{X'}$ is embedded into $\hX'_{d}\times_{\Pic^{d}_{X}}\hX'_{d}$ by viewing $\Pic^{d}_{X'}$ as the zero section of $\hX'_{d}$ in both factors.

\begin{prop}\label{p:M} 
\begin{enumerate}
\item\label{M DM stack} The morphism $\iota$ is an isomorphism of functors, and $\calM_{d}$ is a proper Deligne--Mumford stack over $k$. \footnote{The properness of $\calM_{d}$ will not be used elsewhere in this paper.}

\item\label{M smooth} For $d\geq 2g'-1$,  $\calM_{d}$ is a smooth Deligne--Mumford stack over $k$ of pure dimension $2d-g+1$.
\item\label{nud proper} The morphism $\wh{\nu}_{d}:\hX'_{d}\to \hX_{d}$ is proper.
\item\label{fM proper} We have a Cartesian diagram
\begin{equation}\label{fM nu}
\xymatrix{\CM_{d}\ar@{^{(}->}[rr]^{\iota}\ar[d]^{f_{\calM}} && \hX'_{d}\times_{\Pic^{d}_{X}}\hX'_{d}\ar[d]^{\wh{\nu}_{d}\times\wh{\nu}_{d}}\\
\calA_{d}\ar@{^{(}->}[rr] && \hX_{d}\times_{\Pic^{d}_{X}}\hX_{d}}
\end{equation}
Moreover, the map $f_{\calM}$ is proper.
\end{enumerate}
\end{prop}
\begin{proof}
(1) Let $(\Pic_{X'}\times\Pic_{X'})_{d}$ be the disjoint union of $\Pic^{i}_{X'}\times\Pic^{i+d}_{X'}$ over all $i\in\ZZ$. Consider the morphism $\theta: (\Pic_{X'}\times\Pic_{X'})_{d}/\Pic_{X}\to \Pic^{d}_{X'}\times_{\Pic^{d}_{X}}\Pic^{d}_{X'}$ (the fiber product is taken with respect to the norm map) that sends $(\calL,\calL')$ to $(\calL'\otimes\calL^{-1},\sigma^{*}\calL'\otimes\calL^{-1},\tau)$, where $\tau$ is the tautological isomorphism between $\Nm(\calL'\otimes\calL^{-1})\cong\Nm(\calL')\otimes\Nm(\calL)^{-1}$ and $\Nm(\sigma^{*}\calL'\otimes\calL^{-1})\cong\Nm(\calL')\otimes\Nm(\calL)^{-1}$. By definition, we have a Cartesian diagram
\begin{equation}\label{Md to Pic}
\xymatrix{\calM_{d}\ar[r]^{\iota}\ar[d]^{\om} & \hX'_{d}\times_{\Pic^{d}_{X}}\hX'_{d}-Z'_{d}\ar[d]\\
(\Pic_{X'}\times\Pic_{X'})_{d}/\Pic_{X}\ar[r]^{\theta} & \Pic^{d}_{X'}\times_{\Pic^{d}_{X}}\Pic^{d}_{X'}}
\end{equation}
where the map $\om$ sends $(\calL,\calL',\alpha,\beta)$ to $(\calL,\calL')$. Therefore it suffices to check that $\theta$ is an isomorphism. For this we will construct an inverse to $\theta$.

From the exact sequence of \'etale sheaves
\begin{equation*}
1\to \calO^{\times}_{X}\xrightarrow{\nu^{*}}\nu_{*}\calO^{\times}_{X'}\xrightarrow{\id-\sigma}\nu_{*}\calO^{\times}_{X'}\xrightarrow{\Nm}\calO^{\times}_{X}\to 1
\end{equation*}
we get an exact sequence of Picard stacks
\begin{equation*}
1\to \Pic_{X'}/\Pic_{X}\xrightarrow{\id-\sigma} \Pic^{0}_{X'}\xrightarrow{\Nm}\Pic^{0}_{X}\to 1.
\end{equation*}
Given $(\calK_{1},\calK_{2},\tau)\in\Pic^{d}_{X'}\times_{\Pic^{d}_{X}}\Pic^{d}_{X'}$ (where $\tau:\Nm(\calK_{1})\cong\Nm(\calK_{2})$), there is a unique object $\calL'\in\Pic_{X'}/\Pic_{X}$ such that $\calL'\otimes\sigma^{*}\calL'^{-1}\cong\calK_{1}\otimes\calK^{-1}_{2}$ compatible with the trivializations of the norms to $X$ of both sides. We then define $\psi(\calK_{1},\calK_{2},\tau)=(\calL'\otimes\calK^{-1}_{1},\calL')$, which is a well-defined object  in $(\Pic_{X'}\times\Pic_{X'})_{d}/\Pic_{X}$. It is easy to check that $\psi$ is an inverse to $\theta$. This proves that $\theta$ is an isomorphism, and so is $\iota$.

We show that $\calM_{d}$  is a proper Deligne--Mumford stack over $k$. By extending $k$ we may assume that $X'$ contains a $k$-point, and we fix a point $y \in X'(k)$. We consider the moduli stack $\hM_{d}$ classifying $(\calK_{1},\gamma_{1}, \calK, \rho, \alpha,\beta)$ where $\calK_{1}\in\Pic_{X'}^{d}$, $\gamma_{1}$ is a trivialization of the stalk $\calK_{1,y}$, $\calK\in\Pic^{0}_{X'}$, $\rho$ is an isomorphism $\Nm(\calK)\cong\calO_{X}$, $\alpha$ is a section of $\calK_{1}$ and $\beta$ is a section of $\calK_{1}\otimes\calK$ such that $\alpha$ and $\beta$ are not both zero. There is a canonical map $p: \hM_{d}\to \hX'_{d}\times_{\Pic^{d}_{X}}\hX'_{d}-Z'_{d}$ sending $(\calK_{1},\gamma_{1}, \calK, \rho, \alpha,\beta)$ to $(\calK_{1}, \calK_{2}:=\calK_{1}\otimes\calK, \tau, \alpha,\beta)$ (the isomorphism $\tau:\Nm(\calK_{1})\cong\Nm(\calK_{2})$ is induced from the trivialization $\rho$). Clearly $p$ is the quotient map for the $\Gm$-action on $\hM_{d}$ that scales $\gamma_{1}$. There is another $\Gm$-action on $\hM_{d}$ that scales $\alpha$ and $\beta$ simultaneously. Using automorphisms of $\calK_{1}$, we have a canonical identification of the two $\Gm$-actions on $\hM_{d}$; however,  to distinguish them, we call the first torus $\Gm(y)$ and the second $\Gm(\alpha,\beta)$. By the above discussion, $\iota^{-1}\circ p$ gives an isomorphism $\hM_{d}/\Gm(y)\cong \calM_{d}$, hence also an isomorphism $\hM_{d}/\Gm(\alpha,\beta)\cong \calM_{d}$. 

Let $\Prym_{X'/X}:=\ker(\Nm:\Pic^{0}_{X'}\to \Pic^{0}_{X})$ which classifies a line bundle $\calK$ on $X'$ together with a trivialization of $\Nm(\calL)$.  \index{$\Prym_{X'/X}$}%
This is a Deligne--Mumford stack isomorphic to the usual Prym variety divided by the trivial action of $\mu_{2}$. Let $J^{d}_{X'}$ be the degree $d$-component of the Picard {\em scheme} of $X'$, which classifies a line bundle $\calK_{1}$ on $X'$ of degree $d$ together with a trivialization of the stalk $\calK_{1,y}$. We have a natural map $h: \hM_{d}\to J^{d}_{X'}\times\Prym_{X'/X}$ sending $(\calK_{1},\gamma_{1}, \calK, \rho, \alpha,\beta)$  to $(\calK_{1},\gamma_{1})\in J^{d}_{X'}$ and $(\calK,\rho)\in\Prym_{X'/X}$. The map $h$ is invariant under the $\Gm(\alpha,\beta)$-action, hence induces a map
\begin{equation}\label{Md to Prym}
\ov{h}: \hM_{d}/\Gm(\alpha,\beta)\cong\calM_{d}\to J^{d}_{X'}\times\Prym_{X'/X}
\end{equation}
The fiber of $\ov{h}$ over a point $((\calK_{1},\gamma_{1}),(\calK,\rho))\in J^{d}_{X'}\times\Prym_{X'/X}$ is the projective space $\PP(\Gamma(X',\calK_{1})\oplus\Gamma(X',\calK_{1}\otimes\calK))$. In particular, the map $\ov{h}$ is proper and schematic. Since $J^{d}_{X'}\times\Prym_{X'/X}$ is a proper Deligne--Mumford stack over $k$, so is $\calM_{d}$.

(2) Since $\calM_{d}$ is covered by open substacks $X'_{d}\times_{\Pic^{d}_{X}}\hX'_{d}$ and $\hX'_{d}\times_{\Pic^{d}_{X}}X'_{d}$, it suffices to show that both of them are smooth over $k$. For $d\geq 2g'-1$ the Abel-Jacobi map $\AJ_{d}: X'_{d}\to\Pic^{d}_{X'}$ is smooth of relative dimension $d-g'+1$, hence $X'_{d}$ is smooth over $\Pic^{d}_{X}$ of relative dimension $d-g+1$. Therefore both $X'_{d}\times_{\Pic^{d}_{X}}\hX'_{d}$ and $\hX'_{d}\times_{\Pic^{d}_{X}}X'_{d}$ are smooth over $X'_{d}$ of relative dimension $d-g+1$. We conclude that $\calM_{d}$ is a smooth Deligne--Mumford stack of dimension $2d-g+1$ over $k$.

(3) We introduce a compactification $\ov{X}'_{d}$ of $\hX'_{d}$ as follows. Consider the product $\hX'_{d}\times\AA^{1}$ with the natural $\Gm$-action scaling both the section of the line bundle and the scalar in $\AA^{1}$. Let $z_{0}:\Pic^{d}_{X'}\incl \hX'_{d}\times\AA^{1}$ sending $\calL$ to $(\calL,0,0)$. Let $\ov{X}'_{d}:=(\hX'_{d}\times\AA^{1}-z_{0}(\Pic^{d}_{X'}))/\Gm$. Then the fiber of $\ov{X}'_{d}$ over $\calL\in\Pic^{d}_{X'}$ is the projective space $\PP(\Gamma(X', \calL\oplus\calO_{X'}))$. In particular, $\ov{X}'_{d}$ is proper and schematic over $\Pic^{d}_{X'}$. The stack $\ov{X}'_{d}$ contains $\hX'_{d}$ as an open substack where the $\AA^{1}$-coordinate is invertible, whose complement is isomorphic to the projective space bundle $X'_{d}/\Gm$ over $\Pic^{d}_{X'}$. Similarly we a have compactification $\ov{X}_{d}$ of $\hX_{d}$. 

Consider the quadratic map $\hX'_{d}\times\AA^{1}\to \hX_{d}\times\AA^{1}$ sending $(\calL,s,\l)\mapsto (\Nm(\calL),\Nm(s),\l^{2})$. This quadratic map passes to the projectivizations because $(\Nm(s),\l^{2})=(0,0)$ implies $(s,\l)=(0,0)$ on the level of field-valued points. The resulting map $\ov{\nu}_{d}: \ov{X}'_{d}\to \ov{X}_{d}$ extends $\wh{\nu}_{d}$. We may factorize $\ov{\nu}_{d}$ as the composition
\begin{equation*}
\ov{\nu}_{d}:\ov{X}'_{d}\to \ov{X}_{d}\times_{\Pic^{d}_{X}}\Pic^{d}_{X'}\to \ov{X}_{d}
\end{equation*}
Here the first map is proper because both the source and the target are proper over $\Pic^{d}_{X'}$; the second map is proper by the properness of the norm map $\Nm: \Pic^{d}_{X'}\to \Pic^{d}_{X}$. We conclude that $\ov{\nu}_{d}$ is proper. Since $\wh{\nu}_{d}$ is the restriction of $\ov{\nu}_{d}$ to $\hX_{d}\incl \ov{X}_{d}$, it is also proper.
 
(4) The commutativity of the diagram \eqref{fM nu} is clear from the construction of $\iota$. Note that $Z'_{d}$ is the preimage of $Z_{d}$ under $\wh{\nu}_{d}\times\wh{\nu}_{d}$, and $\calM_{d}$ and $\calA_{d}$ are complements of $Z'_{d}$ and $Z_{d}$ respectively. Therefore \eqref{fM nu} is also Cartesian. Now the properness of $f_{\calM}$ follows from the properness of $\wh{\nu}_{d}$ proved in part (3) together with the Cartesian diagram \eqref{fM nu}.

\end{proof}

\subsection{A formula for $\BI_{r}(h_{D})$} 

\subsubsection{The correspondence $\Hk^{\mu}_{\calM,d}$}\label{sss:alt HkM} Fix any tuple $\mu=(\mu_{1},\cdots,\mu_{r})$ as in \S\ref{sss:arb mu}. We define $\wt\Hk^{\mu}_{\CM,d}$ to be the moduli functor whose $S$-points classify the following data
\begin{enumerate}
\item For $i=1,\cdots, r$, a map $x'_{i}:S\to X'$ with graph $\Gamma_{x'_{i}}$.
\item For each $i=0,1,\cdots, r$, an $S$-point $(\CL_{i}, \CL'_{i}, \alpha_{i}, \beta_{i})$ of $\tcM_{d}$:
\begin{equation*}
\alpha_{i}: \CL_{i}\to \CL'_{i},\quad \beta_{i}: \CL_{i}\to \sigma^{*}\CL'_{i}.
\end{equation*}
In particular, $\deg\CL'_{i}-\deg\CL_{i}=d$ and $\alpha_{i}$ and $\beta_{i}$ are not both zero.
\item A commutative diagram of $\calO_{X'}$-linear maps between line bundles on $X'$
\begin{equation}\label{Hmu}
\xymatrix{\calL_{0}\ar[d]^{\alpha_{0}}\ar@{-->}[r]^{f_{1}} & \calL_{1}\ar[d]^{\alpha_{1}} \ar@{-->}[r]^{f_{2}} & \cdots\ar@{-->}[r]^{f_{r}} & \calL_{r}\ar[d]^{\alpha_{r}}\\
\calL'_{0}\ar@{-->}[r]^{f'_{1}} & \calL'_{1}\ar@{-->}[r]^{f'_{2}} & \cdots\ar@{-->}[r]^{f'_{r}} & \calL'_{r}}
\end{equation}
where the top and bottom rows are $S$-points of $\Hk^{\mu}_{T}$ over the same point $(x'_{1},\cdots, x'_{r})\in X'^{r}(S)$, such that the following diagram is also commutative 
\begin{equation}\label{Hmu2}
\xymatrix{\calL_{0}\ar[d]^{\beta_{0}}\ar@{-->}[r]^{f_{1}} & \calL_{1}\ar[d]^{\beta_{1}} \ar@{-->}[r]^{f_{2}} & \cdots\ar@{-->}[r]^{f_{r}} & \calL_{r}\ar[d]^{\beta_{r}}\\
\sigma^{*}\calL'_{0}\ar@{-->}[r]^{\sigma^{*}f'_{1}} & \sigma^{*}\calL'_{1}\ar@{-->}[r]^{\sigma^{*}f'_{2}} & \cdots\ar@{-->}[r]^{\sigma^{*}f'_{r}} & \sigma^{*}\calL'_{r}}
\end{equation}
\end{enumerate}

There is an action of $\Pic_{X}$ on $\wt\Hk^{\mu}_{\CM,d}$ by tensoring on the line bundles $\CL_{i}$ and $\CL'_{i}$. We define
\begin{equation*}
\Hk^{\mu}_{\CM,d}:=\wt\Hk^{\mu}_{\CM,d}/\Pic_{X}.		\index{$\Hk^{\mu}_{\CM,d}, \wt\Hk^{\mu}_{\CM,d}$}%
\end{equation*}
The same argument as \S\ref{HkT indep mu} (applying the isomorphism \eqref{HkT mu mu'} to both rows of \eqref{Hmu2}) shows that for different choices of $\mu$, the stacks $\Hk^{\mu}_{\CM,d}$ are canonically isomorphic to each other. However, as in the case for $\Hk^{\mu}_{T}$, the morphism $\Hk^{\mu}_{\CM,d}\to X'^{r}$ does depend on $\mu$. 

\subsubsection{}\label{sss:same proj A} Let $\gamma_{i}: \Hk^{\mu}_{\calM,d}\to\calM_{d}$ be the projections given by taking the diagram \eqref{Hmu} to its $i$-th column. It is clear that this map is schematic, therefore $\Hk^{\mu}_{\calM,d}$ itself is a Deligne--Mumford stack.

In the diagram \eqref{Hmu}, the line bundles $\Delta_{i}=\Nm(\calL_{i}')\otimes\Nm(\calL_{i})^{-1}$ are all canonically isomorphic to each other for $i=0,\cdots, r$. Also the sections $a_{i}=\Nm(\alpha_{i})$ (resp. $b_{i}=\Nm(\beta_{i})$) of $\Delta_{i}$ can be identified with each other for all $i$ under the isomorphisms between the $\Delta_{i}$'s. Therefore, composing $\gamma_{i}$ with the map $f_{\calM}:\CM_{d}\to \CA_{d}$ all give the same map. We may view $\Hk^{\mu}_{\CM,d}$ as a self-correspondence of $\calM_{d}$ over $\calA_{d}$ via the maps $(\gamma_{0},\gamma_{r})$.

There is a stronger statement. Let us define $\tcA_{d}\subset \hX'_{d}\times_{\Pic^{d}_{X}}\hX_{d}$ to be preimage of  $\calA_{d}$ under $\Nm\times\id:\hX'_{d}\times_{\Pic^{d}_{X}}\hX_{d}\to \hX_{d}\times_{\Pic^{d}_{X}}\hX_{d}$. Then $\tcA_{d}$ classifies triples $(\calK,\alpha,b)$ where $\calK\in\Pic_{X'}$, $\alpha$ is a section of $\calK$ and $b$ is a section of $\Nm(\calK)$ such that $\alpha$ and $b$ are not simultaneously zero. Then $f_{\calM}$ factors through the map
\begin{equation*}
\tf_{\calM}:\calM_{d}\to\tcA_{d}		\index{$\tcA_{d}$} \index{$\tf_{\calM}$}%
\end{equation*}
sending $(\CL,\CL',\alpha,\beta)$ to $(\CL'\otimes\CL^{-1}, \alpha,\Nm(\beta))$. 

Consider a point of $\Hk^{\mu}_{\CM,d}$ giving among others the diagram \eqref{Hmu}. Since the maps $f_{i}$ and $f'_{i}$ are simple modifications at the same point $x'_{i}$, the line bundles $\CL'_{i}\otimes\CL^{-1}_{i}$ are all isomorphic to each other for all $i=0,1,\cdots, r$. Under these isomorphisms, their sections given by $\alpha_{i}$ correspond to each other. Therefore the maps $\tf_{\calM}\circ\gamma_{i}:\Hk^{\mu}_{\CM,d}\to \tcA_{d}$ are the same for all $i$.

\subsubsection{}\label{sss:define calH} The particular case $r=1$ and $\mu=(\mu_{+})$ gives a moduli space $\calH:=\Hk^{1}_{\CM,d}$ classifying commutative diagrams up to simultaneous tensoring by $\Pic_{X}$:
\begin{equation}\label{Hdiag}
\xymatrix{\calL_{0}\ar[r]^{f}\ar[d]^{\alpha_{0}} & \calL_{1}\ar[d]^{\alpha_{1}} & \calL_{0}\ar[r]^{f}\ar[d]^{\beta_{0}} & \calL_{1}\ar[d]^{\beta_{1}} \\
\calL'_{0}\ar[r]^{f'} & \calL'_{1} & \sigma^{*}\calL'_{0}\ar[r]^{\sigma^{*}f'} & \sigma^{*}\calL'_{1}}				\index{$\calH=\Hk^{1}_{\CM,d}$}%
\end{equation}
such that the cokernel of $f$ and $f'$ are invertible sheaves supported at the same point $x'\in X'$, and the data $(\CL_{0},\CL'_{0},\alpha_{0},\beta_{0})$ and $(\CL_{1},\CL'_{1},\alpha_{1},\beta_{1})$ are objects of $\CM_{d}$. 

We have two maps $(\gamma_{0},\gamma_{1}):\calH\to \calM_{d}$, and we view $\calH$ as a self-correspondence of $\CM_{d}$ over $\calA_{d}$. We also have a map $p:\calH\to X'$ recording the point $x'$ (support of $\CL_{1}/\CL_{0}$ and $\CL'_{1}/\CL'_{0}$).

The following lemma follows directly from the definition of $\Hk^{\mu}_{\CM,d}$.
\begin{lemma}\label{l:Hk r comp} As a self-correspondence of $\calM_{d}$, $\Hk^{\mu}_{\calM,d}$ is canonically isomorphic to the $r$-fold composition of $\calH$
\begin{equation*}
\Hk^{\mu}_{\CM,d}\cong\calH\times_{\gamma_{1},\calM_{d},\gamma_{0}}\times\calH\times_{\gamma_{1},\calM_{d},\gamma_{0}}\times\cdots\times_{\gamma_{1},\calM_{d},\gamma_{0}}\calH.
\end{equation*}
\end{lemma}

\subsubsection{}\label{sss:Ads} Let $\Ads_{d}\subset\calA_{d}$ be the open subset consisting of $(\Delta,a,b)$ where $b\neq0$, i.e., $\Ads_{d}=\hX_{d}\times_{\Pic^{d}_{X}}X_{d}$ under the isomorphism \eqref{AXX}. Let $\Mds_{d}$, $\Hk^{\mu}_{\Mds,d}$ and $\calH^{\ds}$ be the preimages of $\Ads_{d}$ in $\calM_{d}$, $\Hk^{\mu}_{\CM,d}$ and $\calH$.
\index{$\calA_{d}^{\diam}$} \index{$\calM_{d}^\diam$} \index{$\Hk^{\mu}_{\Mds,d}$} \index{$\calH^{\ds}$}%

\begin{lemma}\label{l:HInc} Let $I'_{d}\subset X'_{d}\times X'$ be the incidence scheme, i.e., $I'_{d}\to X'_{d}$ is the universal family of degree $d$ effective divisors on $X'$. There is a natural map $\calH^{\ds}\to I'_{d}$ such that the diagram 
\begin{eqnarray}
\label{HInc}\xymatrix{ & \calH^{\ds}\ar@/^1pc/[rr]^{p}\ar[r]\ar[d]^{\gamma_{1}} & I'_{d}\ar[d]^{q}\ar[r]_{p_{I'}} & X'\\
\Mds_{d}\ar@{=}[r] & \hX'_{d}\times_{\Pic^{d}_{X}}X'_{d}\ar[r]^(.6){\pr_{2}} & X'_{d}}					\index{$I'_d$}%
\end{eqnarray}
is commutative and the square is Cartesian. Here the $q:I'_{d}\to X'_{d}$ sends $(D,y)\in X'_{d}\times X'$ to $D-y+\sigma(y)$, and $p_{I'}:I'_{d}\to X'$ sends $(D,y)$ to $y$.
\end{lemma}
\begin{proof}
A point in $\calH^{\ds}$ is a diagram as in \eqref{Hdiag} with $\beta_{i}$ nonzero (hence injections). Such a diagram is uniquely determined by $(\calL_{0},\calL'_{0},\alpha_{0},\beta_{0})\in\Mds_{d}$ and $y=\div(f)\in X'$ for then $\calL_{1}=\calL_{0}(y)$, $\calL'_{1}=\calL'_{0}(y)$ are determined, and $f, f'$ are the obvious inclusions and $\alpha_{1}$ the unique map making the first diagram in \eqref{Hdiag} commutative; the commutativity of the second diagram uniquely determines $\beta_{1}$, but there is a condition on $y$ to make it possible:
\begin{equation*}
\div(\beta_{0})+\sigma(y)=\div(\beta_{1})+y\in X'_{d+1}.
\end{equation*}
Since $\sigma$ acts on $X'$ without fixed points, $y$ must appear in $\div(\beta_{0})$. The assignment $\calH^{\ds}\ni(y,\calL_{0},\cdots,\beta_{0}, \calL_{1},\cdots, \beta_{1})\mapsto (\div(\beta_{0}), y)$ then gives a point in $I'_{d}$.  The above argument shows that the square in \eqref{HInc} is Cartesian and the triangle therein is commutative.
\end{proof}

\begin{lemma}\label{l:dim HkM} We have
\begin{enumerate}
\item The map $\gamma_{0}: \Hk^{\mu}_{\Mds,d}\to\Mds_{d}$ is finite and surjective. In particular, $\dim\Hk^{\mu}_{\Mds,d}=\dim\Mds_{d}=2d-g+1$.
\item The dimension of the image of $\Hk^{\mu}_{\CM,d}-\Hk^{\mu}_{\Mds,d}$ in $\CM_{d}\times\CM_{d}$ is at most $d+2g-2$.
\end{enumerate}
\end{lemma}
\begin{proof}
(1) In the case $r=1$, this follows from the Cartesian square in \eqref{HInc}, because the map $q:I'_{d}\to X'_{d}$ is finite. For general $r$, the statement follows by induction from Lemma \ref{l:Hk r comp}.

(2) The closed subscheme $Y=\Hk^{\mu}_{\CM,d}-\Hk^{\mu}_{\Mds,d}$ classifies diagrams \eqref{Hmu} only because all the $\beta_{i}$ are zero. Its image $Z\subset \CM_{d}\times\CM_{d}$  under $(\gamma_{0},\gamma_{r})$ consists of pairs of points $(\CL_{0},\CL'_{0},\alpha_{0},0)$ and $(\CL_{r},\CL'_{r},\alpha_{r},0)$ in $\calM_{d}$ such that there exists a diagram of the form \eqref{Hmu} connecting them. In particular, the divisors of $\alpha_{0}$ and $\alpha_{r}$ are the same. Therefore such a point in  $Z$ is completely determined by two points $\CL_{0},\CL_{r}\in\Bun_{T}$ and a divisor $D\in X'_{d}$ (as the divisor of $\alpha_{0}$ and $\alpha_{r}$). We see that $\dim Z\leq 2\dim \Bun_{T}+\dim X'_{d}=d+2g-2$.
\end{proof}

\subsubsection{}\label{sss:Hds} Recall $\calH=\Hk^{1}_{\CM,d}$ is a self-correspondence of $\calM_{d}$ over $\calA_{d}$ (see the discussion in \S\ref{sss:same proj A}). Let
\begin{equation*}
[\calH^{\ds}]\in\Ch_{2d-g+1}(\calH)_{\QQ}
\end{equation*}
denote the class of the closure of $\calH^{\ds}$. The image of $[\calH^{\ds}]$ in the Borel-Moore homology group $\hBM{2(2d-g+1)}{\calH\otimes_{k}\kbar}(-2d+g-1)$ defines a cohomological self-correspondence of the constant sheaf $\Ql$ on $\calM_{d}$. According the discussion in \S\ref{sss:corr}, it induces an endomorphism
\begin{equation*}
f_{\calM,!}[\calH^{\ds}]: \bR f_{\calM, !}\Ql\to \bR f_{\calM, !}\Ql
\end{equation*}
For a point $a\in\calA_{d}(k)$, we denote the action of $f_{\calM,!}[\calH^{\ds}]$ on the geometric stalk $(\bR f_{\calM, !}\Ql)_{\ov{a}}=\cohoc{*}{f_{\calM}^{-1}(\ov{a})\otimes_{k}\kbar}$ by $(f_{\calM,!}[\calH^{\ds}])_{a}$.

Recall from \S\ref{ss:AD} that $\calA_{D}=\delta^{-1}(D)\subset\Ah_{d}$ is the fiber of $D$ under $\delta: \calA_{d}\to \hX_{d}$. The  main result of this section is the following.

\begin{theorem}\label{th:I trace} Suppose $D$ is an effective divisor on $X$ of degree $d\geq\max\{2g'-1,2g\}$. Then we have
\begin{equation}\label{I trace}
\BI_{r}(h_{D})=\sum_{a\in\calA_{D}(k)}\Tr\left((f_{\calM,!}[\calH^{\ds}])^{r}_{a}\circ\Frob_{a}, (\bR f_{\calM,!}\Ql)_{\ov{a}}\right).
\end{equation}
\end{theorem}

\subsubsection{Orbital decomposition of $\BI_{r}(h_{D})$}\label{sss:orb I} According Theorem \ref{th:I trace}, we may write
\begin{equation}\label{I orb decomp}
\BI_{r}(h_{D})=\sum_{u\in\PP^{1}(F)-\{1\}}\BI_{r}(u,h_{D})
\end{equation}
where
\begin{align}\label{orb I}
\BI_{r}(u,h_{D})=\begin{cases} \Tr\left((f_{\calM,!}[\calH^{\ds}])^{r}_{a}\circ\Frob_{a}, (\bR f_{\calM,!}\Ql)_{\ov{a}}\right) & \text{ if $u=\inv_{D}(a)$ for some $a\in\calA_{D}(k)$};\\
0 & \text{ otherwise.} \end{cases} \index{$\BI_{r}(u,h_{D})$}%
\end{align}

The rest of the section is devoted to the proof of this theorem. In the rest of this subsection we assume $d\geq\max\{2g'-1,2g\}$.

\subsubsection{} We apply the discussion in Appendix \S\ref{sss:corr Frob} to $M=\calM_{d}\xrightarrow{f_{\calM}}S=\calA_{d}$ and the self-correspondence $C=\Hk^{\mu}_{\CM,d}$ of $\CM_{d}$. We define $\Sht^{\mu}_{\CM,d}$ by the Cartesian diagram
\begin{equation}\label{NW}
\xymatrix{\Sht^{\mu}_{\CM,d}\ar[rr]\ar[d]& & \Hk^{\mu}_{\CM,d}\ar[d]^{(\gamma_{0},\gamma_{r})}\\
\CM_{d}\ar[rr]^{(\id,\Fr_{\CM_{d}})} && \CM_{d}\times\CM_{d}}			\index{$\Sht^{\mu}_{\CM,d}$}%
\end{equation}
This fits into the situation of \S\ref{sss:corr Frob} because $f_{\calM}\circ\gamma_{0}=f_{\calM}\circ\gamma_{r}$ by the discussion in \S\ref{sss:same proj A}, hence $\Hk^{\mu}_{\calM,d}$ is a self-correspondence of $\calM_{d}$ over $\calA_{d}$ while $(\id,\Fr_{\CM_{d}})$ covers the map $(\id,\Fr_{\calA_{d}}):\calA_{d}\to\calA_{d}\times\calA_{d}$. In particular we have a decomposition
\begin{equation}\label{ShtM into a}
\Sht^{\mu}_{\CM,d}=\coprod_{a\in \calA_{d}(k)}\Sht^{\mu}_{\CM,d}(a).  \index{$\Sht^{\mu}_{\CM,d}(a)$}%
\end{equation}
For $D\in X_{d}(k)$, we let
\begin{equation}\label{decomp D into a}
\Sht^{\mu}_{\CM,D}:=\coprod_{a\in \calA_{D}(k)}\Sht^{\mu}_{\CM,d}(a)\subset \Sht^{\mu}_{\CM,d}. \index{$\Sht^{\mu}_{\CM,D}$}%
\end{equation}
Using the decompositions \eqref{ShtM into a} and \eqref{decomp D into a}, we get a decomposition
\begin{equation}\label{Chow0 MD}
\Ch_{0}(\Sht^{\mu}_{\CM,d})_{\QQ}=\left(\bigoplus_{D\in X_{d}(k)}\Ch_{0}(\Sht^{\mu}_{\CM,D})_{\QQ}\right)\oplus\left(\bigoplus_{a\in\calA_{d}(k)-\Ah_{d}(k)}\Ch_{0}(\Sht^{\mu}_{\CM,d}(a))_{\QQ}\right).
\end{equation}

Let $\zeta\in\Ch_{2d-g+1}(\Hk^{\mu}_{\CM,d})_{\QQ}$. Since $\CM_{d}$ is a smooth Deligne--Mumford stack by Proposition \ref{p:M}\eqref{M smooth},  $(\id, \Fr_{\CM_{d}})$ is a regular local immersion, the refined Gysin map  (which is the same as intersecting with the Frobenius graph $\Gamma(\Fr_{\CM_{d}})$ of $\CM_{d}$) is defined
\begin{equation*}
(\id, \Fr_{\CM_{d}})^{!}: \Ch_{2d-g+1}(\Hk^{\mu}_{\CM,d})_{\QQ}\to\Ch_{0}(\Sht^{\mu}_{\CM,d})_{\QQ}
\end{equation*}
Under the decomposition \eqref{Chow0 MD}, we denote the component of $(\id,\Fr_{\CM_{d}})^{!}\zeta$ in the direct summand $\Ch_{0}(\Sht^{\mu}_{\CM,D})_{\QQ}$ by 
\begin{equation*}
\left((\id,\Fr_{\calM_{d}})^{!}\zeta\right)_{D}\in\Ch_{0}(\Sht^{\mu}_{\CM,D})_{\QQ}.
\end{equation*}
Composing with the degree map (which exists because $\Sht^{\mu}_{\CM,D}$ is proper over $k$, see the discussion after \eqref{s component of ShtC}), we define
\begin{equation*}
\jiao{\zeta,\Gamma(\Fr_{\CM_{d}})}_{D}:=\deg\left((\id,\Fr_{\calM_{d}})^{!}\zeta\right)_{D}\in\QQ.		
\end{equation*}

As the first step towards the proof of Theorem \ref{th:I trace}, we have the following result.
\begin{theorem}\label{th:alt I} Suppose $D$ is an effective divisor on $X$ of degree $d\geq\max\{2g'-1,2g\}$, then there exists a class $\zeta\in\Ch_{2d-g+1}(\Hk^{\mu}_{\CM,d})_{\QQ}$ whose restriction to $\Hk^{\mu}_{\CM,d}|_{\Ah_{d}\cap\Ads_{d}}$ is the fundamental cycle, such that 
\begin{equation}\label{alternative I}
\BI_{r}(h_{D})=\jiao{\zeta,\Gamma(\Fr_{\CM_{d}})}_{D}.
\end{equation}
\end{theorem}
This theorem will be proved in \S\ref{sss:proof alt I}, after introducing some auxiliary moduli stacks in the next subsection.

\subsubsection{Proof of Theorem \ref{th:I trace}}
Granting Theorem \ref{th:alt I}, we now prove Theorem \ref{th:I trace}. Let $\zeta\in\Ch_{2d-g+1}(\Hk^{\mu}_{\CM,d})_{\QQ}$ be the class as in Theorem \ref{th:alt I}. By \eqref{decomp D into a}, we have a decomposition
\begin{equation}\label{Ch0 D a}
\Ch_{0}(\Sht^{\mu}_{\CM,D})_{\QQ}=\bigoplus_{a\in\calA_{D}(k)}\Ch_{0}(\Sht^{\mu}_{\CM,d}(a))_{\QQ}.
\end{equation}
We write
\begin{equation*}
\jiao{\zeta,\Gamma(\Fr_{\CM_{d}})}_{D}=\sum_{a\in\calA_{D}(k)}\jiao{\zeta,\Gamma(\Fr_{\CM_{d}})}_{a}
\end{equation*}
under the decomposition \eqref{Ch0 D a}, where $\jiao{\zeta,\Gamma(\Fr_{\CM_{d}})}_{a}$ is the degree of $\left((\id,\Fr_{\CM_{d}})^{!}\zeta\right)_{a}\in\Ch_{0}(\Sht^{\mu}_{\CM,d}(a))_{\QQ}$. Combining this with Theorem \ref{th:alt I} we get
\begin{equation}\label{I sum a}
\BI_{r}(h_{D})=\sum_{a\in\calA_{D}(k)}\jiao{\zeta,\Gamma(\Fr_{\CM_{d}})}_{a}.
\end{equation}

On the other hand, by Proposition \ref{p:Fix}, we have for any $a\in \calA_{D}(k)$
\begin{equation}\label{trace zeta}
\jiao{\zeta, \Gamma(\Fr_{\CM_{d}})}_{a}=\Tr((f_{\calM,!}\cl(\zeta))_{a}\circ\Frob_{a}, (\bR f_{\calM,!}\Ql)_{\ov{a}}).
\end{equation}
Here we are viewing the cycle class $\cl(\zeta)\in\hBM{2(2d-g+1)}{\Hk^{\mu}_{\CM,d}}(-2d+g-1)$ as a cohomological self-correspondence of the constant sheaf $\Ql$ on $\calM_{d}$, which induces an endomorphism
\begin{equation}\label{clzeta}
f_{\calM,!}\cl(\zeta): \bR f_{\calM, !}\Ql\to \bR f_{\calM, !}\Ql,
\end{equation}
and $(f_{\calM,!}\cl(\zeta))_{a}$ is the induced endomorphism on the geometric stalk $(\bR f_{\calM,!}\Ql)_{\ov{a}}$. Since we only care about the action of $f_{\calM,!}\cl(\zeta)$ on stalks in $\Ah_{d}$, only the restriction $\zeta^{\hs}:=\zeta|_{\Ah_{d}}\in Z_{2d-g+1}(\Hk^{\mu}_{\CM,d}|_{\Ah_{d}})_{\QQ}$ matters. Combining \eqref{trace zeta} with \eqref{I sum a}, we see that in order to prove \eqref{I trace}, it suffices to show that $f_{\calM,!}\cl(\zeta^{\hs})$ and $(f_{\calM,!}[\calH^{\ds}])^{r}$ give the same endomorphism of the complex $\bR f_{\calM, !}\Ql|_{\Ah_{d}}$.  This is the following lemma, which is applicable because $d\geq 3g-2$ is implied by $d\geq 2g'-1=4g-3$ (since $g\geq 1$).	\index{$\zeta^{\hs}$}%

\begin{lemma} Suppose $d\geq 3g-2$, and $\zeta^{\hs}\in Z_{2d-g+1}(\Hk^{\mu}_{\CM,d}|_{\Ah_{d}})_{\QQ}$. Suppose the restriction of $\zeta^{\hs}$ to $\Hk^{\mu}_{\CM,d}|_{\Ah_{d}\cap\Ads_{d}}$ is the fundamental cycle, then the endomorphism $f_{\calM,!}\cl(\zeta^{\hs})$ of $\bR f_{\calM, !}\Ql|_{\Ah_{d}}$ is equal to the $r$-th power of the endomorphism $f_{\calM,!}[\calH^{\ds}]$.
\end{lemma}
\begin{proof} Let $[\calH^{\ds}]^{r}$ denotes the $r$-th self-convolution of $[\calH^{\ds}]$, which is a cycle on the $r$-th self composition of $\calH$, hence on $\Hk^{\mu}_{\CM,d}$ by Lemma \eqref{l:Hk r comp}. We have two cycle  $\zeta^{\hs}$ and (the restriction of) $[\calH^{\ds}]^{r}$ in $Z_{2d-g+1}(\Hk^{\mu}_{\CM,d}|_{\Ah_{d}})_{\QQ}$. We temporarily denote $\calM_{d}|_{\Ah_{d}}$ by $\Mh_{d}$ (although the same notation will be defined in an a priorily different way in \S\ref{ss:aux}). We need to show that they are in the same cycle class when projected to $\Mh_{d}\times\Mh_{d}$ under $(\gamma_{0},\gamma_{r}):\Hk^{\mu}_{\CM,d}|_{\Ah_{d}}\to\Mh_{d}\times\Mh_{d}$.

By assumption, when restricted to $\Hk^{\mu}_{\CM,d}|_{\Ah_{d}\cap\Ads_{d}}$, both $\zeta^{\hs}$ and $[\calH^{\ds}]^{r}$ are the fundamental cycle. Therefore the difference $(\gamma_{0},\gamma_{r})_{*}(\zeta^{\hs}-[\calH^{\ds}]^{r})\in Z_{2d-g+1}(\Mh_{d}\times\Mh_{d})_{\QQ}$ is supported on the image of $\Hk^{\mu}_{\CM,d}|_{\Ah_{d}-\Ads_{d}}$ in $\Mh_{d}\times\Mh_{d}$, which is contained in the image of $\Hk^{\mu}_{\CM,d}-\Hk^{\mu}_{\Mds,d}$ in $\CM_{d}\times\CM_{d}$. By Lemma \ref{l:dim HkM}(2), the latter has dimension $\leq d+2g-2$. Since $d>3g-3$, we have $d+2g-2<2d-g+1$, therefore $(\gamma_{0},\gamma_{r})_{*}(\zeta^{\hs}-[\calH^{\ds}]^{r})=0\in Z_{2d-g+1}(\Mh_{d}\times\Mh_{d})_{\QQ}$, and the lemma follows.
\end{proof}

\subsection{Auxiliary moduli stacks}\label{ss:aux} The goal of this subsection is to prove Theorem \ref{th:alt I}. Below we fix an integer $d\geq \max\{2g'-1,2g\}$. In this subsection, we will introduce moduli stacks $\Hk'^{r}_{G,d}$ and $H_{d}$ that will fit into the following commutative diagram
\begin{equation}\label{Tian}
\xymatrix{\Hk^{\mu}_{T}\times\Hk^{\mu}_{T}\ar@<-3ex>[d]_{(\gamma_{0},\gamma_{r})}\ar@<3ex>[d]^{(\gamma_{0},\gamma_{r})}\ar[rr]^{\Pi^{\mu}\times\Pi^{\mu}} && \Hk'^{r}_{G}\times\Hk'^{r}_{G} \ar@<-3ex>[d]_{(\gamma'_{0},\gamma'_{r})}\ar@<3ex>[d]^{(\gamma'_{0},\gamma'_{r})} && \Hk'^{r}_{G,d}\ar[d]^{(\gamma'_{0},\gamma'_{r})}\ar[ll]_{(\oll{\rho}',\orr{\rho}')}\\
(\Bun_{T})^{2}\times (\Bun_{T})^{2}\ar[rr]^{\Pi\times\Pi\times\Pi\times\Pi} && (\Bun_{G})^{2}\times (\Bun_{G})^{2} && H_{d}\times H_{d}\ar[ll]_(.4){\olr{p}_{13}\times \olr{p}_{24}}\\
\Bun_{T}\times \Bun_{T}\ar@<-3ex>[u]_{(\id,\Fr)}\ar@<3ex>[u]^{(\id,\Fr)}\ar[rr]^{\Pi\times\Pi} && \Bun_{G}\times\Bun_{G}\ar@<-3ex>[u]_{(\id,\Fr)}\ar@<3ex>[u]^{(\id,\Fr)} && H_{d}\ar[ll]_{\olr{p}=(\oll{p},\orr{p})}\ar[u]_{(\id,\Fr)}}
\end{equation}
The maps in this diagram will be introduced later. The fiber products of the three columns are
\begin{equation}\label{3 columns}
\Sht^{\mu}_{T}\times\Sht^{\mu}_{T}\xrightarrow{\theta^{\mu}\times\theta^{\mu}}\Sht'^{r}_{G}\times\Sht'^{r}_{G}\xleftarrow{(\oll{p}',\orr{p}')}\Sht'^{r}_{G,d}
\end{equation}
where $\Sht'^{r}_{G,d}$ is defined as the fiber product of the third column.		\index{$\Sht'^{r}_{G,d}$}%

The fiber products of the three rows will be denoted
\begin{equation}\label{3 rows}
\xymatrix{\Hk^{\mu}_{\Mh,d}\ar[d]^{(\gamma_{0},\gamma_{r})}\\
\Mh_{d}\times\Mh_{d}\\
\Mh_{d}\ar[u]_{(\id,\Fr)}}		 \index{$\calM_{d}^\heart$} \index{$\Hk^{\mu}_{\Mh, d}$}%
\end{equation}
These stacks will turn out to be the restrictions of $\CM_{d}$ and $\Hk^{\mu}_{\CM,d}$ to $\Ah_{d}$, as we will see in Lemma \ref{l:MCart ok}\eqref{Mh M} and Lemma \ref{l:Hk Mh}.

\subsubsection{}\label{two ShtM same} In \S\ref{ss:oct} we discuss an abstract situation as in the above diagrams, which can be pictured using a subdivided octahedron. By Lemma \ref{l:same north pole}, the fiber products of the two diagrams \eqref{3 columns} and \eqref{3 rows} are canonically isomorphic. We denote this stack by
\begin{equation*}
\Sht^{\mu}_{\Mh,d}. \index{$\Sht^{\mu}_{\Mh,d}$}%
\end{equation*}

Below we will introduce $H_{d}$ and $\Hk'^{r}_{G,d}$.

\subsubsection{} We define $\wt{H}_{d}$ to be the moduli stack whose $S$-points is the groupoid of maps
\begin{equation*}
\phi: \CE\incl\CE'
\end{equation*}
where $\CE,\CE'$ are vector bundles over $X\times S$ of rank two, $\phi$ is an injective map of $\calO_{X\times S}$-modules (so its cokernel has support finite over $S$) and $\pr_{S*}\coker(\phi)$ is a locally free $\calO_{S}$-module of rank $d$ (where $\pr_{S}:X\times S\to S$ is the projection). We have an action of $\Pic_{X}$ on $\wt{H}_{d}$ by tensoring, and we form the quotient
\begin{equation*}
H_{d}:=\wt{H}_{d}/\Pic_{X}			\index{$H_{d}, \wt{H}_{d}$}%
\end{equation*}
Taking the map $\phi$ to its source and target gives two maps $\oll{p},\orr{p}: H_{d}\to \Bun_{G}$. The map $\olr{p}_{13}\times\olr{p}_{24}$ that appears in \eqref{Tian} is the map
\begin{eqnarray*}
\olr{p}_{13}\times\olr{p}_{24}: H_{d}\times H_{d}&\to& \Bun_{G}\times\Bun_{G}\times\Bun_{G}\times\Bun_{G}\\
(h,h')&\mapsto&(\oll{p}(h), \oll{p}(h'), \orr{p}(h), \orr{p}(h')).
\end{eqnarray*}

On the other hand we have the morphism $\Pi: \Bun_{T}\to \Bun_{G}$ sending $\calL$ to $\nu_{*}\calL$, see \S\ref{sss:Pi}. We form the following Cartesian diagram, and take it as the definition of $\Mh_{d}$
\begin{equation}\label{MCart}
\xymatrix{\Mh_{d}\ar[r]\ar[d] & H_{d}\ar[d]^{(\oll{p},\orr{p})}\\
\Bun_{T}\times\Bun_{T}\ar[r]^{\Pi\times\Pi} & \Bun_{G}\times \Bun_{G}}		\index{$\calM_{d}^\heart$}%
\end{equation}

\begin{lemma}\label{l:MCart ok}
\begin{enumerate}
\item The morphisms $\oll{p},\orr{p}: H_{d}\to\Bun_{G}$ are representable and smooth of pure relative dimension $2d$. In particular, $H_{d}$ is a smooth algebraic stack over $k$ of pure dimension $2d+3g-3$.
\item\label{Mh M} There is a canonical open embedding $\Mh_{d}\incl \calM_{d}$ whose image is $f_{\calM}^{-1}(\Ah_{d})$ (for the definition of $\Ah_{d}$, see \S\ref{sss:Ah}). In particular, $\Mh_{d}$ is a smooth Deligne--Mumford stack over $k$ of pure dimension $2d-g+1$.
\end{enumerate}
\end{lemma}
\begin{proof}
(1) Let $\Coh^{d}_{0}$ be the stack classifying torsion coherent sheaves on $X$ of length $d$. By \cite[(3.1)]{Laumon}, $\Coh^{d}_{0}$ is smooth of dimension $0$. Consider the map $q: H_{d}\to\Coh^{d}_{0}$ sending $\phi: \CE\incl\CE'$ to $\coker(\phi)$. Then the map $(\oll{p}, q): H_{d}\to \Bun_{G}\times \Coh^{d}_{0}$ is a vector bundle of rank $2d$ whose fiber over $(\CE, \CQ)$ is $\Ext^{1}(\CQ, \CE)$. Therefore $\oll{p}$ is smooth of relative dimension $2d$.

There is an involution $\delta_{H}$ on $H_{d}$ sending $\phi:\CE\to \CE'$ to $\phi^{\vee}:\CE'^{\vee}\to \CE^{\vee}$. We have $\oll{p}\circ\delta_{H}=\orr{p}$ because $\CE'^{\vee}=\CE'$ canonically as $G$-bundles. Therefore $\orr{p}$ is also smooth of relative dimension $2d$.

(2) By the diagram \eqref{MCart}, $\Mh_{d}$ classifies $(\calL,\calL',\psi)$ up to the action of $\Pic_{X}$, where $\calL$ and $\calL'$ are as in the definition of $\tcM_{d}$, and $\psi$ is an injective $\calO_{X}$-linear map $\nu_{*}\calL\to\nu_{*}\calL'$.

The discussion in \S\ref{push L} turns a point $(\calL,\calL',\psi: \nu_{*}\calL\to\nu_{*}\calL')\in\Mh_{d}$ into a point $(\calL,\calL',\alpha:\calL\to\calL',\beta:\calL\to\sigma^{*}\calL')\in\calM_{d}$. The condition that $\psi$ be injective is precisely the condition that $\det(\psi)\neq0$, which is equivalent to saying that $f_{\calM}(\calL,\calL',\psi)\in\Ah_{d}$, according to \eqref{det psi}.

Proposition \ref{p:M}\eqref{M smooth} shows that $\CM_{d}$ is a smooth Deligne--Mumford stack over $k$ of pure dimension $2d-g+1$, hence the same is true for its open substack $\Mh_{d}$. 
\end{proof}

\subsubsection{} Recall the Hecke stacks $\Hk^{r}_{G}$ and $\Hk^{\mu}_{T}$ defined in \eqref{HkrG} and \eqref{HkrT}.  Let $\Hk^{\mu}_{2,d}$ be the moduli stack of commutative diagrams
\begin{equation}\label{HmuGd}
\xymatrix{\calE_{0}\ar@{-->}[r]\ar[d]^{\phi_{0}} & \calE_{1}\ar@{-->}[r]\ar[d]^{\phi_{1}} & \cdots\ar@{-->}[r] &\calE_{r}\ar[d]^{\phi_{r}}\\
\calE'_{0}\ar@{-->}[r] & \calE'_{1}\ar@{-->}[r] & \cdots\ar@{-->}[r] &\calE'_{r}}
\end{equation}
where both rows are points in $\Hk^{\mu}_{2}$ with the same image in $X^{r}$, and the vertical maps $\phi_{j}$ are points in $H_{d}$ (i.e., injective maps with colength $d$). Let
\begin{equation*}
\Hk^{r}_{G,d}=\Hk^{\mu}_{2,d}/\Pic_{X}		\index{$\Hk^{r}_{G,d}, \Hk^{\mu}_{2,d}$}%
\end{equation*}
where $\Pic_{X}$ simultaneously acts on all $\calE_{i}$ and $\calE'_{i}$ by tensor product. The same argument of Lemma \ref{l:Hk indep mu} shows that $\Hk^{r}_{G,d}$ is independent of $\mu$. 

There are  natural maps $\Hk^{r}_{G}\to X^{r}$ and $\Hk^{r}_{G,d}\to X^{r}$. We define
\begin{equation*}
\Hk'^{r}_{G}=\Hk^{r}_{G}\times_{X^{r}}X'^{r}; \quad \Hk'^{r}_{G,d}:=\Hk^{r}_{G,d}\times_{X^{r}}X'^{r}.	\index{$\Hk'^{r}_{G},\Hk'^{r}_{G,d}$}%
\end{equation*}
The map $\Hk^{\mu}_{T}\to \Hk^{r}_{G}$ given by $\calE_{i}=\nu_{*}\calL_{i}$ induces a map
\begin{equation*}
\Pi^{\mu}:\Hk^{\mu}_{T}\to \Hk'^{r}_{G}.
\end{equation*}

We have two maps
\begin{equation*}
\oll{\rho},\orr{\rho}:\Hk^{r}_{G,d}\to\Hk^{r}_{G}
\end{equation*}
sending the diagram \eqref{HmuGd} to its top and bottom row. We denote their base change to $X'^{r}$ by
\begin{equation*}
\oll{\rho}',\orr{\rho}':\Hk'^{r}_{G,d}\to\Hk'^{r}_{G}
\end{equation*}

We define $\Hk^{\mu}_{\Mh, d}$ by the following Cartesian diagram
\begin{equation}\label{HCart}
\xymatrix{\Hk^{\mu}_{\Mh, d}\ar[rr]\ar[d] && \Hk'^{r}_{G,d}\ar[d]^{(\oll{\rho}',\orr{\rho}')}\\
\Hk^{\mu}_{T}\times\Hk^{\mu}_{T}\ar[rr]^{\Pi^{\mu}\times\Pi^{\mu}} & &\Hk'^{r}_{G}\times\Hk'^{r}_{G}}		\index{$\Hk^{\mu}_{\Mh, d}$}%
\end{equation}

The same argument of Lemma \ref{l:MCart ok}\eqref{Mh M} shows the following result. Recall that the stack $\Hk^{\mu}_{\CM,d}$ is defined in \S\ref{sss:alt HkM}. 
\begin{lemma}\label{l:Hk Mh}
There is a canonical isomorphism between $\Hk^{\mu}_{\Mh,d}$ and the preimage of $\Ah_{d}$ under the natural map $f_{\calM}\circ\gamma_{0}: \Hk^{\mu}_{\CM,d}\to\calA_{d}$. 
\end{lemma}

\subsubsection{} We have a map
\begin{equation*}
s:\Hk^{r}_{G,d}\to X_{d}\times X^{r}
\end{equation*}
which sends a diagram \eqref{HmuGd} to $(D;x_{1},\cdots, x_{r})$ where $D$ is the divisor of $\det(\phi_{i})$ for all $i$. Let $(X_{d}\times X^{r})^{\circ}\subset X_{d}\times X^{r}$ be the open subscheme consisting of those $(D;x_{1},\cdots, x_{r})$ where $x_{i}$ is disjoint from the support of $D$ for all $i$. Let
\begin{equation*}
\Hk^{r,\circ}_{G,d}=s^{-1}((X_{d}\times X^{r})^{\circ}).		\index{$\Hk^{r,\circ}_{G,d}$}%
\end{equation*}
be an open substack of $\Hk^{r}_{G,d}$. Let $\Hk'^{r,\circ}_{G,d}\subset \Hk'^{r}_{G,d}$ and  $\Hk^{\mu,\circ}_{\Mh,d}\subset \Hk^{\mu}_{\Mh,d}$ be the preimages of $\Hk^{r,\circ}_{G,d}$. \index{$\Hk'^{r,\circ}_{G,d}, \Hk^{\mu,\circ}_{\Mh,d}$}%

\begin{lemma}\label{l:HCart dim}
\begin{enumerate}
\item\label{dim HkcircGd} The stacks $\Hk^{r,\circ}_{G,d}$ and $\Hk'^{r,\circ}_{G,d}$ are smooth of pure dimension $2d+2r+3g-3$.
\item\label{dim HkrGd} The dimensions of all geometric fibers of $s$ are $d+r+3g-3$. In particular, $\dim\Hk^{r}_{G,d}=\dim \Hk'^{r}_{G,d}=2d+2r+3g-3$.
\item\label{dim HkM} Recall that $\Hk^{\mu}_{\Mds,d}$ is the restriction $\Hk^{\mu}_{\CM,d}|_{\Ads_{d}}$, where $\Ads_{d}\subset\calA_{d}$ is defined in \S\ref{sss:Ads}. Suppose $d\geq\max\{2g'-1,2g\}$. Let $\Hk^{\mu,\circ}_{\Mds,d}$ be the intersection of $\Hk^{\mu}_{\Mds,d}$ with $\Hk^{\mu,\circ}_{\Mh,d}$ inside $\Hk^{\mu}_{\CM,d}$. Then $\dim(\Hk^{\mu}_{\Mds,d}-\Hk^{\mu,\circ}_{\Mds,d})<2d-g+1=\dim \Hk^{\mu}_{\Mds,d}$. 
\end{enumerate}
\end{lemma}
The proof of this lemma will be postponed to \S\ref{sss:proof HCart dim circ}-\S\ref{sss:proof HCart dim HkM}.

\begin{lemma}\label{l:HCart ok} Suppose $d\geq\max\{2g'-1,2g\}$.
\begin{enumerate}
\item The diagram \eqref{HCart} satisfies the conditions in \S\ref{sss:Gysin}. In particular, the refined Gysin map 
\begin{equation*}
(\Pi^{\mu}\times\Pi^{\mu})^{!}:\Ch_{*}(\Hk'^{r}_{G,d})_{\QQ}\to \Ch_{*-2(2g-2+r)}(\Hk^{\mu}_{\Mh,d})_{\QQ}
\end{equation*}
is defined.
\item Let
\begin{equation}\label{define zeta}
\zeta^{\hs}=(\Pi^{\mu}\times\Pi^{\mu})^{!}[\Hk'^{r}_{G,d}]\in\Ch_{2d-g+1}(\Hk^{\mu}_{\Mh,d})_{\QQ}.
\end{equation}
Then the restriction of $\zeta^{\hs}$ to $\Hk^{\mu}_{\Mh,d}|_{\Ads_{d}\cap \Ah_{d}}$ is the fundamental cycle.
\end{enumerate}
\end{lemma}
\begin{proof}
(1) We first check that $\Hk^{\mu}_{\Mh,d}$ admits a finite flat presentation. The map $\gamma_{0}:\Hk^{\mu}_{\Mh,d}\to \Mh_{d}$ is schematic, so it suffices to check that $\Mh_{d}$ or $\CM_{d}$ admits a finite flat presentation. In the proof of Proposition \ref{p:M}\eqref{M DM stack} we constructed a proper and schematic map $\ov{h}: \calM_{d}\to J^{d}_{X'}\times\Prym_{X'/X}$, see \eqref{Md to Prym}.  Since $J^{d}_{X'}$ is a scheme and $\Prym_{X'/X}$ is the quotient of the usual Prym variety by the trivial action of $\mu_{2}$, $J^{d}_{X'}\times\Prym_{X'/X}$ admits a finite flat presentation, hence so do $\CM_{d}$ and $\Hk^{\mu}_{\Mh,d}$.

Next we verify the condition \eqref{global lci} of \S\ref{sss:Gysin}. Extending $k$ if necessary, we may choose a point $y\in X(k)$ that is split into $y',y''\in X'(k)$. Let $\Bun_{G}(y)$ be the moduli stack of $G$-torsors over $X$ with a Borel reduction at $y$. Let $\Hk'^{r}_{G}(y)=\Hk'^{r}_{G}\times_{\Bun_{G}}\Bun_{G}(y)$ where the map $\Hk'^{r}_{G}\to\Bun_{G}$ sends $(\CE_{i};x_{i};f_{i})$ to $\CE_{0}$. We may lift the morphism $\Pi^{\mu}$ to a morphism
\begin{equation*}
\Pi^{\mu}(y): \Hk^{\mu}_{T}\to \Hk'^{r}_{G}(y)
\end{equation*}
where the Borel reduction of $\CE_{0}=\nu_{*}\CL_{0}$ at $y$ (i.e., a line in the stalk $\CE_{0,y}$) is given by the stalk of $\CL_{0}$ at $y'$. The projection $p:\Hk'^{r}_{G}(y)\to \Hk'^{r}_{G}$ is smooth, and $\Pi^{\mu}=p\circ\Pi^{\mu}(y)$. So to check the condition \eqref{global lci} of \S\ref{sss:Gysin}, it suffices to show that $\Pi^{\mu}(y)$ is a regular local immersion. 

We will show by tangential calculations that $\Hk'^{r}_{G}(y)$ is a Deligne--Mumford stack in a neighborhood of the image of $\Pi^{\mu}(y)$, and the tangent map of $\Pi^{\mu}(y)$ is injective. For this it suffices to make tangential calculations at geometric points of $\Hk^{\mu}_{T}$ and its image in $\Hk'^{r}_{G}(y)$. We identify $\Hk^{\mu}_{T}$ with $\Bun_{T}\times X'^{r}$ as in \S\ref{sss:HkT smooth}. Fix a geometric point $(\CL;\ul{x'})\in \Pic_{X'}(K)\times X'(K)^{r}$. For notational simplicity, we base change the situation from $k$ to $K$ without changing notation. So $X$ means $X\otimes_{k}K$, etc. 

The relative tangent space of $\Hk^{\mu}_{T}\to X'^{r}$ at $(\CL;\ul{x'})$ is  $\cohog{1}{X, \calO_{X'}/\calO_{X}}$. The relative tangent complex of $\Hk'^{r}_{G}(y)\to X'^{r}$ at $\Pi^{\mu}(y)(\CL;\ul{x'})=(\nu_{*}\CL\to\nu_{*}\CL(x'_{1})\to\cdots; \CL_{y'})$ is $\cohog{*}{X,\Ad^{\ul{x'},y}(\nu_{*}\CL)}[1]$, where $\Ad^{\ul{x'},y}(\nu_{*}\CL)=\un\End^{\ul{x'},y}(\nu_{*}\CL)/\calO_{X}\cdot\id$, and $\un\End^{\ul{x'},y}(\nu_{*}\CL)$ is the endomorphism sheaf of the chain of vector bundles $\nu_{*}\CL\to\nu_{*}\CL(x'_{1})\to\cdots$ preserving the line $\CL_{y'}$ of the stalk $(\nu_{*}\CL)_{y}$.  Note that 
\begin{eqnarray}\notag
\un\End^{\ul{x'},y}(\nu_{*}\CL)&\subset&\un\End^{y}(\nu_{*}\CL)=\nu_{*}\un\Hom(\CL\oplus(\sigma^{*}\CL)(y''), \CL)\\
\label{Endy}&=&\nu_{*}\calO_{X'}\oplus\nu_{*}(\CL\otimes\sigma^{*}\CL^{-1}(-y''))
\end{eqnarray}
We also have a natural inclusion
\begin{equation*}
\gamma:\nu_{*} \calO_{X'}\incl  \un\End^{\ul{x'},y}(\nu_{*}\CL)
\end{equation*}
identifying the LHS as those endomorphisms of $\nu_{*}\CL$ that are $\calO_{X'}$-linear. Now $\gamma(\nu_{*}\calO_{X'})$ maps isomorphically to $\nu_{*}\calO_{X'}$ on the RHS of \eqref{Endy}. Combining these we get a canonical decomposition $\un\End^{\ul{x'},y}(\nu_{*}\CL)=\nu_{*}\calO_{X'}\oplus\calK$ for some line bundle $\calK$ on $X$ with $\deg(\calK)<0$. Consequently, we have a canonical decomposition 
\begin{equation}\label{Ady}
\Ad^{\ul{x'},y}(\nu_{*}\CL)=\calO_{X'}/\calO_{X}\oplus \calK.
\end{equation}
In particular $\cohog{0}{X,\Ad^{\ul{x'},y}(\nu_{*}\CL)}=\cohog{0}{X,\calO_{X'}/\calO_{X}}=0$. This shows that $\Hk'^{r}_{G}(y)$ is a Deligne--Mumford stack in a neighborhood of $\Pi^{\mu}(y)(\CL;\ul{x'})$.

The tangent map of $\Pi^{\mu}(y)$ is the map  $\cohog{1}{X, \calO_{X'}/\calO_{X}}\to\cohog{1}{X,\Ad^{\ul{x'}}(\nu_{*}\CL)}$ induced by $\gamma$, hence it corresponds to the inclusion of the first factor in the decomposition \eqref{Ady}. In particular, the tangent map of $\Pi^{\mu}(y)$ is injective. This finishes the verification of all conditions in \S\ref{sss:Gysin} for the diagram \eqref{HCart}.

(2) Let $\Hk^{\mu,\circ}_{\Mh,d}$ be the preimage of $\Hk'^{r,\circ}_{G,d}$. By Lemma \ref{l:HCart dim}\eqref{dim HkcircGd}, $\Hk'^{r,\circ}_{G,d}$ is smooth of dimension $2d+2r+3g-3$. On the other hand, by Lemma \ref{l:dim HkM}, $\Hk^{\mu}_{\Mds,d}$ has dimension $2d-g+1$. Combining these facts, we see that $\Hk^{\mu,\circ}_{\Mh,d}\cap \Hk^{\mu}_{\Mds,d}$ has the expected dimension in the Cartesian diagram \eqref{HCart}. This implies that $\zeta^{\hs}|_{\Hk^{\mu,\circ}_{\Mh,d}\cap \Hk^{\mu}_{\Mds,d}}$ is the fundamental cycle. By Lemma \ref{l:HCart dim}\eqref{dim HkM}, $\Hk^{\mu}_{\Mds,d}-\Hk^{\mu,\circ}_{\Mh,d}$ has lower dimension than $\Hk^{\mu}_{\Mds,d}$, therefore $\zeta^{\hs}|_{\Hk^{\mu}_{\Mds,d}}$ must be the fundamental cycle.
\end{proof}

\subsubsection{} There are $r+1$ maps $\gamma_{i}$ ($0\leq i\leq r$) from the diagram \eqref{HCart} to \eqref{MCart}: it sends the diagram \eqref{HmuGd} to its $i$-th column, etc. In particular, we have maps $\gamma_{i}:\Hk^{r}_{G,d}\to H_{d}$ and $\gamma'_{i}: \Hk'^{r}_{G,d}\to H_{d}$. The maps $\gamma'_{0}$ and $\gamma'_{r}$ appear in the diagram \eqref{Tian}.

We define the stack $\Sht^{r}_{G,d}$ by the following  Cartesian diagram
\begin{equation}\label{ShtrGd}
\xymatrix{\Sht^{r}_{G,d}\ar[d]\ar[r] & \Hk^{r}_{G,d}\ar[d]^{(\gamma_{0},\gamma_{r})}\\
H_{d}\ar[r]^{(\id,\Fr)} & H_{d}\times H_{d}}			\index{$\Sht^{r}_{G,d}$}%
\end{equation}
Similarly we define $\Sht'^{r}_{G,d}$ as the fiber product of the third column of \eqref{Tian}:
\begin{equation}\label{Sht'rGd}
\xymatrix{\Sht'^{r}_{G,d}\ar[d]\ar[r] & \Hk'^{r}_{G,d}\ar[d]^{(\gamma'_{0},\gamma'_{r})}\\
H_{d}\ar[r]^{(\id,\Fr)} & H_{d}\times H_{d}}					\index{$\Sht'^{r}_{G,d}$}%
\end{equation} 
We have $\Sht'^{r}_{G,d}\cong\Sht^{r}_{G,d}\times_{X^{r}}X'^{r}$.

\begin{lemma}\label{l:same hD} There are canonical isomorphisms of stacks
\begin{eqnarray*}
\Sht^{r}_{G,d}\cong \coprod_{D\in X_{d}(k)}\Sht^{r}_{G}(h_{D});\\
\Sht'^{r}_{G,d}\cong \coprod_{D\in X_{d}(k)}\Sht'^{r}_{G}(h_{D}).
\end{eqnarray*}
For the definitions of $\Sht^{r}_{G}(h_{D})$ and $\Sht'^{r}_{G}(h_{D})$, see \S\ref{sss:Sht hD} and \S\ref{sss:Hk Sht'}.
\end{lemma}
\begin{proof}
From the definitions, $(\gamma_{0},\gamma_{r})$ factors through the map $\Hk^{r}_{G,d}\to H_{d}\times_{X_{d}}H_{d}$. On the other hand, $(\id,\Fr):H_{d}\to H_{d}\times H_{d}$ covers the similar map $(\id,\Fr):X_{d}\to X_{d}\times X_{d}$. By the discussion in \S\ref{sss:decomp ShtC},  we have a  decomposition
\begin{equation*}
\Sht^{r}_{G,d}=\coprod_{D\in X_{d}(k)}\Sht^{r}_{G,D}
\end{equation*}
Let $H_{D}$ and $\Hk^{r}_{G,D}$ be the fibers of $H_{d}$ and $\Hk^{r}_{G,d}$ over $D$. Then the $D$-component $\Sht^{r}_{G,D}$ of $\Sht^{r}_{G,d}$ fits into a Cartesian diagram
\begin{equation}\label{Sht G D}
\xymatrix{\Sht^{r}_{G,D}\ar[r]\ar[d] & \Hk^{r}_{G,D}\ar[d]^{(\gamma_{0},\gamma_{r})}\\
H_{D}\ar[r]^{(\id,\Fr)} & H_{D}\times H_{D}}
\end{equation}
Comparing this with the definition in \S\ref{sss:Sht hD}, we see that $\Sht^{r}_{G,D}\cong \Sht^{r}_{G}(h_{D})$. The statement for $\Sht'^{r}_{G,d}$ follows from the statement for $\Sht^{r}_{G,d}$ by base change to $X'^{r}$.
\end{proof}

\begin{cor}\label{c:ShtD} 
Let $D\in X_{d}(k)$ (i.e., an effective divisor on $X$ of  degree $d$). Recall the stack $\Sht^{\mu}_{\CM,d}$ defined in \eqref{NW} and $\Sht^{\mu}_{\Mh,d}$ defined in \S\ref{two ShtM same}. Then $\Sht^{\mu}_{\Mh,d}$ is canonically isomorphic to the restriction of $\Sht^{\mu}_{\CM,d}$ to $\Ah_{d}(k)\subset \calA_{d}(k)$. 

Moreover, there is a canonical decomposition
\begin{equation*}
\Sht^{\mu}_{\Mh,d}=\coprod_{D\in X_{d}(k)}\Sht^{\mu}_{\CM,D},
\end{equation*}
where $\Sht^{\mu}_{\CM,D}$ is defined in \eqref{decomp D into a}. In particular, we have a Cartesian diagram
\begin{equation}\label{SCart D}
\xymatrix{\Sht^{\mu}_{\CM,D}\ar[d]\ar[rr] && \Sht'^{r}_{G}(h_{D})\ar[d]^{(\oll{p}',\orr{p}')}\\
\Sht^{\mu}_{T}\times\Sht^{\mu}_{T}\ar[rr]^{\theta^{\mu}\times\theta^{\mu}} && \Sht'^{r}_{G}\times\Sht'^{r}_{G}}
\end{equation}
\end{cor}
\begin{proof} Note that $\Sht^{\mu}_{\Mh,d}$ is defined as a fiber product in two ways: one as the fiber product of \eqref{3 columns} and the other as the fiber product of \eqref{3 rows}. Using the first point of view and the decomposition of $\Sht'^{r}_{G,d}$ given by Lemma \ref{l:same hD}, we get a decomposition of $\Sht^{\mu}_{\Mh,d}=\coprod_{D\in X_{d}(k)}\Sht^{\mu}_{\Mh,D}$, where $\Sht^{\mu}_{\Mh,D}$ is by definition the stack to put in the upper left corner of \eqref{SCart D} to make the diagram Cartesian. 

On the other hand, using the second point of view of $\Sht^{\mu}_{\Mh,d}$ as the fiber product of \eqref{3 rows}, and using the fact that $\Hk^{\mu}_{\Mh,d}$ is the restriction of $\Hk^{\mu}_{\CM,d}$ over $\Ah_{d}$ by Lemma \ref{l:Hk Mh}, we see that $\Sht^{\mu}_{\Mh,d}$ is the restriction of $\Sht^{\mu}_{\CM,d}$ over $\Ah_{d}$ by comparing \eqref{3 rows} and \eqref{NW}. By \eqref{ShtM into a} and \eqref{decomp D into a}, and the fact that $\Ah_{d}(k)=\coprod_{D\in X_{d}(k)}\calA_{D}(k)$,  we get a decomposition $\Sht^{\mu}_{\Mh,d}=\coprod_{D\in X_{d}(k)}\Sht^{\mu}_{\CM,D}$. Therefore, both  $\Sht^{\mu}_{\Mh,D}$ and $\Sht^{\mu}_{\CM,D}$ are the fiber of the map $\Sht^{\mu}_{\CM,d}\to \calA_{d}\to \hX_{d}$ over $D$, and they are canonically isomorphic. Hence we may replace the upper left corner of \eqref{SCart D} by $\Sht^{\mu}_{\Mh,D}$, and the new diagram is Cartesian by definition.\end{proof}

\begin{lemma}\label{l:Sht'rGd ok} 
\begin{enumerate}
\item The diagram \eqref{Sht'rGd} satisfies the conditions in \S\ref{sss:proper int}. In particular, the refined Gysin map
\begin{equation*}
(\id,\Fr_{H_{d}})^{!}: \Ch_{*}(\Hk'^{r}_{G,d})_{\QQ}\to \Ch_{*-\dim H_{d}}(\Sht'^{r}_{G,d})_{\QQ}
\end{equation*}
is defined.
\item We have
\begin{equation*}
[\Sht'^{r}_{G,d}]=(\id,\Fr_{H_{d}})^{!}[\Hk'^{r}_{G,d}]\in\Ch_{2r}(\Sht'^{r}_{G,d}).
\end{equation*}
\end{enumerate}
\end{lemma}
\begin{proof}
(1) Since  $\oll{p}:\Sht^{r}_{G}(h_{D})\to \Sht^{r}_{G}$ is representable by Lemma \ref{l:Sht hD proper}, $\Sht^{r}_{G}(h_{D})$ is also a Deligne--Mumford stack. Since $\Sht^{r}_{G,d}$ is the disjoint union of $\Sht^{r}_{G}(h_{D})$ by Lemma \ref{l:same hD},  $\Sht^{r}_{G,d}$ is Deligne--Mumford, hence so is $\Sht'^{r}_{G,d}$. The map $\gamma'_{0}:\Hk'^{r}_{G,d}\to H_{d}$ is representable because its fibers are closed subschemes of iterated Quot schemes (fixing $\CE_{0}\incl\CE'_{0}$, building $\CE_{i}$ and $\CE'_{i}$ step by step and imposing commutativity of the maps).  Therefore $(\gamma'_{0},\gamma'_{r})$ is also representable. This verifies the condition (1) in \S\ref{sss:proper int}.

Since $H_{d}$ is smooth by Lemma \ref{l:MCart ok}, the normal cone stack of the map $(\id, \Fr_{H_{d}}):H_{d}\to H_{d}\times H_{d}$ is the vector bundle stack $\Fr^{*}TH_{d}$, the Frobenius pullback of the tangent bundle stack of $H_{d}$. Therefore $(\id,\Fr_{H_{d}})$ satisfies condition \eqref{normal f} in \S\ref{sss:proper int}. It also satisfies condition \eqref{lci} of \S\ref{sss:proper int} by the discussion in Remark \ref{r:smooth lci}. 

Finally the dimension condition \eqref{n n-d} in \S\ref{sss:proper int} for $\Hk'^{r}_{G,d}$ and $\Sht'^{r}_{G,d}=\coprod_{D}\Sht'^{r}_{G}(h_{D})$ follow from Lemma \ref{l:HCart dim}\eqref{dim HkrGd} and Lemma \ref{l:dim hD}. We have verified all conditions in \S\ref{sss:proper int}.

(2) Take the open substack $\Hk'^{r,\circ}_{G,d}\subset\Hk'^{r}_{G,d}$ as in Lemma \ref{l:HCart dim}. Then $\Hk'^{r,\circ}_{G,d}$ is smooth of pure dimension $2d+2r+3g-3$. According to Lemma \ref{l:same hD}, the corresponding open part $\Sht'^{r,\circ}_{G,d}$ is the disjoint union of $\Sht'^{r,\circ}_{G}(h_{D})$, where
\begin{equation*}
\Sht'^{r,\circ}_{G}(h_{D})=\Sht'^{r}_{G}(h_{D})|_{(X'-\nu^{-1}(D))^{r}}.
\end{equation*}
It is easy to see that both projections $\Sht'^{r,\circ}_{G}(h_{D})\to \Sht'^{r}_{G}$ are \'etale, hence $\Sht'^{r,\circ}_{G}(h_{D})$ is smooth of dimension $2r=\dim \Hk'^{r,\circ}_{G,d}-\codim(\id,\Fr_{H_{d}})$, the expected dimension. This implies that if we replace $\Hk'^{r}_{G,d}$ with $\Hk'^{r,\circ}_{G,d}$, and replace $\Sht'^{r}_{G,d}$ with $\Sht'^{r,\circ}_{G,d}$ in the diagram \eqref{Sht'rGd}, it becomes a complete intersection diagram. Therefore $(\id,\Fr_{H_{d}})^{!}[\Hk'^{r}_{G,d}]$ is the fundamental cycle when restricted to $\Sht'^{r,\circ}_{G,d}$.  Since $\Sht'^{r}_{G,d}-\Sht'^{r,\circ}_{G,d}$ has lower dimension than $2r$ by Lemma \ref{l:dim hD}, we see that $(\id,\Fr_{H_{d}})^{!}[\Hk'^{r}_{G,d}]$ must be equal to the fundamental cycle over the whole $\Sht'^{r}_{G,d}$.
\end{proof}

\subsubsection{Proof of Theorem \ref{th:alt I}}\label{sss:proof alt I} 
Consider the diagram \eqref{SCart D}. Since $\Sht^{\mu}_{T}$ is a proper Deligne--Mumford stack over $k$ and the map $(\oll{p}',\orr{p}')$ is proper and representable, $\Sht^{\mu}_{\CM,D}$ is also a proper Deligne--Mumford stack over $k$. A simple manipulation using the functoriality of Gysin maps gives
\begin{equation*}
\BI_{r}(h_{D})=\jiao{\theta^{\mu}_{*}[\Sht^{\mu}_{T}],h_{D}*\theta^{\mu}_{*}[\Sht^{\mu}_{T}]}_{\Sht'^{r}_{G}}=\deg\left((\theta^{\mu}\times\theta^{\mu})^{!}[\Sht'^{r}_{G}(h_{D})]\right).
\end{equation*}
Here $(\theta^{\mu}\times\theta^{\mu})^{!}:\Ch_{2r}(\Sht'^{r}_{G}(h_{D}))_{\QQ}\to \Ch_{0}(\Sht^{\mu}_{\CM,D})_{\QQ}$ is the refined Gysin map attached to the map $\theta^{\mu}\times\theta^{\mu}$. By Corollary \ref{c:ShtD}, $(\theta^{\mu}\times\theta^{\mu})^{!}[\Sht'^{r}_{G}(h_{D})]$ is the $D$-component of the $0$-cycle
\begin{equation*}
(\theta^{\mu}\times \theta^{\mu})^{!}[ \Sht'^{r}_{G,d}]\in \Ch_{0}(\Sht^{\mu}_{\Mh,d})_{\QQ}=\bigoplus_{D\in X_{d}(k)}\Ch_{0}(\Sht^{\mu}_{\CM,D})_{\QQ}.
\end{equation*}
Therefore, to prove \eqref{alternative I} simultaneously for all $D$ of degree $d$, it suffices to find a cycle class $\zeta^{\hs}\in\Ch_{2d-g+1}(\Hk^{\mu}_{\Mh,d})_{\QQ}$ whose restriction to $\Hk^{\mu}_{\Mh,d}\cap\Hk^{\mu}_{\Mds,d}=\Hk^{\mu}_{\CM,d}|_{\Ah_{d}\cap\Ads_{d}}$ is the fundamental class, and that
\begin{equation}\label{two int}
(\theta^{\mu}\times \theta^{\mu})^{!}[\Sht'^{r}_{G,d}]=(\id,\Fr_{\Mh_{d}})^{!}\zeta^{\hs}\in\Ch_{0}(\Sht^{\mu}_{\Mh,d})_{\QQ}.
\end{equation}
The statement of Theorem \ref{th:alt I} asks for a cycle $\zeta$ on $\Hk^{\mu}_{\CM,d}$, but we may extend the above $\zeta^{\hs}$ arbitrarily to a $(2d-g+1)$-cycle in $\Hk^{\mu}_{\calM,d}$.

To prove \eqref{two int}, we would like to apply Theorem \ref{th:oct} to the situation of \eqref{Tian}. We check the assumptions:
\begin{enumerate}
\item The smoothness of $\Bun_{T}$ and $\Bun_{G}$ is well-known. The smoothness of $\Hk'^{r}_{G}$ and $\Hk^{\mu}_{T}$ follow from Remark \ref{r:HkG smooth} and \S\ref{sss:HkT smooth}. Finally, by Lemma \ref{l:MCart ok}, $H_{d}$ is smooth of pure dimension $2d+3g-3$. This checks the smoothness of all members in \eqref{Tian} except $B=\Hk'^{r}_{G,d}$. 
\item By Corollary \ref{c:ShtG smooth}, $\Sht^{r}_{G}$ and hence $\Sht'^{r}_{G}$ is smooth of pure dimension $2r$; by Lemma \ref{l:ShtT proper}, $\Sht^{\mu}_{T}$ is smooth of pure dimension $r$. By Lemma \ref{l:MCart ok}, $\Mh_{d}$ is smooth of pure dimension $2d-g+1$. All of them have the dimension expected from the Cartesian diagrams defining them.
\item The diagram \eqref{Sht'rGd} satisfies the conditions in \S\ref{sss:proper int} by Lemma \ref{l:Sht'rGd ok}. The diagram \eqref{HCart} satisfies the conditions in \S\ref{sss:Gysin} by Lemma \ref{l:HCart ok}.

\item We check that the Cartesian diagram formed by \eqref{3 rows}, or rather \eqref{NW}, satisfies the conditions in \S\ref{sss:Gysin}. The map $\Sht^{\mu}_{\CM,d}\to \CM_{d}$ is representable because $\Hk^{\mu}_{\CM,d}\to \CM_{d}\times\CM_{d}$ is. In the proof of Lemma \ref{l:HCart ok}(1) we have proved that $\CM_{d}$ admits a finite flat presentation, hence so does  $\Sht^{\mu}_{\CM,d}$. This verifies the first condition in \S\ref{sss:Gysin}. Since $\calM_{d}$ is a smooth Deligne--Mumford stack by Lemma \ref{p:M}\eqref{M smooth}, $(\id,\Fr_{\calM_{d}}):\CM_{d}\to\CM_{d}\times\CM_{d}$ is  a regular local immersion, which verifies condition \eqref{global lci} of \S\ref{sss:Gysin}. 

Finally we consider the Cartesian diagram formed by \eqref{3 columns} (or equivalently, the disjoint union of the diagrams \eqref{SCart D} for all $D\in X_{d}(k)$).  We have already showed above that $\Sht^{\mu}_{\CM,d}$ admits a finite flat presentation. All members  in these diagrams are Deligne--Mumford stacks, and $\Sht^{\mu}_{T}$ and $\Sht'^{r}_{G}$ are smooth Deligne--Mumford stacks by Lemma \ref{l:ShtT proper} and Corollary \ref{c:ShtG smooth}. Hence the map $\theta^{\mu}\times\theta^{\mu}$ satisfies the conditions \eqref{global lci} of \S\ref{sss:Gysin} by Remark \ref{r:source target smooth ok}. 
\end{enumerate}

Now we can apply Theorem \ref{th:oct} to the situation \eqref{Tian}. Let $\zeta^{\hs}=(\Pi^{\mu}\times\Pi^{\mu})^{!}[\Hk'^{r}_{G,d}]\in\Ch_{2d-g+1}(\Hk^{\mu}_{\Mh,d})_{\QQ}$ as defined in \eqref{define zeta}. Then the restriction of $\zeta^{\hs}$ to $\Hk^{\mu}_{\CM,d}|_{\Ads_{d}\cap\Ah_{d}}$ is the fundamental cycle by Lemma \ref{l:HCart ok}(2). Finally, 
\begin{eqnarray*}
&&(\id,\Fr_{\Mh_{d}})^{!}\zeta^{\hs}=(\id,\Fr_{\Mh_{d}})^{!}(\Pi^{\mu}\times\Pi^{\mu})^{!}[\Hk'^{r}_{G,d}]\\
&=&(\theta^{\mu}\times\theta^{\mu})^{!}(\id,\Fr_{H_{d}})^{!}[\Hk'^{r}_{G,d}]\quad\mbox{(Theorem \ref{th:oct})}\\
&=&(\theta^{\mu}\times\theta^{\mu})^{!}[\Sht'^{r}_{G,d}] \quad\mbox{(Lemma \ref{l:Sht'rGd ok}(2))}
\end{eqnarray*}
which is \eqref{two int}. This finishes the proof of \eqref{alternative I}.

\subsection{Some dimension calculation}
In this subsection, we give the proofs of several lemmas we stated previously concerning the dimensions of certain moduli stacks.

\subsubsection{Proof of Lemma \ref{l:HCart dim}\eqref{dim HkcircGd}}\label{sss:proof HCart dim circ}
In the diagram \eqref{HmuGd}, when the divisors of the $\phi_{i}$ are disjoint from the divisors of the horizontal maps, namely the $x_{i}$'s, the diagram is uniquely determined by its left column $\phi_{0}:\CE_{0}\to\CE'_{0}$ and top row. Therefore we have
\begin{equation*}
\Hk^{r,\circ}_{G,d}=(H_{d}\times_{\Bun_{G}}\Hk^{r}_{G})|_{(X_{d}\times X^{r})^{\circ}}.
\end{equation*}
Since $H_{d}$ is smooth of pure dimension $2d+3g-3$ by Lemma \ref{l:MCart ok}, and the map $p_{0}: \Hk^{r}_{G}\to \Bun_{G}$ is smooth of relative dimension $2r$, we see that $H_{d}\times_{\Bun_{G}}\Hk^{r}_{G}$ is smooth of pure dimension $2d+2r+3g-3$. 

\subsubsection{Proof of Lemma \ref{l:HCart dim}\eqref{dim HkrGd}}\label{sss:proof HCart dim HkrGd} Over $(X_{d}\times X^{r})^{\circ}$, we have $\dim \Hk^{r,\circ}_{G,d}=2d+2r+3g-3$, therefore the generic fiber of $s$ has dimension $d+r+3g-3$. By the semicontinuity of fiber dimensions, it suffices to show that the geometric fibers of $s$ have dimension $\leq d+r+3g-3$. We will actually show that the geometric fibers of the map $(s,p_{0}): \Hk^{r}_{G,d}\to X_{d}\times X^{r}\times \Bun_{G}$ sending the diagram \eqref{HmuGd} to $(D;x_{i};\CE'_{r})$ have dimension $\leq d+r$.

We present $\Hk^{r}_{G,d}$ as the quotient of $\Hk^{\mu}_{2,d}/\Pic_{X}$ with $\mu=\mu_{+}^{r}$. Therefore a point in $\Hk^{r}_{G,d}$ is a diagram of the form \eqref{HmuGd} with all arrows $f_{i},f'_{i}$ pointing to the right.

Let $(D;\ul{x}=(x_{i}))\in X_{d}\times X^{r}$ and $\CE'_{r}\in \Bun_{G}$ be geometric points. For notational simplicity we base change the whole situation to the field of definition of this point without changing notation. Let $H_{D, \ul{x}, \CE'_{r}}$ be the fiber of $(s,p_{0})$ over $(D;x_{i};\CE'_{r})$. We consider the scheme $H'=H'_{D, \ul{x}, \CE'_{r}}$ classifying commutative diagrams
\begin{equation}\label{H'}
\xymatrix{\CE_{0}\ar[r]^{f_{1}}\ar[d]^{\phi_{0}} & \CE_{1}\ar[r]^{f_{2}} & \cdots\ar[r]^{f_{r}} & \CE_{r}\ar[d]^{\phi_{r}}\\
\CE'_{0}\ar[r]^{f'_{1}} & \CE'_{1}\ar[r]^{f'_{2}} & \cdots\ar[r]^{f'_{r}} & \CE'_{r}}
\end{equation}
where $\div(\det\phi_{0})=D=\div(\det\phi_{r})$ and $\div(\det f_{i})=x_{i}=\div(\det f'_{i})$. The only difference between $H'$ and $H_{D,\ul{x},\CE'_{r}}$ is that we do not require the maps $\phi_{i}$ for $1\leq i\leq r-1$ to exist (they are unique if exist). There is a natural embedding $H_{D,\ul{x},\CE'_{r}}\incl H'$, and it suffices to show that $\dim(H')\leq d+r$. We isolate this part of the argument into a separate Lemma below, because it will be used in another proof. This finishes the proof of Lemma \ref{l:HCart dim}\eqref{dim HkrGd}.

\begin{lemma}\label{l:H'} Consider the scheme $H'=H'_{D,\ul{x},\CE'_{r}}$ introduced in the proof of Lemma \ref{l:HCart dim}\eqref{dim HkrGd}. We have $\dim H'=d+r$.
\end{lemma}
\begin{proof}
We only give the argument for the essential case where all $x_{i}$ are equal to the same point $x$ and $D=dx$. The general case can be reduced to this case by factorizing $H'$ into a product indexed by points that appear in $|D|\cup\{x_{1},\cdots, x_{r}\}$. Let $\Gr_{1^{r}, d}$ be the iterated version of the affine Schubert variety classifying chains of lattices $\L_{0}\subset\L_{1}\subset\L_{2}\subset\cdots\subset \L_{r}\subset \L'_{r}=\calO^{2}_{x}$ in $F_{x}^{2}$ where all inclusions have colength 1 except for the last one, which has colength $d$. Similarly let $\Gr_{d,1^{r}}$ be the iterated affine Schubert variety classifying chains of lattices $\L_{0}\subset\L'_{0}\subset\L'_{1}\subset\cdots\subset \L'_{r}=\calO^{2}_{x}$ in $F_{x}^{2}$ where the first inclusion has  colength $d$ and all other inclusions have colength 1. Let $\Gr_{d+r}\subset\Gr_{G,x}$ be the affine Schubert variety classifying $\calO_{x}$-lattices $\L\subset \calO^{2}_{x}$ with colength $d+r$. We have natural maps $\pi: \Gr_{1^{r},d}\to \Gr_{d+r}$ and $\pi':\Gr_{d,1^{r}}\to\Gr_{d+r}$ sending the lattice chains to $\L_{0}$. By the definition of $H'$, after choosing a trivialization of $\CE'_{r}$ in the formal neighborhood of $x$, we have an isomorphism
\begin{equation}\label{H'Gr}
H'\cong \Gr_{1^{r}, d}\times_{\Gr_{d+r}}\Gr_{d,1^{r}}.
\end{equation}
Since $\pi$ and $\pi'$ are surjective, therefore $\dim H'\geq \dim \Gr_{d+r}=d+r$. 

Now we show $\dim H'\le d+r$. Since the natural projections $\Gr_{1^{d+r}}\to \Gr_{1^{r}, d}$ and $\Gr_{1^{d+r}}\to \Gr_{d,1^{r}}$ are surjective, it suffices to show that $\dim (\Gr_{1^{d+r}}\times_{\Gr_{d+r}}\Gr_{1^{d+r}})\le d+r$. In other words, letting $m=d+r$, we have to show that $\pi_{m}: \Gr_{1^{m}}\to \Gr_{m}$ is a semismall map. This is a very special case of the semismallness of convolution maps in the geometric Satake equivalence, and we shall give a direct argument. The scheme $\Gr_{m}$ is stratified into $Y^{i}_{m}$ ($0\le i\le [m/2]$) where $Y^{i}_{m}$ classifies those $\L\subset \calO_{x}^{2}$ such that $\calO_{x}^{2}/\L\cong \calO_{x}/\varpi_{x}^{i}\oplus \calO_{x}/\varpi_{x}^{m-i}$. We may identify $Y^{i}_{m}$ with the open subscheme $Y^{0}_{m-2i}\subset \Gr_{m-2i}$ by sending $\L\in Y^{i}_{m}$ to  $\varpi_{x}^{-i}\L\subset \calO_{x}^{2}$, hence $\dim Y^{i}_{m}=m-2i$ and $\codim_{\Gr_{m}} Y^{i}_{m}=i$. We need to show that for $\L\in Y_{i}$, $\dim \pi_{m}^{-1}(\L)\le i$. We do this by induction on $m$. By definition, $\pi_{m}^{-1}(\L)$ classifies chains $\L=\L_{0}\subset \L_{1}\subset\cdots\subset\L_{m}=\calO_{x}^{2}$ each step of which has colength one. For $i=0$ such a chain is unique. For $i>0$, the choices of $\L_{1}$ are parametrized by $\PP^{1}$, and the map $\rho: \pi_{m}^{-1}(\L)\to \PP^{1}$ recording $\L_{1}$ has fibers $\pi^{-1}_{m-1}(\L_{1})$. Either $\calO_{x}^{2}/\L_{1}\cong \calO_{x}/\varpi_{x}^{i-1}\oplus \calO_{x}/\varpi_{x}^{m-i}$, in which case $\dim \rho^{-1}(\L_{1})=\dim\pi^{-1}_{m-1}(\L_{1})\le i-1$ by inductive hypothesis,  or $\calO_{x}^{2}/\L_{1}\cong \calO_{x}/\varpi_{x}^{i}\oplus \calO_{x}/\varpi_{x}^{m-i-1}$ (which happens for exactly one $\L_{1}$), in which case $\dim \rho^{-1}(\L_{1})=\dim\pi^{-1}_{m-1}(\L_{1})\le i$. These imply that  $\dim \pi_{m}^{-1}(\L)\le i$. The lemma is proved.

\end{proof}

\subsubsection{Proof of Lemma \ref{l:HCart dim}\eqref{dim HkM}}\label{sss:proof HCart dim HkM} We denote $\Hk^{\mu}_{\Mds,d}-\Hk^{\mu,\circ}_{\Mds,d}$ by $\partial\Hk^{\mu}_{\Mds,d}$. \index{$\partial\Hk^{\mu}_{\Mds,d}$}%
By Lemma \ref{l:Hk r comp} and Lemma \ref{l:HInc}, $\Hk^{\mu}_{\Mds,d}\cong \hX'_{d}\times_{\Pic^{d}_{X}}B_{r,d}$, where $B_{r,d}$ classifies $(r+1)$-triples of divisors $(D_{0},D_{1},\cdots, D_{r})$ of degree $d$ on $X'$, such that for each $1\leq i\leq r$, $D_{i}$ is obtained from  $D_{i-1}$  by changing some point $x'_{i}\in D_{i-1}$ to $\sigma(x'_{i})$. In particular, all $D_{i}$ have the same image $D_{b}:=\pi(D_{i})\in X_{d}$. We denote a point in $\Hk^{\mu}_{\Mds,d}$ by $z=(\CL,\alpha, D_{0},\cdots, D_{r})\in \hX'_{d}\times_{\Pic^{d}_{X}}B_{r,d}$, where $(\CL,\alpha)\in\hX'_{d}$ denotes a line bundle $\CL$ on $X'$ and a section $\alpha$ of it, together with an isomorphism $\Nm(\calL)\cong\calO_{X}(D_{b})$. Therefore both $\Nm(\alpha)$ and $1$ give sections of $\calO_{X}(D_{b})$. The image of $z$ under $\Hk^{\mu}_{\Mds,d}\to \calA_{d}\xrightarrow{\delta}\hX_{d}$ is the pair $(\calO_{X}(D_{b}), \Nm(\alpha)-1)$. Therefore $z\in \partial\Hk^{\mu}_{\Mds,d}$ if and only if $\div(\Nm(\alpha)-1)$ contains $\pi(x'_{i})$ for some $1\leq i\leq r$ ($\Nm(\alpha)=1$ is allowed). Since $x'_{i}\in D_{i-1}$, we have $\pi(x'_{i})\in \pi(D_{i-1})=D_{b}$, therefore $\pi(x_{i}')$ also appears in the divisor of $\Nm(\alpha)$. So we have two cases: either $\alpha=0$ or $\div(\Nm(\alpha))$ shares a common point with $D_{b}$. 

In the former case, $z$ is contained in $\Pic^{d}_{X'}\times_{\Pic^{d}_{X}}B_{r,d}$ which has dimension $g-1+d<2d-g+1$ since $d\geq 2g$. 

In the latter case, the image of $z$ in $\calA_{d}$ lies in the subscheme $\calC_{d}\subset X_{d}\times_{\Pic^{d}_{X}}X_{d}$ consisting of triples $(D_{1},D_{2}, \gamma:\calO(D_{1})\cong\calO(D_{2}))$ such that the divisors $D_{1}$ and $D_{2}$ have a common point. There is a surjection $X\times (X_{d-1}\times_{\Pic^{d-1}_{X}}X_{d-1})\to \calC_{d}$ which implies that $\dim \calC_{d}\leq 1+2(d-1)-g+1=2d-g$. Here we are using the fact that $d-1\geq 2g-1$ to compute the dimension of $X_{d-1}\times_{\Pic^{d-1}_{X}}X_{d-1}$. The conclusion is that in the latter case, $z$ lies in the preimage of $\calC_{d}$ in $\Hk^{\mu}_{\Mds,d}$, which has dimension equal to $\dim \calC_{d}$ (because $\Hk^{\mu}_{\CM,d}\to \calA_{d}$ is finite when restricted to $\calC_{d}\subset X_{d}\times_{\Pic^{d}_{X}}X_{d}$), which is $\leq 2d-g<2d-g+1$. 

Combining the two cases we conclude that $\dim\partial\Hk^{\mu}_{\Mds,d}<2d-g+1=\dim\Hk^{\mu}_{\Mds,d}$.

\subsubsection{Proof of Lemma \ref{l:dim hD}}\label{sss:proof dim hD} Let $\ul{x}=(x_{1},\cdots, x_{r})\in X^{r}$ be a geometric point. Let $\Sht^{r}_{G}(h_{D})_{\ul{x}}$ be the fiber of $\Sht^{r}_{G}(h_{D})$ over $\ul{x}$. When $\ul{x}$ is disjoint from $|D|$, $\oll{p}:\Sht^{r}_{G}(h_{D})_{\ul{x}}\to \Sht^{r}_{G,\ul{x}}$ is \'etale, hence in this case $\dim \Sht^{r}_{G}(h_{D})_{\ul{x}}=r$. By semicontinuity of fiber dimensions, it remains to show that $\dim\Sht^{r}_{G}(h_{D})_{\ul{x}}\leq r$ for all geometric points $\ul{x}$ over closed points of $X^{r}$. To simplify notation we assume $x_{i}\in X(k)$. The general case can be argued similarly. 

We use the same notation as in \S\ref{sss:proof HCart dim HkrGd}. In particular, we will use $\Hk^{r}_{G,d}$, and think of it as $\Hk^{\mu}_{2,d}/\Pic_{X}$ with $\mu=\mu_{+}^{r}$. Let $H_{D}$ be the fiber over $D$ of $H_{d}\to X_{d}$ sending $(\phi:\CE\incl\CE')$ to the divisor of $\det(\phi)$. Let $\Hk^{r}_{D, \ul{x}}$ be the fiber of $s: \Hk^{r}_{G,d}\to X_{d}\times X^{r}$ over $(D;\ul{x})$. 

Taking the fiber of the diagram \eqref{Sht G D} over $\ul{x}$ we get a Cartesian diagram
\begin{equation}\label{Sht hD x}
\xymatrix{\Sht^{r}_{G}(h_{D})_{\ul{x}}\ar[d]\ar[r] & \Hk^{r}_{D,\ul{x}}\ar[d]^{(p_{0},p_{r})}\\
H_{D}\ar[r]^{(\id,\Fr)} & H_{D}\times H_{D}}
\end{equation}
For each divisor $D'\leq D$ such that $D-D'$ has even coefficients, we have a closed embedding $H_{D'}\incl H_{D}$ sending $(\phi:\CE\to\CE')\in H_{D'}$ to $\CE\xrightarrow{\phi}\CE'\incl \CE'(\frac{1}{2}(D-D'))$. Let $H_{D,\leq D'}$ be the image of this embedding. Also let $H_{D,D'}=H_{D,\leq D'}-\cup_{D''<D'}H_{D,\leq D''}$. Then $\{H_{D,D'}\}$ give a stratification of $H_{D}$ indexed by divisors $D'\leq D$ such that $D-D'$ is even. We may restrict the diagram \eqref{Sht hD x} to $H_{D,D'}\times H_{D,D'}\incl H_{D}\times H_{D}$ and get a Cartesian diagram
\begin{equation}\label{DD'}
\xymatrix{\Sht^{r}_{G}(h_{D})_{D',\ul{x}}\ar[d]\ar[r] & \Hk^{r}_{D,D',\ul{x}}\ar[d]^{(p_{0},p_{r})}\\
H_{D,D'}\ar[r]^{(\id,\Fr)} & H_{D,D'}\times H_{D,D'}}				\index{$H_{D},  H_{D,D'}$}%
\end{equation}
We will show that $\dim\Sht^{r}_{G}(h_{D})_{D',\ul{x}}\leq r$ for each $D'\leq D$ and $D-D'$ even.

The embedding $H_{D'}\incl H_{D}$ above restricts to an isomorphism $H_{D',D'}\cong H_{D,D'}$. Similarly we have an isomorphism $\Hk^{r}_{D',D',\ul{x}}\cong\Hk^{r}_{D,D',\ul{x}}$ sending a diagram of the form \eqref{HmuGd} to the diagram of the same shape with each $\CE'_{i}$ changed to $\CE'(\frac{1}{2}(D-D'))$. Therefore we have $\Sht_{G}(h_{D'})_{D',\ul{x}}\cong\Sht_{G}(h_{D})_{D',\ul{x}}$, and it suffices to show that the open stratum $\Sht_{G}(h_{D})_{D,\ul{x}}$ has dimension at most $r$.  This way we reduce to treating the case $D'=D$.

Let $\tD=D+\ul{x}=D+x_{1}+\cdots+x_{r}\in X_{d+r}$ be the effective divisor of degree $d+r$. Let $\Bun_{G,\tD}$ be the moduli stack of $G$ bundles with a trivialization over $\tD$. A point of $\Bun_{G,\tD}$ is a pair $(\CE', \tau: \CE'_{\tD}\cong \calO^{2}_{\tD})$ (where $\CE'$ is a vector bundle of rank two over $X$) up to the action of $\Pic_{X}(\tD)$ (line bundles with a trivialization over $\tD$). There is a map $h: \Bun_{G,\tD}\to H_{D,D}$ sending $(\CE',\tau)$ to $(\phi:\CE\incl\CE')$ where $\CE$ is the preimage of the first copy of $\calO_{D}$ under the surjective map $\CE'\surj\CE'_{D}\xrightarrow{\tau}\calO^{2}_{\tD}\surj\calO^{2}_{D}$. Let $B_{D}\subset \Res^{\calO_{D}}_{k}G=\PGL_{2}(\calO_{D})$ be the subgroup stabilizing the first copy of $\calO^{2}_{D}$, and let $\tB_{D}\subset \Res^{\calO_{\tD}}_{k}G=\PGL_{2}(\calO_{\tD})$ be the preimage of $B_{D}$. Then $h$ is a $\tB_{D}$-torsor. In particular, $H_{D,D}$ is smooth, and the map $h$ is also smooth. Since smooth maps have sections \'etale locally, we may choose an \'etale surjective map $\om:Y\to H_{D,D}$ and a map $s: Y\to \Bun_{G,\tD}$ such that $hs=\om$.

Let $W=\Hk^{r}_{D,D,\ul{x}}\times_{H_{D,D}}Y$ (using the projection $\gamma_{r}: \Hk^{r}_{D,D,\ul{x}}\to H_{D,D}$). We claim that the projection $W\to Y$ is in fact a trivial fibration. In fact, let $T$ be the moduli space of diagrams of the form \eqref{HmuGd} with $\CE'_{r}=\calO^{2}_{X}$ and $\CE_{r}=\calO_{X}(-D)\oplus\calO_{X}$ and $\phi_{r}$ is the obvious embedding $\CE_{r}\incl\CE'_{r}$. In such a diagram all $\CE_{i}$ and $\CE'_{i}$ contain $\CE'_{r}(-\tD)$, therefore it contains the same amount of information as the diagram formed by the torsion sheaves $\CE_{i}/\CE'_{r}(-\tD)$ and $\CE'_{i}/\CE'_{r}(-\tD)$. For a point $y\in Y$ with image $(\phi_{r}:\CE_{r}\incl\CE'_{r})\in H_{D,D}$, $s(y)\in\Bun_{G,\tD}$ gives a trivialization of $\CE'_{r}|_{\tD}$. Therefore, completing $\phi_{r}$ into a diagram of the form \eqref{HmuGd} is the same as completing the standard point $(\CE_{r}=\calO_{X}(-D)\oplus\calO_{X}\incl \calO_{X}^{2})\in H_{D,D}$ into such a diagram. This shows that $W\cong Y\times T$. We have a diagram
\begin{equation*}
\xymatrix{\calU\ar[d]^{u}\ar[r] & W\ar[d]^{w}\ar[r]^{\sim} & Y\times T\ar[r]^{\om\times\id} & H_{D,D}\times T\\
\Sht^{r}_{G}(h_{D})_{D,\ul{x}}\ar[r]\ar[d] & \Hk^{r}_{D,D,\ul{x}}\ar[d]^{(\gamma_{0},\gamma_{r})}\\
H_{D,D}\ar[r]^{(\id,\Fr)} & H_{D,D}\times H_{D,D}}
\end{equation*}
where $\calU$ is defined so that the top square is Cartesian. The outer Cartesian diagram fits into the situation of \cite[Lemme 2.13]{VL}, and we have used the same notation as in {\em loc.cit}, except that we take $Z=H_{D,D}$. Applying {\em loc. cit.},  we conclude that the map $\calU\to T$ is \'etale. Since $w:W\to \Hk^{r}_{D,D,\ul{x}}$ is \'etale surjective, so is $u:\calU\to \Sht^{r}_{G}(h_{D})_{D,\ul{x}}$.  Therefore $\Sht^{r}_{G}(h_{D})_{D,\ul{x}}$ is \'etale locally isomorphic to $T$, and in particular they have the same dimension.

It remains to show that $\dim T\leq r$. Recall the moduli space $H'=H'_{D,\ul{x},\CE'_{r}}$ introduced in the proof of Lemma \ref{l:HCart dim}\eqref{dim HkrGd} classifying diagrams of the form \eqref{H'}. Here we fix $\CE'_{r}=\calO^{2}_{X}$. Let $T'$ be subscheme of $H'$ consisting of diagrams of the form \eqref{H'} where $(\phi_{r}:\CE_{r}\incl\CE'_{r})$ is fixed to be $(\CE_{r}=\calO_{X}(-D)\oplus\calO_{X}\incl \calO_{X}^{2})$. Then we have a natural embedding $T\incl T'$, and it suffices to show that $\dim T'\leq r$. Again we treat only the case where $D$ and $\ul{x}$ are both supported at a single point $x\in X$. The general case easily reduces to this by factorizing $T'$ into a product indexed by points in $|D|\cup\{x_{1},\cdots, x_{r}\}$.

Let $\Gr_{d}\subset \Gr_{G,x}$ be the affine Schubert variety classifying lattices $\L\subset\calO^{2}_{x}$ of colength $d$. Let $\Gr^{\hs}_{d}\subset \Gr_{d}$ be the open Schubert stratum consisting of lattices $\L\subset\calO^{2}_{x}$ such that $\calO^{2}_{x}/\L\cong\calO_{x}/\varpi^{d}_{x}$ ($\varpi_{x}$ is a uniformizer at $x$). We have a natural projection $\rho: H'\to \Gr_{d}$ sending the diagram \eqref{H'} to $\L:=\CE_{r}|_{\Spec\calO_{x}}\incl\CE'_{r}|_{\Spec\calO_{x}}=\calO^{2}_{x}$. Then $T'$ is the fiber of $\rho$ at the point $\L=\varpi^{d}_{x}\calO_{x}\oplus\calO_{x}$. Let $H^{\hs}=\rho^{-1}(\Gr^{\hs}_{d})$. There is a natural action of the positive loop group $L^{+}_{x}G$ on both $H'$ and $\Gr_{d}$ making $\rho$ equivariant under these actions. Since the action of $L^{+}_{x}G$ on $\Gr^{\hs}_{d}$ is transitive, all fibers of $\rho$ over points of $\Gr^{\hs}_{d}$ have the same dimension, i.e.,
\begin{equation}\label{dim T'}
\dim T'=\dim H^{\hs}-\dim \Gr^{\hs}_{d}=\dim H^{\hs}-d.
\end{equation}
By Lemma \ref{l:H'},  $\dim H'=d+r$. Therefore $\dim H^{\hs}=d+r$ and $\dim T'\leq r$ by \eqref{dim T'}.  We are done.


\section{Cohomological spectral decomposition}\label{s:c spec}

In this section, we give a decomposition of the cohomology of $\Sht^{r}_{G}$ under the action of the Hecke algebra $\sH$, generalizing the classical spectral decomposition for the space of automorphic forms. The main result is Theorem \ref{th:Spec decomp} which shows that $\cohoc{2r}{\Sht^{r}_{G,\kbar},\Ql}$ is an orthogonal  direct sum of an Eisenstein part and finitely many (generalized) Hecke eigenspaces. We then use a variant of such a decomposition for $\Sht'^{r}_{G}$ to make a decomposition for the Heegner-Drinfeld cycle.

\subsection{Cohomology of the moduli stack of Shtukas}\label{ss:coho Sht}

\subsubsection{Truncation of $\Bun_{G}$  by index of instability} For a rank two vector bundle $\calE$ over $X$,  we define its {\em index of instability} to be 
\begin{equation*}
\inst(\calE):=\max\{2\deg\calL-\deg \calE\},			\index{$\inst(\calE)$}%
\end{equation*}
where $\calL$ runs over line subbundle of $\calE$. When $\inst(\calE)>0$, $\calE$ is called unstable, in which case there is a unique line subbundle $\calL\subset\calE$ such that $\deg\calL>\frac{1}{2}\deg\calE$. We call this line subbundle the {\em maximal line subbundle} of $\calE$. Note that there is a constant $c(g)$ depending only on the genus $g$ of $X$ such that $\inst(\calE)\geq c(g)$ for all rank two vector bundles $\CE$ on $X$.

The function $\inst:\Bun_{2}\to \ZZ$ is upper semi-continuous, and descends to a function $\inst:\Bun_{G}\to \ZZ$. For an integer $a$, $\inst^{-1}((-\infty,a])=:\Bun^{\leq a}_{G}$ is an open substack of $\Bun_{G}$ of finite type over $k$.		\index{$\Bun^{\leq a}_{G}$}%

\subsubsection{Truncation of $\Sht^{r}_{G}$  by index of instability} For $\Sht^{r}_{G}$ we define a similar stratification by the index of instability of the various $\calE_{i}$. We choose $\mu$ as in \S\ref{sss:mu} and present $\Sht^{r}_{G}$ as $\Sht^{\mu}_{2}/\Pic_{X}(k)$. 

Consider the set $\calD$ of functions $d:\ZZ/r\ZZ\to \ZZ$ such that $d(i)-d(i-1)=\pm1$ for all $i$. There is a partial order on $\calD$ by pointwise comparison. 		 \index{$\calD$}%

For any $d\in\calD$, let $\Sht^{\mu,\leq d}_{2}$ be the open substack of  $\Sht^{\mu}_{2}$ consisting of those $(\calE_{i};x_{i};f_{i})$ such that $\inst(\calE_{i})\leq d(i)$. Then each $\Sht^{\mu,\leq d}_{2}$ is preserved by the $\Pic_{X}(k)$-action, and we define $\Sht^{\mu,\leq d}_{G}:=\Sht^{\mu,\leq d}_{2}/\Pic_{X}(k)$, an open substack of $\Sht^{r}_{G}$ of finite type. If we change $\mu$ to $\mu'$, the canonical isomorphism $\Sht^{\mu}_{G}\cong\Sht^{\mu'}_{G}$ in Lemma \ref{l:indep mu} preserves the $G$-torsors $\CE_{i}$, therefore the open substacks $\Sht^{\mu,\leq d}_{G}$ and $\Sht^{\mu',\leq d}_{G}$ correspond to each other under the isomorphism. This shows that $\Sht^{\mu,\leq d}_{G}$ is canonically independent of the choice of $\mu$, and we will simply denote it by $\Sht^{\leq d}_{G}$.		\index{$\Sht^{\leq d}_{G}$}%

In the sequel, the superscript on $\Sht_{G}$ will be reserved for the truncation parameters $d\in\calD$, and we will omit $r$ from the superscripts. In the rest of the section, $\Sht_{G}$ means $\Sht^{r}_{G}$.

Define $\Sht^{d}_{G}:=\Sht^{\leq d}_{G}-\cup_{d'<d}\Sht^{\leq d'}_{G}$. This is a locally closed substack of $\Sht_{G}$ of finite type classifying Shtukas $(\calE_{i};x_{i};f_{i})$ with $\inst(\calE_{i})=d(i)$ for all $i$. A priori we could define $\Sht^{d}_{G}$ for any function $d:\ZZ/r\ZZ\to\ZZ$; however, only for those $d\in\calD$ is $\Sht^{d}_{G}$ nonempty,  because for $(\calE_{i};x_{i};f_{i})\in\Sht^{\mu}_{2}$, $\inst(\calE_{i})=\inst(\calE_{i-1})\pm1$. The locally closed substacks  $\{\Sht^{d}_{G}\}_{d\in\calD}$ give a stratification of $\Sht_{G}$. 		 \index{$\Sht^{d}_{G}$}%

\subsubsection{Cohomology of $\Sht_{G}$}
Let $\pi^{\leq d}_{G}:\Sht^{\leq d}_{G}\to X^{r}$ be the restriction of $\pi_{G}$, and similarly define $\pi^{<d}_{G}$ and $\pi^{d}_{G}$. \index{$\pi^{\leq d}_{G},\pi^{<d}_{G},\pi^{d}_{G}$}%
For $d\leq d'\in\calD$ we have a map induced by the open inclusion $\Sht^{\leq d}_{G}\incl\Sht^{\leq d'}_{G}$:
\begin{equation*}
\iota_{d,d'}:\bR \pi^{\leq d}_{G,!}\Ql\to \bR \pi^{\leq d'}_{G,!}\Ql
\end{equation*} 

The total cohomology $\cohoc{*}{\Sht_{G}\otimes_{k}\kbar}$ is defined as the inductive limit
\begin{equation*}
\cohoc{*}{\Sht_{G}\otimes_{k}\kbar}:=\varinjlim_{d\in\calD}\cohoc{*}{\Sht^{\leq d}_{G}\otimes_{k}\kbar}=\varinjlim_{d\in\calD}\cohog{*}{X^{r}\otimes_{k}\kbar,\bR \pi^{\leq d}_{G,!}\Ql}.
\end{equation*}

\subsubsection{The action of Hecke algebra on the cohomology of $\Sht_{G}$} For each effective divisor $D$ of $X$, we have defined in \S\ref{sss:Sht hD} a self-correspondence $\Sht_{G}(h_{D})$ of $\Sht_{G}$ over $X^{r}$. 

For  any $d\in\calD$, let $\leftexp{\leq d}{\Sht}_{G}(h_{D})\subset \Sht_{G}(h_{D})$ be the preimage of $\Sht^{\leq d}_{G}$ under $\oll{p}$. \index{$\leftexp{\leq d}{\Sht}_{G}(h_{D})$}%
For a point $(\CE_{i}\incl\CE'_{i})$ of $\leftexp{\leq d}{\Sht}_{G}(h_{D})$, we have $\inst(\CE_{i})\leq d(i)$, hence $\inst(\CE'_{i})\leq d(i)+\deg D$. Therefore the image of $\leftexp{\leq d}{\Sht}_{G}(h_{D})$ under $\orr{p}$ lies in $\Sht^{\leq d+\deg D}_{G}$. For any $d'\geq d+\deg D$, we may view $\leftexp{\leq d}{\Sht}_{G}(h_{D})$ as a correspondence between $\Sht^{\leq d}_{G}$ and $\Sht^{\leq d'}_{G}$ over $X^{r}$. By Lemma \ref{l:dim hD}, $\dim \Sht_{G}(h_{D})=\dim \Sht_{G}=2r$, the fundamental cycle of $\leftexp{\leq d}{\Sht}_{G}(h_{D})$ gives a cohomological correspondence between the constant sheaf on $\Sht^{\leq d}_{G}$ and the constant sheaf on $\Sht^{\leq d'}_{G}$ (see \S\ref{sss:corr}), and induces a map
\begin{equation}\label{Hk d d'}
C(h_{D})_{d,d'}:\bR\pi^{\leq d}_{G,!}\Ql \to \bR\pi^{\leq d'}_{G,!}\Ql.		\index{$C(h_D), C(h_{D})_{d,d'}$}%
\end{equation}
Here we are using the fact that $\leftexp{\leq d}{\oll{p}}:\leftexp{\leq d}{\Sht}_{G}(h_{D})\to\Sht^{\leq d}_{G}$ is proper (which is necessary for the construction \eqref{push coho corr}), which follows from the properness of $\oll{p}: \Sht_{G}(h_{D})\to\Sht_{G}$ by Lemma \ref{l:Sht hD proper}. 

For any $e\geq d$ and $e'\geq e+\deg D$ and $e'\geq d'$, we have a commutative diagram
\begin{equation*}
\xymatrix{\bR\pi^{\leq d}_{G,!}\Ql\ar[rr]^{C(h_{D})_{d,d'}}\ar[d]^{\iota_{d,e}}&& \bR\pi^{\leq d'}_{G,!}\Ql\ar[d]^{\iota_{d',e'}}\\
\bR\pi^{\leq e}_{G,!}\Ql\ar[rr]^{C(h_{D})_{e,e'}} && \bR\pi^{\leq e'}_{G,!}\Ql}
\end{equation*}
which follows from the definition of cohomological correspondences. Taking $\cohog{*}{X^{r}\otimes_{k}\kbar,-}$ and taking inductive limit over $d$ and $e$, we get an endomorphism of $\cohoc{*}{\Sht_{G}\otimes_{k}\kbar}$
\begin{eqnarray*}
C(h_{D})&:& \cohoc{*}{\Sht_{G}\otimes_{k}\kbar}=\varinjlim_{d\in\calD}\cohog{*}{X^{r}\otimes_{k}\kbar,\bR \pi^{\leq d}_{G,!}\Ql}\\
&&\xrightarrow{\varinjlim C(h_{D})_{d,d'}}\varinjlim_{d'\in\calD}\cohog{*}{X^{r}\otimes_{k}\kbar,\bR \pi^{\leq d'}_{G,!}\Ql}=\cohoc{*}{\Sht_{G}\otimes_{k}\kbar}.		\index{$C(h_D), C(h_{D})_{d,d'}$}%
\end{eqnarray*}

The following result is a cohomological analog of Proposition \ref{p:Hecke action on Chow}.
\begin{prop} The assignment $h_{D}\mapsto C(h_{D})$ gives a ring homomorphism for each $i\in\ZZ$
\begin{equation*}
C:\sH\to\End(\cohoc{i}{\Sht_{G}\otimes_{k}\kbar}).
\end{equation*}
\end{prop}
\begin{proof}
The argument is similar to that of Proposition \ref{p:Hecke action on Chow}, for this reason we only give a sketch here. For two effective divisors $D$ and $D'$, we need to check that the action of $C(h_{D}h_{D'})$ is the same as the composition $C(h_{D})\circ C(h_{D'})$. 

Let $d,d^{\dagger}$ and $d'\in\calD$ satisfy $d^{\dagger}\geq d+\deg D'$ and $d'\geq d^{\dagger}+\deg D$, then the map
\begin{equation*}
C(h_{D})_{d^{\dagger},d'}\circ C(h_{D'})_{d,d^{\dagger}}:\bR\pi^{\leq d}_{G,!}\Ql\to \bR\pi^{\leq d^{\dagger}}_{G,!}\Ql\to\bR\pi^{\leq d'}_{G,!}\Ql
\end{equation*}
is induced from a cohomological correspondence $\zeta$ between the constant sheaves on $\Sht^{\leq d}_{G}$ and on $\Sht^{\leq d'}_{G}$ supported on $\leftexp{\leq d}{\Sht}_{G}(h_{D})*\leftexp{\leq d^{\dagger}}{\Sht}_{G}(h_{D'}):=\leftexp{\leq d}{\Sht}_{G}(h_{D})\times_{\orr{p},\Sht_{G},\oll{p}}\leftexp{\leq d^{\dagger}}{\Sht}_{G}(h_{D'})$; i.e., $\zeta\in\hBM{4r}{\leftexp{\leq d}{\Sht}_{G}(h_{D})*\leftexp{\leq d^{\dagger}}{\Sht}_{G}(h_{D'})\otimes_{k}\kbar}$.

On the other hand, the Hecke function $h_{D}h_{D'}$ is a linear combination of $h_{E}$ where $E\leq D+D'$ and $D+D'-E$ is even. Since $d\in\calD$ and $d'\geq d+\deg D+\deg D'$,  the map
\begin{equation*}
C(h_{D}h_{D'})_{d,d'}:\bR\pi^{\leq d}_{G,!}\Ql\to \bR\pi^{\leq d'}_{G,!}\Ql
\end{equation*}
is induced from a cohomological correspondence $\xi$ between the constant sheaves on $\Sht^{\leq d}_{G}$ and on $\Sht^{\leq d'}_{G}$ supported on the union of $\leftexp{\leq d}{\Sht}_{G}(h_{E})$ for $E\leq D+D'$ and $D+D'-E$ even,  i.e., supported on $\leftexp{\leq d}{\Sht}_{G}(h_{D+D'})$. In other words, $\xi\in\hBM{4r}{\leftexp{\leq d}{\Sht}_{G}(h_{D+D'})\otimes_{k}\kbar}$. 

There is a proper map of correspondences $\theta: \leftexp{\leq d}{\Sht}_{G}(h_{D})*\leftexp{\leq d^{\dagger}}{\Sht}_{G}(h_{D'})\to \leftexp{\leq d}{\Sht}_{G}(h_{D+D'})$, and the action of $C(h_{D})_{d^{\dagger},d'}\circ \,C(h_{D'})_{d,d^{\dagger}}$ is also induced from the class $\theta_{*}\zeta\in \hBM{4r}{\leftexp{\leq d}{\Sht}_{G}(h_{D+D'})\otimes_{k}\kbar}$, viewed as a cohomological correspondence supported on $\leftexp{\leq d}{\Sht}_{G}(h_{D+D'})$. Let $U=X-|D|-|D'|$. It is easy to check that $\xi|_{U^{r}}=\theta_{*}\zeta|_{U^{r}}$ using that, over $U^{r}$, the correspondences $\Sht(h_{D}), \Sht(h_{D'})$ and $\Sht(h_{D+D'})$ are finite \'etale over $\Sht_{G}$. By Lemma \ref{l:dim hD}, $\leftexp{\leq d}{\Sht}_{G}(h_{D+D'})-\leftexp{\leq d}{\Sht}_{G}(h_{D+D'})|_{U^{r}}$ has dimension $<2r$, therefore $\xi=\theta_{*}\zeta$ holds as elements in $\hBM{4r}{\leftexp{\leq d}{\Sht}_{G}(h_{D+D'})\otimes_{k}\kbar}$, and hence $C(h_{D}h_{D'})_{d,d'}=C(h_{D})_{d^{\dagger},d'}\circ C(h_{D'})_{d,d^{\dagger}}$. Applying $\cohog{*}{X^{r}\otimes_{k}\kbar, -}$ and taking inductive limit over $d$ and $d'$, we see that $C(h_{D}h_{D'})=C(h_{D})\circ C(h_{D'})$ as endomorphisms of $\cohoc{*}{\Sht_{G}\otimes_{k}\kbar}$.
\end{proof}

\subsubsection{Notation} For $\alpha\in\cohoc{*}{\Sht_{G}\otimes_{k}\kbar}$ and $f\in\sH$, we denote the action of $C(f)$ on $\alpha$ simply by $f*\alpha\in \cohoc{*}{\Sht_{G}\otimes_{k}\kbar}$.

\subsubsection{}
Cup product gives a symmetric bilinear pairing on $\cohoc{*}{\Sht_{G}\otimes_{k}\kbar}$
\begin{equation*}
(-,-): \cohoc{i}{\Sht_{G}\otimes_{k}\kbar}\times \cohoc{4r-i}{\Sht_{G}\otimes_{k}\kbar}\to \cohoc{4r}{\Sht_{G}\otimes_{k}\kbar}\cong\Ql(-2r).
\end{equation*}

We have a cohomological analog of Lemma \ref{l:Chow self adj}.
\begin{lemma}\label{l:coho self adj} The action of any $f\in\sH$ on $\cohoc{*}{\Sht_{G}\otimes_{k}\kbar}$ is self-adjoint with respect to the cup product pairing.
\end{lemma}
\begin{proof}
Since $\{h_{D}\}$ span $\sH$, it suffices to show that the action of $h_{D}$ is self-adjoint. From the construction of the endomorphism $C(h_{D})$ of $\cohoc{i}{\Sht_{G}\otimes_{k}\kbar}$, we see that for $\alpha\in\cohoc{i}{\Sht_{G}\otimes_{k}\kbar}$ and $\beta\in \cohoc{4r-i}{\Sht_{G}\otimes_{k}\kbar}$, the pairing $(h_{D}*\alpha,\beta)$ is the same as the pairing $([\Sht_{G}(h_{D})], \oll{p}^{*}\alpha\cup\orr{p}^{*}\beta)$ (i.e, the pairing of  $\oll{p}^{*}\alpha\cup\orr{p}^{*}\beta\in\cohoc{4r}{\Sht_{G}(h_{D})\otimes_{k}\kbar}$ with the fundamental class of $\Sht_{G}(h_{D})$). Similarly, $(\alpha,h_{D}*\beta)$ is the pairing $([\Sht_{G}(h_{D})], \oll{p}^{*}\beta\cup\orr{p}^{*}\alpha)$. Applying the involution $\tau$ on $\Sht_{G}(h_{D})$ constructed in the proof of Lemma \ref{l:Chow self adj} that switches the two projections $\oll{p}$ and $\orr{p}$, we get
\begin{equation*}
\Big([\Sht_{G}(h_{D})], \oll{p}^{*}\alpha\cup\orr{p}^{*}\beta\Big)=\Big([\Sht_{G}(h_{D})], \oll{p}^{*}\beta\cup\orr{p}^{*}\alpha\Big)
\end{equation*}
which is equivalent to the self-adjointness of $h_{D}$: $(h_{D}*\alpha,\beta)=(\alpha,h_{D}*\beta)$.
\end{proof}

\subsubsection{} The cycle class map gives a $\QQ$-linear map (see \S\ref{sss:cycle class map})
\begin{equation*}
\cl: \Ch_{c,i}(\Sht_{G})_{\QQ}\to \cohoc{4r-2i}{\Sht_{G}\otimes_{k}\kbar}(2r-i). 		\index{$\cl$, cycle class}%
\end{equation*}

\begin{lemma} The map $\cl$ is $\sH$-equivariant for any $i$.
\end{lemma}
\begin{proof}
Since $\{h_{D}\}$ span $\sH$, it suffices to show that $\cl$ intertwines the actions of  $h_{D}$ on $\Ch_{c,i}(\Sht_{G})$ and on $\cohoc{4r-2i}{\Sht_{G}\otimes_{k}\kbar}(2r-i)$. Let $\zeta\in\Ch_{c,i}(\Sht_{G})$. By the definition of the $h_{D}$-action on $\Ch_{c,i}(\Sht_{G})$, $h_{D}*\zeta\in\Ch_{c,i}(\Sht_{G})$ is $\pr_{2*}((\pr^{*}_{1}\zeta)\cdot_{\Sht_{G}\times\Sht_{G}}(\oll{p},\orr{p})_{*}[\Sht_{G}(h_{D})])$. Taking its cycle class we get that $\cl(h_{D}*\zeta)\in\cohoc{4r-2i}{\Sht_{G}\otimes_{k}\kbar}(2r-i)$ can be identified with the class
\begin{equation*}
\orr{p}_{*}(\oll{p}^{*}\cl(\zeta)\cap [\Sht_{G}(h_{D})])\in \homog{2i}{\Sht_{G}\otimes_{k}\kbar}(-i)
\end{equation*}
under the Poincar\'e duality isomorphism $\cohoc{4r-2i}{\Sht_{G}\otimes_{k}\kbar}\cong\homog{2i}{\Sht_{G}\otimes_{k}\kbar}(-2r)$.

On the other hand, by \eqref{push coho corr}, the action of $h_{D}$ on $\cohoc{4r-2i}{\Sht_{G}\otimes_{k}\kbar}$ is the composition
\begin{eqnarray*}
&&\cohoc{4r-2i}{\Sht_{G}\otimes_{k}\kbar}\xrightarrow{\oll{p}^{*}}\cohoc{4r-2i}{\Sht_{G}(h_{D})\otimes_{k}\kbar}\xrightarrow{\cap[\Sht_{G}(h_{D})]}\\
&&\homog{2i}{\Sht_{G}(h_{D})\otimes_{k}\kbar}(-2r)\xrightarrow{\orr{p}_{*}}\homog{2i}{\Sht_{G}\otimes_{k}\kbar}(-2r)\cong\cohoc{4r-2i}{\Sht_{G}\otimes_{k}\kbar}.
\end{eqnarray*}
Therefore we have $\cl(h_{D}*\zeta)=h_{D}*\cl(\zeta)$.
\end{proof}

\subsubsection{} We are most interested in the middle dimensional cohomology
\begin{equation*}
V_{\Ql}:=\cohoc{2r}{\Sht_{G}\otimes_{k}\kbar,\Ql}(r).
\end{equation*}
This is a $\Ql$-vector space with an action of $\sH$. In the sequel, we simply write $V$ for $V_{\Ql}$.

For the purpose of proving our main theorems, it is the cohomology of $\Sht'_{G}$ rather than $\Sht_{G}$ that matters. However, for most of this section, we will study $V$. The main result in this section (Theorem \ref{th:Spec decomp}) provides a decomposition of $V$ into a direct sum of two $\sH$-modules, an infinite-dimensional one called the Eisenstein part and a finite-dimensional complement. The same result holds when $\Sht_{G}$ is replaced by $\Sht'_{G}$ with the same proof. We will only state the corresponding result for $\Sht'_{G}$ in the final subsection \S\ref{ss:decomp HD cycle} and use it to decompose the Heegner-Drinfeld cycle.

\subsection{Study of horocycles}
Let $B\subset G$ be a Borel subgroup with quotient torus $H\cong\Gm$. \index{$B, H$}%
We think of $H$ as the universal Cartan of $G$, which is to be distinguished with the subgroup $A$ of $G$. We shall define horocycles in $\Sht_{G}$ corresponding to $B$-Shtukas.

\subsubsection{$\Bun_{B}$} Let $\wt{B}\subset \GL_{2}$ be the preimage of $B$. Then $\Bun_{\tB}$ classifies pairs $(\CL\incl\CE)$ where $\CE$ is a rank two vector bundle over $X$ and $\CL$ is a line subbundle of it. We have $\Bun_{B}=\Bun_{\tB}/\Pic_{X}$ \index{$\Bun_{B}, \Bun_{\tB}$}%
where $\Pic_{X}$ acts by simultaneous tensoring on $\CE$ and on $\CL$. We have a decomposition
\begin{equation*}
\Bun_{B}=\coprod_{n\in\ZZ}\Bun^{n}_{B}
\end{equation*}
where $\Bun^{n}_{B}=\Bun^{n}_{\tB}/\Pic_{X}$, and $\Bun^{n}_{\tB}$ is the open and closed substack of $\Bun_{\tB}$ classifying those $(\CL\incl\CE)$ such that $2\deg\CL-\deg\CE=n$.

\subsubsection{Hecke stack for $\tB$}
Fix $d\in\calD$. Choose any $\mu$ as in \S\ref{sss:mu}. Consider the moduli stack $\Hk^{\mu,d}_{\wt{B}}$ whose $S$-points classify the  data $(\CL_{i}\incl\CE_{i}; x_{i};f_{i})$ where
\begin{enumerate}
\item A point $(\CE_{i};x_{i}; f_{i})\in\Hk^{\mu}_{2}(S)$.
\item For each $i=0,\cdots, r$, $(\CL_{i}\incl\CE_{i})\in\Bun^{d(i)}_{\tB}$ such that the isomorphism $f_{i}: \CE_{i-1}|_{X\times S-\Gamma_{x_{i}}}\cong \CE_{i}|_{X\times S-\Gamma_{x_{i}}}$ restricts to an isomorphism $\alpha'_{i}: \CL_{i-1}|_{X\times S-\Gamma_{x_{i}}}\cong \CL_{i}|_{X\times S-\Gamma_{x_{i}}}$. \index{$\Hk^{\mu,d}_{\wt{B}}$}%
\end{enumerate}
We have $(r+1)$ maps $p_{i}:\Hk^{\mu,d}_{\tB}\to \Bun^{d(i)}_{\tB}$ by sending the above data to $(\CL_{i}\incl \CE_{i})$, $i=0,1,\cdots, r$. We define $\Sht^{\mu,d}_{\tB}$ by the Cartesian diagram
\begin{equation}\label{define ShtB}
\xymatrix{\Sht^{\mu,d}_{\tB}\ar[rr]\ar[d] && \Hk^{\mu,d}_{\tB}\ar[d]^{(p_{0},p_{r})}\\
\Bun^{d(0)}_{\tB}\ar[rr]^{(\id,\Fr)} && \Bun^{d(0)}_{\tB}\times\Bun^{d(0)}_{\tB}}				\index{$\Sht^{\mu,d}_{\wt{B}}$}%
\end{equation} 
In other words, $\Sht^{\mu,d}_{\tB}$ classifies $(\CL_{i}\incl\CE_{i};x_{i};f_{i};\iota)$, where $(\CL_{i}\incl\CE_{i};x_{i};f_{i})$ is a point in $\Hk^{\mu,d}_{\tB}$ and $\iota$ is an isomorphism $\CE_{r}\cong\leftexp{\tau}{\CE_{0}}$ sending $\CL_{r}$ isomorphically to $\leftexp{\tau}{\CL_{0}}$.

We may summarize the data classified by $\Sht^{\mu,d}_{\tB}$ as a commutative diagram
\begin{equation}\label{ShtB}
\xymatrix{0\ar[r] & \calL_{0}\ar@{-->}[d]^{\alpha'_{1}}\ar[r] & \calE_{0}\ar@{-->}[d]^{f_{1}}\ar[r] & \calM_{0}\ar@{-->}[d]^{\alpha''_{1}}\ar[r] & 0\\
&\cdots\ar@{-->}[d]^{\alpha'_{r}}&\cdots\ar@{-->}[d]^{f_{r}}&\cdots\ar@{-->}[d]^{\alpha''_{r}} \\
0\ar[r] & \calL_{r}\ar[d]^{\wr}_{\iota'}\ar[r] & \calE_{r}\ar[d]_{\iota}^{\wr}\ar[r] & \calM_{r}\ar[d]^{\wr}_{\iota''}\ar[r] & 0\\
0\ar[r] & \leftexp{\tau}{\calL}_{0}\ar[r] & \leftexp{\tau}{\calE}_{0}\ar[r] & \leftexp{\tau}{\calM}_{0}\ar[r] & 0}
\end{equation}
Here we denote the quotient line bundle $\CE_{i}/\CL_{i}$ by $\CM_{i}$.

\subsubsection{$B$-Shtukas} There is an action of $\Pic_{X}(k)$ on $\Sht^{\mu,d}_{\wt{B}}$ by tensoring each member in \eqref{ShtB} by a line bundle defined over $k$. We define 
\begin{equation*}
\Sht^{d}_{B}:=\Sht^{\mu, d}_{\wt{B}}/\Pic_{X}(k).		\index{$\Sht^{d}_{B}$}%
\end{equation*}
Equivalently we may first define $\Hk^{\mu, d}_{B}:=\Hk^{\mu,d}_{\tB}/\Pic_{X}$ and define $\Sht^{d}_{B}$ by a diagram similar to \eqref{define ShtB}, using $\Hk^{d}_{B}$ and $\Bun^{d(0)}_{B}$ instead of $\Hk^{\mu,d}_{\tB}$ and $\Bun^{d(0)}_{\tB}$. The same argument as Lemma \ref{l:Hk indep mu} shows that $\Hk^{\mu, d}_{B}$ is canonically independent of the choice of $\mu$ and these isomorphisms preserve the maps $p_{i}$, hence $\Sht^{d}_{B}$ is also independent of the choice of $\mu$.


\subsubsection{Indexing by degrees}  In the  definition of Shtukas in \S\ref{sss:Sht n}, we may decompose $\Sht^{\mu}_{n}$  according to the degrees of $\CE_{i}$. More precisely, for $d\in\calD$, we let $\mu(d)\in\{\pm1\}^{r}$ be defined as
\begin{equation}\label{mu d}
\mu_{i}(d)=d(i)-d(i-1).
\end{equation}
Let $\Sht^{d}_{n}\subset \Sht^{\mu(d)}_{n}$ be the open and closed substack classifying rank $n$ Shtukas $(\CE_{i};\cdots)$ with $\deg\CE_{i}=d(i)$. 

Consider the action of $\ZZ$ on $\calD$ by adding a constant integer to a function $d\in\calD$. The assignment $d\mapsto \mu(d)$ descends to a function $\calD/\ZZ\to \{\pm1\}^{r}$. For a $\ZZ$-orbit $\delta\in\calD/\ZZ$, we write $\mu(d)$ as $\mu(\delta)$ for any $d\in\delta$. Then for  any $\delta\in\calD/\ZZ$, we have a decomposition
\begin{equation}\label{Sht into components}
\Sht^{\mu(\delta)}_{n}=\coprod_{d\in\delta}\Sht^{d}_{n}
\end{equation}

In particular, after identifying $H$ with $\Gm$, we define $\Sht^{d}_{H}$ to be $\Sht^{d}_{1}$ for any $d\in\calD$.

\subsubsection{The horocycle correspondence} From the definition of $\Sht^{d}_{B}$, we have a forgetful map
\begin{equation*}
p_{d}:\Sht^{d}_{B}\to\Sht_{G}
\end{equation*}
sending the data in \eqref{ShtB} to the middle column. 

On the other hand, mapping the diagram \eqref{ShtB} to  $(\calL_{i}\otimes\calM^{-1}_{i}; x_{i}; \alpha'_{i}\otimes\alpha''_{i};\iota'\otimes\iota'')$ we get a morphism
\begin{equation*}
q_{d}:\Sht^{d}_{B}\to \Sht^{d}_{H}. \index{$\Sht^{d}_{H}$}%
\end{equation*}

Via the maps $p_{d}$ and $q_{d}$, we may view $\Sht^{d}_{B}$ as a correspondence between $\Sht_{G}$ and $\Sht^{d}_{H}$ over $X^{r}$:
\begin{equation}\label{ShtB as corr}
\xymatrix{ & \Sht^{d}_{B}\ar[dl]_{p_{d}}\ar[dr]^{q_{d}}\ar[dd]^{\pi^{d}_{B}} \\
\Sht_{G}\ar[dr]_{\pi_{G}} & & \Sht^{d}_{H}\ar[dl]^{\pi^{d}_{H}}\\
& X^{r}}
\end{equation}

\begin{lemma}\label{l:ShtBG} Let $\calD^{+}\subset\calD$ be the subset consisting of functions $d$ such that $d(i)>0$ for all $i$.  \index{$\calD^{+}$}%
Suppose $d\in\calD^{+}$.  Then the map $p_{d}:\Sht^{d}_{B}\to\Sht_{G}$ has image $\Sht^{d}_{G}$ and induces an isomorphism $\Sht^{d}_{B}\cong\Sht^{d}_{G}$.
\end{lemma}
\begin{proof}We first show that $p_{d}(\Sht^{d}_{B})\subset\Sht^{d}_{G}$. If $(\calL_{i}\incl\calE_{i};x_{i};f_{i};\iota)\in\Sht^{d}_{B}$ (up to tensoring with a line bundle), then $\deg\calL_{i}\geq\frac{1}{2}(\deg\calE_{i}+d(i))>\frac{1}{2}\deg\CE_{i}$, hence $\calL_{i}$ is the maximal line subbundle of $\calE_{i}$. Therefore $\inst(\calE_{i})=d(i)$ and $(\calE_{i};x_{i};f_{i})\in\Sht^{d}_{G}$. 

Conversely, we will define a map $\Sht^{d}_{G}\to \Sht^{d}_{B}$. Let $(\calE_{i};x_{i};f_{i};\iota)\in\Sht^{d}_{G}(S)$, then the maximal line bundle $\calL_{i}\incl\calE_{i}$ is well-defined since each $\calE_{i}$ is unstable. 

We claim that for each geometric point $s\in S$, the generic fibers of $\calL_{i}|_{X\times \{s\}}$ map isomorphically to each other under the rational maps $f_{i}$ between the $\calE_{i}$'s. For this we may assume $S=\Spec(K)$ for some field $K$ and we base change the situation to $K$ without changing notation. Let $\calL'_{i+1}\subset \calE_{i+1}$ be the line bundle obtained by saturating $\calL_{i}$ under the rational map $f_{i+1}:\calE_{i}\dashrightarrow\calE_{i+1}$. Then $d'(i+1):=2\calL'_{i+1}-\deg\calE_{i+1}=d(i)\pm1$. If $d'(i+1)>0$, then $\calL'_{i+1}$ is also the maximal line subbundle of $\calE_{i+1}$, hence $\calL'_{i+1}=\calL_{i+1}$. If $d'(i+1)\leq0$, then we must have $d(i)=1$ and $d'(i+1)=0$. Since $d\in\calD^{+}$, we must have $d(i+1)=2$. In this case the map $\calL'_{i+1}\oplus\calL_{i+1}\to\calE_{i+1}$ cannot be injective because the source has degree at least $\frac{1}{2}(\deg\calE_{i+1}+d'(i+1))+\frac{1}{2}(\deg\calE_{i+1}+d(i+1))=\deg\calE_{i+1}+1>\deg\calE_{i+1}$. Therefore $\calL'_{i+1}$ and $\calL_{i+1}$ have the same generic fiber, which is impossible since they are both line subbundles of $\calE_{i+1}$ but have different degrees. This proves the claim.

Moreover, the isomorphism $\iota:\CE_{r}\cong\leftexp{\tau}{\CE_{0}}$ must send $\CL_{r}$ isomorphically onto $\leftexp{\tau}{\CL_{0}}$ by the uniqueness of the maximal line subbundle. This together with the claim above implies that $(\calL_{i}; x_{i};f_{i}|_{\CL_{i}};\iota|_{\CL_{r}})$ is a rank one sub-Shtuka of $(\calE_{i};x_{i};f_{i};\iota)$, and therefore $(\calL_{i}\incl\calE_{i};x_{i};f_{i};\iota)$ gives a point in $\Sht^{d}_{B}$. This way we have defined a map $\Sht^{d}_{G}\to \Sht^{d}_{B}$. It is easy to check that this map is inverse to $p_{d}: \Sht^{d}_{B}\to \Sht^{d}_{G}$.
\end{proof}

\begin{lemma}\label{l:ShtB smooth} Let $d\in \calD$ be such that $d(i)>2g-2$ for all $i$. Then the morphism $q_{d}:\Sht^{d}_{B}\to \Sht^{d}_{H}$ is smooth of relative dimension $r/2$, and its geometric fibers are isomorphic to $[\Ga^{r/2}/Z]$ for some finite \'etale group scheme $Z$ acting on $\Ga^{r/2}$ via a homomorphism $Z\to \Ga^{r/2}$.
\end{lemma}
\begin{proof} We pick $\mu$ as in \S\ref{sss:mu} to realize $\Sht_{G}$ as the quotient $\Sht^{\mu}_{2}/\Pic_{X}(k)$, and $\Sht^{d}_{B}$ as the quotient $\Sht^{\mu,d}_{\tB}/\Pic_{X}(k)$.  

In the definition of Shtukas in \S\ref{sss:Sht n}, we may allow some coordinates $\mu_{i}$ of the modification type $\mu$ to be $0$, which means that the corresponding $f_{i}$ is an isomorphism. Therefore we may define $\Sht^{\mu}_{n}$ for more general $\mu\in\{0,\pm1\}^{r}$ such that $\sum\mu_{i}=0$.

We define the sequence $\mu'(d)=(\mu'_{1}(d),\cdots,\mu'_{r}(d))\in\{0,\pm1\}^{r}$ by
\begin{equation*}
\mu'_{i}(d):=\frac{1}{2}(\sgn(\mu_{i})+d(i)-d(i-1))
\end{equation*}
We also define $\mu''(d)=(\mu''_{1}(d),\cdots,\mu''_{r}(d))\in\{0,\pm1\}^{r}$ by
\begin{equation*}
\mu''_{i}(d):=\frac{1}{2}(\sgn(\mu_{i})-d(i)+d(i-1))=\sgn(\mu_{i})-\mu'_{i}(d)
\end{equation*}
We write $\mu'(d)$ and $\mu''(d)$ simply as $\mu'$ and $\mu''$.  Mapping the diagram \eqref{ShtB} to the rank one Shtuka $(\CL_{i};x_{i};\alpha'_{i}; \iota')$ defines a map $\Sht^{\mu,d}_{\tB}\to \Sht^{\mu'}_{1}$; similarly, sending the diagram \eqref{ShtB} to the rank one Shtuka $(\CM_{i};x_{i};\alpha''_{i}; \iota'')$ defines a map $\Sht^{\mu,d}_{\tB}\to \Sht^{\mu''}_{1}$. Combining the two maps we get
\begin{equation*}
\tilq_{d}: \Sht^{\mu,d}_{\wt{B}}\to \Sht^{\mu'}_{1}\times_{X^{r}}\Sht^{\mu''}_{1}.
\end{equation*}

Fix a pair $\calL_{\bu}:=(\calL_{i};x_{i}; \alpha'_{i};\iota')\in\Sht^{\mu'}_{1}(S)$ and $\CM_{\bu}:=(\calM_{i};x_{i};\alpha''_{i};\iota'')\in\Sht^{\mu''}_{1}(S)$.  Then the fiber of $q_{d}$ over $(\calL_{i}\otimes\calM_{i}^{-1};x_{i};\cdots)\in\Sht^{d}_{H}(S)$ is isomorphic to the fiber of $\wt{q}_{d}$ over $(\CL_{\bu},\CM_{\bu})$, the latter being the moduli stack $E_{\Sht}(\calM_{\bu},\calL_{\bu})$ (over $S$) of extensions of $\calM_{\bu}$ by $\calL_{\bu}$ as Shtukas. 

Since $\deg(\calL_{i})-\deg(\calM_{i})=d(i)>2g-2$, we have $\Ext^{1}(\calM_{i},\calL_{i})=0$. For each $i$, let $E(\calM_{i},\calL_{i})$ be the stack classifying extensions of $\calM_{i}$ by $\calL_{i}$. Then $E(\calM_{i},\calL_{i})$ is canonically isomorphic to the classifying space of the additive group $H_{i}:=\un\Hom(\calM_{i},\calL_{i})$ over $S$. For each $i=1,\cdots, r$, we have another moduli stack $C_{i}$ classifying commutative diagrams of extensions
\begin{equation*}
\xymatrix{0\ar[r] & \calL_{i-1}\ar[r]\ar@{-->}[d]^{\alpha'_{i}} & \calE_{i-1}\ar[r]\ar@{-->}[d]^{f_{i}} & \calM_{i-1}\ar[r]\ar@{-->}[d]^{\alpha''_{i}} & 0\\
0\ar[r] & \calL_{i}\ar[r] & \calE_{i}\ar[r] & \calM_{i}\ar[r] & 0}
\end{equation*}
Here the left and right columns are fixed. We have four cases:
\begin{enumerate}
\item When $(\mu'_{i},\mu''_{i})=(1,0)$, then $\alpha'_{i}:\CL_{i-1}\incl\CL_{i}$ with colength one and $\alpha''_{i}$ is an isomorphism. In this case,  the bottom row is the pushout of the top row along $\alpha'_{i}$, hence determined by the top row. Therefore $C_{i}=E(\calM_{i-1},\calL_{i-1})$ in this case. 
\item When $(\mu'_{i},\mu''_{i})=(-1,0)$, then  $\alpha'^{-1}_{i}:\CL_{i}\incl \CL_{i-1}$ with colength one and $\alpha''_{i}$ is an isomorphism.  In this case,  the top row is the pushout of the bottom row along $\alpha'^{-1}_{i}$, hence determined by the bottom row. Therefore $C_{i}=E(\calM_{i},\calL_{i})$ in this case. 
\item When $(\mu'_{i},\mu''_{i})=(0,1)$, then  $\alpha'_{i}$ is an isomorphism and $\alpha''_{i}:\CM_{i-1}\incl\CM_{i}$ with colength one.  In this case, the top row is the pullback of the bottom row along $\alpha''_{i}$, hence determined by the bottom row. Therefore $C_{i}=E(\calM_{i},\calL_{i})$ in this case.
\item When $(\mu'_{i},\mu''_{i})=(0,-1)$, then $\alpha'_{i}$ is an isomorphism and $\alpha''^{-1}_{i}:\CM_{i}\incl\CM_{i-1}$ with colength one.  In this case,  the bottom row is the pullback of the top row along $\alpha''^{-1}_{i}$, hence determined by the top row. Therefore $C_{i}=E(\calM_{i-1},\calL_{i-1})$ in this case. 
\end{enumerate}
From the combinatorics of $\mu'$ and $\mu''$ we see that the cases (1)(4) and (2)(3) each appear $r/2$ times. In all cases, we view $C_{i}$ as a correspondence
\begin{equation*}
E(\calM_{i-1},\calL_{i-1})\leftarrow C_{i} \rightarrow E(\calM_{i},\calL_{i})
\end{equation*}
then $C_{i}$ is the graph of a natural map $E(\calM_{i-1},\calL_{i-1})\to E(\calM_{i},\calL_{i})$ in cases (1) and (4) and the graph of a natural map $E(\calM_{i},\calL_{i})\to E(\calM_{i-1},\calL_{i-1})$ in cases (2) and (3). We see that $C_{i}$ is canonically the classifying space of an additive group scheme $\Om_{i}$ over $S$, which is either $H_{i-1}$ in cases (1) and (4) or $H_{i}$ in cases (2) and (3).

Consider the composition of these correspondences
\begin{equation*}
C(\calM_{\bu},\calL_{\bu}):=C_{1}\times_{E(\calM_{1},\calL_{1})}C_{2}\times_{E(\calM_{2},\calL_{2})}\cdots\times_{E(\calM_{r-1},\calL_{r-1})}C_{r}.
\end{equation*}
This is viewed as a correspondence
\begin{equation*}
E(\calM_{0},\calL_{0})\leftarrow C(\calM_{\bu},\calL_{\bu})\rightarrow E(\calM_{r},\calL_{r})\cong E(\leftexp{\tau}{\calM}_{0},\leftexp{\tau}{\calL}_{0}).
\end{equation*}

To compute $C(\calM_{\bu},\calL_{\bu})$ more explicitly, we consider the following situation. Let $\CG$ be a group scheme over $S$ with two subgroup schemes $\CG_{1}$ and $\CG_{2}$. Then we have a canonical isomorphism of stacks over $S$
\begin{equation*}
\BB(\CG_{1})\times_{\BB(\CG)}\BB(\CG_{2})\cong \CG_{1}\bs\CG/\CG_{2}.
\end{equation*}
Using this fact repeatedly, and using that $E(\calM_{i},\calL_{i})=\BB(H_{i})$ and $C_{i}=\BB(\Om_{i})$, we see that
\begin{equation}
C(\calM_{\bu},\calL_{\bu})\cong \Om_{1}\bs H_{1}\twtimes{\Om_{2}} H_{2}\twtimes{\Om_{3}}\cdots\twtimes{\Om_{r-1}}H_{r-1}/\Om_{r}.
\end{equation}
where $H_{i-1}\twtimes{\Om_{i}}H_{i}$ means dividing by the diagonal action of $\Om_{i}$ on both $H_{i-1}$ and $H_{i}$ by translations.  Let
\begin{equation*}
A(\calM_{\bu},\calL_{\bu}):=H_{0}\twtimes{\Om_{1}}H_{1}\twtimes{\Om_{2}} \cdots\twtimes{\Om_{r-1}}H_{r-1}\twtimes{\Om_{r}}H_{r}
\end{equation*}
Since $\Om_{i}$ is always the smaller of $H_{i-1}$ and $H_{i}$, $A(\calM_{\bu},\calL_{\bu})$ is an additive group scheme over $S$. Then we have
\begin{equation}\label{C HAH}
C(\calM_{\bu},\calL_{\bu})\cong H_{0}\bs A(\calM_{\bu},\calL_{\bu})/H_{r}.
\end{equation}

Note that $H_{r}\cong\leftexp{\tau}H_{0}$ is the pullback of $H_{0}$ via $\Fr_{S}$. We have a relative Frobenius map over $S$
\begin{equation*}
\Fr_{/S}: E(\calM_{0},\calL_{0})=\BB(H_{0})\xrightarrow{\Fr_{H_{0}/S}}\BB(H_{r})=E(\calM_{r},\calL_{r}).
\end{equation*}
By the moduli meaning of $E_{\Sht}(\CM_{\bu},\CL_{\bu})$, we have a Cartesian diagram of stacks
\begin{equation*}
\xymatrix{E_{\Sht}(\calM_{\bu},\calL_{\bu})\ar[rr]\ar[d] && C(\calM_{\bu},\calL_{\bu})\ar[d] \\
E(\calM_{0},\calL_{0})\ar[rr]^{(\id, \Fr_{/S})} & & E(\calM_{0},\calL_{0})\times E(\calM_{r},\calL_{r})}
\end{equation*}
Using the isomorphism \eqref{C HAH}, the above diagram becomes
\begin{equation}\label{ESht}
\xymatrix{E_{\Sht}(\calM_{\bu},\calL_{\bu})\ar[rr]\ar[d] && H_{0}\bs A(\calM_{\bu},\calL_{\bu})/H_{r}\ar[d] \\
\BB(H_{0})\ar[rr]^{(\id, \Fr_{H_{0}/S})} & & \BB(H)\times \BB(H_{r})}
\end{equation}
This implies that
\begin{equation}\label{E quot A}
E_{\Sht}(\calM_{\bu},\calL_{\bu})\cong [A(\calM_{\bu},\calL_{\bu})/_{(\id,\Fr_{H_{0}/S})}H_{0}]
\end{equation}
where $H_{0}$ acts on  $A(\calM_{\bu},\calL_{\bu})$ via the embedding $(\id,\Fr_{H_{0}/S}):H_{0}\to H_{0}\times H_{r}$ and the natural action of $H_{0}\times H_{r}$ on $A(\calM_{\bu},\calL_{\bu})$. Since $A$ is an additive group scheme over $S$, hence smooth over $S$, the isomorphism \eqref{E quot A} shows that $E_{\Sht}(\calM_{\bu},\calL_{\bu})$ is smooth over $S$.

To compute the dimension of $A(\calM_{\bu},\calL_{\bu})$, we compare $\dim \Om_{i}$ with $\dim H_{i}$. We have $\dim H_{i}-\dim \Om_{i}=1$ in cases (1) and (4) and  $\dim H_{i}-\dim \Om_{i}=0$ in cases (2) and (3). Since (1)(4) and (2)(3) each appear $r/2$ times, we have
\begin{equation*}
\dim A(\calM_{\bu},\calL_{\bu})=\dim H_{0}+\sum_{i=1}^{r}(\dim H_{i}-\dim\Om_{i})=\dim H_{0}+r/2.
\end{equation*}
This implies $E_{\Sht}(\calM_{\bu},\calL_{\bu})$ is smooth of dimension $r/2$.

When $S$ is a geometric point $\Spec(K)$, $H_{0}$ and $H_{r}$ can be viewed as subspaces of the $K$-vector space $A:=A(\calM_{\bu},\calL_{\bu})$, and $\phi=\Fr_{H_{0}/K}: H_{0}\to H_{r}$ is a morphism of group schemes over $K$. Choose a $K$-subspace $L\subset A$  complement to $H_{0}$, then $L\cong\Ga^{r/2}$ as a group scheme over $K$. Consider the homomorphism
\begin{equation*}
\alpha: H_{0}\times L\to A
\end{equation*}
given by $(x,y)\mapsto x+y+\phi(x)$. By computing the tangent map of $\alpha$ at the origin, we see that $\alpha$ is \'etale, therefore $Z=\ker(\alpha)$ is a finite \'etale group scheme over $K$. We conclude that in this case the fiber of $q_{d}$ over $S=\Spec(K)$ is
\begin{equation*}
E_{\Sht}(\calM_{\bu},\calL_{\bu})\cong [A/_{(\id,\phi)}H_{0}]\cong [L/Z]\cong [\Ga^{r/2}/Z].
\end{equation*}
\end{proof}

\begin{cor}\label{c:Eis loc sys} Suppose $d\in\calD$ satisfies $d(i)>2g-2$ for all $i$,  then the cone of the map $\bR\pi^{<d}_{G,!}\Ql\to\bR\pi^{\leq d}_{G,!}\Ql$ is isomorphic to $\pi^{d}_{H,!}\Ql[-r](-r/2)$, which is a local system concentrated in degree $r$.
\end{cor}
\begin{proof}The cone of $\bR\pi^{<d}_{G,!}\Ql\to\bR\pi^{\leq d}_{G,!}\Ql$ is isomorphic to $\bR\pi^{d}_{G,!}\Ql$, where $\pi^{d}_{G}:\Sht^{d}_{G}\to X^{r}$. By Lemma \ref{l:ShtBG}, for $d\in\calD^{+}$, we have $\bR\pi^{d}_{G,!}\Ql\cong\bR\pi^{d}_{B,!}\Ql$. By Lemma \ref{l:ShtB smooth}, $q_{d}$ is smooth of relative dimension $r/2$, the relative fundamental cycles gives  $\bR q_{d,!}\Ql\to \bR^{r}q_{d,!}\Ql[-r]\to \Ql[-r](-r/2)$, which is an isomorphism by checking the stalks (using the description of the geometric fibers of $q_{d}$ given in Lemma \ref{l:ShtB smooth}). Therefore $\bR \pi^{d}_{B,!}\Ql\cong\bR \pi^{d}_{H,!}\Ql[-r](-r/2)$. Finally, $\pi^{d}_{H}:\Sht^{d}_{H}\to X^{r}$ is a $\Pic^{0}_{X}(k)$-torsor by an argument similar to Lemma \ref{l:ShtT proper}.  Therefore $\bR \pi^{d}_{H,!}\Ql$ is a local system on $X^{r}$, and $\bR\pi^{d}_{G,!}\Ql\cong\pi^{d}_{H,!}\Ql[-r](-r/2)$ is a local system shifted to degree $r$.
\end{proof}

\subsection{Horocycles in the generic fiber}\label{ss:horo gen}
Fix a geometric generic point $\gen$ of $X^{r}$. For a stack $\frX$ over $X^{r}$, we denote its fiber over $\gen$ by $\frX_{\gen}$. Next we study the cycles in $\Sht_{G,\gen}$ given by images of $\Sht^{d}_{B,\gen}$.

\begin{lemma}[Drinfeld {\cite[Prop 4.2]{D87}} for $r=2$; Varshavsky {\cite[Prop 5.7]{Va}} in general]\label{l:pgen finite} For each $d\in\calD$, the map $p_{d,\gen}: \Sht^{d}_{B,\gen}\to \Sht_{G,\gen}$ is finite and unramified.
\end{lemma}

\subsubsection{The cohomological constant term} Taking the geometric generic fiber of the diagram \eqref{ShtB as corr}, we view $\Sht^{d}_{B,\gen}$ as a correspondence between $\Sht_{G,\gen}$ and $\Sht^{d}_{H,\gen}$. The fundamental cycle of $\Sht^{d}_{B,\gen}$ (of dimension $r/2$) gives a cohomological correspondence between the constant sheaf on $\Sht_{G,\gen}$ and the shifted constant sheaf $\Ql[-r](-r/2)$ on $\Sht^{d}_{H,\gen}$. Therefore $[\Sht^{d}_{B,\gen}]$ induces a map
\begin{equation}\label{coho const term}
\gamma_{d}: \cohoc{r}{\Sht_{G,\gen}}(r/2)\xrightarrow{p^{*}_{d,\gen}} \cohoc{r}{\Sht^{d}_{B,\gen}}(r/2)\xrightarrow{[\Sht^{d}_{B,\gen}]} \homog{0}{\Sht^{d}_{B,\gen}}\xrightarrow{q_{d,\gen,!}}\homog{0}{\Sht^{d}_{H,\gen}}. 	\index{$\gamma_{d}, \gamma_{\delta}$}%
\end{equation}
Here we are implicitly using Lemma \ref{l:pgen finite} to conclude that $p_{d,\gen}$ is proper, hence $p^{*}_{d,\gen}$ induces a map between compactly supported cohomology groups.

Taking the product of $\gamma_{d}$ for all $d$ in a fixed $\ZZ$-orbit $\delta\in \calD/\ZZ$, using the decomposition \eqref{Sht into components}, we get a map
\begin{equation}\label{c horocycle gen}
\gamma_{\delta}: \cohoc{r}{\Sht_{G,\gen}}(r/2) \to \prod_{d\in\delta}\homog{0}{\Sht^{d}_{H,\gen}}\cong\cohog{0}{\Sht^{\mu(\delta)}_{1,\gen}}.
\end{equation}
When $r=0$,   \eqref{c horocycle gen}  is exactly the constant term map for automorphic forms. Therefore we may call $\gamma_{\delta}$ the {\em cohomological constant term map}.

The RHS of \eqref{c horocycle gen} carries an action of the Hecke algebra $\sH_{H}=\otimes_{x\in|X|}\QQ[t_{x},t_{x}^{-1}]$. In fact, $\Sht^{\mu(\delta)}_{1,\gen}$ is a $\Pic_{X}(k)$-torsor over $\Spec k(\gen)$. The action of $\sH_{H}$ on $\Sht^{\mu(\delta)}_{H,\gen}$ is via the natural map $\sH_{H}\cong\QQ[\Div(X)]\to\QQ[\Pic_{X}(k)]$.	\index{$\sH_H$}%

\begin{lemma}\label{l:horo action} The map $\gamma_{\delta}$ in \eqref{c horocycle gen} intertwines the $\sH$-action on the LHS and the $\sH_{H}$-action on the RHS via the Satake transform $\Sat: \sH\incl\sH_{H}$.
\end{lemma}
\begin{proof}
Since $\sH$ is generated by $\{h_{x}\}_{x\in |X|}$ as a $\QQ$-algebra, it suffices to show that for any $x\in|X|$, the following diagram is commutative
\begin{equation}\label{hx tx}
\xymatrix{\cohoc{r}{\Sht_{G,\gen}}\ar[d]^{C(h_{x})}\ar[rr]^{\gamma_{\delta}} && \prod_{d\in\delta} \homog{0}{\Sht^{d}_{H,\gen}}\ar[d]^{t_{x}+q_{x}t^{-1}_{x}}\\
\cohoc{r}{\Sht_{G,\gen}}\ar[rr]^{ \gamma_{\delta}} && 
\prod_{d\in\delta}\homog{0}{\Sht^{d}_{H,\gen}}}
\end{equation}

Let $U=X-\{x\}$. For a stack $\frX$ over $X^{r}$, we use $\frX_{U^{r}}$ to denote its restriction to $U^{r}$. Similar notation applies to morphisms over $X^{r}$.

Recall that $\Sht_{G, U^{r}}(S)$ classifies $(\CE_{i};x_{i};f_{i};\iota)$ such that $x_{i}$ are disjoint from  $x$. Hence the composition $\iota\circ f_{r}\cdots f_{1}:\CE_{0}\dashrightarrow\leftexp{\tau}\CE_{0}$ is an isomorphism near $x$. In particular, the fiber $\CE_{0,x}=\CE_{0}|_{S\times\{x\}}$ carries a Frobenius structure $\CE_{0,x}\cong\leftexp{\tau}\CE_{0,x}$, hence $\CE_{0,x}$ descends to a two-dimensional vector space over $\Spec k_{x}$ ($k_{x}$ is the residue field of $X$ at $x$) up to tensoring with a line.  In other words, there is a morphism $\om_{x}: \Sht_{G, U^{r}}\to \BB(G(k_{x}))$ sending $(x_{i};\CE_{i};f_{i};\iota)$ to the descent of $\CE_{0,x}$ to $\Spec k_{x}$. In the following we shall understand that $\CE_{0,x}$ is a 2-dimensional vector space over $k_{x}$, up to tensoring with a line over $k_{x}$.

The correspondence $\Sht_{G}(h_{x})_{U^{r}}$ classifies diagrams of the form \eqref{Sht hD} where the vertical maps have divisor $x$. Therefore, if the first row in \eqref{Sht hD} is fixed, the bottom row is determined by $\CE'_{0}$, which in turn is determined by the line $e_{x}=\ker(\CE_{0,x}\to\CE'_{0,x})$ over $k_{x}$.  Recall that $\oll{p}$ and $\orr{p}:\Sht_{G}(h_{x})\to \Sht_{G}$ are the projections sending \eqref{Sht hD} to the top and  bottom row respectively. Then we have a Cartesian diagram
\begin{equation*}
\xymatrix{\Sht_{G}(h_{x})_{U^{r}}\ar[r]\ar[d]^{\oll{p}_{U^{r}}} & \BB(B(k_{x}))\ar[d]\\
\Sht_{G, U^{r}}\ar[r]^{\om_{x}} & \BB(G(k_{x}))}
\end{equation*}
where $B\subset G$ is a Borel subgroup.   We have a similar Cartesian diagram where $\oll{p}_{U^{r}}$ is replaced with $\orr{p}_{U^{r}}$. In particular, $\oll{p}_{U^{r}}$ and $\orr{p}_{U^{r}}$ are finite \'etale of degree $q_{x}+1$.

Let $\Sht^{d}_{B}(h_{x})$ be the base change of $\oll{p}$ along $\Sht^{d}_{B}\to \Sht_{G}$. Let $\oll{p}_{B}: \Sht^{d}_{B}(h_{x})_{U^{r}}\to \Sht^{d}_{B,U^{r}}$ be the base-changed map restricted to $U^{r}$. A point  $(\CL_{i}\incl \CE_{i}; x_{i};f_{i};\iota)\in \Sht^{d}_{B}$ gives another line $\ell_{x}:=\CL_{0,x}\subset \CE_{0,x}$. Therefore, for a point $(\CL_{i}\incl\CE_{i}\to \CE'_{i};\cdots)\in\Sht^{d}_{B}(h_{x})|_{U^{r}}$, we get two lines $\ell_{x}$ and $e_{x}$ inside $\CE_{0,x}$. In other words we have a morphism
\begin{equation*}
\om: \Sht^{d}_{B}(h_{x})_{U^{r}}\to \BB(B(k_{x}))\times_{\BB(G(k_{x}))}\BB(B(k_{x}))=B(k_{x})\bs G(k_{x})/B(k_{x})
\end{equation*}
This allows us to decompose $\Sht^{d}_{B}(h_{x})_{U^{r}}$ into the disjoint union of two parts
\begin{equation*}
\Sht^{d}_{B}(h_{x})_{U^{r}}=C_{1}\coprod C_{2}
\end{equation*}
where $C_{1}$ is the preimage of the unit coset $B(k_{x})\bs B(k_{x})/B(k_{x})$ and $C_{2}$ is the preimage of the complement.

For a point $(\CL_{i}\incl\CE_{i}\incl\CE'_{i};\cdots)\in C_{1}$, $\CE'_{i}$ is determined by $e_{x}=\ell_{x}=\CL_{0,x}$. Therefore the map $\oll{p_{B,1}}:=\oll{p}_{B}|_{C_{1}}: C_{1}\to \Sht^{d}_{B,U^{r}}$ is an isomorphism. In this case, $\CE'_{i}$ is obtained via the pushout of $\CL_{i}\to \CE_{i}$ along $\CL_{i}\incl\CL_{i}(x)$. This way we get an exact sequence $0\to \CL_{i}(x)\to\CE_{i}(x)\to \CM_{i}\to0$ where $\CM_{i}=\CE_{i}/\CL_{i}$. We define a map $\orr{p_{B,1}}: C_{1}\to \Sht^{d+d_{x}}_{B,U^{r}}$ sending $(\CL_{i}\incl\CE_{i}\incl\CE'_{i};\cdots)\in C_{1}$ to $(\CL_{i}(x)\incl\CE'_{i};\cdots)$. Since $\oll{p_{B,1}}$ is an isomorphism, $C_{1}$ viewed as a correspondence between $\Sht^{d}_{B,U^{r}}$ and $\Sht^{d+d_{x}}_{B,U^{r}}$ can be identified with the graph of the map $\varphi_{x}:=\orr{p_{B,1}}\circ\, \oll{p_{B,1}}^{-1}: \Sht^{d}_{B,U^{r}}\to\Sht^{d+d_{x}}_{B,U^{r}}$. Note that $\varphi_{x}$ is a finite \'etale map of degree $q_{x}$.  We have a commutative diagram
\begin{equation*}
\xymatrix{\Sht^{d}_{H,U^{r}} & \Gamma(t_{x})\ar[l]_{\id}\ar[r]^{t_{x}} & \Sht^{d+d_{x}}_{H,U^{r}}\\
\Sht^{d}_{B,U^{r}}\ar[d]^{p_{d}}\ar[u]_{q_{d}} & C_{1}=\Gamma(\varphi_{x})\ar[u]\ar[d]\ar[l]_{\oll{p_{B,1}}}^{\sim}\ar[r]^{\orr{p_{B,1}}} &\Sht^{d+d_{x}}_{B,U^{r}}\ar[d]^{p_{d+d_{x}}}\ar[u]_{q_{d+d_{x}}}\\
\Sht_{G,U^{r}} & \Sht_{G}(h_{x})_{U^{r}}\ar[l]_{\oll{p}}\ar[r]^{\orr{p}} & \Sht_{G,U^{r}}}
\end{equation*} 
Here $\Gamma(t_{x})$ is the graph of the isomorphism $\Sht^{d}_{H,U^{r}}\to \Sht^{d+d_{x}}_{H,U^{r}}$ given by tensoring the line bundles with $\calO(x)$.  Therefore the action of $[C_{1}]$ on the compactly supported cohomology of the generic fiber of $\Sht^{d}_{B}$ fits into a commutative diagram
\begin{equation}\label{action C1}
\xymatrix{\cohoc{r}{\Sht^{d}_{B,\gen}}(r/2)\ar[d]^{[C_{1}]}\ar[rr]^{[\Sht^{d}_{B,\gen}]} && \homog{0}{\Sht^{d}_{B,\gen}}\ar[d]^{\varphi_{x,*}}\ar[r] & \homog{0}{\Sht^{d}_{H,\gen}}\ar[d]^{t_{x}}\\
\cohoc{r}{\Sht^{d+d_{x}}_{B,\gen}}(r/2)\ar[rr]^{[\Sht^{d+d_{x}}_{B,\gen}]} && \homog{0}{\Sht^{d+d_{x}}_{B,\gen}}\ar[r] & \homog{0}{\Sht^{d+d_{x}}_{H,\gen}}}
\end{equation}

Similarly for $C_{2}$, we define a morphism $\orr{p_{B,2}}: C_{2}\to \Sht^{d-d_{x}}_{B,U^{r}}$ sending $(\CL_{i}\incl\CE_{i}\incl\CE'_{i};\cdots)\in C_{2}$ to $(\CL_{i}\incl\CE'_{i};\cdots)$. Then $\orr{p_{B,2}}$ is an isomorphism while $\oll{p}_{B,2}=\oll{p}_{B}|_{C_{2}}$ is finite \'etale of degree $q_{x}$. Therefore $C_{2}$ viewed as a correspondence between $\Sht^{d}_{B,U^{r}}$ and $\Sht^{d-d_{x}}_{B,U^{r}}$ can be identified with the {\em transpose} of the graph of the map $\varphi_{x}:\Sht^{d-d_{x}}_{B,U^{r}}\to \Sht^{d}_{B,U^{r}}$ defined previously. We also have a commutative diagram
\begin{equation*}
\xymatrix{\Sht^{d}_{H,U^{r}} & \Gamma(t^{-1}_{x})\ar[l]_{\id}\ar[r]^{t^{-1}_{x}} & \Sht^{d-d_{x}}_{H,U^{r}}\\
\Sht^{d}_{B,U^{r}}\ar[d]^{p_{d}}\ar[u]_{q_{d}} & C_{2}=\leftexp{t}\Gamma(\varphi_{x})\ar[u]\ar[d]\ar[l]_{\oll{p_{B,2}}}\ar[r]^{\orr{p_{B,2}}}_{\sim} &\Sht^{d-d_{x}}_{B,U^{r}}\ar[d]^{p_{d-d_{x}}}\ar[u]_{q_{d-d_{x}}}\\
\Sht_{G,U^{r}} & \Sht_{G}(h_{x})_{U^{r}}\ar[l]_{\oll{p}}\ar[r]^{\orr{p}} & \Sht_{G,U^{r}}}
\end{equation*} 
The action of $[C_{2}]$ on the compactly supported cohomology of the generic fibers of $\Sht^{d}_{B}$ fits into a commutative diagram
\begin{equation}\label{action C2}
\xymatrix{\cohoc{r}{\Sht^{d}_{B,\gen}}(r/2)\ar[d]^{[C_{2}]=\varphi^{*}_{x}}\ar[rr]^{[\Sht^{d}_{B,\gen}]} && \homog{0}{\Sht^{d}_{B,\gen}}\ar[r] & \homog{0}{\Sht^{d}_{H,\gen}}\\
\cohoc{r}{\Sht^{d-d_{x}}_{B,\gen}}(r/2)\ar[rr]^{[\Sht^{d-d_{x}}_{B,\gen}]} && \homog{0}{\Sht^{d-d_{x}}_{B,\gen}}\ar[r]\ar[u]_{q_{x}^{-1}\varphi_{x,*}} & \homog{0}{\Sht^{d-d_{x}}_{H,\gen}}\ar[u]_{q^{-1}_{x}t_{x}}}
\end{equation}
The appearance of $q_{x}$ in the above diagram is because the degree of $\varphi_{x}$ is $q_{x}$. Combining \eqref{action C1} and \eqref{action C2} we get a commutative diagram
\begin{equation}\label{C1C2}
\xymatrix{\prod_{d\in\delta}\cohoc{r}{\Sht^{d}_{B,\gen}}(r/2)\ar[r]\ar[d]^{[C_{1}]+[C_{2}]} & \prod_{d\in\delta}\homog{0}{\Sht^{d}_{H,\gen}}\ar[d]^{t_{x}+q_{x}t^{-1}_{x}}\\
\prod_{d\in\delta}\cohoc{r}{\Sht^{d}_{B,\gen}}(r/2)\ar[r] & \prod_{d\in\delta}\homog{0}{\Sht^{d}_{H,\gen}}}
\end{equation}

Finally, let $\orr{p_{B}}: \Sht^{d}_{B}(h_{x})_{U^{r}}\to \Sht^{d}_{B,U^{r}}$ be $\orr{p_{B,1}}$ on $C_{1}$ and $\orr{p_{B,2}}$ on $C_{2}$. Consider the commutative diagram 
\begin{equation*}
\xymatrix{\coprod_{d\in\delta}\Sht^{d}_{B}(h_{x})_{U^{r}}\ar[d]^{(p_{d})_{d\in\delta}}\ar[r]^{\orr{p_{B}}} & \coprod_{d\in\delta}\Sht^{d}_{B,U^{r}}\ar[d]^{(p_{d})_{d\in\delta}}\\
\Sht_{G}(h_{x})_{U^{r}}\ar[r]^{\orr{p}} & \Sht_{G,U^{r}}}
\end{equation*}
Since $\orr{p_{B}}$ and $\orr{p}$ are both finite \'etale of degree $q_{x}+1$, by examining geometric fibers we conclude that the above diagram is Cartesian. The similar diagram with $\orr{p_{B}}$ and $\orr{p}$ replaced with  $\oll{p_{B}}$ and $\oll{p}$ is Cartesian by definition. From these facts we get a commutative diagram
\begin{equation*}
\xymatrix{\cohoc{r}{\Sht_{G,\gen}}\ar[d]^{C(h_{x})}\ar[rr]^{(p^{*}_{d})_{d\in\delta}} && \prod_{d\in\delta} \cohoc{r}{\Sht^{d}_{B,\gen}}\ar[d]^{[C_{1}]+[C_{2}]}\\
\cohoc{r}{\Sht_{G,\gen}}\ar[rr]^{ (p^{*}_{d})_{d\in\delta}} && 
\prod_{d\in\delta}\cohoc{r}{\Sht^{d}_{B,\gen}}}
\end{equation*}
Combining this with \eqref{C1C2} we obtain \eqref{hx tx}, as desired.
\end{proof}

\subsection{Finiteness} 

For fixed $d\in\calD$, the Leray spectral sequence associated with  the map $\pi^{\leq d}_{G}$ gives an increasing filtration $L_{\leq i}\cohoc{2r}{\Sht^{\leq d}_{G}\otimes_{k}\kbar}$ on $\cohoc{2r}{\Sht^{\leq d}_{G}\otimes_{k}\kbar}$, with $L_{\leq i}\cohoc{2r}{\Sht^{\leq d}_{G}\otimes_{k}\kbar}$ being the image of $\cohog{2r}{X^{r}\otimes_{k}\kbar, \tau_{\leq i}\bR\pi^{\leq d}_{G,!}\Ql}\to \cohog{2r}{X^{r}\otimes_{k}\kbar, \bR\pi^{\leq d}_{G,!}\Ql}\cong\cohoc{2r}{\Sht^{\leq d}_{G}\otimes_{k}\kbar}$. Here $\tau_{\leq i}$ means the truncation in the usual $t$-structure of $D^{b}_{c}(X^{r},\Ql)$. 	\index{$\tau_{\leq i}$}%
Let $L_{\leq i}V$ be the inductive limit $\varinjlim_{d\in\calD}L_{\leq i}\cohoc{2r}{\Sht^{\leq d}_{G}\otimes_{k}\kbar}(r)$, which is a subspace of $V$. This way we get a filtration on $V$
\begin{equation*}
0\subset L_{\leq 0}V\subset L_{\leq 1}V\subset\cdots\subset L_{\leq 2r}V=V.		\index{$L_{\leq i}V$}%
\end{equation*}

\begin{lemma}\label{l:L stable under Hk}
Each $L_{\leq i}V$ is stable under the action of $\sH$. 
\end{lemma}
\begin{proof}
The map $C(h_{D})_{d,d'}$ in \eqref{Hk d d'} induces $\tau_{\leq i}C(h_{D})_{d,d'}: \tau_{\leq i}\bR\pi^{\leq d}_{G,!}\Ql\to \tau_{\leq i}\bR\pi^{\leq d'}_{G,!}\Ql$. By the construction of $C(h_{D})$ we have a commutative diagram
\begin{equation*}
\xymatrix{\varinjlim_{d}\cohog{2r}{X^{r}\otimes_{k}\kbar, \tau_{\leq i}\bR\pi^{\leq d}_{G,!}\Ql}\ar[rr]^{\varinjlim\tau_{\leq i}C(h_{D})_{d,d'}}\ar[d] && \varinjlim_{d'}\cohog{2r}{X^{r}\otimes_{k}\kbar, \tau_{\leq i}\bR\pi^{\leq d'}_{G,!}\Ql}\ar[d]\\
\cohoc{2r}{\Sht_{G}\otimes_{k}\kbar}\ar[rr]^{C(h_{D})} && \cohoc{2r}{\Sht_{G}\otimes_{k}\kbar}}
\end{equation*}
The image of the vertical maps are both $L_{\leq i}V$ up to a Tate twist, therefore $L_{\leq i}V$ is stable under $C(h_{D})$. When $D$ runs over all effective divisors on $X$, $C(h_{D})$ span $\sH$, hence $L_{\leq i}V$ is stable under $\sH$.
\end{proof}

\begin{lemma}\label{l:GrL}
For $i\neq r$, $\Gr^{L}_{i}V:=L_{\leq i}V/L_{\leq i-1}V$ is finite-dimensional over $\Ql$.	\index{$\Gr^{L}_{i}V$}%
\end{lemma}
\begin{proof} We say $d\in\calD$ is large if $d(i)>2g-2$ for all $i$.	 In the following argument it is convenient to choose a total order on $\calD$ that extends its partial order. Under the total order, $\Sht^{<d}_{G}=\coprod_{d'<d}\Sht^{d'}_{G}$ and $\Sht^{\leq d}_{G}=\coprod_{d'\leq d}\Sht^{d}_{G}$ are different from their original meanings, and we will use the new notion during the proof.

By Corollary \ref{c:Eis loc sys}, the inductive system $\tau_{\leq r-1}\bR\pi^{\leq d}_{G,!}\Ql$ stabilizes for $d$ large. Hence so does $L_{\leq r-1}\cohoc{2r}{\Sht^{\leq d}_{G}\otimes_{k}\kbar}$. Therefore $L_{\leq r-1}V$ is finite dimensional. 

It remains to show that $V/L_{\leq r}V$ is finite dimensional. Again by Corollary \ref{c:Eis loc sys}, for $d$ large, the map $\bR^{r+1}\pi^{<d}_{G,!}\Ql\to \bR^{r+1}\pi^{\leq d}_{G,!}\Ql$ is surjective because the next term in the long exact sequence is $\bR^{r+1}\pi^{d}_{G,!}\Ql=0$. This implies that the inductive system $\bR^{r+1}\pi^{\leq d}_{G,!}\Ql$ is eventually stable because any chain of surjections $\CF_{1}\surj \CF_{2}\surj\cdots$ of constructible sheaves on $X^{r}$ has to stabilize (i.e., constructible $\Ql$-sheaves satisfy the ascending chain condition). Also by Corollary \ref{c:Eis loc sys}, the inductive system $\tau_{>r+1}\bR\pi^{\leq d}_{G,!}\Ql$ is stable. Combined with the stability of $\bR^{r+1}\pi^{\leq d}_{G,!}\Ql$, we see that the system $\tau_{>r}\bR\pi^{\leq d}_{G,!}\Ql$ is stable. In other words, there exists a large $d_{0}\in\calD$ such that for any $d\geq d_{0}$, the natural map $\tau_{>r}\bR\pi^{<d}_{G,!}\Ql\to\tau_{>r}\bR\pi^{\leq d}_{G,!}\Ql$ is an isomorphism.

We abbreviate $\cohoc{2r}{\Sht^{<d}_{G}\otimes_{k}\kbar}$ by $H_{<d}$ and $\cohoc{2r}{\Sht^{\leq d}_{G}\otimes_{k}\kbar}$ by $H_{\leq d}$. For $d\geq d_{0}$, the distinguished triangle of functors $\tau_{\leq r}\to \id\to \tau_{>r}\to$ applied to $\bR\pi^{< d}_{G,!}\Ql$ and $\bR\pi^{\leq d}_{G,!}\Ql$ gives a morphism of exact sequences
\begin{equation*}
\xymatrix{ L_{\leq r} H_{<d}\ar[r]\ar[d] &  H_{<d}\ar[r]\ar[d] & \cohog{2r}{X^{r}\otimes_{k}\kbar, \tau_{>r}\bR\pi^{< d}_{G,!}\Ql}\ar[r]\ar[d]^{\wr} & \cdots\\
L_{\leq r} H_{\leq d}\ar[r] & H_{\leq d} \ar[r] & \cohog{2r}{X^{r}\otimes_{k}\kbar, \tau_{>r}\bR\pi^{\leq d}_{G,!}\Ql}\ar[r] & \cdots}
\end{equation*}
Therefore the inductive system $H_{\leq d}/L_{\leq r}H_{\leq d}$ is a subsystem of $\cohog{2r}{X^{r}\otimes_{k}\kbar, \tau_{>r}\bR\pi^{\leq d}_{G,!}\Ql}$ which is stable with finite-dimensional inductive limit. Hence the inductive system $H_{\leq d}/L_{\leq r}H_{\leq d}$ is itself stable with finite-dimensional inductive limit.
Taking inductive limit on $d$, using that $V=\varinjlim_{d}H_{\leq d}(r)$ and $L_{\leq i}V=\varinjlim_{d}L_{\leq i}H_{\leq d}(r)$, we see that $V/L_{\leq r}V\cong \varinjlim_{d}H_{\leq d}(r)/L_{\leq r}H_{\leq d}(r)$ is finite dimensional.
\end{proof}

\begin{lemma}\label{l:Hcusp image}
The space $\calI_{\Eis}\cdot (L_{\leq r}V)$ is finite-dimensional over $\Ql$.
\end{lemma}
\begin{proof} Let $U\subset \Sht_{G}$ be the union of those $\Sht^{\leq d}_{G}$ for $d\in\calD$ such that $\min_{i\in\ZZ/r\ZZ}\{d(i)\}\leq 2g-2$. Since $\inst(\CE)$ has an absolute lower bound, there are only finitely many such $d$ with $\Sht^{d}_{G}\neq\varnothing$,  hence $U$ is an open substack of finite type. Let $\pi^{U}_{G}:U\to X^{r}$ be the restriction of $\pi_{G}$. For $f\in\sH$, and any $d\in\calD$, its action defines a map $C(f)_{d,d'}:\bR^{i}\pi^{\leq d}_{G,!}\Ql\to\bR^{i}\pi^{\leq d'}_{G,!}\Ql$ for sufficiently large $d'$. We may assume $d'>2g-2$ , which means $d'(j)>2g-2$ for all $j$. We shall show that when $f\in\calI_{\Eis}$ and $i\leq r$, the image of $C(f)_{d,d'}$ is contained in the image of the map $\iota_{U,d'}: \bR^{i}\pi^{U}_{G,!}\Ql\to \bR^{i}\pi^{\leq d'}_{G,!}\Ql$ induced by the inclusion $U\subset \Sht^{\leq d'}_{G}$, which implies the proposition. By Corollary \ref{c:Eis loc sys}, either $\iota_{U,d'}$ is an isomorphism (if $i<r$) or when $i=r$, the cokernel of $\iota$ is a local system on $X^{r}$. Therefore, it suffices to show the same statement for the generic stalk of the relevant complexes. 

Let $\gen$ be a geometric generic point of $X^{r}$ and we use a subscript $\gen$ to denote the fibers over $\gen$, as in \S\ref{ss:horo gen}. Let $\iota_{U}: \cohoc{r}{U_{\gen}}\to \cohoc{r}{\Sht_{G,\gen}}$ be the map induced by the inclusion of $U$. It suffices to show that for $f\in\calI_{\Eis}$, the composition $\cohoc{r}{\Sht_{G,\gen}}\xrightarrow{f*}\cohoc{r}{\Sht_{G,\gen}}\surj \cohoc{r}{\Sht_{G,\gen}}/\iota_{U}(\cohoc{r}{U_{\gen}})$ is zero. 

Recall from \eqref{coho const term} the cohomological constant term map $\gamma_{d}: \cohoc{r}{\Sht_{G,\gen}}\to \homog{0}{\Sht^{d}_{H,\gen}}$. By the definition of $\gamma_{d}$, for $d>2g-2$, $\gamma_{d}$ factors through the quotient $\cohoc{r}{\Sht_{G,\gen}}/\iota_{U}(\cohoc{r}{U_{\gen}})$, and induces a map
\begin{equation*}
\gamma_{+}:=\prod_{d>2g-2}\gamma_{d}: \cohoc{r}{\Sht_{G,\gen}}/\iota_{U}(\cohoc{r}{U_{\gen}})\to \prod_{d>2g-2}\homog{0}{\Sht^{d}_{H,\gen}}.
\end{equation*}
Both sides of the above map admit filtrations indexed by the poset $\{d\in \calD;d>2g-2\}$: on the LHS this is given by the image of $\cohoc{r}{\Sht^{\leq d}_{G,\gen}}$ and on the RHS this is given by $\prod_{2g-2<d'\leq d}\homog{0}{\Sht^{d'}_{H,\gen}}$. The map $\gamma_{+}$ respects these filtrations and by Corollary \ref{c:Eis loc sys}, the associated graded map of $\gamma_{+}$ under these filtrations is injective. Therefore $\gamma_{+}$ is injective. 

By Lemma \ref{l:horo action}, we have a commutative diagram
\begin{equation*}
\xymatrix{\cohoc{r}{\Sht_{G,\gen}}(r/2)\ar[r]^{f*}\ar[d]^{\prod\gamma_{d}} & \cohoc{r}{\Sht_{G,\gen}}(r/2) \ar[r]\ar[d]^{\prod\gamma_{d}} & \cohoc{r}{\Sht_{G,\gen}}(r/2)/\iota_{U}(\cohoc{r}{U_{\gen}})(r/2)\ar[d]^{\gamma_{+}}\\
\prod_{d\in\calD}\homog{0}{\Sht^{d}_{H,\gen}}\ar[r]^{\Sat(f)*} & \prod_{d\in\calD}\homog{0}{\Sht^{d}_{H,\gen}}\ar[r] & \prod_{d>2g-2}\homog{0}{\Sht^{d}_{H,\gen}}}
\end{equation*}
Since the action of $\sH_{H}$ on $\prod_{d\in\calD}\homog{0}{\Sht^{d}_{H,\gen}}$ factors through $\Ql[\Pic_{X}(k)]$, $\Sat(f)$ acts by zero in the bottom arrow above. Since $\gamma_{+}$ is injective, the composition of the top row is also zero, as desired.
\end{proof}

\begin{defn}\label{def:bsH} We define the $\Ql$-algebra $\bsH$ to be the image of the map
\begin{equation*}
\sH\otimes\Ql\to\End_{\Ql}(V)\times\Ql[\Pic_{X}(k)]^{\iota_{\Pic}},		\index{$\overline{\sH_{\ell}}$}%
\end{equation*}
the product of the action map on $V$ and $a_{\Eis}\otimes\Ql$.
\end{defn}

\begin{lemma}\label{l:fg}
\begin{enumerate}
\item For any $x\in|X|$, $V$ is a finitely generated $\sH_{x}\otimes\Ql$-module. 
\item The $\Ql$-algebra $\bsH$ is finitely generated over $\Ql$ and is a ring with Krull dimension one.
\end{enumerate}
\end{lemma}
\begin{proof} (1) Let $\calD_{\leq n}\subset\calD$ be the subset of those $d$ such that $\min_{i}\{d(i)\}\leq2g-2+nd_{x}$. Let $\calD_{n}=\calD_{\leq n}-\calD_{\leq n-1}$. \index{$\calD_{n},\calD_{\leq n}$}%
For each $n\geq 0$, let $U_{n}=\cup_{d\in\calD_{\leq n}}\Sht^{\leq d}_{G}$, then $U_{0}\subset U_{1}\subset\cdots$ are finite type open substacks of $\Sht_{G}$ which exhaust $\Sht_{G}$. Let $\pi_{n}:U_{n}\to X^{r}$ be the restriction of $\pi_{G}$, and let $K_{n}=\bR\pi_{n,!}\Ql$. The inclusion $U_{n}\incl U_{n+1}$ induces maps $\iota_{n}:K_{n}\to K_{n+1}$. Let $C_{n+1}$ be the cone of $\iota_{n}$. Then by Corollary \ref{c:Eis loc sys}, when $n\geq 0$, $C_{n+1}$ is a successive extension of shifted local systems $\pi^{d}_{H,!}\Ql[-r](-r/2)$ for those $d\in\calD_{n+1}$. In particular, for $n\geq 0$, $C_{n+1}$ is a shifted local system in degree $r$ and pure of weight $0$ as a complex.

By construction, the action of $h_{x}\in\sH_{x}$ on $\cohoc{*}{\Sht_{G}\otimes_{k}\kbar}$ is induced from the correspondence $\Sht_{G}(h_{x})$, which restricts to a correspondence $\leftexp{\leq n}\Sht_{G}(h_{x})=\oll{p}^{-1}(U_{n})$ between $U_{n}$ and $U_{n+1}$. Similar to the construction of $C(h_{x})_{d,d'}$ in \eqref{Hk d d'}, the fundamental class of $\leftexp{\leq n}\Sht_{G}(h_{x})$ gives  a map $C(h_{x})_{n}: K_{n}\to K_{n+1}$. Since $C(h_{x})_{n}\circ \iota_{n-1}=\iota_{n}\circ C(h_{x})_{n-1}$, we have the induced map $\tau_{n}: C_{n}\to C_{n+1}$. We claim that $\tau_{n}$ is an isomorphism for $n>0$. In fact, since $C_{n}$ and $C_{n+1}$ are local systems in degree $r$, it suffices to check that $\tau_{n}$ induces an isomorphism between the geometric generic stalks $C_{n,\gen}$ and $C_{n+1,\gen}$. By Corollary \ref{c:Eis loc sys}, we have an isomorphism induced from the maps $\gamma_{d}$ for $d\in\calD_{n}$ (cf. \eqref{coho const term})
\begin{equation*}
C_{n,\gen}\cong\bigoplus_{d\in\calD_{n}}\homog{0}{\Sht^{d}_{H,\gen}}.
\end{equation*}
By Lemma \ref{l:horo action}, $\tau_{n,\gen}:C_{n,\gen}\to C_{n+1,\gen}$ is the same as the direct sum of the isomorphisms $t_{x}: \homog{0}{\Sht^{d}_{H,\gen}}\to \homog{0}{\Sht^{d+d_{x}}_{H,\gen}}$ (the other term $q_{x}t^{-1}_{x}: \homog{0}{\Sht^{d}_{H,\gen}}\to \homog{0}{\Sht^{d-d_{x}}_{H,\gen}}$ does not appear because $d-d_{x}\in\calD_{\leq n-1}$ hence the corresponding contribution becomes zero in $C_{n+1,\gen}$). Therefore $\tau_{n,\gen}$ is an isomorphism, hence so is $\tau_{n}$.

We claim that there exists $n_{0}\geq 0$ such that for any $n\geq n_{0}$, the map $$W_{\leq 2r}\cohog{2r+1}{X^{r}\otimes_{k}\kbar, K_{n}}\to W_{\leq 2r}\cohog{2r+1}{X^{r}\otimes_{k}\kbar, K_{n+1}}$$ is an isomorphism. Here $W_{\leq 2r}$ is the weight filtration using Frobenius weights. In fact, the next term in the long exact sequence is $W_{\leq 2r}\cohog{2r+1}{X^{r}\otimes_{k}\kbar, C_{n+1}}$, which is zero because $C_{n+1}$ is pure of weight $0$. Therefore the natural map $W_{\leq 2r}\cohog{2r+1}{X^{r}\otimes_{k}\kbar, K_{n}}\to W_{\leq 2r}\cohog{2r+1}{X^{r}\otimes_{k}\kbar, K_{n+1}}$ is always surjective for $n\geq0$, hence it has to be an isomorphism for sufficiently large $n$.

The triangle $K_{n}\to K_{n+1}\to C_{n+1}\to K_{n}[1]$ gives a long exact sequence
\begin{eqnarray}
\label{ex C}&&\cohog{2r}{X^{r}\otimes_{k}\kbar, K_{n}}\to \cohog{2r}{X^{r}\otimes_{k}\kbar,K_{n+1}}\to \cohog{2r}{X^{r}\otimes_{k}\kbar, C_{n+1}}\to\\
\notag&\to& W_{\leq 2r}\cohog{2r+1}{X^{r}\otimes_{k}\kbar,K_{n}}\to W_{\leq 2r}\cohog{2r+1}{X^{r}\otimes_{k}\kbar, K_{n+1}}
\end{eqnarray}
Here we are using the fact that $\cohog{2r}{X^{r}\otimes_{k}\kbar,C_{n+1}}$ is pure of weight $2r$ (since $C_{n+1}$ is pure of weight $0$). For $n\geq n_{0}$, the last map above is an isomorphism, therefore the first row of \eqref{ex C} is exact on the right. 

Let $F_{\leq n}V$ be the image of $\cohog{2r}{X^{r}\otimes_{k}\kbar,K_{n}}(r)\to \varinjlim_{n}\cohog{2r}{X^{r}\otimes_{k}\kbar,K_{n}}(r)=V$. Then for $n\geq n_{0}$, the exactness of \eqref{ex C} implies $\cohog{2r}{X^{r}\otimes_{k}\kbar,C_{n+1}}(r)\surj \Gr^{F}_{n+1}V$ for $n\geq n_{0}$. The Hecke operator $C(h_{x})$ sends $F_{\leq n}V$ to $F_{\leq n+1} V$ and induces a map $\Gr^{F}_{n}C(h_{x}):\Gr^{F}_{n}V\to \Gr^{F}_{n+1}V$. We have a commutative diagram for $n\geq n_{0}$
\begin{equation*}
\xymatrix{\cohog{2r}{X^{r}\otimes_{k}\kbar,C_{n}}(r)\ar[d]\ar[rr]^{\cohog{2r}{X^{r}\otimes_{k}\kbar, \tau_{n}}} && \cohog{2r}{X^{r}\otimes_{k}\kbar,C_{n+1}}(r)\ar@{->>}[d]\\
\Gr^{F}_{n}V\ar[rr]^{\Gr^{F}_{n}C(h_{x})} && \Gr^{F}_{n+1}V}
\end{equation*}
The fact that $\tau_{n}:C_{n}\to C_{n+1}$ is an isomorphism implies that $\Gr^{F}_{n}C(h_{x})$ is surjective for $n\geq n_{0}$. Therefore the action map
\begin{equation*}
\sH_{x}\otimes_{\QQ} F_{\leq n_{0}}V=\QQ[h_{x}]\otimes_{\QQ} F_{\leq n_{0}}V\to V
\end{equation*}
is surjective by checking the surjectivity on the associated graded. Since $F_{\leq n_{0}}V$ is finite-dimensional over $\Ql$, $V$ is finitely generated as an $\sH_{x}\otimes\Ql$-module.

(2) We have $\bsH\subset\End_{\sH_{x}\otimes\Ql}(V\oplus \Ql[\Pic_{X}(k)]^{\iota_{\Pic}})$. Since both $V$ and $\Ql[\Pic_{X}(k)]^{\iota_{\Pic}}$ are finitely generated $\sH_{x}\otimes\Ql$-modules by Part (1) and Lemma \ref{l:Eis fin gen}, $\End_{\sH_{x}\otimes\Ql}(V\oplus \Ql[\Pic_{X}(k)]^{\iota_{\Pic}})$ is also finitely generated as an $\sH_{x}\otimes\Ql$-module. Since $\sH_{x}\otimes\Ql$ is a polynomial ring in one variable over $\Ql$, $\bsH$ is a finitely generated algebra over $\Ql$ of Krull dimension at most one. Since $\bsH\to\Ql[\Pic_{X}(k)]^{\iota_{\Pic}}$ is surjective by Lemma \ref{l:Eis fin gen} and $\Ql[\Pic_{X}(k)]^{\iota_{\Pic}}$ has Krull dimension one, $\bsH$  also has Krull dimension one.
\end{proof}

The map $\ov{a}_{\Eis}:\bsH\to \Ql[\Pic_{X}(k)]^{\iota_{\Pic}}$ is surjective by Lemma \ref{l:Eis fin gen} (2). It induces a closed embedding $\Spec(\ov{a}_{\Eis}): Z_{\Eis,\Ql}=\Spec\Ql[\Pic_{X}(k)]^{\iota_{\Pic}}\incl\Spec\bsH$. 

\begin{theorem}[Cohomological spectral decomposition]\label{th:Spec decomp} 
\begin{enumerate}
\item There is a decomposition of the reduced scheme of $\Spec\bsH$ into a disjoint union
\begin{equation}\label{decomp bcH}
\Spec\left(\bsH\right)^{\red}=Z_{\Eis,\Ql}\coprod Z^{r}_{0,\ell}		\index{$Z^{r}_{0,\ell}$}%
\end{equation}
where $Z^{r}_{0,\ell}$ consists of a finite set of closed points. There is a unique decomposition
\begin{equation*}
V=V_{\Eis}\oplus V_{0}
\end{equation*}
into $\sH\otimes\Ql$-submodules, such that $\Supp(V_{\Eis})\subset Z_{\Eis,\Ql}$ and $\Supp(V_{0})=Z^{r}_{0,\ell}$. \footnote{When we talk about the support of a coherent module $M$ over a Noetherian ring $R$, we always mean a closed subset of $\Spec R$ with the reduced scheme structure.} 

\item The subspace $V_{0}$ is finite dimensional over $\Ql$.
\end{enumerate}
\end{theorem}
\begin{proof} (1) Let $V'=L_{\leq r}V$. Let $\ov{\calI}_{\Eis}\subset\bsH$ be the ideal generated by the image of $\calI_{\Eis}$. By Lemma \ref{l:fg}, $V'$ is a submodule of a finitely generated module $V$ over the noetherian ring $\bsH$, therefore $V'$ is also finitely generated. By Lemma \ref{l:Hcusp image}, $\ov{\calI}_{\Eis}V'$ is a finite-dimensional $\bsH$-submodule of $V'$. Let $Z'\subset\Spec(\bsH)^{\red}$ be the finite set of closed points corresponding to the action of $\bsH$ on $\ov{\calI}_{\Eis}V'$. We claim that $\Supp(V')$ is contained in the union $Z_{\Eis,\Ql}\cup Z'$. In fact, suppose $f\in\bsH$ lies in the defining radical ideal $\calJ$ of $Z_{\Eis,\Ql}\cup Z'$, then after replacing $f$ by a power of it, we have $f\in \ov{\calI}_{\Eis}$ (since $\calJ$ is contained in the radical of $\ov{\calI}_{\Eis}$) and $f$ acts on $\ov{\calI}_{\Eis}V'$ by zero. Therefore $f^{2}$ acts on $V'$ by zero, hence $f$ lies in the radical ideal defining $\Supp(V')$. 

By Lemma \ref{l:GrL}, $V/V'$ is finite-dimensional. Let $Z''\subset\Spec(\bsH)^{\red}$ be the support of $V/V'$ as a  $\bsH$-module, which is a finite set. Then $\Spec(\bsH)^{\red}=\Supp(V)\cup Z_{\Eis,\Ql} =Z_{\Eis,\Ql}\cup Z'\cup Z''$.  Let $Z^{r}_{0,\ell}=(Z'\cup Z'')-Z_{\Eis,\Ql}$, we get the desired decomposition \eqref{decomp bcH}. 

According to \eqref{decomp bcH}, the finitely generated $\bsH$-module $V$, viewed as a coherent sheaf on $\Spec\bsH$, can be uniquely decomposed into
\begin{equation*}
V=V_{\Eis}\oplus V_{0}
\end{equation*}
with $\Supp(V_{\Eis})\subset Z_{\Eis,\Ql}$ and $\Supp(V_{0})=Z^{r}_{0,\ell}$.

(2) We know that $V_{0}$ is a coherent sheaf on the scheme $\Spec\bsH$ which is of finite type over $\Ql$ and that $\Supp(V_{0})=Z^{r}_{0,\ell}$ is finite. Therefore $V_{0}$ is finite dimensional over $\Ql$.
\end{proof}

\subsubsection{The case $r=0$} Let us reformulate the result in Theorem \ref{th:Spec decomp} in the case $r=0$ in terms of automorphic forms. Let $\calA=C_{c}(G(F)\backslash G(\BA_{F})/K,\QQ)$ be the space of compactly supported $\QQ$-valued unramified automorphic forms, where $K=\prod_{x}G(\calO_{x})$. This is a $\QQ$-form of the $\Ql$-vector space $V$ for $r=0$. Let $\sH_{\aut}$ be the image of the action map $\sH\to\End_{\QQ}(\calA)\times\QQ[\Pic_{X}(k)]^{\iota_{\Pic}}$. The $\Ql$-algebra $\sH_{\aut,\Ql}:=\sH_{\aut}\otimes\Ql$ is the algebra $\bsH$ defined in Definition \ref{def:bsH} for $r=0$.	\index{$\sH_{\aut},\sH_{\aut,\Ql}$}%

Theorem \ref{th:Spec decomp} for $r=0$ reads
\begin{equation}\label{Haut Ql}
\Spec\sH^{\red}_{\aut,\Ql}=Z_{\Eis,\Ql}\coprod Z^{0}_{0,\ell}
\end{equation}
where $Z^{0}_{0,\ell}$ is a finite set of closed points.  Below we will strengthen this decomposition to work over $\QQ$, and link $Z^{0}_{0,\ell}$ to the set of cuspidal automorphic representations.

\subsubsection{Positivity and reducedness} The first thing to observe is that $\sH_{\aut}$ is already reduced. In fact, we may extend the Petersson inner product on $\calA$ to a positive definitive quadratic form on $\calA_{\RR}$. By the $r=0$ case of Lemma \ref{l:coho self adj}, $\sH_{\aut}$ acts on $\calA_{\RR}$ as self-adjoint operators, its image in $\End(\calA)$ is therefore reduced. Since $\QQ[\Pic_{X}(k)]^{\iota_{\Pic}}$ is reduced as well, we conclude that $\sH_{\aut}$ is reduced.

Let $\calA_{\cusp}\subset \calA$ be the finite-dimensional $\QQ$-vector space of cusp forms. Let $\sH_{\cusp}$ be the image of $\sH_{\aut}$ in $\End_{\QQ}(\calA_{\cusp})$. 	\index{$\sH_{\cusp}$}%
Then $\sH_{\cusp}$ is a reduced artinian $\QQ$-algebra, hence a product of fields. Let $Z_{\cusp}=\Spec\sH_{\cusp}$. \index{$Z_{\cusp}$}%
Then a  point in $Z_{\cusp}$ is the same as an everywhere unramified cuspidal automorphic representation $\pi$ of $G$ in the sense of \S\ref{ss:intro L}. Therefore we have a canonical isomorphism
\begin{equation*}
\sH_{\cusp}=\prod_{\pi\in Z_{\cusp}}E_{\pi}
\end{equation*} 
where $E_{\pi}$ is the coefficient field of $\pi$.

\begin{lemma}\label{l:Haut decomp} 
\begin{enumerate}
\item There is a canonical isomorphism of $\QQ$-algebras
\begin{equation*}
\sH_{\aut}\cong\QQ[\Pic_{X}(k)]^{\iota_{\Pic}}\times\sH_{\cusp}
\end{equation*}
Equivalently, we have a decomposition into disjoint reduced closed subschemes
\begin{equation}\label{Haut Q}
\Spec\sH_{\aut}=Z_{\Eis}\coprod Z_{\cusp}.
\end{equation}
\item We have $Z^{0}_{0,\ell}=Z_{\cusp,\Ql}$, the base change of $Z_{\cusp}$ from $\QQ$ to $\Ql$.
\end{enumerate}
\end{lemma}
\begin{proof}
(1) The $\QQ$ version of Lemma \ref{l:fg} says that $\sH_{\aut}$ is a finitely generated $\QQ$-algebra, and that $\calA$ is a finitely generated $\sH_{\aut}$-module. By the same argument of Theorem \ref{th:Spec decomp}, we get a decomposition
\begin{equation}\label{ZEis Z0}
\Spec\sH^{\red}_{\aut}=\Spec\sH_{\aut}=Z_{\Eis}\coprod Z_{0}
\end{equation}
where $Z_{0}$ is a finite collection of closed points. Correspondingly we have a decomposition
\begin{equation*}
\calA=\calA_{\Eis}\oplus\calA_{0}
\end{equation*}
with $\Supp(\calA_{\Eis})\subset Z_{\Eis}$ and $\Supp(\calA_{0})=Z_{0}$. Since $\calA_{0}$ is finitely generated over $\sH_{\aut}$ with finite support, it is finite dimensional over $\QQ$.  Since $\calA_{0}$ is finite dimensional and stable under $\sH$, we necessarily have $\calA_{0}\subset\calA_{\cusp}$ (see \cite[Lemme 8.13]{VL}; in fact in our case it can be easily deduced from the $r=0$ case of Lemma \ref{l:horo action}). 

We claim that $\calA_{0}=\calA_{\cusp}$. To show the inclusion in the other direction, it suffices to show that any cuspidal Hecke eigenform $\varphi\in\calA_{\cusp}\otimes\Qbar$ lies in $\calA_{0}\otimes\Qbar$. Suppose this is not the case for $\varphi$, letting $\l:\sH\to \Qbar$ be the character by which $\sH$ acts on $\varphi$, then $\l\notin Z_{0}(\Qbar)$. By \eqref{ZEis Z0}, $\l\in Z_{\Eis}(\Qbar)$, which means that the action of $\sH$ on $\varphi$ factors through $\QQ[\Pic_{X}(k)]$ via $a_{\Eis}$, which is impossible.

Now $\calA_{0}=\calA_{\cusp}$ implies that $Z_{0}=\Supp(\calA_{0})=\Supp(\calA_{\cusp})=Z_{\cusp}$. Combining with \eqref{ZEis Z0}, we get \eqref{Haut Q}.

Part (2) follows from comparing \eqref{Haut Ql} to the base change of \eqref{Haut Q} to $\Ql$.
\end{proof}

\subsection{Decomposition of the Heegner--Drinfeld cycle class}\label{ss:decomp HD cycle} 
In previous subsections, we have been working with the middle-dimensional cohomology (with compact support) of $\Sht_{G}=\Sht^{r}_{G}$, and we established a decomposition of it as an $\sH_{\Qlbar}$-module. Exactly the same argument works if we replace $\Sht_{G}$ with $\Sht'_{G}=\Sht'^{r}_{G}$. Instead of repeating the argument we simply state the corresponding result for $\Sht'_{G}$ in what follows.

Let 
\begin{equation*}
V'=\cohoc{2r}{\Sht'_{G}\otimes_{k}\kbar,\Ql}(r). 
\index{$V'$}%
\end{equation*}
Then $V'$ is equipped with a $\Ql$-valued cup product pairing
\begin{equation}\label{cup prod V'}
(\cdot,\cdot): V'\otimes_{\Ql}V'\to \Ql
\end{equation}
and an action of $\sH$ by self-adjoint operators. 

Similar to Definition \ref{def:bsH}, we define the $\Ql$-algebra $\bsH'$ to be the image of the map
\begin{equation*}
\sH\otimes\Ql\to\End_{\Ql}(V')\times\Ql[\Pic_{X}(k)]^{\iota_{\Pic}}.		\index{$\bsH'$}%
\end{equation*}

\begin{theorem}[Variant of Lemma \ref{l:fg} and Theorem \ref{th:Spec decomp}]\label{th:Spec decomp Sht'} 
\begin{enumerate}
\item For any $x\in|X|$, $V'$ is a finitely generated $\sH_{x}\otimes\Ql$-module. 
\item The $\Ql$-algebra $\bsH'$ is finitely generated over $\Ql$ and is one-dimensional as a ring.
\item There is a decomposition of the reduced scheme of $\Spec\bsH'$ into a disjoint union
\begin{equation}\label{decomp bcH'}
\Spec\left(\bsH'\right)^{\red}=Z_{\Eis,\Ql}\coprod Z'^{r}_{0,\ell}		\index{$Z'^{r}_{0,\ell}$}%
\end{equation}
where $Z'^{r}_{0,\ell}$ consists of a finite set of closed points. There is a unique decomposition
\begin{equation*}
V'=V'_{\Eis}\oplus V'_{0}
\end{equation*}
into $\sH\otimes\Ql$-submodules, such that $\Supp(V'_{\Eis})\subset Z_{\Eis,\Ql}$ and $\Supp(V'_{0})=Z'^{r}_{0,\ell}$.
\item The subspace $V'_{0}$ is finite dimensional over $\Ql$.
\end{enumerate}
\end{theorem}

We may further decompose $V'_{\Qlbar}:=V'\otimes_{\Ql}\Qlbar$ according to points in $Z'^{r}_{0,\ell}(\Qlbar)$. A point in $Z'^{r}_{0,\ell}(\Qlbar)$ is a maximal ideal $\fkm\subset \sH_{\Qlbar}$, or equivalently a ring homomorphism $\sH\to \Qlbar$ whose kernel is $\fkm$. We have a decomposition
\begin{equation}\label{decomp V}
V'_{\Qlbar}=V'_{\Eis,\Qlbar}\oplus(\bigoplus_{\fkm\in Z'^{r}_{0,\ell}(\Qlbar)}V'_{\fkm}).
\end{equation}
Then $V'_{\fkm}$ is characterized as the largest $\Qlbar$-subspace of $V'_{\Qlbar}$ on which the action of $\fkm$ is locally nilpotent.  By Theorem \ref{th:Spec decomp Sht'}, $V'_{\fkm}$ turns out to be the localization of $V'$ at the maximal ideal $\fkm$, hence our notation $V'_{\fkm}$ is consistent with the standard notation used in commutative algebra.

We may decompose the cycle class $\cl(\theta^{\mu}_{*}[\Sht^{\mu}_{T}])\in V'_{\Qlbar}$ according to the decomposition  \eqref{decomp V}
\begin{equation}\label{Sht T decomp}
\cl(\theta^{\mu}_{*}[\Sht^{\mu}_{T}])=[\Sht_{T}]_{\Eis}+\sum_{\fkm\in Z^{r}_{0,\ell}(\Qlbar)}[\Sht_{T}]_{\fkm}		\index{$[\Sht_{T}]_{\Eis}$}%
\end{equation}
where $[\Sht_{T}]_{\Eis}\in V'_{\Eis}$ and $[\Sht_{T}]_{\fkm}\in V'_{\fkm}$.

\begin{cor}\label{c:decomp Ir} 
\begin{enumerate}
\item The decomposition \eqref{decomp V} is an orthogonal decomposition under the cup product pairing \eqref{cup prod V'} on $V'$.
\item For any $f\in \sH$, we have
\begin{equation}\label{decomp Ir}
\BI_{r}(f)=\left([\Sht_{T}]_{\Eis},f*[\Sht_{T}]_{\Eis}\right)+\sum_{\fkm\in Z'^{r}_{0,\ell}(\Qlbar)}\BI_{r}(\fkm,f)
\end{equation}
where
\begin{equation*}
\BI_{r}(\fkm,f):=\left([\Sht_{T}]_{\fkm}, f*[\Sht_{T}]_{\fkm}\right).
\end{equation*}
\end{enumerate}
\end{cor}
\begin{proof}
The orthogonality of the decomposition \eqref{decomp V} follows from the self-adjointness of $\sH$ with respect to the cup product pairing, i.e.,  variant of Lemma \ref{l:coho self adj} for $\Sht'_{G}$. The formula \eqref{decomp Ir} then follows from the orthogonality of the terms in the decomposition \eqref{Sht T decomp}.
\end{proof}


\part{The comparison}

\section{Comparison for most Hecke functions}\label{s:most}

The goal of this section is to prove the key identity \eqref{key id} for most Hecke functions. More precisely, we will prove the following theorem.

\begin{thm}\label{th:IJ}Let $D$ be an effective divisor on $X$ of degree $d\geq\max\{2g'-1,2g\}$. Then for any $u\in\PP^{1}(F)-\{1\}$ we have
\begin{equation}\label{u IJ}
(\log q)^{-r}\BJ_{r}(u,h_{D})=\BI_{r}(u,h_{D}).
\end{equation}
In particular, we have
\begin{equation}\label{IJ hD}
(\log q)^{-r}\BJ_{r}(h_{D}) =\BI_{r}(h_{D}).
\end{equation}
\end{thm}
For the definition of $\BJ_{r}(u,h_{D})$ and $\BI_{r}(u,h_{D})$, see \eqref{orb J} and \eqref{orb I} respectively.

\subsection{Direct image of $f_{\CM}$}\label{ss:Loc K} 

\subsubsection{The local system $L(\rho_{i})$}\label{sss:Lrhoi} Let $j:X^{\c}_{d}\subset X_{d}\subset \hX_{d}$ be the locus of multiplicity-free divisors. \index{$X^{\c}_{d}$}%
Taking the preimage of $X^{\c}_{d}$ under the branched cover $X'^{d}\to X^{d}\to X_{d}$, we get an \'etale Galois cover
\begin{equation*}
u: X'^{d,\c}\to X^{d,\c}\to X^{\c}_{d}
\end{equation*}
with Galois group $\Gamma_{d}:=\{\pm1\}^{d}\rtimes S_{d}$. \index{$\Gamma_d, S_d$}%
For $0\leq i\leq d$, let $\chi_{i}$ be the character $\{\pm1\}^{d}\to\{\pm1\}$ that is nontrivial on the first $i$ factors and trivial on the rest. Let $S_{i,d-i}\cong S_{i}\times S_{d-i}$ be the subgroup of $S_{d}$ stabilizing $\{1,2,\cdots, i\}\subset \{1,\cdots, d\}$. Then $\chi_{i}$ extends to the subgroup $\Gamma_{d}(i)=\{\pm1\}^{d}\rtimes S_{i,d-i}$ of $\Gamma_{d}$ with the trivial representation on the $S_{i,d-i}$-factor. The induced representation
\begin{equation}\label{rhoi}
\rho_{i}=\Ind^{\Gamma_{d}}_{\Gamma_{d}(i)}(\chi_{i}\boxtimes \one)
\end{equation}
is an irreducible representation of $\Gamma_{d}$. This representation gives rise to an irreducible local system $L(\rho_{i})$ on $X^{\c}_{d}$.  Let $K_{i}:=j_{!*}(L(\rho_{i})[d])[-d]$ be the middle extension of $L(\rho_{i})$ (see \cite[2.1.7]{BBD}). Then $K_{i}$ is a shifted simple perverse sheaf on $\hX_{d}$. \index{$K_{i}$}%

\begin{prop}\label{p:Rf decomp} Suppose $d\geq 2g'-1$. Then we have a canonical isomorphism of shifted perverse sheaves
\begin{equation}\label{Rf decomp}
\bR f_{\calM,*}\Ql\cong\bigoplus_{i,j=0}^{d}(K_{i}\boxtimes K_{j})|_{\calA_{d}}.
\end{equation}
Here $K_{i}\boxtimes K_{j}$ lives on $\hX_{d}\times_{\Pic^{d}_{X}}\hX_{d}$, which contains $\calA_{d}$ as an open subscheme. 
\end{prop}
\begin{proof}
By Proposition \ref{p:M}\eqref{fM proper}, $f_{\calM}$ is the restriction of $\wh\nu_{d}\times\wh\nu_{d}: \hX'_{d}\times_{\Pic^{d}_{X}}\hX'_{d}\to \hX_{d}\times_{\Pic^{d}_{X}}\hX_{d}$, where $\wh{\nu}_{d}:\hX'_{d}\to \hX_{d}$ is the norm map. By Proposition \ref{p:M}\eqref{nud proper}, $\wh{\nu}_{d}$ is also proper. Therefore by the K\"unneth formula, it suffices to show that
\begin{equation}\label{sumKd}
\bR\wh\nu_{d,*}\Ql\cong\bigoplus_{i=0}^{d}K_{i}.
\end{equation}

We claim that $\wh\nu_{d}$ is a small map (see \cite[6.2]{GM}). In fact the only positive dimension fibers are over the zero section $\Pic^{d}_{X}\incl \hX_{d}$, which has codimension $d-g+1$. On the other hand, the restriction of $\wh\nu_{d}$ to the zero section is the norm map $\Pic^{d}_{X'}\to \Pic^{d}_{X}$, which has fiber dimension $g-1$. The condition $d\geq 2g'-1\geq 3g-2$ implies $d-g+1\geq 2(g-1)+1$, therefore $\wh\nu_{d}$ is a small map.

Now $\wh\nu_{d}$ is proper, small with smooth and geometrically irreducible source, $\bR\wh\nu_{d,*}\Ql$ is the middle extension of its restriction to any dense open subset of $\hX_{d}$ (see \cite[Theorem at the end of 6.2]{GM}). In particular, $\bR\wh\nu_{d,*}\Ql$ is the middle extension of its restriction to $X^{\c}_{d}$. It remains to show
\begin{equation}\label{whpi L}
\bR\wh\nu_{d,*}\Ql|_{X^{\c}_{d}}\cong\bigoplus_{i=0}^{d}L(\rho_{i}).
\end{equation}

Let $\nu^{\c}_{d}: X'^{\c}_{d}=\nu_{d}^{-1}(X^{\c}_{d})\to X^{\c}_{d}$ be the restriction of $\nu_{d}:X'_{d}\to X_{d}$ over $X^{\c}_{d}$. Then $\bR\nu^{\c}_{d,*}\Ql$ is the local system on $X^{\c}_{d}$ associated with the representation $\Ind^{\Gamma_{d}}_{S_{d}}\Ql=\Ql[\Gamma_{d}/S_{d}]$ of $\Gamma_{d}$. A basis $\{\one_{\ep}\}$ of $\Ql[\Gamma_{d}/S_{d}]$ is given by the indicator functions of the $S_{d}$-coset of  $\ep\in\{\pm1\}^{d}$. For any character $\chi:\{\pm1\}^{d}\to\{\pm1\}$, let $\one_{\chi}:=\sum_{\ep}\chi(\ep)\one_{\ep}\in \Ql[\Gamma_{d}/S_{d}]$. For the character $\chi_{i}$ considered in \S\ref{sss:Lrhoi},  $\one_{\chi_{i}}$ is invariant under $S_{i,d-i}$, and therefore we have a $\Gamma_{d}$-equivariant embedding $\rho_{i}=\Ind^{\Gamma_{d}}_{\Gamma_{d}(i)}(\chi_{i}\boxtimes\one)\incl\Ql[\Gamma_{d}/S_{d}]$. Checking total dimensions we conclude that
\begin{equation*}
\Ql[\Gamma_{d}/S_{d}]\cong\bigoplus_{i=0}^{d}\rho_{i}.
\end{equation*}
This gives a canonical isomorphism of local systems $\bR\nu^{\c}_{d,*}\Ql\cong\oplus_{i=0}^{d}L(\rho_{i})$, which is \eqref{whpi L}.
\end{proof}

In \S\ref{sss:define calH}, we have defined a self-correspondence $\calH=\Hk^{1}_{\CM,d}$ of $\CM_{d}$ over $\calA_{d}$. Recall that  $\Ads_{d}\subset\calA_{d}$ is the open subscheme $\hX_{d}\times_{\Pic^{d}_{X}}X_{d}$, and $\Mds_{d}$ and $\calH^{\ds}$ are the restrictions of $\calM_{d}$ and $\calH$ to $\Ads_{d}$. Recall that $[\calH^{\ds}]\in\Ch_{2d-g+1}(\calH)_{\QQ}$ is the fundamental cycle of the closure of $\calH^{\ds}$.	 \index{$\calA_{d}^{\diam}$} \index{$\calM_{d}^\diam$} \index{$\calH^{\ds}$} %

\begin{prop}\label{p:H action} Suppose $d\geq 2g'-1$. Then the action $f_{\calM,!}[\calH^{\ds}]$  on $\bR f_{\calM*}\Ql$ preserves each direct summand $K_{i}\boxtimes K_{j}$ under the decomposition \eqref{Rf decomp}, and acts on $K_{i}\boxtimes K_{j}$ by the scalar $(d-2j)$. 
\end{prop}
\begin{proof} By Proposition \ref{p:Rf decomp}, $\bR f_{\calM*}\Ql$ is a shifted perverse sheaf all of whose simple constituents have full support. Therefore it suffices to prove the same statement after restricting to any dense open subset $U\subset \calA_{d}$. We work with $U=\Ads_{d}$.

Recall $\calH$ is indeed a self-correspondence of $\calM_{d}$ over $\tcA_{d}$ (see \S\ref{sss:same proj A}):
\begin{equation}\label{Corr H}
\xymatrix{ & \calH\ar[dl]_{\gamma_{0}}\ar[dr]^{\gamma_{1}} \\
\calM_{d}\ar[dr]_{\tf_{\calM}} & & \calM_{d}\ar[dl]^{\tf_{\calM}}\\
& \tcA_{d}}
\end{equation}

By Lemma \ref{l:HInc}, the diagram \eqref{Corr H} restricted to $\tcA^{\ds}_{d}$ (the preimage of $\Ads_{d}$ in $\tcA_{d}$) is obtained from the following correspondence via base change along the second projection $\pr_{2}: \tcA^{\ds}_{d}\cong \hX'_{d}\times_{\Pic^{d}_{X}} X_{d}\to X_{d}$ which is smooth
\begin{equation*}
\xymatrix{& I'_{d}\ar[dl]_{\pr}\ar[dr]^{q}\\
X'_{d}\ar[dr]_{\nu_{d}} & & X'_{d}\ar[dl]^{\nu_{d}}\\
& X_{d}}
\end{equation*}
Here for $(D,y)$ in the universal divisor $I'_{d}\subset X'_{d}\times X'$, $\pr(D,y)=D$ and $q(D,y)=D-y+\sigma(y)$.

Let $T_{d}:=\nu_{d,!}[I'_{d}]: \bR\nu_{d,*}\Ql\to\bR\nu_{d,*}\Ql$ be the operator on $\bR\nu_{d,*}\Ql$ induced from the cohomological correspondence between the constant sheaf $\Ql$ on $X'_{d}$ and itself given by the fundamental class of $I'_{d}$. Under the isomorphism $\bR \tf_{\calM,!}\Ql|_{\Ads_{d}}\cong \pr_{2}^{*}\bR\nu_{d,*}\Ql$, the action of $\tf_{\calM,!}[\calH^{\ds}]$ is the pullback along the smooth map $\pr_{2}$ of the action of $T_{d}=\nu_{d,!}[I'_{d}]$. Therefore it suffices to show that $T_{d}$ preserves the decomposition \eqref{sumKd} (restricted to $X_{d}$), and acts on each $K_{j}$ by the scalar $(d-2j)$.

Since $\bR \nu_{d,*}\Ql$ is the middle extension of the local system $L=\oplus_{j=0}^{d}L(\rho_{j})$ on $X^{\c}_{d}$, it suffices to calculate the action of $T_{d}$ on $L$, or rather calculate its action over a geometric generic point $\eta\in X_{d}$. Write  $\eta=x_{1}+x_{2}+\cdots +x_{d}$ and name the two points in $X'$ over $x_{i}$ by $x^{+}_{i}$ and $x^{-}_{i}$ (in one of the two ways). The fiber $\nu^{-1}_{d}(\eta)$ consists of points $\xi_{\ep}$ where $\ep\in\{\pm\}^{r}$, and $\xi_{\ep}=\sum_{i=1}^{d}x^{\ep_{i}}_{i}$ . As in the proof of Proposition \ref{p:Rf decomp}, we may identify the stalk $L_{\eta}$ with $\Ql[\Gamma_{d}/S_{d}]=\Span\{\one_{\ep};\ep\in\{\pm\}^{r}\}$ (we identify $\{\pm\}$ with $\{\pm1\}$). Now we denote $\one_{\ep}$ formally by the monomial $x^{\ep_{1}}_{1}\cdots x^{\ep_{d}}_{d}$. The stalk $L(\rho_{j})_{\eta}$ has a basis given by $\{P_{\delta}\}$, where
\begin{equation*}
P_{\delta}:=\prod_{i=1}^{d}(x^{+}_{i}+\delta_{i}x^{-}_{i})
\end{equation*}
and $\delta$ runs over those elements $\delta=(\delta_{1},\cdots,\delta_{d})\in \{\pm\}^{d}$ with exactly $i$ minuses. The action of $T_{d}$ on $L_{\eta}$ turns each monomial basis element $x^{\ep_{1}}_{1}\cdots x^{\ep_{d}}_{d}$ into $\sum_{t=1}^{d}x^{\ep_{1}}_{1}\cdots x^{-\ep_{t}}_{t}\cdots x^{\ep_{d}}_{d}$. Therefore, $T_{d}$ is a derivation in the following sense: for any linear form $\ell_{i}$ in $x^{+}_{i}$ and $x^{-}_{i}$, we have
\begin{equation*}
T_{d}\prod_{i=1}^{d}\ell_{i}=(T_{d}\ell_{1})\cdot \ell_{2}\cdots\ell_{d}+\ell_{1}(T_{d}\ell_{2})\ell_{3}\cdots\ell_{d}+\cdots+\ell_{1}\cdots\ell_{d-1}(T_{d}\ell_{d}).
\end{equation*}
Also $T_{d}(x^{+}_{i}+x^{-}_{i})=x^{+}_{i}+x^{-}_{i}$ and $T_{d}(x^{+}_{i}-x^{-}_{i})=-(x^{+}_{i}-x^{-}_{i})$. From these we easily calculate that $T_{d}P_{\delta}=(d-2|\delta|)P_{\delta}$ where $|\delta|$ is the number of minuses in $\delta$. Since $L(\rho_{j})_{\eta}$ is the span of $P_{\delta}$ with $|\delta|=j$, it is exactly the eigenspace of $T_{d}$ with eigenvalue $(d-2j)$. This finishes the proof.
\end{proof}

Combining Theorem \ref{th:I trace}, \eqref{orb I} with Proposition \ref{p:H action}, we get

\begin{cor}\label{c:BIr} Suppose $d\geq \max\{2g'-1,2g\}$. Let $D\in X_{d}(k)$. Then
\begin{align*}
\BI_{r}(u,h_{D})=\begin{cases} \sum_{i,j=0}^{d}(d-2j)^{r}\Tr(\Frob_{a},(K_{i})_{\ov{a}}\otimes(K_{j})_{\ov{a}}) & u=\inv_{D}(a), a\in\calA_{D}(k) \\
0 & \text{ otherwise.}\end{cases}				\index{$\BI_{r}(u,h_{D})$}%
\end{align*}
\end{cor}

\subsection{Direct image of $f_{\calN_{\un{d}}}$}
Recall the moduli space $\calN_{\un{d}}$ defined in \S\ref{sss:calN} for $\un{d}\in\Sigma_{d}$. It carries a local system $L_{\un{d}}$, see \S\ref{sss:Ld}. \index{$\calN_{\un{d}}$} \index{$L_{\un{d}}$}%

\begin{prop}\label{p:RfN decomp} Let $d\geq 2g'-1$ and $\un{d}\in\Sigma_{d}$. Then there is a canonical isomorphism 
\begin{equation}\label{RfN}
\bR f_{\calN_{\un{d}},*}L_{\un{d}}\cong (K_{d_{11}}\boxtimes K_{d_{12}})|_{\calA_{d}}.
\end{equation}
\end{prop}
\begin{proof}
The condition $d\geq 2g'-1$ does not imply that $f_{\calN_{\un{d}}}$ is small. Nevertheless we shall show that the complex $K_{\un{d}}:=\bR f_{\calN_{\un{d}},*}L_{\un{d}}$ is the middle extension from its restriction to $\calB:=X_{d}\times_{\Pic^{d}_{X}}X_{d}\subset\calA_{d}$. By Proposition \ref{p:Nsm}(2),  $\calN_{\un{d}}$ is smooth hence $L_{\un{d}}[\dim \calN_{\un{d}}]$ is Verdier self-dual up to a Tate twist. By Proposition \ref{p:Nsm}(3), $f_{\calN_{\un{d}}}$ is proper,  hence the complex $K_{\un{d}}[\dim \calN_{\un{d}}]$ is also Verdier self-dual up to a Tate twist. The morphism $f_{\calN_{\un{d}}}$ is finite over the open stratum $\calB$, therefore $K_{\un{d}}|_{\calB}$ is concentrated in degree $0$. The complement $\calA_{d}-\calB$ is the disjoint union of $\calC=\{0\}\times X_{d}$ and $\calC'=X_{d}\times \{0\}$.  We compute the restriction $K_{\un{d}}|_{\calC}$. 

When $d_{11}<d_{22}$, by the last condition in the definition of $\calN_{\un{d}}$, $\varphi_{22}$ is allowed to be zero but $\varphi_{11}$ is not. The fiber of $f_{\calN_{\un{d}}}$ over a point $(0,D)\in \calC$ is of the form $X_{d_{11}}\times\add^{-1}_{d_{12},d_{21}}(D)$, where $\add_{j,d-j}: X_{j}\times X_{d-j}\to X_{d}$ is the addition map. We have $(K_{\un{d}})_{(0,D)}=\cohog{*}{X_{d_{11}}\otimes_{k}\kbar,L_{d_{11}}}\otimes M$ where $M=\cohog{0}{\add^{-1}_{d_{12},d_{21}}(D)\otimes_{k}\kbar, L_{d_{12}}}$ is a finite-dimensional vector space. We have $\cohog{*}{X_{d_{11}}\otimes_{k}\kbar,L_{d_{11}}}\cong \bigwedge^{d_{11}}(\cohog{1}{X\otimes_{k}\kbar,L_{X'/X}})[-d_{11}]$ which is concentrated in degree $d_{11}$, and is zero for $d_{11}>2g-2$. Therefore $(K_{\un{d}})_{(0,D)}$ is concentrated in some degree $\leq 2g-2$, which is smaller than $\codim_{\calA_{d}}\calC=d-g+1$. 

When $d_{11}\geq d_{22}$, $\varphi_{11}$ may be zero but $\varphi_{22}$ is nonzero. The fiber of $f_{\calN_{\un{d}}}$ over a point $(0,D)\in \calC$ is of the form $X_{d_{22}}\times\add^{-1}_{d_{12},d_{21}}(D)$. For $(D_{22}, D_{12}, D_{21})\in X_{d_{22}}\times\add^{-1}_{d_{12},d_{21}}(D)$,  its image in $\Pic^{d_{11}}_{X}$ is $\calO_{X}(D-D_{22})$, therefore the restriction of $L_{d_{11}}$ to $f^{-1}_{\calN_{\un{d}}}(0,D)$ is isomorphic to $L^{-1}_{d_{22}}$ on the $X_{d_{22}}$ factor. Therefore $(K_{\un{d}})_{(0,D)}=\cohog{*}{X_{d_{22}}\otimes_{k}\kbar,L^{-1}_{d_{22}}}\otimes \cohog{0}{\add^{-1}_{d_{12},d_{21}}(D)\otimes_{k}\kbar, L_{d_{12}}}$, which is again concentrated in some degree $\leq 2g-2<\codim_{\calA_{d}}\calC=d-g+1$. 

Same argument shows that the stalks of $K_{\un{d}}$ over $\calC'$ are concentrated in some degree $\leq 2g-2<\codim_{\calA_{d}}\calC'=d-g+1$. Using Verdier self-duality of $K_{\un{d}}[\dim \calN_{\un{d}}]$, we conclude that $K_{\un{d}}$ is the middle extension from its restriction to $\calB$. 

By Proposition \ref{p:Nsm}(3) and the Kunneth formula, we have
\begin{equation*}
K_{\un{d}}|_{\calB}\cong\add_{d_{11},d_{22},*}(L_{d_{11}}\boxtimes\Ql)\boxtimes\add_{d_{12},d_{21},*}(L_{d_{12}}\boxtimes\Ql).
\end{equation*}
To prove the proposition, it suffices to give a canonical isomorphism
\begin{equation}\label{LK}
\add_{j,d-j,*}(L_{j}\boxtimes\Ql)\cong K_{j}|_{X_{d}}
\end{equation}
for every $0\leq j\leq d$.
Both sides of \eqref{LK} are middle extensions from $X^{\c}_{d}$, we only need to give an isomorphism between their restrictions to $X^{\c}_{d}$. Over $X^{\c}_{j}$, the local system $L_{j}$ is given by the representation $\pi_{1}(X^{\c}_{j})\to\pi_{1}(X)^{j}\rtimes S_{j}\surj\Gal(X'/X)^{j}\rtimes S_{j}\cong\{\pm1\}^{j}\rtimes S_{j}\to\{\pm1\}$ which is nontrivial on each factor $\Gal(X'/X)$ and trivial on the $S_{j}$-factor. The finite \'etale cover $\add^{\c}_{j,d-j}:(X_{j}\times X_{d-j})^{\c}\to X^{\c}_{d}$ (restriction of $\add_{j,d-j}$ to $X^{\c}_{d}$) is the quotient $X^{d,\c}/S_{j,d-j}$ where $S_{j,d-j}\subset S_{d}$ is the subgroup defined in \S\ref{sss:Lrhoi}. Therefore the local system $\add^{\c}_{j,d-j,*}(L_{j}\boxtimes\Ql)$ corresponds to the representation $\rho_{j}$ of $\Gamma_{d}$, and $\add^{\c}_{j,d-j,*}(L_{j}\boxtimes\Ql)\cong L(\rho_{j})$ as local systems over $X^{\c}_{d}$. This completes the proof of \eqref{LK}, and the proposition is proved.
\end{proof}

Combining  Propositions \ref{p:Rf decomp} and \ref{p:RfN decomp}, we get
\begin{cor} Assume $d\geq 2g'-1$. Then there is a canonical isomorphism
\begin{equation*}
\bR f_{\calM,*}\Ql\cong \bigoplus_{\un{d}\in\Sigma_{d}}\bR f_{\calN_{\un{d}},*}L_{\un{d}}
\end{equation*}
such that the $(i,j)$-grading of the LHS appearing in \eqref{Rf decomp} corresponds to the $(d_{11},d_{12})$-grading on the RHS. 
\end{cor}

\subsection{Proof of Theorem \ref{th:IJ}}
By Corollary \ref{c:Jr geom} and \eqref{orb J}, both $\JJ_{r}(u,h_{D})$ and $\II_{r}(u,h_{D})$ vanish when $u$ is not of the form $\inv_{D}(a)$ for $a\in\calA_{D}(k)$. We only need to prove \eqref{u IJ} when $u=\inv_{D}(a)$ for  $a\in\calA_{D}(k)$. In this case we have
\begin{eqnarray*}
(\log q)^{-r}\BJ_{r}(u,h_{D})&=&\sum_{\un{d}\in\Sigma_{d}}(2d_{12}-d)^{r}\Tr\left(\Frob_{a}, \left(\bR f_{\calN_{\un{d}},*}L_{\un{d}}\right)_{\ov{a}}\right) \quad\mbox{(Corollary \ref{c:Jr geom})}\\
&=&\sum_{d_{11}, d_{12}=0}^{d}(2d_{12}-d)^{r}\Tr\left(\Frob_{a},(K_{d_{11}})_{\ov{a}}\otimes(K_{d_{12}})_{\ov{a}}\right) \quad\mbox{(Prop.  \ref{p:RfN decomp})}\\
&=&\sum_{i,j=0}^{d}(d-2j)^{r}\Tr\left(\Frob_{a},(K_{i})_{\ov{a}}\otimes(K_{j})_{\ov{a}}\right)  \quad\mbox{($r$ is even)}\\
&=&\BI_{r}(u,h_{D}) \quad\mbox{(Corollary \ref{c:BIr})}
\end{eqnarray*}
Therefore \eqref{u IJ} is proved. By  \eqref{eqn orb decomp} and \eqref{I orb decomp}, \eqref{u IJ} implies \eqref{IJ hD}.

\section{Proof of the main theorems}\label{s:pf}
In this section we complete the proofs of our main results stated in the Introduction.

\subsection{The identity $(\log q)^{-r}\BJ_{r}(f)=\BI_{r}(f)$ for all Hecke functions} By Theorem \ref{th:IJ}, we have $(\log q)^{-r}\BJ_{r}(f)=\BI_{r}(f)$ for all $f=h_{D}$ where $D$ is an effective divisor with $\deg(D)\geq \max\{2g'-1,2g\}$. Our goal in this subsection is to show by some algebraic manipulations that this identity holds for all $f\in\sH$.

We first fix a place $x\in|X|$. Recall the Satake transform identifies $\sH_{x}=\QQ[h_{x}]$ with the subalgebra of $\QQ[t_{x}^{\pm1}]$ generated by $h_{x}=t_{x}+q_{x}t_{x}^{-1}$. For $n\geq0$,  we have $\Sat_{x}(h_{nx})=t^{n}_{x}+q_{x}t^{n-2}_{x}+\cdots+q^{n-1}_{x}t^{-n+2}_{x}+q^{n}_{x}t^{-n}_{x}$. 

\begin{lemma}\label{l:loc surj} Let $E$ be any field containing $\QQ$. Let $I$ be a nonzero ideal of $\sH_{x,E}:=\sH_{x}\otimes_{\QQ} E$ and let $m$ be a positive integer. Then $I+\Span_{E}\{h_{mx}, h_{(m+1)x}, \cdots\}=\sH_{x,E}$.
\end{lemma}
\begin{proof} Let $t=q^{-1/2}_{x}t_{x}$. Then $h_{nx}=q^{n/2}_{x}T_{n}$ where $T_{n}=t^{n}+t^{n-2}+\cdots+t^{2-n}+t^{-n}$ for any $n\geq 0$. It suffices to show that $I+\Span_{E}\{T_{m}, T_{m+1},\cdots\}=\sH_{x,E}$. 

Let $\pi: \sH_{x,E}\to \sH_{x,E}/I$ be the quotient map. Let $\sH_{m,E}\subset\sH_{x,E}$ be the $E$-span of $t^{n}+t^{-n}$ for $n\geq m$. Note that $T_{n}-T_{n-2}=t^{n}+t^{-n}$, therefore it suffices to show that $\pi(\sH_{m,E})=\sH_{x,E}/I$ for all $m$. To show this, it suffices to show the same statement after base change from $E$ to an algebraic closure $\ov{E}$. From now on we use the notation $\sH_{x}, I$ and $\sH_{m}$ to denote their base changes to $\ov{E}$. 

To show that $\pi(\sH_{m})=\sH_{x}/I$, we take any nonzero linear function $\ell:\sH_{x}/I\to \ov{E}$. We only need to show that $\ell(\pi(t^{n}+t^{-n}))\neq0$ for some $n\geq m$. We prove this by contradiction: suppose $\ell(\pi(t^{n}+t^{-n}))=0$ for all $n\geq m$.

Let $\nu:\Gm\to\BA^{1}=\Spec\sH_{x}$ be the morphism given by $t\mapsto T=t+t^{-1}$. This is the quotient by the involution $\sigma(t)=t^{-1}$. Consider the finite subscheme $Z=\Spec(\sH_{x}/I)$ and its preimage $\tilZ=\nu^{-1}(Z)$ in $\Gm$. We have $\calO_{Z}=\sH_{x}/I=\calO^{\sigma}_{\tilZ}\subset\calO_{\tilZ}$. One can uniquely extend $\ell$ to a $\sigma$-invariant linear function $\tl: \calO_{\tilZ}\to\ov{E}$. Note that $\calO_{\tilZ}$ is a product of the form $\ov{E}[t]/(t-z)^{d_{z}}$ for a finite set of points $z\in\ov{E}^{\times}$, and that $z\in\tilZ$ if and only if $\sigma(z)=z^{-1}\in\tilZ$. Any linear function $\tl$ on $\calO_{\tilZ}$, when pulled back to $\calO_{\Gm}=\ov{E}[t,t^{-1}]$, takes the form
\begin{equation*}
\ov{E}[t,t^{-1}]\ni f\mapsto \sum_{z\in\tilZ}(D_{z}f)(z)
\end{equation*}
with $D_{z}=\sum_{j\geq 0}c_{j}(z)(t \frac{d}{dt})^{j}$ (finitely many terms) a differential operator on $\Gm$ with constant coefficients $c_{j}(z)$ depending on $z$. The $\sigma$-invariance of $\tl$ is equivalent to
\begin{equation}\label{cj sym}
c_{j}(z)=(-1)^{j}c_{j}(z^{-1}),\quad\text{ for all $z\in\tilZ$ and $j$}.
\end{equation}
Evaluating at $f=t^{n}+t^{-n}$, we get that
\begin{equation*}
\ell(\pi(t^{n}+t^{-n}))=\sum_{z\in\tilZ}P_{z}(n)z^{n}+P_{z}(-n)z^{-n}
\end{equation*}
where $P_{z}(T)=\sum_{j}c_{j}(z)T^{j}\in \ov{E}[T]$ is a polynomial depending on $z$. The symmetry \eqref{cj sym} implies $P_{z}(T)=P_{z^{-1}}(-T)$. Using this symmetry, we may collect the terms corresponding to $z$ and $z^{-1}$ and re-organize the sum above as
\begin{equation*}
\ell(\pi(t^{n}+t^{-n}))=2\sum_{z\in\tilZ}P_{z}(n)z^{n}=0, \quad\text{ for all $n\geq m$}.
\end{equation*}
By linear independence of $\phi_{a,z}: n\mapsto n^{a}z^{n}$ as functions on $\{m,m+1,m+2,\cdots\}$, we see that all polynomials $P_{z}(T)$ are identically zero. Hence $\tl=0$ and $\ell=0$, which is a contradiction!
\end{proof}

\begin{theorem}\label{th:full IJ} For any $f\in\sH$, we have the identity
\begin{equation*}
(\log q)^{-r}\BJ_{r}(f)=\BI_{r}(f).
\end{equation*}
\end{theorem}
\begin{proof}
Let $\tsH$ be the image of $\sH\otimes\Ql$ in $\End_{\Ql}(V')\times \End_{\Ql}(\calA\otimes\Ql)\times \Ql[\Pic_{X}(k)]^{\iota_{\Pic}}$. \index{$\tsH$}%
Denote the quotient map $\sH\otimes\Ql\surj\tsH$ by $a$. Then for any $x\in |X|$, $\tsH\subset\End_{\sH_{x}\otimes\Ql}(V'\oplus\calA\otimes\Ql\oplus\Ql[\Pic_{X}(k)]^{\iota_{\Pic}})$. The latter being finitely generated over $\sH_{x}\otimes\Ql$ by Lemma \ref{l:fg} (or rather, the analogous assertion for $V'$), $\tsH$ is also a finitely generated $\sH_{x}\otimes\Ql$-module, and hence a finitely generated $\Ql$-algebra. Clearly for $f\in\sH$, $\BI_{r}(f)$ and $\BJ_{r}(f)$ only depend on the image of $f$ in $\tsH$. Let $\sH^{\dagger}\subset\sH$ be the linear span of the functions $h_{D}$ for effective divisors $D$ such that $\deg D\geq \max\{2g'-1,2g\}$. By Theorem \ref{th:IJ}, we have $(\log q)^{-r}\BJ_{r}(f)=\BI_{r}(f)$ for all $f\in\sH^{\dagger}$. Therefore it suffices to show that the composition $\sH^{\dagger}\otimes\Ql\to \sH\otimes\Ql\xrightarrow{a}\tsH$ is surjective.

Since $\tsH$ is finitely generated as an algebra, there exists a finite set $S\subset |X|$ such that $\{a(h_{x})\}_{x\in S}$ generate $\tsH$. We may enlarge $S$ and assume that $S$ contains all places with degree $\leq \max\{2g'-1,2g\}$. Let $y\in |X|-S$, then for any $f\in\sH_{S}=\otimes_{x\in S}\sH_{x}$, we have $fh_{y}\in\sH^{\dagger}$. Therefore  $a(\sH^{\dagger}\otimes\Ql)\supset a(\sH_{S}\otimes\Ql)a(h_{y})=\tsH a(h_{y})$. In other words, $a(\sH^{\dagger}\otimes\Ql)$ contains the ideal $I$ generated by the $a(h_{y})$ for $y\notin S$.  

We claim that the quotient $\tsH/I$ is finite-dimensional over $\Ql$. Since $\tsH$ is finitely generated over $\Ql$, it suffices to show that $\Spec(\tsH/I)$ is finite. Combining Theorem \ref{th:Spec decomp Sht'} and \eqref{Haut Ql}, $\Spec\tsH=\Spec\bsH'\cup\Spec\sH_{\aut,\Ql}=Z_{\Eis,\Ql}\cup Z'^{r}_{0,\ell}\cup Z^{0}_{0,\ell}$. Let $\sigma:\tsH/I\to \Qlbar$ be a $\Qlbar$-point of $\Spec(\tsH/I)$. If $\sigma$ lies in $Z_{\Eis,\Ql}$, then the composition $\sH\to\tsH/I\xrightarrow{\sigma}\Qlbar$ factors as $\sH\xrightarrow{\Sat}\QQ[\Pic_{X}(k)]\xrightarrow{\chi}\Qlbar$ for some character $\chi:\Pic_{X}(k)\to \Qlbar^{\times}$. Since $h_{y}$ vanishes in $\tsH/I$ for any $y\notin S$, we have $\chi(\Sat(h_{y}))=\chi(t_{y})+q_{y}\chi(t^{-1}_{y})=0$ for all $y\notin S$, which implies that $\chi(t_{y})=\pm(-q_{y})^{1/2}$ for all $y\notin S$. Let $\chi':\Pic_{X}(k)\to \Qlbar^{\times}$ be the character $\chi'=\chi\cdot q^{-\deg/2}$. Then $\chi'$ is a character with finite image satisfying $\chi'(t_{y})=\pm\sqrt{-1}$ for all but finitely $y$. This contradicts Chebotarev density since there should be a positive density of $y$ such that $\chi'(t_{y})=1$. Therefore $\Spec(\tsH/I)$ is disjoint from $Z_{\Eis,\Ql}$ hence $\Spec(\tsH/I)^{\red}\subset Z'^{r}_{0,\ell}\cup Z^{0}_{0,\ell}$, hence finite. 

Let $\ov{a}: \sH\otimes\Ql\xrightarrow{a}\tsH\to\tsH/I$ be the quotient map. For each $x\in |X|$, consider the surjective ring homomorphism $\sH_{x}\otimes\Ql\to \ov{a}(\sH_{x}\otimes\Ql)$. Note that $\sH^{\dagger}\cap \sH_{x}$ is spanned by elements of the form $h_{nx}$ for $n\deg(x)\geq\max\{2g'-1,2g\}$. Since $\ov{a}(\sH_{x}\otimes\Ql)\subset\tsH/I$ is finite-dimensional over $\Ql$, Lemma \ref{l:loc surj} implies that $(\sH^{\dagger}\cap\sH_{x})\otimes\Ql\to \ov{a}(\sH_{x}\otimes\Ql)$ is surjective.  Therefore $\ov{a}(\sH^{\dagger}\otimes\Ql)$ contains $\ov{a}(\sH_{x}\otimes\Ql)$ for all $x\in|X|$. Since $\ov{a}$ is surjective, $\ov{a}(\sH_{x}\otimes\Ql)$ (all $x\in|X|$) generate the image $\tsH/I$ as an algebra, hence $\ov{a}(\sH^{\dagger}\otimes\Ql)=\tsH/I$. Since $a(\sH^{\dagger}\otimes\Ql)$ already contains $I$, we conclude that $a(\sH^{\dagger}\otimes\Ql)=\tsH$. 
\end{proof}

\subsubsection{Proof of Theorem \ref{th:limit}}\label{proof:limit} Apply Theorem \ref{th:full IJ} to the unit function $h=\one_{K}$, we get
\begin{equation*}
(\theta^{\mu}_{*}[\Sht^{\mu}_{T}], \quad\theta^{\mu}_{*}[\Sht^{\mu}_{T}])_{\Sht'^{r}_{G}}=(\log q)^{-r}\JJ_{r}(\one_{K}).
\end{equation*}
We then apply Corollary \ref{c:one K} to write the RHS using the $r$-th derivative of $L(\eta,s)$, as desired.

\begin{remark}\label{rem limit r=0}
Let $r=0$. Note that $\Sht^{\mu}_{T}$, resp. $\Sht'^{r}_{G}$, is the constant groupoid $\Bun_T(k)$, resp. $\Bun_G(k)$.  We write $\theta_\ast[\Bun_{T}(k)]$ for $\theta^{\mu}_{*}[\Sht^{\mu}_{T}]$, as an element in $C_c^\infty(\Bun_{G}(k),\BQ)$. 
The analogous statement of Theorem \ref{th:limit} should be
\begin{equation}\label{limit r=0}
\jiao{\theta_\ast[\Bun_{T}(k)] ,\quad \theta_\ast[\Bun_{T}(k)] }_{\Bun_{G}(k)}=4L(\eta,0)+q-2.
\end{equation}
Here the left side $\jiao{-,-}_{\Bun_{G}(k)}$ is the inner product on $C_c^\infty(\Bun_{G}(k),\BQ)$ defined such that the characteristic functions $\{{\bf 1}_{[\CE]}\}_{\CE\in \Bun_G(k)}$ are orthogonal to each other and that
$$
\jiao{{\bf 1}_{[\CE]},{\bf 1}_{[\CE]}}_{\Bun_G(k)}=\frac{1}{\#\Aut(\CE)}.
$$
The equality \eqref{limit r=0} can be proved directly. We leave the detail to the reader.
\end{remark}

\subsection{Proof of Theorem \ref{th:main coho}}\label{proof:main coho}  The theorem was formulated as an equality in $E_{\pi,\l}$, but for the proof we shall extend scalars from $\Ql$ to $\Qlbar$, and use the decomposition \eqref{decomp V} instead. For any embedding $\iota:E_{\pi}\incl\Qlbar$, we have the point $\fkm(\pi,\iota)\in Z_{\cusp}(\Qlbar)$ corresponding to the homomorphism $\sH\xrightarrow{\l_{\pi}} E_{\pi}\xrightarrow{\iota}\Qlbar$. To prove the theorem, it suffices to showing that for all embeddings $\iota: E_{\pi}\incl\Qlbar$, we have an identity in $\Qlbar$
\begin{equation*}
\frac{|\omega_X|}{2(\log q)^{r}}\iota(\sL^{(r)}(\pi_{F'},1/2))=\Big([\Sht^{\mu}_{T}]_{\fkm(\pi,\iota)}, [\Sht^{\mu}_{T}]_{\fkm(\pi,\iota)}\Big)_{\Qlbar}
\end{equation*}
where  $(\cdot,\cdot)_{\Qlbar}$ is the $\Qlbar$-bilinear extension of the cup product pairing \eqref{cup prod V'} on $V'$. In other words, for any everywhere unramified cuspidal automorphic $\Qlbar$-representation $\pi$ of $G(\AA_{F})$ that corresponds to the homomorphism $\l_{\pi}: \sH_{\Qlbar}\to\Qlbar$, we need to show
\begin{equation}\label{iota identity}
\frac{|\omega_X|}{2(\log q)^{r}}\sL^{(r)}(\pi_{F'},1/2)=\Big([\Sht^{\mu}_{T}]_{\fkm_{\pi}}, [\Sht^{\mu}_{T}]_{\fkm_{\pi}}\Big)_{\Qlbar}
\end{equation}
where $\fkm_{\pi}=\ker(\l_{\pi})$ is the maximal ideal of $\sH_{\Qlbar}$, and $[\Sht^{\mu}_{T}]_{\fkm_{\pi}}$ is understood to be zero if $\fkm_{\pi}\notin Z'^{r}_{0,\ell}(\Qlbar)$.

As in the proof of Theorem \ref{th:full IJ}, let $\tsH$ be the image of $\sH_{\Ql}$ in $\End_{\Ql}(V')\times \End_{\Ql}(\calA\otimes\Ql)\times \Ql[\Pic_{X}(k)]^{\iota_{\Pic}}$. By Theorem \ref{th:Spec decomp Sht'} and \eqref{Haut Ql}, we may write $\Spec\tsH$ as a disjoint union of closed subsets
\begin{equation}\label{decomp Z}
\Spec\tsH^{\red}=Z_{\Eis,\Ql}\coprod \wt{Z}_{0,\ell}.
\end{equation}
where $\wt{Z}_{0,\ell}=Z'^{r}_{0,\ell}\cup Z^{0}_{0,\ell}$ is a finite collection of closed points. This gives a product decomposition of the ring $\tsH$
\begin{equation}\label{decomp tcH}
\tsH=\wt\sH_{\ell,\Eis}\times\wt\sH_{\ell,0}
\end{equation}
with $\Spec\wt\sH^{\red}_{\ell,\Eis}=Z_{\Eis,\Ql}$ and $\Spec\wt\sH^{\red}_{\ell,0}=\wt{Z}_{0,\ell}$. For any element $h\in\wt\sH_{\ell,0}$, we view it as the element $(0,h)\in\tsH$.  By Corollary \ref{c:decomp Ir} we have for any $h\in\wt\sH_{\ell,0}$
\begin{eqnarray}\label{simple I}
\BI_{r}(h)&=&\sum_{\fkm\in Z'^{r}_{0,\ell}(\Qlbar)}\Big([\Sht_{T}]_{\fkm}, h\ast [\Sht_{T}]_{\fkm}\Big)
\end{eqnarray}
Extending by linearity, the above formula also holds for all $h\in \wt\sH_{\ell,0}\otimes_{\Ql}\Qlbar$. Note that the linear function $h\mapsto ([\Sht_{T}]_{\fkm}, h\ast [\Sht_{T}]_{\fkm})$ on $\wt\sH_{\ell,0}\otimes\Qlbar$ factors through the localization $\wt\sH_{\ell,0}\otimes\Qlbar\to (\wt\sH_{\ell,0}\otimes\Qlbar)_{\fkm}$ (viewing $\fkm$ as a maximal ideal of $\wt\sH_{\ell,0}\otimes\Qlbar$).

On the other hand, let $\wt\calI_{\Eis}$ be the ideal of $\tsH$ generated by the image of $\calI_{\Eis}$. We have $(0,h)\in\wt\calI_{Eis}$. By Theorem \ref{th: coeff E}, we have for any $h\in\wt\sH_{\ell,0}\otimes\Qlbar$
\begin{eqnarray}\label{simple J}
\BJ_{r}(h)&=&\sum_{\pi\in Z_{\cusp}(\Qlbar)}\frac{d^{r}}{ds^{r}}\Big|_{s=0}\BJ_{\pi}((0,h))\\
&=&\sum_{\pi\in Z_{\cusp}(\Qlbar)}\frac{|\omega_X|}{2}\l_{\pi}(h)\sL^{(r)}(\pi_{F'},1/2).
\end{eqnarray}
By Lemma \ref{l:Haut decomp}(3), $Z_{\cusp}(\Qlbar)=Z^{0}_{0,\ell}(\Qlbar)$, hence can be viewed as a subset of $\wt{Z}_{0,\ell}$. Comparing the RHS of \eqref{simple I} and \eqref{simple J}, and using Theorem \ref{th:full IJ}, we get for any $h\in\wt\sH_{\ell,0}\otimes\Qlbar$
\begin{equation}\label{IJh}
\sum_{\fkm\in Z^{r}_{0,\ell}(\Qlbar)}\Big([\Sht_{T}]_{\fkm}, h\ast\Sht_{T}]_{\fkm}\Big)=\sum_{\pi\in Z_{\cusp}(\Qlbar)}\frac{|\omega_X|}{2(\log q)^{r}}\l_{\pi}(h)\sL^{(r)}(\pi_{F'},1/2).
\end{equation}

Since $\wt\sH_{\ell,0}\otimes\Qlbar$ is an artinian algebra, we have a canonical decomposition into local artinian algebras
\begin{equation}\label{decomp fkm}
\wt\sH_{\ell,0}\otimes\Qlbar\cong\prod_{\fkm\in\wt{Z}_{0,\ell}(\Qlbar)}(\wt\sH_{\ell,0}\otimes\Qlbar)_{\fkm}.
\end{equation}
As linear functions on $\wt\sH_{\ell,0}\otimes\Qlbar$, the $\fkm$-summand of the left side of \eqref{IJh} factors through $(\wt\sH_{\ell,0}\otimes\Qlbar)_{\fkm}$ while the $\pi$-summand of the right side of \eqref{IJh} factors through $(\wt\sH_{\ell,0}\otimes\Qlbar)_{\fkm_{\pi}}$. By the decomposition \eqref{decomp fkm}, we conclude that
\begin{itemize}
\item If $\fkm\in Z'^{r}_{0,\ell}(\Qlbar)-Z_{\cusp}(\Qlbar)$, then for any $h\in\wt\sH_{\ell,0}\otimes\Qlbar$,
\begin{equation*}
([\Sht_{T}]_{\fkm}, \quad h\ast [\Sht_{T}]_{\fkm})_{\Qlbar}=0.
\end{equation*}
\item If $\fkm\in Z'^{r}_{0,\ell}(\Qlbar)\cap Z_{\cusp}(\Qlbar)$, i.e., there is a (necessarily unique) $\pi\in Z_{\cusp}(\Qlbar)$ such that $\fkm=\fkm_{\pi}$, then  for any $h\in\wt\sH_{\ell,0}\otimes\Qlbar$ we have
\begin{equation*}
\l_{\pi}(h)\frac{|\omega_X|}{2(\log q)^{r}}\sL^{(r)}(\pi_{F'},1/2)=([\Sht_{T}]_{\fkm_{\pi}}, \quad h\ast[\Sht_{T}]_{\fkm_{\pi}})_{\Qlbar}.
\end{equation*}
In particular, taking $h=1$ we get
\begin{equation*}
\frac{|\omega_X|}{2(\log q)^{r}}\sL^{(r)}(\pi_{F'},1/2)=([\Sht_{T}]_{\fkm_{\pi}}, \quad [\Sht_{T}]_{\fkm_{\pi}})_{\Qlbar}.
\end{equation*}
\item If $\pi\in Z_{\cusp}(\Qlbar)-Z'^{r}_{0,\ell}(\Qlbar)$, then
\begin{equation*}
\sL^{(r)}(\pi_{F'},1/2)=0.
\end{equation*}
\end{itemize}
These together imply \eqref{iota identity}, which finishes the proof of Theorem \ref{th:main coho}.

\subsection{The Chow group version of the main theorem}\label{ss:ch}
In \S\ref{ss:intro cycle} we defined an $\sH$-module $W$ equipped with a perfect symmetric bilinear pairing $(\cdot,\cdot)$. Recall that $\tilW$ is the $\sH$-submodule of $\Ch_{c,r}(\Sht'^{r}_{G})_{\QQ}$ generated by $\theta^{\mu}_{*}[\Sht^{\mu}_{T}]$, and $W$ is by definition the quotient of $\tilW$ by the kernel $\tilW_{0}$ of the intersection pairing. \index{$W, \wt W_0, \wt W$}%

\begin{cor}[of Theorem \ref{th:full IJ}]
The action of $\sH$ on $W$ factors through $\sH_{\aut}$. In particular, $W$ is a cyclic $\sH_{\aut}$-module, hence finitely generated module over $\sH_{x}$ for any $x\in |X|$.
\end{cor}
\begin{proof}
Suppose $f\in\sH$ is in the kernel of $\sH\to \sH_{\aut}$, then $\JJ_{r}(f)=0$ hence $\II_{r}(f)=0$ by Theorem \ref{th:full IJ}. In particular, for any $h\in\sH$, we have $\JJ_{r}(hf)=0$. Therefore $\jiao{h*\theta^{\mu}_{*}[\Sht^{\mu}_{T}], f*\theta^{\mu}_{*}[\Sht^{\mu}_{T}]}=\II_{r}(hf)=0$. This implies that $f*\theta^{\mu}_{*}[\Sht^{\mu}_{T}]\in\tilW_{0}$, hence $f*\theta^{\mu}_{*}[\Sht^{\mu}_{T}]$ is zero in $W$, i.e., $f$ acts as zero on $W$. 
\end{proof}

\subsubsection{Proof of Theorem \ref{th:intro W decomp}}\label{proof:intro W decomp} By the decomposition \eqref{Haut Q}, we have an orthogonal decomposition 
\begin{equation*}
W=W_{\Eis}\oplus W_{\cusp}
\end{equation*}
with $\Supp(W_{\Eis})\subset Z_{\Eis}$ and $\Supp(W_{\cusp})\subset Z_{\cusp}$. Since $W_{\cusp}$ is a finitely generated $\sH_{\aut}$-module with finite support, it is finite-dimensional over $\QQ$. By Lemma \ref{l:Haut decomp},  $Z_{\cusp}$ is the set of unramified cuspidal automorphic representations in $\calA$, which implies the finer decomposition \eqref{intro W decomp}.
Since $W$ is a cyclic $\sH_{\aut}$-module, we have $\dim_{E_{\pi}}W_{\pi}\leq \dim_{E_{\pi}}\sH_{\aut,\pi}=1$ by the decomposition in Lemma \ref{l:Haut decomp}(1).

\subsubsection{Proof of Theorem \ref{th:main cycle}}\label{proof:main cycle} Pick any place $\l$ of $E_{\pi}$ over $\ell$, then by the compatibility of the intersection pairing and the cup product pairing under the cycle class map, we have
\begin{equation*}
([\Sht^{\mu}_{T}]_{\pi}, \quad[\Sht^{\mu}_{T}]_{\pi})_{\pi}=([\Sht^{\mu}_{T}]_{\pi,\l}, \quad[\Sht^{\mu}_{T}]_{\pi,\l})_{\pi,\l}
\end{equation*}
both as elements in the local field $E_{\pi,\l}$. Therefore Theorem \ref{th:main cycle} follows from Theorem \ref{th:main coho}.




\printindex


\appendix

\section{Results from intersection theory}\label{s:int}

In this appendix, we use Roman letters $X,Y,V,W,$ etc to denote algebraic stacks over a  field $k$. In particular, $X$ does not mean an algebraic curve. All algebraic stacks we consider are locally of finite type over $k$.

\subsection{Rational Chow groups for Deligne--Mumford stacks}\label{ss:proper cycles}

\subsubsection{Generalities about intersection theory on stacks} We refer to \cite{Kresch} for the definition of the Chow group $\Ch_{*}(X)$ of an algebraic stack $X$ over $k$.

For a Deligne--Mumford stack of finite type over $k$, the rational Chow group $\Ch_{*}(X)_{\QQ}$ can be defined in a more naive way using $\QQ$-coefficient cycles modulo rational equivalence, see \cite{Vistoli}.

\subsubsection{Chow group of proper cycles}  Let $X$ be a Deligne--Mumford stack locally of finite type over $k$. Let $Z_{c,i}(X)_{\QQ}$ denote the $\QQ$-vector space spanned by irreducible $i$-dimensional closed substacks $Z\subset X$ that are {\em proper over $k$}. Let $\Ch_{c,i}(X)_{\QQ}$ be the quotient of $Z_{c,i}(X)_{\QQ}$ modulo rational equivalence which comes from rational functions on cycles which are proper over $k$.  Equivalently, $\Ch_{c,i}(X)_{\QQ}=\varinjlim_{Y\subset X}\Ch_{i}(Y)_{\QQ}$ where $Y$ runs over closed substacks of $X$ that are proper over $k$, partially ordered by inclusion.

From the definition, we see that if $X$ is exhausted by open substacks $X_{1}\subset X_{2}\subset\cdots$, then we have
\begin{equation*}
\Ch_{c,i}(X)_{\QQ}\cong\varinjlim_{n}\Ch_{c,i}(X_{n})_{\QQ}.
\end{equation*}

\subsubsection{The degree map} When $X$ is a Deligne--Mumford stack, we have a degree map
\begin{equation*}
\deg: \Ch_{c,0}(X)_{\QQ}\to\QQ. 
\end{equation*}
Suppose  $x\in X$ is a closed point with residue field $k_{x}$ and automorphism group $\Aut(x)$ (a finite group scheme over $k_{x}$). Let $|\Aut(x)|_{k_{x}}$ be the order of $\Aut(x)$ as a finite group scheme over $k_{x}$. Let $[x]\in\Ch_{c,0}(X)_{\QQ}$ be the cycle class of the closed point $x$. Then 
\begin{equation*}
\deg([x])=[k_{x}:k]/|\Aut(x)|_{k_{x}}.
\end{equation*}

\subsubsection{Intersection pairing}\label{sss:int pairing} For the rest of \S\ref{ss:proper cycles}, we assume that $X$ is a smooth separated Deligne--Mumford stack,  locally of finite type over $k$ with pure dimension $n$. There is an intersection product
\begin{eqnarray*}
(-)\cdot_{X}(-): \Ch_{c,i}(X)_{\QQ}\times \Ch_{c,j}(X)_{\QQ}\to \Ch_{c,i+j-n}(X)_{\QQ}
\end{eqnarray*}
defined as follows. For closed substacks $Y_{1}$ and $Y_{2}$ of $X$ that are proper over $k$, the refined Gysin map attached to the regular local immersion $\Delta:X\to X\times X$ gives an intersection product
\begin{eqnarray*}
\Ch_{i}(Y_{1})_{\QQ}\times \Ch_{j}(Y_{2})_{\QQ}&\to& \Ch_{i+j-n}(Y_{1}\cap Y_{2})_{\QQ}\to \Ch_{c,i+j-n}(X)_{\QQ}\\
(\zeta_{1},\zeta_{2})&\mapsto & \Delta^{!}(\zeta_{1}\times\zeta_{2})
\end{eqnarray*}
Note that $Y_{1}\cap Y_{2}=Y_{1}\times_{X}Y_{2}\to Y_{1}$ is proper, hence $Y_{1}\cap Y_{2}$ is proper over $k$. Taking direct limits for $Y_{1}$ and $Y_{2}$, we get the intersection product on $\Ch_{c,*}(X)_{\QQ}$.

Composing with the degree map, we get an intersection pairing
\begin{equation}\label{int pairing Chc}
\jiao{\cdot,\cdot}_{X}:\Ch_{c,j}(X)_{\QQ}\times \Ch_{c,n-j}(X)_{\QQ}\to {\QQ}
\end{equation}
defined as
\begin{equation*}
\jiao{\zeta_{1},\zeta_{2}}_{X}=\deg(\zeta_{1}\cdot_{X}\zeta_{2}), \quad \zeta_{1}\in \Ch_{c,j}(X)_{\QQ}, \zeta_{2}\in\Ch_{c,n-j}(X)_{\QQ}.
\end{equation*}

\subsubsection{The cycle class map}\label{sss:cycle class map}  For any closed substack $Y\subset X$ that is proper over $k$, we have the usual cycle class map into the $\ell$-adic (Borel-Moore) homology of $Y$
\begin{equation*}
\cl_{Y}: \Ch_{j}(Y)_{\QQ}\to \hBM{2j}{Y\otimes_{k}\kbar,\Ql}(-j)\cong \homog{2j}{Y\otimes_{k}\kbar,\Ql}(-j).
\end{equation*}
Composing with the proper map $i:Y\incl X$ we get
\begin{equation}\label{cycle to homo}
\cl_{Y,X}:\Ch_{j}(Y)_{\QQ}\xrightarrow{\cl_{Y}}\homog{2j}{Y\otimes_{k}\kbar,\Ql}(-j)\xrightarrow{i_{*}}\homog{2j}{X\otimes_{k}\kbar,\Ql}(-j)\cong\cohoc{2n-2j}{X\otimes_{k}\kbar,\Ql}(n-j).
\end{equation}
where the last isomorphism is the Poincar\'e duality for $X$.  Taking inductive limit over all such proper $Y$, we get a cycle class map for proper cycles on $X$
\begin{equation*}
\cl_{X}: \Ch_{c,j}(X)_{\QQ}=\varinjlim_{Y}\Ch_{j}(Y)_{\QQ}\xrightarrow{\varinjlim \cl_{Y,X}}\cohoc{2n-2j}{X\otimes_{k}\kbar,\Ql}(n-j).
\end{equation*}
This map intertwines the intersection pairing \eqref{int pairing Chc} with the cup product pairing
\begin{equation*}
\cohoc{2j}{X\otimes_{k}\kbar,\Ql}(j)\times \cohoc{2n-2j}{X\otimes_{k}\kbar,\Ql}(n-j)\xrightarrow{\cup}\cohoc{2n}{X\otimes_{k}\kbar,\Ql}(n)\xrightarrow{\cap[X]}\Ql.
\end{equation*}

\subsubsection{A ring of correspondences}\label{sss:corr operator}
Let 
\begin{equation*}
\cCh_{n}(X\times X)_{\QQ}=\varinjlim_{Z\subset X\times X,\pr_{1}:Z\to X \text{ is proper}}\Ch_{n}(Z)_{\QQ}.
\end{equation*}

For closed substacks $Z_{1},Z_{2}\subset X\times X$ that are proper over $X$ via the first projections, we have a bilinear map
\begin{eqnarray*}
\Ch_{n}(Z_{1})_{\QQ}\times\Ch_{n}(Z_{2})_{\QQ} &\to & \Ch_{n}((Z_{1}\times X)\cap(X\times Z_{2}))_{\QQ}\xrightarrow{\pr_{13*}}\cCh_{n}(X\times X)_{\QQ}\\
(\rho_{1},\rho_{2}) &\mapsto & \rho_{1}*\rho_{2}:=\pr_{13*}\big((\rho_{1}\times[X])\cdot_{X^{3}} ([X]\times\rho_{2})\big).
\end{eqnarray*}
Note that $(Z_{1}\times X)\cap(X\times Z_{2})=Z_{1}\times_{\pr_{2}, X,\pr_{1}}Z_{2}$ is proper over $Z_{1}$, hence is proper over $X$ via the first projection. Taking direct limit over such $Z_{1}$ and $Z_{2}$, we get a convolutiton product
\begin{equation*}
(-)*(-): \cCh_{n}(X\times X)_{\QQ}\times\cCh_{n}(X\times X)_{\QQ}\to\cCh_{n}(X\times X)_{\QQ}.
\end{equation*}
This gives $\cCh_{n}(X\times X)_{\QQ}$ the structure of an associative $\QQ$-algebra.

For a closed substack $Z\subset X\times X$ such that $\pr_{1}$ is proper, and a closed substack $Y\subset X$ which is proper over $k$, we have a bilinear map
\begin{eqnarray*}
\Ch_{n}(Z)_{\QQ}\times\Ch_{i}(Y)_{\QQ}&\to&\Ch_{i}(Z\cap(Y\times X))_{\QQ}\xrightarrow{\pr_{2*}} \Ch_{c,i}(X)_{\QQ}\\
(\rho,\zeta) &\mapsto& \rho *\zeta:=\pr_{2*}\big(\rho\cdot_{X\times X}(\zeta\times[X])\big)
\end{eqnarray*}
Note here $Z\cap (Y\times X)=Z\times_{\pr_{1},X}Y$ is proper over $Y$, hence is itself proper over $k$.  Taking direct limit over such $Z$ and $Y$, we get a bilinear map
\begin{equation*}
\cCh_{n}(X\times X)_{\QQ}\times\Ch_{c,i}(X)_{\QQ}\to \Ch_{c,i}(X)_{\QQ}.
\end{equation*} 
This defines an action of the $\QQ$-algebra $\cCh_{n}(X\times X)_{\QQ}$ on $\Ch_{c,i}(X)_{\QQ}$.

\subsection{Graded $K'_{0}$ and Chow groups for Deligne--Mumford stacks}
\subsubsection{A naive filtration on $K'_{0}(X)_{\QQ}$} For an algebraic stack $X$ over $k$, let $\Coh(X)$ be the abelian category of coherent $\calO_{X}$-modules on $X$. Let $K'_{0}(X)$ denote the Grothendieck group of $\Coh(X)$.  

Let $\Coh(X)_{\leq n}$ be the full subcategory of coherent sheaves of $\calO_{X}$-modules with support dimension $\leq n$. We define $K'_{0}(X)^{\naive}_{\QQ,\leq n}$ to be the the image of $K_{0}(\Coh(X)_{\leq n})_{\QQ}\to K'_{0}(X)_{\QQ}$. They give an increasing filtration on $K'_{0}(X)_{\QQ}$. This is not yet the correct filtration to put on $K'_{0}(X)_{\QQ}$, but let us first review the case where $X$ is a scheme.

Let $X$ be a scheme of finite type over $k$. Recall from \cite[\S15.1.5]{Fulton} that there is a natural graded map $\phi_{X}: \Ch_{*}(X)_{\QQ}\to \Gr^{\naive}_{*}K'_{0}(X)_{\QQ}$ sending the class of an irreducible subvariety $V\subset X$ of dimension $n$ to the image of $\calO_{V}$ in $\Gr^{\naive}_{n}K'_{0}(X)_{\QQ}$. This map is in fact an isomorphism, with inverse $\psi_{X}: \Gr^{\naive}_{*}K'_{0}(X)_{\QQ}\to \Ch_{*}(X)_{\QQ}$ given by the leading term of the Riemann--Roch map $\tau_{X}: K'_{0}(X)_{\QQ}\to \Ch_{*}(X)_{\QQ}$. For details, see \cite[Theorem 18.3, and proof of Corollary 18.3.2]{Fulton}. These results also hold for algebraic spaces $X$ over $k$ by Gillet \cite{Gillet}.

\subsubsection{} A naive attempt to generalize the map $\psi_{X}$ to stacks is the following. Let $Z_{n}(X)_{\QQ}$ be the naive cycle group of $X$, namely the $\QQ$-vector space with a basis given by integral closed substacks $V\subset X$ of dimension $n$.

We define a linear map $\supp_{X}: K_{0}(\Coh(X)_{\leq n})_{\QQ}\to Z_{n}(X)_{\QQ}$ sending a coherent sheaf $\calF$ to $\sum_{V}m_{V}(\calF)[V]$, where $V$ runs over all integral substacks of $X$ of dimension $n$ and $m_{V}(\calF)$ is the length of $\calF$ at the generic point of $V$.

Clearly this map kills the image of $K_{0}(\Coh(X)_{\leq n-1})_{\QQ}$ but what is not clear is whether or not the composition $K_{0}(\Coh(X)_{\leq n})_{\QQ}\xrightarrow{\supp_{X}} Z_{n}(X)_{\QQ}\to \Ch_{n}(X)_{\QQ}$  factors through $K'_{0}(X)^{\naive}_{\QQ, \leq n}$. For this reason we will look for another filtration on $K'_{0}(X)_{\QQ}$. \footnote{Our definition in \S\ref{sss:fil} may still seem naive to experts, but it suffices for our applications. We wonder if there is a way to put a natural $\l$-structure on $K'_{0}(X)_{\QQ}$ when $X$ is a Deligne--Mumford stack, and then one may define a filtration on it using eigenvalues of the Adams operations.}

When $X$ is an algebraic space, the map $\supp_{X}$ does induce a map $\Gr^{\naive}_{n}K'_{0}(X)_{\QQ}\to \Ch_{n}(X)_{\QQ}$, and it is the same as the map $\psi_{X}$, the top term of the Riemann--Roch map.

\subsubsection{Another filtration on $K'_{0}(X)_{\QQ}$}\label{sss:fil} Now we define another filtration on $K'_{0}(X)_{\QQ}$ when $X$ is a Deligne--Mumford stack satisfying the following condition. 

\begin{defn}\label{def:ffp} Let $X$ be a Deligne--Mumford stack over $k$. A finite flat surjective map $U\to X$ from an algebraic space $U$ of finite type over $k$ is called a {\em finite flat presentation} of $X$. We say that $X$ admits a finite flat presentation if such a map $U\to X$ exists.
\end{defn}

We define $K'_{0}(X)_{\QQ,\leq n}$ to be the subset of elements $\alpha\in K'_{0}(X)_{\QQ}$ such that there exists a finite flat presentation $\pi: U\to X$ such that $\pi^{*}\alpha\in K'_{0}(U)^{\naive}_{\QQ,\leq n}$. 

We claim that $K'_{0}(X)_{\QQ,\leq n}$ is a $\QQ$-linear subspace of $K'_{0}(X)_{\QQ}$. In fact, for any two elements $\alpha_{1},\alpha_{2}\in K'_{0}(X)_{\QQ,\leq n}$, we find finite flat presentations $\pi_{i}:U_{i}\to X$ such that $\pi_{i}^{*}\alpha_{i}\in K'_{0}(U_{i})^{\naive}_{\QQ,\leq n}$ for $i=1,2$.  Then the pullback of the sum $\alpha_{1}+\alpha_{2}$ to the finite flat presentation $U_{1}\times_{X}U_{2}\to X$ lies in $K'_{0}(U_{1}\times_{X}U_{2})^{\naive}_{\QQ,\leq n}$. 

By this definition, $K'_{0}(X)_{\QQ,\leq n}$ may not be zero for $n<0$. For any negative $n$, $K'_{0}(X)_{\QQ,\leq n}$ consists of those classes that vanish when pulled back to some finite flat presentation $U\to X$.

\begin{lemma} When $X$ is an algebraic space of finite type over $k$, the filtration $K'_{0}(X)_{\QQ, \leq n}$ is the same as the naive one $K'_{0}(X)^{\naive}_{\QQ, \leq n}$.
\end{lemma}
\begin{proof}
To see this, it suffices to show that for a finite flat surjective map $\pi: U\to X$ of algebraic spaces over $k$, and an element $\alpha\in K'_{0}(X)_{\QQ}$, if $\pi^{*}\alpha\in K'_{0}(U)^{\naive}_{\QQ,\leq n}$, then $\alpha\in K'_{0}(X)^{\naive}_{\QQ,\leq n}$. In fact, suppose $\alpha\in K'_{0}(X)^{\naive}_{\QQ,\leq m}$ for some $m>n$, let $\alpha_{m}$ be its image in $\Gr^{\naive}_{m}K'_{0}(X)_{\QQ}$. Since the composition $\pi_{*}\pi^{*}:\Ch_{m}(X)_{\QQ}\to\Ch_{m}(U)_{\QQ}\to\Ch_{m}(X)_{\QQ}$ is the multiplication by $\deg(\pi)\neq0$ on each connected component, it is an isomorphism hence $\pi^{*}: \Ch_{m}(X)_{\QQ}\to\Ch_{m}(U)_{\QQ}$ is injective. By the compatibility between the isomorphism $\psi_{X}: \Gr^{\naive}_{m}K'_{0}(X)_{\QQ}\cong\Ch_{m}(X)_{\QQ}$ and flat pullback, the map $\pi^{*}: \Gr_{m}K'_{0}(X)_{\QQ}\to \Gr_{m}K'_{0}(U)_{\QQ}$ is also injective. Now $\pi^{*}(\alpha_{m})=0\in\Gr^{\naive}_{m}K'_{0}(U)_{\QQ}$ because $m>n$, we see that $\alpha_{m}=0$, i.e., $\alpha\in K'_{0}(X)^{\naive}_{\QQ,\leq m-1}$. Repeating the argument we see that $\alpha$ has to lie in $K'_{0}(X)^{\naive}_{\QQ,\leq n}$. 
\end{proof}

For a Deligne--Mumford stack $X$ that admits a finite flat presentation, we denote by $\Gr_{n}K'_{0}(X)_{\QQ}$ the associated graded of $K'_{0}(X)_{\QQ}$ with respect to the filtration $K'_{0}(X)_{\QQ,\leq n}$. We always have $K'_{0}(X)^{\naive}_{\QQ,\leq n}\subset K'_{0}(X)_{\QQ,\leq n}$, but the inclusion can be strict. For example, when $X$ is the classifying space of a finite group $G$, we have $K'_{0}(X)_{\QQ}=R_{k}(G)_{\QQ}$ is the $k$-representation ring of $G$ with $\QQ$-coefficients. Any element  $\alpha\in R_{k}(G)_{\QQ}$ with virtual dimension $0$ vanishes when pulled back along the finite flat map $\Spec k\to X$, therefore $K'_{0}(X)_{\QQ, \leq -1}\subset K'_{0}(X)_{\QQ}$ is the augmentation ideal of classes of virtual degree $0$, and $\Gr_{0}K'_{0}(X)_{\QQ}=\QQ$.

\subsubsection{Functoriality under flat pullback}\label{sss:flat pullback} The filtration $K'_{0}(X)_{\QQ,\leq n}$ is functorial under flat pullback. Suppose $f:X\to Y$ is a flat map of relative dimension $d$ between Deligne--Mumford stacks that admit finite flat presentations, then $f^{*}: K'_{0}(Y)_{\QQ}\to K'_{0}(X)_{\QQ}$ is defined.  Let $\alpha\in K'_{0}(Y)_{\QQ,\leq n}$. We claim that $f^{*}\alpha\in K'_{0}(X)_{\QQ,\leq n+d}$. In fact, choose a finite flat presentation $\pi: V\to Y$ such that $\pi^{*}\alpha\in K'_{0}(V)^{\naive}_{\QQ,\leq n}$. Let $W=V\times_{Y}X$, then $\pi': W\to X$ is representable, finite flat and surjective. Although $W$ itself may not be an algebraic space, we may take any finite flat presentation $\sigma: U\to X$ and let $U':=W\times_{X}U$. Then $U'$ is an algebraic space and $\xi: U'=W\times_{X}U\to X$ is a finite flat presentation.   The map $f':U'\to W\to V$ is flat of relative dimension $d$ between algebraic spaces, hence  $f'^{*}\pi^{*}\alpha\in K'_{0}(U')^{\naive}_{\QQ,\leq d+n}$. Since $f'^{*}\pi^{*}\alpha=\xi^{*}f^{*}\alpha$, we see that $f^{*}\alpha\in K'_{0}(X)_{\QQ,\leq n+d}$.

As a particular case of the above discussion, we have
\begin{lemma}\label{l:any pullback} Let $X$ be a Deligne--Mumford stack that admits a finite flat presentation. Let $\alpha\in K'_{0}(X)_{\QQ,\leq n}$. Then for any finite flat representable map $f: X'\to X$, where $X'$ is a Deligne--Mumford stack (which automatically admits a finite flat presentation), $f^{*}\alpha\in K'_{0}(X')_{\QQ,\leq n}$. 
\end{lemma}

\subsubsection{Functoriality under proper pushforward} The filtration $K'_{0}(X)_{\QQ,\leq n}$ is also functorial under proper representable pushforward. Suppose $f:X\to Y$ is a proper representable map of Deligne--Mumford stacks that admit finite flat presentations. Suppose $\alpha\in K'_{0}(X)_{\QQ,\leq n}$, we claim that $f_{*}\alpha\in K'_{0}(Y)_{\QQ,\leq n}$. Let $\pi: V\to Y$ be a finite flat presentation. Let $\sigma: U=X\times_{Y}V\to X$ be the corresponding finite flat presentation of $X$ ($U$ is an algebraic space because $f$ is representable). Then $f':U\to V$ is a proper map of algebraic spaces. By Lemma \ref{l:any pullback}, $\sigma^{*}\alpha\in K'_{0}(U)_{\QQ,\leq n}=K'_{0}(U)^{\naive}_{\QQ,\leq n}$, therefore $\pi^{*}f_{*}\alpha=f'_{*}\sigma^{*}\alpha\in K'_{0}(V)^{\naive}_{\QQ,\leq n}$, hence $f_{*}\alpha\in K'_{0}(Y)_{\QQ,\leq n}$.

\subsubsection{} For a Deligne--Mumford stack $X$ that admits a finite flat presentation, we now define a graded map $\psi_{X}: \Gr_{*}K'_{0}(X)_{\QQ}\to \Ch_{*}(X)_{\QQ}$ extending the same-named map for algebraic spaces $X$.

We may assume $X$ is connected for otherwise both sides break up into direct summands indexed by the connected components of $X$ and we can define $\psi_{X}$ for each component. Let  $\pi: U\to X$ be a finite flat presentation  of constant degree $d$. For $\alpha\in K'_{0}(X)_{\QQ,\leq n}$, we know from Lemma \ref{l:any pullback} that $\pi^{*}\alpha\in K'_{0}(U)_{\QQ,\leq n}$. Then we define
\begin{equation*}
\psi_{X}(\alpha):=\frac{1}{d}\pi_{*}\psi_{U}(\pi^{*}\alpha)\in \Ch_{n}(X)_{\QQ}.
\end{equation*}
It is easy to check that thus defined $\psi_{X}$ is independent of the choice of the finite flat presentation $U$ by dominating two finite flat presentations by their Cartesian product over $X$.

\subsubsection{} The definition of $\psi_{X}$ is compatible with the support map $\supp_{n}$ in the sense that the following diagram is commutative when $X$ is a Deligne--Mumford stack admitting a finite flat presentation
\begin{equation*}
\xymatrix{K_{0}(\Coh(X)_{\leq n})_{\QQ}\ar[r]\ar[d]^{\supp_{X}}& K'_{0}(X)^{\naive}_{\QQ,\leq n}\ar[r]  & K'_{0}(X)_{\QQ,\leq n}\ar[d]^{\psi_{X}}\\
Z_{n}(X)_{\QQ}\ar[rr] &&\Ch_{n}(X)_{\QQ}}
\end{equation*}

\subsubsection{Compatibility with the Gysin map}\label{sss:Gysin} We need a compatibility result of $\psi_{X}$ and the refined Gysin map. Consider a Cartesian diagram of algebraic stacks
\begin{equation}\label{Cart reg emb}
\xymatrix{X'\ar[r]^{f'}\ar[d]^{g} & Y'\ar[d]^{h}\\
X\ar[r]^{f} & Y}
\end{equation}
satisfying the following conditions
\begin{enumerate}
\item The stack $X'$ is a Deligne--Mumford stack that admits a finite flat presentation.
\item\label{global lci} The morphism $f$ can be factored as $X\xrightarrow{i} P\xrightarrow{p}Y$, where $i$ is a regular local immersion of pure codimension $e$ , and $p$ is a smooth relative Deligne--Mumford type morphism of pure relative dimension $e-d$. 
\end{enumerate}

\begin{remark}\label{r:source target smooth ok} 
Let $X$ and $Y$ be smooth Deligne--Mumford stacks, and $f:X\to Y$ is any morphism. Then we may factor $f$ as $X\xrightarrow{(\id, f)}X\times Y\xrightarrow{\pr_{Y}} Y$, which is the composition of a regular local immersion with a smooth morphism of Deligne--Mumford type. In this case any $f$ always satisfies the condition \eqref{global lci}. 
\end{remark}

\subsubsection{} In the situation of \S\ref{sss:Gysin}, the refined Gysin map \cite[Theorem 2.1.12(xi), and end of p.529]{Kresch} is defined
\begin{equation*} 
f^{!}: \Ch_{*}(Y')_{\QQ}\to \Ch_{*-d}(X')_{\QQ}.
\end{equation*}
We also have a map
\begin{equation}\label{pull K0}
f^{*}: K'_{0}(Y')\to K'_{0}(X')
\end{equation}
defined using derived pullback of coherent sheaves. Let $\calF$ be a coherent sheaf on $Y'$. Then the derived tensor product $f'^{-1}\calF\Ltimes_{(fg)^{-1}\calO_{Y}}g^{-1}\calO_{X}$ has cohomology sheaves only in a bounded range because for a regular local immersion it can be computed locally by a Koszul complex. Then the alternating sum
\begin{equation*}
f^{*}[\calF]=\sum_{i}(-1)^{i}[\Tor^{(fg)^{-1}\calO_{Y}}_{i}(f'^{-1}\calF,g^{-1}\calO_{X})]
\end{equation*}
is a well-defined element in $K'_{0}(X')$. We then extend this definition by linearity to obtain the map $f^{*}$ in \eqref{pull K0}.

\begin{prop}\label{p:AK} In the situation of \eqref{Cart reg emb}, assume all conditions in \S\ref{sss:Gysin} are satisfied. Let $n\geq0$ be an integer. We have
\begin{enumerate}
\item The map $f^{*}$ sends $K'_{0}(Y')^{\naive}_{\QQ,\leq n}$ to $K'_{0}(X')_{\QQ,\leq n-d}$, and hence induces\begin{equation*}
\Gr^{\naive}_{n}f^{*}: \Gr^{\naive}_{n}K'_{0}(Y')_{\QQ}\to\Gr_{n-d}K'_{0}(X')_{\QQ}.
\end{equation*}
\item The following diagram is commutative
\begin{equation}\label{supp psi}
\xymatrix{K_{0}(\Coh(Y')_{\leq n})_{\QQ}\ar[d]^{\supp_{Y'}}\ar[r] & \Gr^{\naive}_{n}K'_{0}(Y')_{\QQ}\ar[r]^{\Gr^{\naive}_{n}f^{*}} & \Gr_{n-d}K'_{0}(X')_{\QQ}\ar[d]^{\psi_{X'}}\\
Z_{n}(Y')_{\QQ}\ar[r] & \Ch_{n}(Y')_{\QQ} \ar[r]^{f^{!}} & \Ch_{n-d}(X')_{\QQ}}
\end{equation}
\item If $Y'$ is also a Deligne--Mumford stack that admits a finite flat presentation, then $f^{*}$ sends $K'_{0}(Y')_{\QQ,\leq n}$ to $K'_{0}(X')_{\QQ,\leq n-d}$, and we have a commutative diagram
\begin{equation*}
\xymatrix{\Gr_{n}K'_{0}(Y')_{\QQ}\ar[d]^{\psi_{Y'}}\ar[r]^{\Gr_{n}f^{*}} &  \Gr_{n-d}K'_{0}(X')_{\QQ}\ar[d]^{\psi_{X'}}\\
\Ch_{n}(Y')_{\QQ}\ar[r]^{f^{!}} & \Ch_{n-d}(X')_{\QQ}}
\end{equation*}
\end{enumerate}
\end{prop}
\begin{proof} 
(1) and (2). Write $f=p\circ i:X\xrightarrow{i}P\xrightarrow{p}Y$ as in condition \eqref{global lci} in \S\ref{sss:Gysin}. Let $P'=P\times_{Y}Y'$. For the smooth morphism $p$ of relative dimension $e-d$, $p^{*}$ sends $\Coh(Y')_{\leq n}$ to $\Coh(P')_{\leq n+e-d}$. Then we have a commutative diagram
\begin{equation}\label{YP}
\xymatrix{K'_{0}(\Coh(Y')_{\leq n})_{\QQ}\ar[d]^{\supp_{Y'}}\ar[r]^{p^{*}} & K'_{0}(\Coh(P')_{\leq n+e-d})_{\QQ}\ar[d]^{\supp_{P'}}\\
Z_{n}(Y')_{\QQ}\ar[r]^{p^{*}} & Z_{n+e-d}(P')_{\QQ}}
\end{equation}
Therefore to prove (1) and (2) we may replace $f:X\to Y$ with $i:X\to P$ hence reducing to the case $f$ is a regular local immersion of pure codimension $d$. 

Let $\alpha\in K_{0}(\Coh(Y')_{\leq n})_{\QQ}$. Then there exists a closed $n$-dimensional closed substack $Y''\subset Y'$ such that $\alpha$ is in the image of $K'_{0}(Y'')_{\QQ}$. We may replace $Y'$ with $Y''$ and replace $X'$ with  $X'':=X'\times_{Y'}Y''$. It suffices to prove the statements (1)(2) for $X''$ and $Y''$ for then we may pushforward along the closed immersion $X''\incl X'$ to get the desired statements for $X'$ and $Y'$. Therefore we may assume that $\dim Y'=n$. 

The construction of the deformation to the normal cone can be extended to our situation, see \cite[p.529]{Kresch}. Let $N_{f}$ be the normal bundle of the regular local immersion $f$. Then the normal cone $C_{X'}Y'$ for the morphism $f': X'\to Y'$ is a closed substack of $g^{*}N_{f}$.  We denote the total space of the deformation by $\Mc_{X'}Y'$. This is a stack over $\PP^{1}$ whose restriction to $\BA^{1}$ is $Y'\times\BA^{1}$ and whose fiber over $\infty$ is the normal cone $C_{X'}Y'$. Let $i_{\infty}: C_{X'}Y'\incl \Mc_{X'}Y'$ be the inclusion of the fiber over $\infty$. We have the specialization map for $K$-groups
\begin{equation*}
\Sp: K'_{0}(Y')_{\QQ}\xrightarrow{\pr^{*}_{Y'}} K'_{0}(Y'\times \BA^{1})_{\QQ}\cong K'_{0}(\Mc_{X'}Y')_{\QQ}/K'_{0}(C_{X'}Y')_{\QQ}\xrightarrow{i^{*}_{\infty}} K'_{0}(C_{X'}Y')_{\QQ}.
\end{equation*}
Similarly, we also have a specialization map for the naive cycle groups
\begin{equation*}
\Sp: Z_{n}(Y')_{\QQ}\xrightarrow{\pr^{*}_{Y'}} Z_{n+1}(Y'\times\BA^{1})_{\QQ}\isom Z_{n+1}(\Mc_{X'}Y')_{\QQ}\xrightarrow{i^{!}_{\infty}} Z_{n}(C_{X'}Y')_{\QQ}.
\end{equation*}
Here we are using the fact that $n=\dim Y'=\dim C_{X'}Y'=\dim \Mc_{X'}Y'-1$, and $Z_{*}(-)_{\QQ}$ is the naive cycle group. For any $n$-dimensional integral closed substack $V\subset Y'$, $\Sp([V])$ is the class of the cone $C_{X'\cap V}V\subset C_{X'}Y'$.

The diagram \eqref{supp psi} can be decomposed into two diagrams
\begin{equation*}
\xymatrix{K'_{0}(Y')_{\QQ}\ar[r]^{\Sp}\ar[d]^{\supp_{Y'}} & K'_{0}(C_{X'}Y')_{\QQ}\ar[d]^{\supp_{C_{X'}Y'}}\ar@{-->}[r]^{s^{*}} & K'_{0}(X')_{\QQ, \leq n-d}\ar[d]^{\psi_{X'}}\\
Z_{n}(Y')_{\QQ}\ar[r]^{\Sp} & Z_{n}(C_{X'}Y')_{\QQ}\ar[r]^{s^{!}} & \Ch_{n-d}(X')_{\QQ}}
\end{equation*}
The dotted arrow is conditional on showing that the image of $s^{*}$ lands $K'_{0}(X')_{\QQ, \leq n-d}$. The left square above is commutative: since we are checking an equality of top-dimensional cycles, we may pass to a smooth atlas and reduce the problem to the case of schemes for which the statement is easy.  Therefore it remains to show that the image of $s^{*}$ lands $K_{0}(X')_{\QQ, \leq n-d}$, and that the right square is commutative. Since $C_{X'}Y'\subset g^{*}N_{f}$, it suffices to replace $C_{X'}Y'$ by $g^{*}N_{f}$ and prove the same original statements (1) and (2), but without assuming that $\dim g^{*}N_{f}=n$. In other words, we have reduced the problem to the following special situation
\begin{eqnarray}\label{more special case}
&&\mbox{$X'=X$, $Y'=Y$ is a vector bundle of rank $d$ over $X$,}\\
\notag&&\mbox{$g=\id_{X}, h=\id_{Y}$ and $f=s$ is the inclusion of the zero section.}
\end{eqnarray}
In this case, let $\pi: U\to X$ be a finite flat presentation, let $Y_{U}$ be the vector bundle $Y$ base changed to $U$. Then $U$ and $Y_{U}$ are both algebraic spaces. Let $s_{U}: U\incl Y_{U}$ be the inclusion of the zero section and let $\sigma: Y_{U}\to Y$ be the projection. For any $\alpha\in K'_{0}(Y)^{\naive}_{\QQ,\leq n}$, we have $\pi^{*}s^{*}\alpha=s^{*}_{U}\sigma^{*}\alpha\in K'_{0}(U)_{\QQ}$.  We have $\sigma^{*}\alpha\in K'_{0}(Y_{U})^{\naive}_{\QQ,\leq n}$. In the case of the regular embedding of algebraic spaces $s_{U}:U\incl Y_{U}$,  $s_{U}^{*}$ sends $K'_{0}(Y_{U})_{\QQ,\leq n}$ to $K'_{0}(U)_{\QQ,\leq n-d}$ by the compatibility of the Riemann--Roch map with the Gysin map (\cite[Theorem 18.3(4)]{Fulton}). Therefore $\pi^{*}s^{*}\alpha=s^{*}_{U}\sigma^{*}\alpha\in K'_{0}(U)_{\QQ,\leq n-d}$, hence $s^{*}\alpha\in K'_{0}(X)_{\QQ,\leq n-d}$.

We finally check the commutativity of \eqref{supp psi} in the special case \eqref{more special case}. For any $\alpha\in K'_{0}(\Coh(Y)_{\leq n})_{\QQ}$, we need to check that $\delta=s^{!}\supp_{Y}(\alpha)-\psi_{X}(s^{*}\alpha)\in\Ch_{n-d}(X)_{\QQ}$
is zero. Since $\pi_{*}\pi^{*}:\Ch_{n-d}(X)_{\QQ}\to \Ch_{n-d}(U)_{\QQ}\to \Ch_{n-d}(X)_{\QQ}$ is the multiplication by $\deg(\pi)$ on each component of $X$, in particular it is an isomorphism and $\pi^{*}$ is injective. Therefore it suffices to check that $\pi^{*}\delta=0\in\Ch_{n-d}(U)_{\QQ}$. Since $\pi^{*}\delta=s^{!}_{U}\supp_{Y_{U}}(\sigma^{*}\alpha)-\psi_{U}s^{*}_{U}(\sigma^{*}\alpha)$, we reduce to the situation of $s_{U}: U\incl Y_{U}$, a regular embedding of algebraic spaces. In this case, the equality $s^{!}_{U}\supp_{U}=\psi_{U}s^{*}_{U}$ follows from the compatibility of the Riemann--Roch map with the Gysin map (\cite[Theorem 18.3(4)]{Fulton}).

(3) Let $\alpha\in K'_{0}(Y')_{\QQ,\leq n}$. Then for some finite flat presentation $\pi_{V}: V\to Y'$, $\pi^{*}_{V}\alpha\in K'_{0}(V)^{\naive}_{\QQ,\leq n}$. Let $W=X'\times_{Y'}V=X\times_{Y}V$, and let $f'': W\to V$ be the projection. Then we have a Cartesian diagram as in \eqref{Cart reg emb} with the top row replaced by $f'':W\to V$. Since $\pi_{W}: W\to X'$ is a finite flat surjective map ($W$ may not be an algebraic space because we are not assuming that $f$ is representable), $\pi^{*}_{W}: \Ch_{n-d}(X')_{\QQ}\to \Ch_{n-d}(W)_{\QQ}$ is injective. Therefore, in order to show that $f^{*}\alpha\in K'_{0}(X')_{\QQ, \leq n-d}$ and that $\psi_{X'}f^{*}\alpha-f^{!}\psi_{Y'}\alpha=0$ in $\Ch_{n-d}(X')_{\QQ}$, it suffices to show that $\pi^{*}_{W}f^{*}\alpha=f^{*}\pi_{V}^{*}\alpha\in K'_{0}(W)_{\QQ, \leq n-d}$ and that $\pi^{*}_{W}(\psi_{X'}f^{*}\alpha-f^{!}\psi_{Y'}\alpha)=\psi_{W}f^{*}(\pi^{*}_{V}\alpha)-f^{!}\psi_{V}(\pi^{*}_{V}\alpha)$ is zero in $\Ch_{n-d}(W)_{\QQ}$. Therefore we have reduced to the case where $Y'=V$ is an algebraic space. In this case $K'_{0}(Y')_{\QQ,\leq n}=K'_{0}(Y')^{\naive}_{\QQ,\leq n}$, and the statements  follows from (1)(2).
\end{proof}

By applying Proposition \ref{p:AK} to the diagonal map $X\to X\times X$ (and taking $g,h$ to be the identity maps), we get the following result, which is not used in the paper.
\begin{cor} Let $X$ be a smooth Deligne--Mumford stack that admits a finite flat presentation, then the map $\psi_{X}$ is a graded ring homomorphism.
\end{cor}

\subsubsection{The case of proper intersection}\label{sss:proper int} There is another situation where an analog of Proposition \ref{p:AK} can be easily proved. We consider a Cartesian diagram as in \eqref{Cart reg emb} satisfying the following conditions
\begin{enumerate}
\item $X'$ is a Deligne--Mumford stack, and $h$ (hence $g$) is representable.
\item\label{normal f} The normal cone stack of $f$ is a vector bundle stack (see \cite[Definition 1.9]{BF}) of some constant virtual rank $d$. 
\item\label{lci} There exists a commutative diagram
\begin{equation}\label{UV lci}
\xymatrix{U\ar[d]^{u}\ar@{^{(}->}[r]^{i} & V\ar[d]^{v}\\
X\ar[r]^{f} & Y}
\end{equation}
where $U$ and $V$ are schemes locally of finite type over $k$, $u$ and $v$ are smooth surjective and $i$ is a regular local immersion. 
\item\label{n n-d} We have $\dim Y'=n$ and $\dim X'=n-d$.
\end{enumerate}

\begin{remark}\label{r:smooth lci} Suppose  $X$ and $Y$ are smooth stacks over $k$. Pick any smooth surjective $W\surj Y$ where $W$ is a smooth scheme, and let $u:U\to X\times_{Y}W$ be any smooth surjective map from a smooth scheme $U$. Take $V=U\times W$, then $i=(\id, \pr_{W}\circ u):U\to V=U\times W$ is a regular local immersion. Therefore, in this case, $f$ satisfies the condition \eqref{lci} above.
\end{remark}

If $f$ satisfies the condition \eqref{normal f} above, the refined Gysin map is defined (see \cite[End of p.529 and footnote]{Kresch}). We only consider the top degree Gysin map
\begin{equation*}
f^{!}:\Ch_{n}(Y')_{\QQ}\to \Ch_{n-d}(X')_{\QQ}
\end{equation*}

On the other hand, derived pullback by $f^{*}$ gives
\begin{equation*}
f^{*}: K'_{0}(Y')\to K'_{0}(X')
\end{equation*}
as in \eqref{pull K0}. Here the boundedness of $\Tor$ can be checked by passing to a smooth cover of $X'$, and we may use the diagram \eqref{UV lci} to reduce to the case where $f$ is a regular local immersion,  where $\Tor$-boundedness can be proved by using the Koszul resolution.

\begin{lemma}\label{l:fund cycle} Under the assumptions of \S\ref{sss:proper int}, we have a commutative diagram
\begin{equation*}
\xymatrix{K'_{0}(Y')_{\QQ}\ar[d]_{\supp_{Y'}}\ar[r]^{f^{*}} & K'_{0}(X')_{\QQ}\ar[d]^{\supp_{X'}} \\
Z_{n}(Y')_{\QQ}\ar[r]^{f^{!}} & Z_{n-d}(X')_{\QQ}}
\end{equation*}
\end{lemma}
\begin{proof} The statement we would like to prove is an equality of top-dimensional cycles in $X'$. Such an equality can be checked after pulling back along a smooth surjective morphism $X''\to X'$. We shall use this observation to reduce the general case to the case where all members of the diagram are algebraic spaces and that $f$ is a regular embedding.

Let $i:U\to V$ be a regular local immersion of schemes as in Condition \eqref{lci} of \S\ref{sss:proper int} that covers $f:X\to Y$. By passing to connected components of $U$ and $V$, we may assume that the maps $u,v$ and $i$ in \eqref{UV lci} have pure (co)dimension. Let $U'=X'\times_{X}U$ and $V'=Y'\times_{Y}V$, then we have a diagram where all three squares and the outer square are Cartesian
\begin{equation*}
\xymatrix{X'\ar[d]^{g}\ar@/^{2pc}/[rrr]_{f'} & U'\ar[l]_{u'}\ar[r]^{i'}\ar[d] & V'\ar[d]\ar[r]^{v'} & Y'\ar[d]^{h}\\
X\ar@/_{2pc}/[rrr]^{f} & U\ar[l]_{u}\ar[r]^{i} & V\ar[r]^{v} & Y}
\end{equation*}
Let $\alpha\in K'_{0}(Y')$. To show $\supp_{X'}(f^{*}\alpha)-f^{!}\supp_{Y'}(\alpha)=0\in Z_{n-d}(X')_{\QQ}$, it suffices to show its pullback to $U'$ is zero. We have
\begin{eqnarray}\label{u pull}
u'^{*}(\supp_{X'}(f'^{*}\alpha)-f^{!}\supp_{Y'}(\alpha))&=&\supp_{U'}(u^{*}f^{*}\alpha)-u^{!}f^{!}\supp_{Y'}(\alpha)\\
\notag &=&\supp_{U'}(i^{*}v^{*}\alpha)-i^{!}v^{!}\supp_{Y'}(\alpha).
\end{eqnarray}
Since $v$ is smooth and representable, we have $v^{!}\supp_{Y'}(\alpha)=\supp_{V'}(v'^{*}\alpha)$. Letting $\beta=v'^{*}\alpha\in K'_{0}(V')$, we get
\begin{equation*}
\supp_{U'}(i^{*}v^{*}\alpha)-i^{!}v^{!}\supp_{Y'}(\alpha)=\supp_{U'}(i^{*}\beta)-i^{!}\supp_{V'}(\beta).
\end{equation*}
To show the LHS of \eqref{u pull} is zero, we only need to show that $\supp_{U'}(i^{*}\beta)-i^{!}\supp_{V'}(\beta)=0$. Therefore we have reduced to the following situation:
\begin{equation*}
\text{ $X$ and $Y$ are schemes and $f$ is a regular local immersion.}
\end{equation*}
In this case, $X'$ and $Y'$ are also algebraic spaces by the representability of $h$ and $g$. In this case we have $\supp_{X'}=\psi_{X'}$ and $\supp_{Y'}=\psi_{Y'}$. The identity $\supp_{X'}(f^{*}\alpha)=\psi_{X'}(f^{*}\alpha)=f^{!}\psi_{Y'}(\alpha)=f^{!}\supp_{Y'}(\alpha)$ follows from the compatibility of the Riemann--Roch map with the Gysin map (\cite[Theorem 18.3(4)]{Fulton}).
\end{proof}

\subsection{The octahedron lemma}\label{ss:oct} We consider the following commutative diagram of algebraic stacks over $k$
\begin{equation}\label{tian}
\xymatrix{A\ar[r]^{a}\ar[d] & X\ar[d] & B\ar[l]\ar[d]\\
U\ar[r] & S & V\ar[l]\\
C\ar[r]\ar[u] & Y\ar[u] & D\ar[l]\ar[u]^{d}}
\end{equation}
Now we enlarge this diagram by one row and one column in the following way:  we  form the fiber product of each row and place it in the corresponding entry of the rightmost column; we form the fiber product of each column and place it in the corresponding entry of the bottom row
\begin{equation}\label{large tian}
\xymatrix{A\ar[r]^{a}\ar[d] & X\ar[d] & B\ar[l]\ar[d] & A\times_{X}B\ar[d]\\
U\ar[r] & S & V\ar[l] & U\times_{S}V\\
C\ar[r]\ar[u] & Y\ar[u] & D\ar[l]\ar[u]^{d} & C\times_{Y}D\ar[u]^{\delta}\\
C\times_{U}A \ar[r]^{\alpha} & Y\times_{S}X & D\times_{V}B\ar[l] & N}
\end{equation}
Finally, the lower right corner $N$ is defined as the fiber product
\begin{equation}\label{north pole}
\xymatrix{N\ar[d]\ar[r] & A\times B\times C\times D\ar[d]\\
X\times_{S}Y\times_{S}U\times_{S}V\ar[r] & (X\times_{S}U)\times(X\times_{S}V)\times(Y\times_{S}U)\times(Y\times_{S}V)=:R}
\end{equation}

We now form the fiber product of the rightmost column of \eqref{large tian}
\begin{equation}\label{hor Cart}
\xymatrix{(C\times_{Y}D)\times_{(U\times_{S}V)} (A\times_{X}B) \ar[r]\ar[d] & A\times_{X}B\ar[d]\\
C\times_{Y}D\ar[r]^{\delta} & U\times_{S}V}
\end{equation}
and the fiber product of the bottom row of \eqref{large tian}
\begin{equation}\label{ver Cart}
\xymatrix{(C\times_{U}A)\times_{(Y\times_{S}X)}(D\times_{V}B)\ar[d]\ar[r] & D\times_{V}B\ar[d]\\
C\times_{U}A\ar[r]^{\alpha} & Y\times_{S}X}
\end{equation}

\begin{lemma}\label{l:same north pole}
There is a canonical isomorphism between $N$ and the stacks that appear in the upper left corners of \eqref{hor Cart} and \eqref{ver Cart}
\begin{equation}\label{ABCD}
(C\times_{Y}D)\times_{(U\times_{S}V)} (A\times_{X}B) \cong N\cong (C\times_{U}A)\times_{(Y\times_{S}X)}(D\times_{V}B)
\end{equation}
\end{lemma}
\begin{proof}
For the first isomorphism, we consider the diagram (to shorten notation, we use $\cdot$ instead of $\times$)
\begin{equation}\label{LHS and N}
\xymatrix@C=5pt{&(C\cdot_{Y}D)\cdot_{(U\cdot_{S}V)}(A\cdot_{X}B)\ar[d]\ar[r] & (C\cdot_{Y}D)\cdot(A\cdot_{X}B)\ar[d]\ar[r] & A\cdot B\cdot C\cdot D\ar[d]\\
&Y\cdot_{S}X\cdot_{S}U\cdot_{S}V\ar[r]\ar[dl] & (Y\cdot_{S}U\cdot_{S}V)\cdot(X\cdot_{S}U\cdot_{S}V)\ar[r]\ar[dl]\ar[d] & R\ar[d]\\
U\cdot_{S}V\ar[r]^{\Delta} & (U\cdot_{S}V)^{2} & Y\cdot X\ar[r]^{\Delta_{Y}\cdot\Delta_{X}} & Y^{2}\cdot X^{2}}
\end{equation}
Here all the squares are Cartesian. The upper two squares combined give the square in \eqref{north pole}. This shows that the LHS of \eqref{ABCD} is canonically isomorphic to $N$.

For the second isomorphism, we argue in the same way using the following diagram instead
\begin{equation}\label{RHS and N}
\xymatrix@C=5pt{&(C\cdot_{U}A)\cdot_{(Y\cdot_{S}X)}(D\cdot_{V}B)\ar[d]\ar[r] & (C\cdot_{U}A)\cdot(D\cdot_{V}B)\ar[d]\ar[r] & A\cdot B\cdot C\cdot D\ar[d]\\
&Y\cdot_{S}X\cdot_{S}U\cdot_{S}V\ar[r]\ar[dl] & (Y\cdot_{S}X\cdot_{S}U)\cdot(Y\cdot_{S}X\cdot_{S}V)\ar[r]\ar[dl]\ar[d] & R\ar[d]\\
Y\cdot_{S}X\ar[r]^{\Delta} & (Y\cdot_{S}X)^{2} & U\cdot V\ar[r]^{\Delta_{U}\cdot\Delta_{V}} & U^{2}\cdot V^{2}}
\end{equation}
\end{proof}

There is a way to label the vertices of the barycentric subdivision of an octahedron by the stacks introduced above. We consider an octahedron with a north pole, a south pole and a square as the equator. We put $S$ at the south pole. The four vertices of the equator are labelled with $A,B,D$ and $C$ clockwisely. The barycenters of the four lower faces are labeled by $U,V,X,Y$ so that their adjacency relation with the vertices labelled by $A,B,C,D$ is consistent with the diagram \eqref{tian}. At the barycenters of the four upper faces we put the fiber products: e.g., for the triangle with bottom edge labeled by $A,B$, we put $A\times_{X}B$ at the barycenter of this triangle. Finally we put $N$ at the north pole.

\begin{theorem}[The Octahedron Lemma]\label{th:oct}
Suppose we are given the commutative diagram \eqref{tian}. Suppose further that
\begin{enumerate}
\item The algebraic stacks $A,C,D,U,V,X,Y$ and $S$ (everybody except $B$, $B$ for bad) are smooth and  equidimensional over $k$. We denote $\dim A$ by $d_{A}$, etc.
\item The fiber products $U\times_{S}V$, $Y\times_{S}X$, $C\times_{Y}D$ and $C\times_{U}A$ have expected dimensions $d_{U}+d_{V}-d_{S}$, etc. 
\item Each of the Cartesian squares
\begin{eqnarray}\label{AB}
\xymatrix{A\times_{X}B\ar[d]\ar[r] & B\ar[d]\\
A\ar[r]^{a} & X}\\
\label{DB}
\xymatrix{D\times_{V}B\ar[d]\ar[r] & B\ar[d]\\
D\ar[r]^{d} & V}
\end{eqnarray}
satisfies either the conditions in \S\ref{sss:Gysin} or the conditions in \S\ref{sss:proper int}.
\item The Cartesian squares \eqref{hor Cart} and \eqref{ver Cart} satisfy the conditions in \S\ref{sss:Gysin}.
\end{enumerate}
Let $n=d_{A}+d_{B}+d_{C}+d_{D}-d_{U}-d_{V}-d_{X}-d_{Y}+d_{S}$.  Then  
\begin{equation*}
\delta^{!}a^{!}[B]=\alpha^{!}d^{!}[B]\in\Ch_{n}(N).
\end{equation*}
\end{theorem}

\begin{proof} Since $U, S$ and $V$ are smooth and pure dimensional, and $U\times_{S}V$ has the expected dimension, it is a local complete intersection and we have
\begin{equation*}
\calO_{U\times_{S}V}\cong \calO_{U}\Ltimes_{\calO_{S}}\calO_{V}
\end{equation*}
Here we implicitly pullback the sheaves $\calO_{U},\calO_{V}$ and $\calO_{S}$ to $U\times_{S}V$ using the plain sheaf pullback. Similar argument shows that the usual structure sheaves $\calO_{Y\times_{S}X}, \calO_{C\times_{Y}D}$ and $\calO_{C\times_{U}A}$ coincide with the corresponding derived tensor products.

We now show a derived version of the isomorphism \eqref{ABCD}. We equip each member of the diagrams \eqref{hor Cart}, \eqref{ver Cart} and \eqref{north pole} with the derived structure sheaves, starting from the usual structure sheaves of $A,B,C,D,X,Y,U,V$ and $S$. For $N$, we use \eqref{north pole} to equip it with the derived structure sheaf
\begin{equation*}
\calO^{\der}_{N}:=(\calO_{X}\Ltimes_{\calO_{S}}\calO_{Y}\Ltimes_{\calO_{S}}\calO_{U}\Ltimes_{\calO_{S}}\calO_{V})\Ltimes_{\calO^{\der}_{R}}(\calO_{A}\boxtimes\calO_{B}\boxtimes\calO_{C}\boxtimes\calO_{D})
\end{equation*}
where $\calO^{\der}_{R}$ is the derived structure sheaf $(\calO_{X}\Ltimes_{\calO_{S}}\calO_{U})\boxtimes\cdots\boxtimes(\calO_{Y}\Ltimes_{\calO_{S}}\calO_{V})$ on $R=(X\times_{S}U)\times\cdots\times(Y\times_{S}V)$. To make sense of this derived tensor product over $\calO^{\der}_{R}$, we need to work with dg categories of coherent complexes rather than the derived category. 

We claim that under the isomorphisms between both sides of \eqref{ABCD} and $N$, their derived structure sheaves are also quasi-isomorphic to each other. In fact we simply put derived structures sheaves on each vertex of the diagram \eqref{LHS and N}. Since the upper two squares combined give the square in \eqref{north pole}, transitivity of the derived tensor product gives a quasi-isomorphism
\begin{equation}\label{N hor}
\calO^{\der}_{N}\cong (\calO_{C}\Ltimes_{\calO_{Y}}\calO_{D})\Ltimes_{(\calO_{U}\Ltimes_{\calO_{S}}\calO_{V})} (\calO_{A}\Ltimes_{\calO_{X}}\calO_{B})
\end{equation}
Similarly, by considering the diagram \eqref{RHS and N}, we get a quasi-isomorphism
\begin{equation}\label{N ver}
\calO^{\der}_{N}\cong (\calO_{C}\Ltimes_{\calO_{U}}\calO_{A})\Ltimes_{(\calO_{Y}\Ltimes_{\calO_{S}}\calO_{X})}(\calO_{D}\Ltimes_{\calO_{V}}\calO_{B})
\end{equation}
Combing the isomorphisms \eqref{N hor} and \eqref{N ver}, and using the fact that $U\times_{S}V$, $Y\times_{S}X$, $C\times_{Y}D$ and $C\times_{U}A$ need not be derived, we get an isomorphism of coherent complexes on $N$
\begin{equation*}
\calO_{C\times_{Y}D}\Ltimes_{\calO_{U\times_{S}V}} (\calO_{A}\Ltimes_{\calO_{X}}\calO_{B}) \cong  
\calO_{C\times_{U}A}\Ltimes_{\calO_{Y\times_{S}X}}(\calO_{D}\Ltimes_{\calO_{V}}\calO_{B})
\end{equation*}
These are bounded complexes because the diagrams \eqref{AB}, \eqref{DB}, \eqref{hor Cart} and \eqref{ver Cart} satisfy the conditions in \S\ref{sss:Gysin} or \S\ref{sss:proper int}. Taking classes in $K'_{0}(N)_{\QQ}$ we get
\begin{equation}\label{class in K0}
\delta^{*}a^{*}\calO_{B}=\alpha^{*}d^{*}\calO_{B}\in K'_{0}(N)_{\QQ}
\end{equation}
Here $a^{*},d^{*},\alpha^{*}$ and $\delta^{*}$ are the derived pullbacks maps between $K'_{0}$-groups defined using the relevant Cartesian diagrams.
Now we apply Proposition \ref{p:AK} to the diagrams \eqref{AB}, \eqref{DB}, \eqref{hor Cart} and \eqref{ver Cart}, to conclude that both sides of \eqref{class in K0} lie in $K'_{0}(N)_{\QQ,\leq n}$ (where $n$ is the expected dimension of $N$). In case \eqref{AB} or \eqref{DB} satisfies \S\ref{sss:proper int} instead of \S\ref{sss:Gysin}, the corresponding statement $K'_{0}(B)_{\QQ,\leq d_{B}}\to K'_{0}(A\times_{X}B)_{\QQ, \leq d_{A}+d_{B}-d_{X}}$ or  $K'_{0}(B)_{\QQ,\leq d_{B}}\to K'_{0}(D\times_{V}B)_{\QQ, \leq d_{D}+d_{B}-d_{V}}$ is automatic for dimension reasons. 

Now we finish the proof. We treat only the case where \eqref{AB} satisfies the conditions in \S\ref{sss:Gysin} and \eqref{DB} satisfies the conditions in \S\ref{DB}. This is the case which we actually use in the main body of the paper, and the other cases can be treated in the same way.

Let $\ov{\delta^{*}a^{*}\calO_{B}}$ and $\ov{\alpha^{*}d^{*}\calO_{B}}$ denote their images in $\Gr_{n}K'_{0}(N)_{\QQ}$. Similarly we let $\ov{a^{*}\calO_{B}}\in \Gr_{d_{A}+d_{B}-d_{X}}K'_{0}(A\times_{X}B)_{\QQ}$ be the images of $a^{*}\calO_{B}$. Applying Proposition \ref{p:AK}  three times and Lemma \ref{l:fund cycle} once we get
\begin{eqnarray*}
&&\delta^{!}a^{!}[B]=\delta^{!}a^{!}\supp_{B}(\calO_{B})\\
&=&\delta^{!}\psi_{A\times_{X}B}(\ov{a^{*}\calO_{B}})\quad \mbox{(Prop \ref{p:AK}(2) applied to \eqref{AB})}\\
&=&\psi_{N}(\ov{\delta^{*}a^{*}\calO_{B}})\quad \mbox{(Prop \ref{p:AK}(3) applied to \eqref{hor Cart})}\\
&=&\psi_{N}(\ov{\alpha^{*}d^{*}\calO_{B}})\quad \eqref{class in K0}\\
&=&\alpha^{!}\supp_{D\times_{V}B}(d^{*}\calO_{B})\quad \mbox{(Prop \ref{p:AK}(2) applied to \eqref{ver Cart})}\\
&=&\alpha^{!}d^{!}\supp_{B}(\calO_{B})\quad \mbox{(Lemma \ref{l:fund cycle} applied to \eqref{DB})}\\
&=&\alpha^{!}d^{!}[B].
\end{eqnarray*}
\end{proof}

\subsection{A Lefschetz trace formula}
In this subsection, we will assume:
\begin{itemize}
\item {\em All sheaf-theoretic functors are derived functors.}
\end{itemize}

\subsubsection{Cohomological correspondences}\label{sss:corr} We first review some basic definitions and properties of cohomological correspondences following \cite{Var}. Consider a diagram of algebraic stacks over $k$
\begin{equation}\label{CXY}
\xymatrix{X& C\ar[l]_{\oll{c}}\ar[r]^{\orr{c}} & Y}
\end{equation}
We call $C$ together with the maps $\oll{c}$ and $\orr{c}$ a correspondence between $X$ and $Y$.

Let $\calF\in D^{b}_{c}(X)$  and $\calG\in D^{b}_{c}(Y)$ be $\Ql$-complexes of sheaves. A {\em cohomological correspondence between $\calF$ and $\calG$ supported on $C$} is a map
\begin{equation*}
\zeta:\oll{c}^{*}\calF\to\orr{c}^{!}\calG
\end{equation*}
in $D^{b}_{c}(C)$.

Suppose we have a map of correspondences
\begin{equation*}
\xymatrix{X\ar[d]^{f}& C\ar[l]_{\oll{c}}\ar[r]^{\orr{c}}\ar[d]^{h} & Y\ar[d]^{g}\\
S& B\ar[l]_{\oll{b}}\ar[r]^{\orr{b}} & T}
\end{equation*}
where $\oll{c}$ and $\oll{b}$ are proper, then we have an induced map between the group of cohomological correspondences supported on $C$ and on $B$ (see \cite[\S1.1.6(a)]{Var})
\begin{equation*}
h_{!}: \Hom_{C}(\oll{c}^{*}\calF, \orr{c}^{!}\calG)\to \Hom_{B}(\oll{b}^{*}f_{!}\calF, \orr{b}^{!}g_{!}\calG).
\end{equation*}
In particular, if $S=B=T$ and $\oll{b}=\orr{b}=\id_{S}$, then $\zeta\in\Hom_{C}(\oll{c}^{*}\calF, \orr{c}^{!}\calG)$ induces a map $h_{!}\zeta$ between $f_{!}\calF$ and $g_{!}\calG$ given by the composition
\begin{equation}\label{push coho corr}
h_{!}\zeta: f_{!}\calF\to  f_{!}\oll{c}_{!}\oll{c}^{*}\calF\xrightarrow{f_{!}\oll{c}_{!}(\zeta)}f_{!}\oll{c}_{!}\orr{c}^{!}\calG=g_{!}\orr{c}_{!}\orr{c}^{!}\calG\to g_{!}\calG.
\end{equation}

When $S=T$, $B$ the diagonal of $S$, $X=Y$ and $f=g$, we call $C$ a self-correspondence of $X$ over $S$. In this case, for a  cohomological correspondence $\zeta$ between $\calF$ and $\calG$ supported on $C$, we also use $f_{!}\zeta$ to denote $h_{!}\zeta\in \Hom_{S}(f_{!}\calF, f_{!}\calG)$ defined above.

\subsubsection{Fixed locus and the trace map}\label{sss:self M} Suppose in the diagram \eqref{CXY} we have $X=Y$. We denote $X$ by $M$. Define the fixed point locus $\Fix(C)$ of $C$ by the Cartesian diagram
\begin{equation*}
\xymatrix{\Fix(C)\ar[r]\ar[d] &  C\ar[d]^{(\oll{c},\orr{c})}\\
M\ar[r]^{\Delta} & M\times M}
\end{equation*}  
For any $\calF\in D^{b}_{c}(M)$,  there is a natural trace map (see \cite[Eqn(1.2)]{Var})
\begin{equation*}
\tau_{C}: \Hom(\oll{c}^{*}\calF,\orr{c}^{!}\calF)\to \hBM{0}{\Fix(C)\otimes_{k}\kbar}.
\end{equation*}
In other words, for a cohomological self-correspondence $\zeta$ of $\calF$ supported on $C$, there is a well-defined Borel-Moore homology class $\tau_{C}(\zeta)\in \hBM{0}{\Fix(C)\otimes_{k}\kbar}$.

\subsubsection{}\label{sss:corr M smooth} In the situation of \S\ref{sss:self M}, we further assume that both $C$ and $M$ are Deligne--Mumford stacks, $M$ is smooth and separated over $k$ of pure dimension $n$, and that $\calF=\const{M}$ is the constant sheaf on $M$.

Using Poincar\'e duality for $M$, a cohomological self-correspondence of the constant sheaf $\const{M}$ supported on $C$ is the same as a map
\begin{equation*}
\const{C}=\oll{c}^{*}\const{M}\to \orr{c}^{!}\const{M}\cong \orr{c}^{!}\DD_{M}[-2n](-n)\cong \DD_{C}[-2n](-n).
\end{equation*}
Over $C\otimes_{k}\kbar$, this is the same thing as an element in $\hBM{2n}{C\otimes_{k}\kbar}(-n)$. In this case, the trace map $\tau_{C}$ becomes the map
\begin{equation*}
\ov\tau_{C}: \hBM{2n}{C\otimes_{k}\kbar}(-n)\to \hBM{0}{\Fix(C)\otimes_{k}\kbar}.
\end{equation*}

On the other hand, we have the cycle class map
\begin{equation*}
\cl_{C}: \Ch_{n}(C)_{\QQ}\to\hBM{2n}{C\otimes_{k}\kbar}(-n)=\Hom(\oll{c}^{*}\const{M},\orr{c}^{!}\const{M}).
\end{equation*}
Therefore, any cycle $\zeta\in \Ch_{n}(C)_{\QQ}$ gives a cohomological self-correspondence of the constant sheaf $\const{M}$ supported on $C$. We will use the same notation $\zeta$ to denote the cohomological self-correspondence induced by it. 
Since $\Delta_{M}:M\to M\times M$ is a regular local immersion of pure codimension $n$, we have the refined Gysin map
\begin{equation*}
\Delta_{M}^{!}:\Ch_{n}(C)_{\QQ}\to \Ch_{0}(\Fix(C))_{\QQ}.
\end{equation*}

\begin{lemma}\label{l:cl Gysin} Under the assumptions of \S\ref{sss:corr M smooth}, we have a commutative diagram
\begin{equation*}
\xymatrix{\Ch_{n}(C)_{\QQ}\ar[d]^{\cl_{C}}\ar[r]^{\Delta_{M}^{!}} &
\Ch_{0}(\Fix(C))_{\QQ}\ar[d]^{\cl_{\Fix(C)}}\\
\hBM{2n}{C\otimes_{k}\kbar}(-n)\ar[r]^{\ov\tau_{C}}& \hBM{0}{\Fix(C)\otimes_{k}\kbar}}
\end{equation*}
\end{lemma}
\begin{proof} Let us base change to $\kbar$ and keep the same notation for $M,C$ etc. Tracing through the definition of $\tau_{C}$, we see that it is the same as the cap product with the relative cycle class of $\Delta(M)$ in $\cohog{2n}{M\times M, M\times M-\Delta(M)}(n)$. Then the lemma follows from \cite[Theorem 19.2]{Fulton}. Note that  \cite[Theorem 19.2]{Fulton} is for schemes over $\CC$ but the argument there works in our situation as well, using the construction of the deformation to the normal cone for Deligne-Mumford stacks in \cite[p.489]{Kresch2}.
\end{proof}

\subsubsection{Intersection with the graph of Frobenius}\label{sss:corr Frob} Suppose we are given a self-correspondence $C$ of $M$ over $S$
\begin{equation*}
\xymatrix{& C\ar[dd]^{h}\ar[dl]_{\oll{c}}\ar[dr]^{\orr{c}} \\
M\ar[dr]_{f} & & M\ar[dl]^{f}\\
&S}
\end{equation*}
satisfying
\begin{itemize}
\item $k$ is a finite field;
\item $S$ is a scheme over $k$;
\item $M$ is a smooth and separated Deligne--Mumford stack over $k$ of pure dimension $n$;
\item $f:M\to S$ is proper;
\item $\oll{c}:C\to M$ is representable and proper.
\end{itemize}

We define $\Sht_{C}$ by the Cartesian diagram
\begin{equation}\label{ShtC}
\xymatrix{\Sht_{C}\ar[r]\ar[d] &  C\ar[d]^{(\oll{c},\orr{c})}\\
M\ar[r]^{(\id,\Fr_{M})} & M\times M}
\end{equation}
Here the notation $\Sht_{C}$ suggests that in applications $\Sht_{C}$ will be a kind of moduli of Shtukas. We denote the image of the fundamental class $[M]$ under $(\id,\Fr_{M})_{*}$ by $\Gamma(\Fr_{M})$. Since $(\id,\Fr_{M})$ is a regular immersion of pure codimension $n$,  the refined Gysin map
\begin{equation*}
(\id,\Fr_{M})^{!}: \Ch_{n}(C)_{\QQ}\to \Ch_{0}(\Sht_{C})_{\QQ}
\end{equation*}
is defined. In particular, for $\zeta\in\Ch_{n}(C)_{\QQ}$, we get a $0$-cycle
\begin{equation*}
(\id,\Fr_{M})^{!}\zeta\in \Ch_{0}(\Sht_{C})_{\QQ}.
\end{equation*}

\subsubsection{}\label{sss:decomp ShtC} Since $C\to M\times_{S}M$, while $(\id,\Fr_{M}):M\to M\times M$ covers the similar map $(\id,\Fr_{S}):S\to S\times S$, the map $\Sht_{C}\to S$ factors through the discrete set $S(k)$, viewed as a discrete closed subscheme of $S$. Since $\Sht_{C}\to S(k)$, we get a decomposition of $\Sht_{C}$ into open and closed subschemes
\begin{equation*}
\Sht_{C}=\coprod_{s\in S(k)}\Sht_{C}(s).
\end{equation*} 
Therefore
\begin{equation*}
\Ch_{0}(\Sht_{C})_{\QQ}=\bigoplus_{s\in S(k)}\Ch_{0}(\Sht_{C}(s))_{\QQ}.
\end{equation*}
For $\zeta\in\Ch_{n}(C)_{\QQ}$, the $0$-cycle $\zeta\cdot _{M\times M}\Gamma(\Fr_{M})$ can be written uniquely as the sum of $0$-cycles
\begin{equation}\label{s component of ShtC}
((\id,\Fr_{M})^{!}\zeta)_{s}\in \Ch_{0}(\Sht_{C}(s))_{\QQ}, \quad\forall s\in S(k).
\end{equation}
Each $\Sht_{C}(s)=\Gamma(\Fr_{M_{s}})\times_{M_{s}\times M_{s}}C_{s}$. Since $\oll{c}: C_{s}\to M_{s}$ is proper and $M_{s}$ is separated (because $f$ is proper), $C_{s}\to M_{s}\times M_{s}$ is proper, therefore $\Sht_{C}(s)$ is proper over $\Gamma(\Fr_{M_{s}})$, hence it is itself proper over $k$ because $\Gamma(\Fr_{M_{s}})\cong M_{s}$ is proper over $k$. Therefore the degree map $\deg: \Ch_{0}(\Sht_{C}(s))_{\QQ}\to \QQ$ is defined, we get an intersection number indexed by $s\in S(k)$:
\begin{equation*}
\jiao{\zeta,\Gamma(\Fr_{M})}_{s}:=\deg((\id,\Fr_{M})^{!}\zeta)_{s}\in\QQ.
\end{equation*}

The main result of this subsection is the following.
 
\begin{prop}\label{p:Fix} Assume all conditions in \S\ref{sss:corr Frob} are satisfied. Let $\zeta\in\Ch_{n}(C)_{\QQ}$. Then for all $s\in S(k)$, we have 
\begin{equation}\label{Lef fix}
\jiao{\zeta, \Gamma(\Fr_{M})}_{s}=\Tr\left((f_{!}\cl_{C}(\zeta))_{s}\circ\Frob_{s},(f_{!}\Ql)_{\ov{s}}\right).
\end{equation}
Here $f_{!}\cl_{C}(\zeta):=h_{!}\cl_{C}(\zeta)$ is the endomorphism of $f_{!}\Ql$ induced by the cohomological correspondence $\cl_{C}(\zeta)$ supported on $C$, and $(f_{!}\cl_{C}(\zeta))_{s}$ is its action on the geometric stalk $(f_{!}\Ql)_{\ov{s}}$.
\end{prop}
\begin{proof}
Let $'C=C$ but viewed as a self-correspondence of $M$ via the following maps
\begin{equation*}
\xymatrix{M && {'C}=C\ar[ll]_{\oll{'c}=\Fr_{M}\circ\oll{c}}\ar[rr]^{\orr{'c}=\orr{c}} && M}
\end{equation*}
However, $'C$ is no longer a self-correspondence of $M$ over $S$. Instead, it maps to the Frobenius graph of $S$:
\begin{equation}\label{MS}
\xymatrix{M\ar[d]^{f} & {'C}\ar[d]^{'h}\ar[l]_{\oll{'c}}\ar[r]^{\orr{c}} & M\ar[d]^{f}\\
S & {'S}\ar[l]_{\Fr_{S}}\ar[r]^{\id} & S}
\end{equation}
Here ${'S}=S$ but viewed as a self-correspondence of $S$ via $(\Fr_{S},\id):{'S}\to S\times S$. The map $'h: {'C}\to {'S}$ is simply the original map $h:C\to S$.

We have the following diagram where both squares are Cartesian and the top square is \eqref{ShtC}
\begin{equation}\label{Sht C C'}
\xymatrix{\Sht_{C}\ar[d]\ar[r] &{{'C}=C}\ar[d]^{(\oll{c},\orr{c})}\\
M\ar[r]^{(\id,\Fr_{M})}\ar[d]^{\Fr_{M}} & M\times M\ar[d]^{(\Fr_{M},\id)}\\
M\ar[r]^{\Delta_{M}} & M\times M}
\end{equation}
Therefore the outer square is also Cartesian, i.e., there is a canonical isomorphism $\Fix({'C})=\Sht_{C}$.

For $\zeta\in \Ch_{n}(C)_{\QQ}=\Ch_{n}({'C})_{\QQ}$, we may also view it as a cohomological self-correspondence of $\const{M}$ supported on $'C$. We denote it by ${'\zeta}\in \Ch_{n}({'C})_{\QQ}$ to emphasize that it is supported on $C'$. We claim that
\begin{equation*}
(\id,\Fr_{M})^{!}\zeta=\Delta_{M}^{!}({'\zeta})\in\Ch_{0}(\Sht_{C})_{\QQ}.
\end{equation*}
In fact, this is a very special case of the Excess Intersection Formula \cite[Theorem 6.3]{Fulton} applied to the diagram \eqref{Sht C C'} where both $(\id,\Fr_{M})$ and $\Delta_{M}$ are regular immersions of the same codimension. In particular, taking the degree of the $s$ components, we have
\begin{equation}\label{tilt Frob}
\jiao{{'\zeta}, \Delta_{*}[M]}_{s}=\jiao{\zeta, \Gamma(\Fr_{M})}_{s} \quad\text{ for all $s\in S(k)$.}
\end{equation}

By \cite[Prop. 1.2.5]{Var} applied to the proper map \eqref{MS} between correspondences, we get a commutative diagram
\begin{equation*}
\xymatrix{\Hom(\oll{'c}^{*}\const{M}, \orr{'c}^{!}\const{M})\ar[r]^{\tau_{'C}}\ar[d]^{'h_{!}(-)} & \hBM{0}{\Fix({'C})\otimes_{k}\kbar}\ar[d]\ar@{=}[r] & \bigoplus_{s\in S(k)}\hBM{0}{\Sht_{C}(s)\otimes_{k}\kbar}\ar[d]^{\deg}\\
\Hom(\Fr^{*}_{S}f_{!}\const{M}, f_{!}\const{M})\ar[r]^{\tau_{'S}}& \hBM{0}{S(k)\otimes_{k}\kbar}\ar@{=}[r] & \bigoplus_{s\in S(k)}\Ql}
\end{equation*}
Combining the with the commutative diagram in Lemma \ref{l:cl Gysin} applied to $'C$, we get a commutative diagram
\begin{equation}\label{CCDD}
\xymatrix{\Ch_{n}({'C})_{\QQ}\ar[r]^{\Delta_{M}^{!}}\ar[d]^{'h_{!}\circ\cl_{'C}} & \Ch_{0}(\Fix({'C}))_{\QQ}\ar[d]\ar@{=}[r] & \bigoplus_{s\in S(k)}\Ch_{0}(\Sht_{C}(s))\ar[d]^{\deg}\\
\Hom(\Fr^{*}_{S}f_{!}\const{M}, f_{!}\const{M})\ar[r]^(.6){\tau_{'S}}& \hBM{0}{S(k)\otimes_{k}\kbar}\ar@{=}[r] & \bigoplus_{s\in S(k)}\Ql}
\end{equation}

Applying \eqref{CCDD} to $'\zeta$, and using \eqref{tilt Frob}, we get that for all $s\in S(k)$
\begin{equation}\label{tau inter}
\tau_{'S}\big({'h}_{!}\cl_{'C}({'\zeta})\big)_{s}=\jiao{{'\zeta}, \Delta_{*}[M]}_{s}=\jiao{\zeta,\Gamma(\Fr_{M})}_{s}.
\end{equation}
Here $\tau_{'S}(-)_{s}\in\Ql$ denotes the $s$-component of the class $\tau_{'S}(-)\in \hBM{0}{S(k)\otimes_{k}\kbar}=\oplus_{s\in S(k)}\Ql$.

Next we would like to express $\tau_{'S}\big({'h}_{!}\cl_{'C}({'\zeta})\big)_{s}$ as a trace. The argument works more generally when $\const{M}$ is replaced with any $\calF\in D^{b}_{c}(M)$ and $\cl_{C}(\zeta)$ replaced with any cohomological self-correspondence $\eta:\oll{c}^{*}\calF\to\orr{c}^{!}\calF$ supported on $C$. So we will work in this generality. For any $\calF\in D^{b}_{c}(M)$ we have a canonical isomorphism $\Phi_{\calF}:\Fr_{M}^{*}\calF\isom \calF$ whose restriction to the geometric stalk at $x\in M(k)$ is given by the {\em geometric} Frobenius $\Frob_{x}$ acting on $\calF_{\ov{x}}$. Similar remark applies to complexes on $S$. Using $\eta$ we define a cohomological self-correspondence $'\eta$ of $\calF$ supported on $'C$ as the composition
\begin{equation*}
'\eta: \oll{'c}^{*}\calF=\oll{c}^{*}\Fr_{M}^{*}\calF\xrightarrow{\oll{c}^{*}\Phi_{\calF}}\oll{c}^{*}\calF\xrightarrow{\eta}\orr{c}^{!}\calF=\orr{'c}^{!}\calF.
\end{equation*}
On the other hand we have a commutative diagram
\begin{equation}\label{Fr comm}
\xymatrix{\Fr_{S}^{*}f_{!}\calF\ar[r]^{adj.}_{\sim}\ar[dr]_{\Phi_{f_{!}\calF}}^{\sim} & f_{!}\Fr^{*}_{M}\calF\ar[d]^{f_{!}\Phi_{\calF}}_{\wr}\ar[r]^{adj.} & h_{!}\oll{c}^{*}\Fr^{*}_{M}\calF\ar@{=}[r]\ar[d]^{h_{!}\oll{c}^{*}\Phi_{\calF}}_{\wr} & h_{!}\oll{'c}^{*}\calF\ar[r]\ar[d]^{h_{!}({'\eta})} & f_{!}\calF\ar@{=}[d]\\
& f_{!}\calF\ar[r]^{adj.} & h_{!}\oll{c}^{*}\calF\ar[r]^{h_{!}(\eta)} & h_{!}\orr{c}^{!}\calF\ar[r]^{adj.} & f_{!}\calF}
\end{equation}
Here the arrows indexed by ``adj.'' are induced from adjunctions, using the properness of $\oll{c}$. The middle square is commutative by the definition of $'\eta$, and the right square is commutative by design. The composition of the top row in \eqref{Fr comm} is by definition the push-forward ${'h}_{!}{'\eta}$ as a cohomological self-correspondence of $f_{!}\calF$ supported on $'S$; the composition of the bottom row in \eqref{Fr comm} is by definition the push-forward ${h}_{!}{\eta}$ as a cohomological self-correspondence of $f_{!}\calF$ supported on the diagonal $S$. Therefore, \eqref{Fr comm} shows that ${'h}_{!}{'\eta}$ may be written as the composition
\begin{equation}\label{zeta Phi zeta}
{'h}_{!}{'\eta}:\Fr_{S}^{*}f_{!}\calF\xrightarrow{\Phi_{f_{!}\calF}} f_{!}\calF\xrightarrow{h_{!}\eta}f_{!}\calF.
\end{equation}

For any cohomological self-correspondence $\xi$ of $\calG\in D^{b}_{c}(S)$ supported on the graph of Frobenius $'S$, i.e., $\xi: \Fr_{S}^{*}\calG\to \calG$, the trace $\tau_{'S}(\xi)_{s}$ at $s\in S(k)$ is simply given by the trace of $\xi_{s}$ acting on the geometric stalk $\calG_{\ov{s}}$: this is because the Frobenius map is contracting at its fixed points, so the local term for the correspondence supported on its graph is the naive local term (a very special case of the main result in \cite[Theorem 2.1.3]{Var}). Applying this observation to $\xi={'h}_{!}{'\eta}$ we get
\begin{eqnarray}\notag
\tau_{{'S}}({'h}_{!}{'\eta})_{s}&=&\Tr\big(({'h}_{!}{'\eta})_{s},(f_{!}\calF)_{\ov{s}}\big)\\
\label{naive local F}
&=&\Tr\big((h_{!}\eta)_{s}\circ\Frob_{s},(f_{!}\calF)_{\ov{s}}\big) \quad\text{ by \eqref{zeta Phi zeta}}.
\end{eqnarray}
Now apply \eqref{naive local F} to $\calF=\const{M}$, $\eta=\cl_{C}(\zeta)$ and note that $'\eta=\cl_{'C}({'\zeta})$. Then \eqref{naive local F} gives
\begin{eqnarray}\label{naive local}
\tau_{{'S}}({'h}_{!}\cl_{'C}({'\zeta}))_{s}=\Tr\big((f_{!}\cl_{C}(\zeta))_{s}\circ\Frob_{s},(f_{!}\const{M})_{\ov{s}}\big).
\end{eqnarray}
Combining \eqref{naive local} with \eqref{tau inter} we get the desired formula \eqref{Lef fix}.
\end{proof}


\section{Super-positivity of $L$-values}
\label{s positivity}
 In this appendix we show the positivity of all derivatives of certain $L$-functions (suitably corrected by their epsilon factors), assuming the Riemann hypothesis. The result is unconditional in the function field case since the Riemann hypothesis is known to hold.

It is well-known that the positivity of the leading coefficient of such $L$-function is implied by the Riemann hypothesis. We have not seen the positivity of non-leading terms in the literature and we provisionally call such phenomenon ``super-positivity".

\subsection{The product expansion of an entire function}
We recall the (canonical) product expansion of an entire function following \cite[\S5.2.3, \S5.3.2]{A}. Let $\phi(s)$ be an entire function in the variable $s\in\BC$. Let $m$ be the  vanishing order of $\phi$ at $s=0$. List all the nonzero roots of $\phi$ as $\alpha_1,\alpha_2,...,\alpha_i,...$ (multiple roots being repeated)  indexed by a subset $I$ of $\BZ_{>0}$, such that $|\alpha_1|\leq |\alpha_2|\leq...$.  Let $E_n$ be the elementary Weierstrass function 
$$
E_n(u)=\begin{cases}\left(1-u \right),& n=0;\\
\left(1-u \right)\,e^{ u+\frac{1}{2}u^2+\cdots+\frac{1}{n}u^n},&n\geq 1.
\end{cases}
$$An entire function $\phi$ is said to have finite genus if it can be written as an absolutely convergent product 
\begin{align}\label{can prod}
\phi(s)=s^m \,e^{h(s)}\,\prod_{i\in\BZ} E_n\left(\frac{s}{\alpha_i}\right)
\end{align}
for a polynomial $h(s)\in\BC[s]$ and an integer $n\geq 0$. The product \eqref{can prod} is unique if we further demand that $n$ is the smallest possible integer, which is characterized as the smallest $n\in\BZ_{\geq 0}$ such that 
\begin{align}
\label{eqn: convergent}
\sum_{i\in I}\frac{1}{\big|\alpha_i\big|^{n+1}}<\infty.
\end{align}
The genus $g(\phi)$ of such $\phi$ is then defined to be $$
g(\phi):=\max\{\deg(h),\,n\}.$$
The order $\rho(\phi)$ of an entire function $\phi$ is defined as the smallest real number $\rho\in[ 0,\infty]$ with the following property: for every $\epsilon>0$, there is a constant $C_\epsilon$ such that 
$$\big|\phi(s)\big|\leq e^{|s|^{\rho+\epsilon}},\quad \mbox{when }\, |s|\geq C_\epsilon.$$
If $\phi$ is a non-constant entire function, an equivalent definition  is 
$$
\rho(\phi)=\limsup_{r\to \infty}\frac{\log\log|\!|\phi|\!|_{\infty,B_r}}{r}
$$
where $|\!|\phi|\!|_{\infty,B_r}$ is the supremum norm of the function $\phi$ on the disc $B_r$ of radius $r$. If the order of $\phi$ is finite, then Hadamard theorem \cite[\S5.3.2]{A} asserts that the function $\phi$ has finite genus and
\begin{align}
\label{eqn: g rho}
g(\phi)\leq \rho(\phi)\leq g(\phi)+1.
\end{align}
In particular, an entire function of finite order admits a product expansion of the form \eqref{can prod}.

\begin{prop}
\label{prop hadamard}
Let $\phi(s)$ be an entire function with the following properties 
\begin{enumerate}
\item It has a functional equation $\phi(-s)=\pm \phi(s)$.
\item For $s\in\BR$ such that $s\gg 0$, we have $\phi(s)\in\BR_{>0}$.
\item The order $\rho(\phi)$ of $\phi(s)$ is at most $1$.
\item {\em (RH)} The only zeros of $\phi(s)$ lie on the imaginary axis $\Re(s)=0$.
\end{enumerate}

Then we have  for all $r\geq 0$,
$$
\phi^{(r)}(0):=\frac{d}{ds}\Big|_{s=0}\phi(s) \geq 0.
$$
Moreover, if $\phi(s)$ is not a constant function, we have
$$
\phi^{(r_0)}(0)\neq0\imp \phi^{(r_0+2i)}(0)\neq0, \quad\mbox{for all }\, r_0 \mbox{ and } i\in\BZ_{\geq 0}.
$$
\end{prop}

\begin{proof}By the functional equation, if $\alpha$ is a root of $\phi$, so is $-\alpha$ with the same multiplicity. 
Therefore we may list all nonzero roots as $\{\alpha_i\}_{i\in \BZ\setminus \{0\}}$  such that
$$
\alpha_{-i}=-\alpha_i,
\quad \mbox{and}\quad |\alpha_1| \leq |\alpha_2|\leq ....
$$
If $\phi$ has only finitely many roots the sequence terminates at a finite number.

Since the order $\rho(\phi)\leq 1$, by \eqref{eqn: g rho} we have $g(\phi)\leq 1$. Hence we may write $\phi$ as a product 
\begin{align*}
\phi(s)=s^m \,e^{h(s)}\,\prod_{i=1}^\infty E_1\left(\frac{s}{\alpha_i}\right)E_1\left(-\frac{s}{\alpha_i}\right),
\end{align*}
where $m$ is the vanishing order at $s=0$.  Note that it is possible that $g(\phi)=0$, in which case one still has a product expansion using $E_1$ by the convergence of \eqref{eqn: convergent}.

By the functional equation, we conclude that $h(s)=h$ is a constant. 

By the condition (4)(RH), all roots $\alpha_i$ are purely imaginary, and hence $\ov\alpha_i=\alpha_{-i}$. We have 
\begin{align*}
\phi(s)&=s^m \,e^{h}\,\prod_{i=1}^\infty E_1\left(\frac{s}{\alpha_i}\right)E_1\left(\frac{s}{\ov\alpha_i}\right)\\
&=s^m \,e^{h}\,\prod_{i=1}^\infty   \left(1+\frac{s^2}{\alpha_i\ov \alpha_i} \right).
\end{align*}
By the condition (2), the leading coefficient $e^{h}$ is a positive real number.  
Then the desired assertion follows from the product above.
\end{proof}

\subsection{Super-positivity}
Let $F$ be  a global field (i.e., a number field, or the function field of a connected smooth projective curve over a finite field $\BF_q$). Let $\BA$ be the ring of ad\`eles of $F$.
Let $\pi$ be an irreducible cuspidal automorphic representation of $\GL_n(\BA)$.  Let $L(\pi,s)$ be the complete (standard) $L$-function associated to $\pi$ \cite{GJ}. We have a functional equation
$$L(\pi,s)= \epsilon(\pi,s) L(\wt\pi,1-s),$$where $\wt\pi$ denotes the contragredient of $\pi$, and
$$\epsilon(\pi,s)=\epsilon(\pi,1/2)N_\pi^{s-1/2} $$ for some positive real number $N_\pi$.
Define $$\Lambda(\pi,s)=N_\pi^{-\frac{(s-1/2)}{2}} L(\pi,s), $$ 
and $$\Lambda^{(r)}(\pi,1/2):=\frac{d}{ds}\Big|_{s=1/2}\Lambda(\pi,s).$$

\begin{thm}\label{th:pos}
Let $\pi$ be a nontrivial cuspidal automorphic representation of $\GL_n(\BA)$. Assume that it is self-dual:
$$\pi\simeq\wt\pi.$$ 
Assume that, if $F$ is a number field, the Riemann hypothesis holds for $L(\pi,s)$, that is, all the roots of $L(\pi,s)$ have real parts equal to $1/2$.
\begin{enumerate}
\item  For all $r\in\BZ_{\geq 0}$, we have
$$
\Lambda^{(r)}(\pi,1/2)\geq 0.
$$
\item If $\Lambda(\pi,s)$ is not a constant function, we have 
$$\Lambda^{(r_0)}(\pi,1/2)\neq0\imp \Lambda^{(r_0+2i)}(\pi,1/2)\neq0, \quad\mbox{for all }\, i\in\BZ_{\geq 0}.
$$
\end{enumerate}
\end{thm}
\begin{proof} We consider $$\lambda(\pi,s):=\Lambda(\pi,s+1/2).$$
Since $\pi$ is cuspidal and nontrivial, its standard $L$-function $L(\pi,s)$  is entire in $s\in \BC$. By the equality  
$
\epsilon(\pi,s)\epsilon(\wt\pi,1-s)=1
$
and the self-duality $\pi\simeq\wt\pi$ we deduce
$$1=\epsilon(\pi,1/2)\epsilon(\pi,1-1/2)=\epsilon(\pi,1/2)^2.
$$
Hence $\epsilon(\pi,1/2)=\pm 1$, and we have a functional equation   
\begin{align}\label{eqn FE}
\lambda(\pi,s)=   \pm \lambda(\pi,-s).
\end{align}
We apply Proposition \ref{prop hadamard} to the entire function $\lambda(\pi,s)$.
The function $L(\pi,s)$ is entire of order one, and so is  $\lambda(\pi,s)$.  In the function field case, the condition (4)(RH) is known by the theorem of Deligne  on Weil conjecture, and of Drinfeld and L. Lafforgue on the global Langlands correspondence. It remains to verify the condition (2) for $\lambda(\pi,s)$. This follows from the following lemma.
\end{proof}

The local $L$-factor $L(\pi_v,s)$ is of the form $\frac{1}{P_{\pi_v}(q_v^{-s})}$ where $P_{\pi_v}$ is a polynomial with constant term equal to one when $v$ is non-archimedean, and a product of functions of the form $\Gamma_{\BC}(s+\alpha)$,  or  $\Gamma_{\BR}(s+\alpha)$, where $\alpha\in \BC$, and
$$\Gamma_\BC(s)=2(2\pi)^{-s}\Gamma(s),\quad \Gamma_\BR(s)=\pi^{-s/2}\Gamma(s/2),$$
when $v$ is archimedean. We say that $L(\pi_v,s)$ has real coefficients if the polynomial $P_{\pi_v}$ has real coefficients when $v$ is non-archimedean, and the factor $\Gamma_{F_v}(s+\alpha)$  in $L(\pi_v,s)$ has real $\alpha$ or the pair $\Gamma_{F_v}(s+\alpha)$ and $\Gamma_{F_v}(s+\ov\alpha)$ show up simultaneously when $v$ is archimedean.  In particular, if $L(\pi_v,s)$ has real coefficients, it takes positive real values when $s$ is real and sufficiently large.

\begin{lem}Let $\pi_v$ be unitary and self-dual. Then $L(\pi_v,s)$ has real coefficients. 
\end{lem}

\begin{proof} 
We suppress the index $v$ in the notation and write $F$ for a local field. Let $\pi$ be irreducible admissible representation of $\GL_n(F)$. It suffices to show that, if $\pi$ is unitary, then we have 
\begin{align}
\label{eqn L =ov L}
\ov{L(\pi,\ov s)}=
L(\wt \pi, s).
\end{align}
Let $\sZ_\pi$ be  the space of local zeta integrals, i.e., the meromorphic continuation of
$$
Z(\Phi,s,f)=\int_{\GL_n(F)}f(g)\Phi(g)|g|^{s+\frac{n-1}{2}}\,dg
$$
where $f$ runs over all matrix coefficients of $\pi$, and $\Phi$ runs over all Bruhat--Schwartz functions  on $\Mat_n(F)$ (a certain subspace, stable under complex conjugation,  if $F$ is archimedean, cf. \cite[\S8]{GJ}). 
We recall that from \cite[Theorem 3.3, 8.7]{GJ} that the Euler factor $L(\pi,s)$ is uniquely determined by the space $\sZ_\pi$ (for instance, it is a certain normalized generator of the $\BC[q^s, q^{-s}]$-module $\sZ_\pi$ if $F$ is non-archimedean).

Let $\sC_\pi$ be the space of matrix coefficients of $\pi$, i.e., the space consisting of all linear combinations of functions on $\GL_n(F)$: $g\mapsto (\pi(g)u,v)$ where $u\in\pi, v\in\wt \pi$ and $(\cdot,\cdot): \pi\times \wt\pi\to\BC$ is the canonical bilinear pairing. We remark that the involution $g\mapsto g^{-1}$ induces an isomorphism between $\sC_\pi$ with $\sC_{\wt\pi}$. 

To show \eqref{eqn L =ov L}, it now suffices to show that, if $\pi$ is unitary, the complex conjugation induces an isomorphism between $\sC_\pi$ and $\sC_{\wt \pi}$. Let $\pair{\cdot,\cdot}: \pi \times \pi\to \BC$ be a non-degenerate {\em Hermitian} pairing invariant under $\GL_n(F)$. Then the space $\sC_\pi$ consists of all functions $f_{u,v}:g\mapsto  \pair{\pi(g)u,v}, u,v\in \pi$. Under complex conjugation we have $\ov{f_{u,v}}(g)=\ov{\pair{\pi(g)u,v}}=\pair{v,\pi(g)u}=\pair{\pi(g^{-1})v,u}=f_{v,u}(g^{-1}).$ This function belongs to $\sC_{\wt\pi}$ by the remark at the end of the previous paragraph. This clearly shows that the complex conjugation induces the desired isomorphism.
\end{proof}

\begin{remark}\label{r:func field pos}
In the case of a function field, we have a simpler proof of Theorem \ref{th:pos}. The function $L(\pi,s)$ is a polynomial in $q^{-s}$ of degree denoted by $d$. Then the function $\lambda(\pi,s)$ is of the form
\begin{align}\label{eqn prod}
\lambda(\pi,s)=q^{ds/2}\prod_{i=1}^{d}\big(1-\alpha_i q^{-s}\big),
\end{align}
where all the roots $\alpha_i$ satisfy $\lvert\alpha_i\rvert=1.$ By the functional equation \eqref{eqn FE}, if $\alpha$ is a root in \eqref{eqn prod}, so is $\alpha^{-1}=\ov\alpha$. We divide all roots not equal to $\pm 1$ into pairs $\alpha_1^{\pm 1},\alpha_2^{\pm 1},...,\alpha_m^{\pm 1}$ (some of them may repeat). Consider
\begin{align*}
A_i(s)&=q^s\big( 1-\alpha_i q^{-s}\big)\big(1-\alpha_i^{-1} q^{-s}\big)
\\&=q^{s}+q^{-s}-\alpha_i-\ov\alpha_i
\\&=\big(2-\alpha_i-\ov\alpha_i\big)+2\sum_{j\geq 1} \frac{(s\log q )^{2j}}{j!}.
\end{align*}
From $\lvert\alpha_i\rvert=1$ and $\alpha_i\neq 1$ it follows that $A_i(s)$ has strictly positive coefficients at all even degrees.
 Now let $a$ (resp., $b$) be the multiplicity of the root $1$ (resp., $-1$). We then have
$$
\lambda(\pi,s)=\big( q^{s/2}-q^{-s/2}\big)^{a}\,\big(q^{s/2}+q^{-s/2}\big)^{b}\,\prod_{i=1}^mA_i(s),\quad 2m+a+b=d.
$$
The desired assertions follow immediately from this product expansion.

\end{remark}
\begin{remark}
In the statement of the theorem, we excludes the trivial representation. In this case the complete $L$-function has a pole at $s=1$. If we replace $\Lambda(\pi,s)$ by $s(s-1)\Lambda(\pi,s)$, the theorem still holds by the same proof. Moreover, if $F=\BQ$, we have the Riemann zeta function,  and the super-positivity is known without assuming the Riemann hypothesis, by P\'olya \cite{CNV}.  The super-positivity also holds when the $L$-function is ``positive definite" as defined by Sarnak in \cite{Sa}. One of such examples is the weight 12 cusp form with $q$-expansion $\Delta=q\prod_{n\geq 1}(1-q^n)^{24}$. More recently, Goldfeld and Huang in \cite{GH} prove that there are infinitely many classical holomorphic cusp forms (Hecke eigenforms) on $\SL_2(\BZ)$ whose $L$-functions satisfy super-positivity.
\end{remark}

\begin{remark}
The positivity of the central value is known for the standard $L$-function attached to a symplectic cuspidal representation of $\GL_n(\BA)$ by \cite{LR}.
\end{remark}

\begin{remark}
The positivity of the first derivative is known for the $L$-function appearing in the Gross--Zagier formula in \cite{GZ}, \cite{YZZ}, for example the $L$-function of an elliptic curve over $\BQ$.
\end{remark}


\end{document}